\documentclass[11pt]{article}
\usepackage{geometry}
 \usepackage{graphicx}
 \usepackage{subcaption}
 \usepackage{titling,authblk}
 \usepackage{enumerate,float}
 \usepackage{enumitem}
 \usepackage{algorithm}
\usepackage{algpseudocode}
\algnewcommand\algorithmicinput{\textbf{INPUT:}}
\algnewcommand\INPUT{\item[\algorithmicinput]}
\algnewcommand\algorithmicoutput{\textbf{OUTPUT:}}
\algnewcommand\OUTPUT{\item[\algorithmicoutput]}
 \usepackage{amsmath}
\usepackage{amssymb}
\usepackage{bbm}
\usepackage{tikz}
\usetikzlibrary{positioning, arrows.meta}
\usepackage{amsthm, amsfonts}
\usepackage[dvipsnames]{xcolor}

\usepackage[round]{natbib}
\usepackage{appendix}
\usepackage[normalem]{ulem}

\usepackage{tocloft}
\usepackage{etoc}

\usepackage[colorlinks=true,
    linkcolor=blue,
    filecolor=magenta,      
    urlcolor=cyan,
      citecolor=blue]{hyperref}
\usepackage{cleveref}

\usepackage[english]{babel}
\newtheorem{theorem}{Theorem}
\newtheorem{lemma}[theorem]{Lemma}
\newtheorem{proposition}[theorem]{Proposition}
\newtheorem{corollary}[theorem]{Corollary}
\newtheorem{definition}{Definition}
\newtheorem{remark}{Remark}
\newtheorem{assumption}{Assumption}

\allowdisplaybreaks

\title{Online change point detection under heavy-tailedness and contamination}
\author[1]{Edwin Y.~N.~Tang}
\author[1]{Yudong Chen}
\author[2]{Mengchu Li}
\author[1]{Yi Yu}

\affil[1]{Department of Statistics, University of Warwick}
\affil[2]{School of Mathematics, University of Birmingham}
\date{\today}

\newcommand{\norm}[1]{\left \lVert #1 \right \rVert}
\DeclareMathOperator{\E}{\mathbb{E}}
\DeclareMathOperator{\Var}{\mathrm{Var}}
\DeclareMathOperator{\R}{\mathbb{R}}
\DeclareMathOperator{\Pb}{\mathbb{P}}

\newcommand{\med}{\mathrm{median}}
\DeclareMathOperator{\esssup}{ess\,sup}
\newcommand{\RUME}{\mathrm{RUME}}
\newcommand{\Gram}{\mathrm{Gram}}

\DeclareMathOperator{\Sum}{\mathrm{Sum}}
\newcommand{\op}{\mathrm{op}}
\newcommand{\iprod}[1]{\langle #1 \rangle}
\newcommand{\supp}[1]{\mathrm{supp}(#1)}

\begin{document}

\maketitle

\begin{abstract}
    We study an online version of the robust mean change point detection problem under a dynamic Huber contamination model with arbitrary contamination distribution and inlier distribution possessing exponentially- or polynomially-decaying tails. This robustness framework is systematically studied for the first time in the change point literature. For univariate data, we characterise the detection delay by partitioning the parameter space into four regimes, in terms of the true change location, signal size and contamination level.  Efficient detection procedures are accompanied by matching lower bounds, up to poly-logarithmic factors. For the multivariate setting, we devise an efficient robust mean testing procedure and apply this to the robust online change point problem.  
    The theoretical analysis of the robust mean testing procedure is the first in dealing with both Huber contamination and heavy-tailedness, and is thus of independent interest. 
    Extensive numerical experiments are conducted to support our theoretical findings.
\end{abstract}

\addtocontents{toc}{\protect\setcounter{tocdepth}{0}}

\section{Introduction}
Change point analysis has a long history \citep[e.g.][]{Page_1954, Lorden_1971} and has experienced a marked renaissance in recent years, particularly due to its applications in fields such as finance, cybersecurity, and manufacturing. Modern technology enables the real-time monitoring of evolving processes, including stock market fluctuations \citep[e.g.][]{Chen_1997}, internet traffic \citep[e.g.][]{Peng_2004} and industrial quality control \citep[e.g.][]{wadinger_2024}. For these streaming data, the ability to detect change points quickly and respond promptly is more desirable than conducting retrospective analyses on the entire dataset \citep{Chen_2022}. However, the presence of outliers or heavy-tailedness can complicate this online detection task. Traditional methods, such as cumulative sum techniques \citep[e.g.][]{WangSamworth2018, Wang_Yu_Rinaldo_2020, Yu_2023}, that rely on Gaussianity or sub-Gaussianity assumptions often fail to perform effectively in these scenarios. 

To address these challenges, we leverage ideas from robust statistics to design robust online change point algorithms.  
The rest of the paper is organised as follows. \Cref{subsect:problemsetup} introduces the problem setup, followed by a summary of our main results in \Cref{subsect:resultsummary}. A review of related work is presented in \Cref{subsect:litreview}. In \Cref{sect:univariate}, we investigate the univariate problem, providing lower bounds on detection delay and presenting an optimal procedure with its theoretical guarantees. We then devise and analyse an efficient robust multivariate mean testing algorithm in \Cref{sect:hdtest}, and apply this to the multivariate change point problem in \Cref{sect:high_dim_cpt}. Simulation studies that verify our theoretical results are presented in \Cref{sect:sim_study}. All the proofs are deferred to the Appendices.   

\subsection{Problem setup}\label{subsect:problemsetup}

We consider the online mean change point detection problem under both Huber contamination and heavy-tailed noise, for both univariate and multivariate data sequences. The model setup is detailed below. 

\subsubsection{Contamination and heavy-tailedness}

We consider a variant of the classical Huber contamination model \citep{Huber1965} tailored for streaming data, namely the dynamic Huber $\varepsilon$-contamination model defined in \Cref{modelasmmp}. This model has been previously studied in the context of offline robust change point detection in \cite{Li_2021}.  

\begin{definition}[Dynamic Huber $\varepsilon$-contamination model with inlier class $\mathcal{D}$]\label{modelasmmp}
    Let $p\in\mathbb{N}$, $\varepsilon \in [0, 1/2)$ and $\mathcal{D}$ be a class of distributions on $\mathbb{R}^p$. Let $\{Q_i\}_{i \in \mathbb{N}}$ be a sequence of distributions on $\mathbb{R}^p$. We say that $\{Q_i\}_{i \in \mathbb{N}}$ is a dynamic Huber $\varepsilon$-contamination model with inlier class $\mathcal{D}$, if 
    \begin{equation}\label{huber}
       Q_i= (1-\varepsilon_i)F_i+\varepsilon_i H_i, \quad i \in \mathbb{N},
    \end{equation}
    where $F_i \in \mathcal{D}$ is an inlier distribution, $H_i$ is an arbitrary distribution and the proportion of contamination satisfies $\varepsilon_i\leq\varepsilon$. 
\end{definition}

We consider inlier distributions with either exponentially- or polynomially-decaying tails.  For univariate random variables, they are specified in Definitions \ref{def-galphak} and \ref{def-palphak}, respectively.

\begin{definition}[$\mathcal{G}_{\theta, M}$ class of distributions] \label{def-galphak}
For $M >0$ and $0 < \theta \leq 2$, let $\mathcal{G}_{\theta, M}$ denote the class of distributions on $\mathbb{R}$ satisfying that 
\[
    \mathbb{E}_{W \sim P \in \mathcal{G}_{\theta, M}} \Big[\exp\bigl\{(|W-\E W|/M)^\theta\bigr\}\Big] \leq 2.
\]
\end{definition}

This $\mathcal{G}_{\theta, M}$ class consists of sub-Weibull distributions of order $\theta$ with the Orlicz $\psi_\theta$-norm upper bounded by $M$, satisfying 
$\mathbb{P}(|W-\E W| \geq x) \leq 2e^{-(x/M)^\theta}$ for any $x > 0$ and $W \sim F \in \mathcal{G}_{\theta,M}$. See \Cref{subsect:subw_prop} for details.
    
\begin{definition}[$\mathcal{P}_{v, \phi}$ class of distributions]\label{def-palphak}
    For $0<\phi<\infty$ and $v \geq 2$, let $\mathcal{P}_{v,\phi}$  denote the class of distributions on $\mathbb{R}$ satisfying that
\[
    \mathbb{E}_{W \sim P \in \mathcal{P}_{v, \phi}}\big[(|W-\E W|/\phi)^v\big] \leq 1.
\]
\end{definition}
In other words, each distribution within the class $\mathcal{P}_{v,\phi}$ has its absolute $v$-th central moment bounded above by $\phi^v$ and possesses a polynomially-decaying tail, as $\mathbb{P}(|W-\E W| \geq x) \leq (\phi/x)^{v}$ for any $x > 0$ and $W \sim F \in \mathcal{P}_{v, \phi}$, using Markov's inequality. This is typically much heavier than an exponentially-decaying tail and thus more challenging statistically.  

Extending to multivariate settings, let $\mathcal{G}_{\theta,M}^p$ and $\mathcal{P}_{v,\phi}^p$ denote the classes of distributions on $\mathbb{R}^p$ whose coordinates are independent, each with a marginal distribution belonging to $\mathcal{G}_{\theta,M}$ and $\mathcal{P}_{v,\phi}$, respectively.

\subsubsection{Online change point detection}

In view of the dynamic Huber contamination model defined in \Cref{modelasmmp}, the change point setup is based on the inlier distribution sequence $\{F_i\}_{i \in \mathbb{N}}$, formalised in \Cref{assum:kappa}.

\begin{assumption}\label{assum:kappa}
Assume that $\{X_i\}_{i\in\mathbb{N}}$ is a sequence of independent random variables drawn from a dynamic Huber $\varepsilon$-contamination model with inlier class $\mathcal{D}$, as defined in \Cref{modelasmmp}. For $i \in \mathbb{N}$, let $f_i = \mathbb{E}_{Y \sim F_i}[Y]$ be the mean of the inlier distribution $F_i$.  Assume that there exists $\Delta\in\mathbb{N} \cup\{\infty\}$ such that
    \begin{equation*}
        f_1=\dots=f_\Delta\neq f_{\Delta+1}=f_{\Delta+2}=\cdots.
    \end{equation*}
When $\Delta \neq \infty$, let $\kappa=\norm{f_\Delta-f_{\Delta+1}}_2$ be the jump size.
\end{assumption}

We write the probability of any event induced by any distribution satisfying \Cref{assum:kappa} as~$\mathbb{P}_{\Delta}(\cdot)$. 

An online change point procedure is characterised by an extended stopping time $\hat{t}\in\mathbb{N}\cup\{\infty\}$, with respect to the natural filtration generated by the data. The false alarm probability of an online change point procedure is $\mathbb{P}_{\Delta}(\hat{t}<\Delta)$, $\Delta \in \mathbb{N} \cup \{\infty\}$.  
For $\Delta < \infty$, define the detection delay as $(\hat{t}-\Delta)^+ = \max\{\hat{t} - \Delta, \, 0\}$.
We aim to develop a procedure that simultaneously controls the false alarm probability at a prescribed level $\alpha > 0$, and guarantees that the detection delay $(\hat{t}-\Delta)^+$ is small for all $\Delta<\infty$. 

\subsection{Summary of the results} \label{subsect:resultsummary}

We study the robust online change point detection problem for both univariate and multivariate data streams.  In \Cref{sect:univariate}, for univariate sequences, we present a complete picture. 
In particular, we characterise the detection delay for the univariate ($p=1$) online change point problem in different regimes of $(\kappa, \varepsilon, \Delta)$, when the inlier distribution has exponentially- or polynomially-decaying tails, shown in Tables~\ref{table:delayG} and~\ref{table:delayP}, respectively.
\begin{table}[H]
\centering
\begin{tabular}{lll}
\hline\hline
Regime & Signal size  & $(\hat{t}-\Delta)^+$\\
\hline
1 & $\kappa \lesssim \varepsilon \log^{1/\theta}({1}/{\varepsilon})\vee\Delta^{-1/2}$ &     No consistent estimator     \\
2 & $\varepsilon \log^{1+1/\theta}({1}/{\varepsilon})\vee\Delta^{-1/2}\lesssim\kappa\lesssim 1$   &   $\kappa^{-2}\log(\Delta/\alpha)$    \\
3 & $1\lesssim\kappa\lesssim \log^{1/\theta}({1}/{\varepsilon})$ &   $\kappa^{-\theta}\log(\Delta/\alpha)$   \\
4 & $\kappa\gtrsim \log^{1/\theta}({1}/{\varepsilon})$ & ${\log(1/\varepsilon)}^{-1}\log(\Delta/\alpha)$ \\
\hline
\end{tabular}
\caption{Order of detection delay for inlier class $\mathcal{D}=\mathcal{G}_{\theta,M}$, as a function of signal size $\kappa$, true change location $\Delta$, contamination level $\varepsilon$ and false alarm probability $\alpha$.}
\label{table:delayG}
\end{table}

\begin{table}[H]
\centering
\begin{tabular}{lll}
\hline\hline
Regime & Signal size       & $(\hat{t}-\Delta)^+$\\
\hline
1&$\kappa\lesssim\varepsilon^{1-1/v}\vee\Delta^{-1/2}$ &     No consistent estimator     \\
2&$\varepsilon^{1-1/v}\vee\Delta^{-1/2}\lesssim\kappa\lesssim 1$   &   $\kappa^{-2}\log(\Delta/\alpha)$    \\
3&$1\lesssim\kappa\lesssim \varepsilon^{-1/v}$ &   $({v\log\kappa})^{-1}\log(\Delta/\alpha)$   \\
4& $\kappa\gtrsim \varepsilon^{-1/v}$ & ${\log(1/\varepsilon)}^{-1}\log(\Delta/\alpha)$ \\
\hline
\end{tabular}
\caption{Order of detection delay for inlier class $\mathcal{D}=\mathcal{P}_{v,\phi}$, as a function of signal size $\kappa$, true change location $\Delta$, contamination level $\varepsilon$ and false alarm probability $\alpha$.}
\label{table:delayP}
\end{table}

\Cref{sect:high_dim_cpt} collects the results for the substantially more challenging multivariate ($p>1$) case. At the core of this multivariate robust online change point detection problem is a multivariate robust mean testing problem, which is studied in \Cref{sect:hdtest} and is of independent interest.  We extend the methodology of \cite{Canonne2023} to handle both Huber-contaminated and heavy-tailed multivariate data, followed by an integration into the change point detection routine.

\subsection{Related literature}\label{subsect:litreview}
\subsubsection{Robust statistics}
Robust statistics aims to develop procedures that remain accurate when classical modelling assumptions, such as Gaussianity or the absence of contamination, fail. Classical work primarily focused on the robustness properties of estimators in asymptotic frameworks. Prominent examples include Huber's M-estimator \citep{Huber1964} and depth-based estimators such as Tukey's halfspace median \citep{Tukey1975Math}. While these estimators provide important conceptual foundations for robustness, some of them are computationally challenging in high dimensions.  For instance, computing the halfspace median requires exponential time in the dimension in general. 

More recent work, particularly from theoretical computer science, has shifted the focus towards algorithms with simultaneous statistical and computational guarantees. Pioneered by \cite{Diakonikolas_2016} and \cite{Lai_2016}, this line of work develops polynomial-time algorithms for robust mean estimation under strong contamination models, together with finite-sample error bounds.  This modern perspective replaces purely asymptotic robustness guarantees with non-asymptotic rates that explicitly capture the effects of sample size, dimension, contamination level and moment assumptions.

A central theme in this literature is the design of algorithms that can handle complex data.  Many robust mean estimators output a weighted average of the observed points where weights are chosen to downweight or remove suspected outliers. Several approaches have been developed for this purpose, including covariance-based filtering methods \citep[see][for a review]{Diakonikolas_Kane_2023}, semidefinite programming approaches \citep[e.g.][]{Cheng_Diakonikolas_Ge_2019}, and non-convex optimization methods \citep[e.g.][]{cheng2020}. Other robust procedures that do not rely on explicit weighting schemes include M-estimators \citep[e.g.][]{Catoni_2012}, median-of-means methods \citep[e.g.][]{Lugosi2019subG} and the trimmed mean \citep[e.g.][]{Lugosi_Mendelson_2021}.  

Modern work on robust mean estimation has led to a detailed study of different sources of distributional challenges: heavy-tailed inlier distributions and adversarial contamination. In the uncontaminated setting, one line of work seeks estimators with sub-Gaussian performance under weak moment assumptions \citep[e.g.][]{Lugosi_Mendelson_2019, Lugosi2019subG, hopkins_2021}. Under strong contamination, optimal estimation error rates have been established for several classes of inlier distributions, including Gaussian inliers with known covariance \citep[e.g.][]{Diakonikolas_2016}, sub-Gaussian inliers with unknown covariance, and distributions with bounded $k$-th moments for $k\geq 2$ \citep[e.g.][]{Steinhardt2018, Diakonikolas_Kane_Pensia_2020}. Importantly, these rates can be achieved by polynomial-time algorithms. Under the Huber contamination model, \cite[e.g.][]{pmlr-v108-prasad20a} developed a mean estimator that is optimal under both heavy-tailed and contaminated models. Recent work has also studied robust estimation under additional structural assumptions, such as sparsity of the mean \citep[e.g.][]{balakrishnan17} or mean-shift contamination models \citep[e.g.][]{Kotekal2025, Diakonikolas2025}.

While there is extensive literature on robust mean estimation, existing finite-sample theory for robust mean testing remains largely concentrated on Gaussian-type models. \cite{Diakonikolas2017} was the first to study sample complexity bounds for testing whether the mean of a Huber-contaminated Gaussian distribution is zero or not, assuming known covariance. Building on this work, \cite{Canonne2023} derived tight sample complexity bounds under strong and weak contamination models. Other recent work has investigated testing with unknown covariance \citep[e.g.][]{Canonne2021, Diakonikolas2023} and under sparsity assumptions \citep[e.g.][]{George2022}. Substantially less is known for contaminated non-Gaussian settings. For heavy-tailed distributions without contamination, practical testing procedures were studied in, e.g.~\cite{ChenQin2010} and \cite{WangPengLi2015}, but only from an asymptotic perspective.

Beyond mean estimation and testing, robustness has been studied in many other statistical problems, including covariance estimation \citep[e.g.][]{Chen_Gao_Ren_2018}, covariance testing \citep[e.g.][]{Diakonikolas2021} and linear regression \citep[see][for a review]{Yu14092017}. 

\subsubsection{Change point detection analysis}\label{subsect:cpt_lit}

Change point detection problems are commonly classified into online and offline settings. In the online setting, observations arrive sequentially, and one must decide in real time whether a change point has occurred. The goal is therefore to minimise the delay in detecting a change point. In contrast, offline change point analysis is conducted retrospectively, using a fixed dataset observed over the entire time horizon. The aim is then to accurately estimate the number and locations of the change points. 

In offline change point detection, three main statistical goals are commonly studied, namely testing, localisation and inference. Testing concerns the problem of distinguishing a no-change model from alternatives containing at least one change point. Minimax testing rates have been derived in several settings under independent-coordinate assumptions, including Gaussian models \citep[e.g.][]{Liu_Gao_Samworth_2021, Verzelen_2023}, and heavy-tailed models \citep[e.g.][]{li_2025}. Testing procedures have also been developed beyond the independent Gaussian setting, for example allowing dependence across coordinates or requiring weaker moment assumptions \citep[e.g.][]{YuChen2021, WangZhu2022}, though the type I error guarantees for these procedures are often asymptotic. Localisation aims to estimate the change point locations. As for the univariate mean change point setting, many methods have been proposed for localisation, including binary segmentation and its variants \citep[e.g.][]{Knott_1974,Fryzlewicz_2014,Kovacs2023}, as well as penalisation-based methods \citep[e.g.][]{Jackson2005, bleakley2011, Killick2012, rojas2014, Maidstone2017}. 
Finally, inference concerns uncertainty quantification after detecting changes, including constructing confidence intervals for change point locations \citep[e.g.][]{Kaul2021, Wangshao2023,  xu2024change, xue2026covariance}, as well as confidence sets for the underlying piecewise constant mean function \citep[e.g.][]{Frick2014, PeinSieling2017, FangLi2020}. 

Offline change point analysis has also generated a large body of methodological, computational, and theoretical work. Methodological developments address settings that go beyond the classical independent univariate model, including temporal dependence in the noise \citep[e.g.][]{Dehling2015, WangZhu2022, ChoFryzlewicz2024, xu2024change}, high-dimensional settings \citep[e.g.][]{Dette_Gosmann_2018, WangSamworth2018, Kaul2021} and missing data settings \citep{Follain2022}, to name but a few.  From a computational perspective, recursive search procedures such as binary segmentation and its variants provide scalable alternatives to exhaustive search, while pruning strategies for penalised segmentation methods improve the efficiency of exact or near-exact optimisation.  On the theoretical side, a central goal is to characterise the optimal localisation error and the minimal signal strength required for consistent localisation, often within a minimax framework. For instance, \cite{Wang_Yu_Rinaldo_2020} established minimax localisation rates in the univariate setting and showed that optimal partitioning and wild binary segmentation attain these rates under suitable conditions. Multivariate extensions have been considered in \cite{Verzelen_2023} and \cite{Pilliat_2023}.

Online mean change point detection has analogous goals, but with additional constraints arising from the sequential nature of the data. The first goal is sequential detection, where the aim is to declare a change as soon as possible while controlling the frequency of false alarms. Classical sequential analysis often formulates this through the average run length in the absence of change \citep{Lorden_1971}. This criterion underlies much of the early work for CUSUM and generalised likelihood-ratio (GLR) procedures \citep[e.g.][]{Page_1954, Lai1995, Mei_2006, Xie_Siegmund_2013}. Among detectors with average run length control, \cite{moustakides1986} proved that the CUSUM procedure of \cite{Page_1954} is optimal under the detection-delay benchmark of \cite{Lorden_1971}. A second line of work controls the overall probability of ever raising an alarm over the monitoring horizon \citep{Chu1996}. In recent years, this framework has attracted increasing attention, where both GLR- and CUSUM-based procedures are studied under this paradigm \citep[e.g.][]{Dette2020, Chen_2022, Gosmann2022, Yu_2023}. Within the class of procedures satisfying false-alarm-probability control, \cite{Yu_2023} showed that, for univariate sub-Gaussian observations, a CUSUM procedure based on the sample mean achieves minimax optimal detection delay. The second goal is post-detection inference, where the objective is to quantify uncertainty about the location of the change after an alarm has been raised. For example, \cite{Chen2024} studied the construction of confidence intervals for the actual change point location in the online framework.

The online setting brings computational considerations to the forefront.  Since observations arrive sequentially, practical procedures must often have low update and storage costs.  Recent work has therefore developed more efficient methods for evaluating generalised likelihood ratios \citep[e.g.][]{Romano2023, Ward_2024, Pischagina2026}, as well as grid-based search strategies that reduce update and storage costs \citep{Chen_2022, moen_2025}.

Although our focus is on mean changes, change point analysis has been studied for many other forms of structural changes. These include changes in covariance structure \citep[e.g.][]{Aue2009, Avanesov_Buzun_2018, Wang_Yu_Rinaldo_2020, LiLi2023}, regression coefficients \citep[e.g.][]{HORVATH2004, Lee2016, Rinaldo2021} and beyond; see \cite{yu2020review} for a recent review.

\subsubsection{Robust change point detection analysis}

A growing body of work has incorporated robustness into mean change point algorithms, primarily in the offline setting. To reduce the effect of outliers or contaminated observations, one approach is to replace the squared-error loss in penalised cost formulations by the Huber loss \citep{Fearnhead2019}. Another approach is to incorporate robust mean estimators into CUSUM-based procedures \citep{Li_2021}. Robustness can also be achieved through nonparametric methods,
many of which are designed to detect general distributional shifts, and are therefore less sensitive to heavy-tailed observations. In high-dimensional settings, U-statistic-based methods have been used to detect change points in distributions with undefined means \citep{Yu2022}, as well as in data with heavy-tailed and mildly dependent errors \citep{Jiang2023, xu2024change}.

In the online setting, existing robust detectors \citep[see, e.g.][]{Moustakides1985, Unnikrishnan2011, Cao2017} primarily focus on sequential testing under parametric uncertainty. For example, in a Gaussian model with unknown mean and known variance, their null and alternative hypotheses specify that the mean belongs to one or two known convex sets. In contrast, our setting involves both unknown inlier means before and after the change, and observations drawn from Huber-contaminated versions of the corresponding inlier distributions. This distinction motivates the development of new methodology.

Beyond mean change point models, robustness has also been considered in stump models \citep{Mukherjee2022}, covariance change point problems \citep{ramsay2021} and high-dimensional linear regression coefficient changes \citep[see, e.g.][]{liu2023changepoint, Cho_Owens_2024, xu2024change, zhao2024}.  These developments highlight the broad relevance of robustness in change point analysis, while also underscoring that online mean change point detection under Huber contamination remains comparatively less explored.

\subsection{Notation}
We introduce the notation used throughout the paper. Let $\mathbb{N}$ denote the set of positive integers. For $d \in \mathbb{N}$, write $[d]=\{1, \dotsc, d\}$. Let $\lceil \cdot \rceil$, $\lfloor \cdot \rfloor$ and $\Gamma(\cdot)$ denote the ceiling, floor and Gamma functions, respectively. 
For a set~$\mathcal{S}$, use $\mathbbm{1}_{\mathcal{S}}$ and $|\mathcal{S}|$ to denote its indicator function and cardinality respectively.  For a sequence of observations $\{X_t\}_{t \in \mathbb{N}}$, denote the sub-sequence as $X_{s:t} = (X_s, \dots, X_t)$.

For two sequences of positive real numbers $\{a_n\}$ and $\{b_n\}$, we write $a_n \lesssim b_n$ (or $b_n \gtrsim a_n$) if there exists a constant $C > 0$ such that $a_n \leq C b_n$ for all $n$ sufficiently large. We write $a_n \asymp b_n$ if both $a_n \lesssim b_n$ and $a_n \gtrsim b_n$ hold. In the proofs of results in Sections \ref{sect:hdtest} and \ref{sect:high_dim_cpt}, we write $a_n \lesssim_{\log} b_n$ if there exists a constant $k \geq 0$ such that $a_n \lesssim b_n \log^k(b_n)$. The notation $a_n \asymp_{\log} b_n$ is defined accordingly. Throughout the paper, $O(\cdot)$ denotes bounds up to universal multiplicative constants, while $\tilde{O}(\cdot)$ additionally suppresses logarithmic factors. Unless stated otherwise, these implicit constants, as well as those denoted by $C, C_1, C_2, \dots$, are understood to depend only on the parameters $\alpha, v, M, \theta$ and $\phi$, which are treated as fixed.

\section{Univariate change point detection: Minimax rates and optimal procedures} \label{sect:univariate} 

In this section, we present a comprehensive analysis of the univariate case.  We first establish lower bounds on the detection delay for all regimes in Tables~\ref{table:delayG} and \ref{table:delayP}, which characterise the fundamental statistical difficulty of robust online change point detection under contamination and heavy-tailed noise.  We then propose a robust online detection procedure, based on a combination of the median and the robust univariate mean estimator (RUME, \citealp{pmlr-v108-prasad20a}), and prove that it is minimax optimal, up to logarithmic factors. 

\subsection{Minimax lower bounds on detection delay} \label{sect:optimal}

We begin by formalising the parameter spaces required to state the lower bounds for the univariate online change point detection problem.  Throughout this subsection, we assume that the observed sequence $\{X_i\}_{i\in\mathbb{N}}$ satisfies \Cref{assum:kappa}. 

For a given mean sequence $f = \{f_i\}_{i \in \mathbb{N}}$ and a class of inlier distributions $\mathcal{D}$, let $\mathcal{H}_\varepsilon(f,\mathcal{D})$ denote the collection of all product measures generated under \Cref{assum:kappa}, namely
\[
    \mathcal{H}_\varepsilon(f,\mathcal{D})=\left\{\bigotimes_{i=1}^\infty Q_i:\, F_i\in\mathcal{D},\, \E_{X\sim F_i}[X]=f_i, \,\Var_{X\sim F_i}[X]=\sigma^2\right\}.
\]
The variance parameter $\sigma^2$ is assumed to be common across $i \in \mathbb{N}$, and will be suppressed from the notation when it plays no explicit role.

We next introduce the classes of mean sequences corresponding to the null and alternative hypotheses.  The null class, representing the absence of a change point, is defined by 
\[
    \mathcal{S}_0 = \{ f: \, f_i = \mu, \, i \in \mathbb{N}_+ \text{ for some } \mu \in \mathbb{R}\}.
\]
For a change point occurring at location $\Delta \in \mathbb{N}$ with jump size $\kappa > 0$, we define
\[
    \mathcal{S}(\Delta, \kappa) = \left\{ f : f_i = \mu_1 \mathbbm{1}_{\{i \le \Delta\}} + \mu_2 \mathbbm{1}_{\{i > \Delta\}} \text{ for some } \mu_1, \mu_2\in \R, \, |\mu_1 - \mu_2| = \kappa \right\}.
\]
Under this formulation, the null and alternative parameter spaces are given respectively by 
\[
    \Theta_0(\mathcal{D}) = \bigcup_{f \in \mathcal{S}_0} \mathcal{H}_\epsilon(f, \mathcal{D}) \quad \mbox{and} \quad \Theta(\Delta, \kappa,\mathcal{D}) = \bigcup_{f \in \mathcal{S}(\Delta, \kappa)} \mathcal{H}_\epsilon(f, \mathcal{D}).
\]
Let $(\mathcal{F}_t)_{t \in \mathbb{N}}$ denote the natural filtration generated by the observations, where $\mathcal{F}_t = \sigma(X_1,\ldots,X_t)$.  For a prescribed false alarm level $\alpha \in (0, 1)$, we define the class of change point estimators by
\begin{align} \label{mathcaltmean}
    &\mathcal{T}(\alpha) = \Big\{T: \, T \text{ is an extended stopping time with respect to } (\mathcal{F}_t)_{t \in \mathbb{N}}, \nonumber\\ 
    & \hspace{6cm} \sup_{P\in\Theta_0(\mathcal{D})}\Pb_{P} ( T < \infty) \leq \alpha\Big\}. 
\end{align}
This class contains all online procedures whose false alarm probability is controlled uniformly over the null parameter space.

We are now ready to state lower bounds on the detection delay over the classes $\mathcal{G}_{\theta,M}$ and $\mathcal{P}_{v,\phi}$.  These bounds reveal the fundamental limits of online change point detection under contamination and heavy-tailedness, and identify the phase transitions that distinguish the different detection regimes.

\begin{theorem}[Detection lower bounds]\label{thm:univ_lower}
Suppose that $\{X_i\}_{i\in\mathbb{N}} \subseteq \mathbb{R}$ satisfies \Cref{assum:kappa}. Let $\Delta \in \mathbb{N}$ be the change point location and $\alpha \in (0,1)$ be the false alarm probability. Consider the classes $\mathcal{G}_{\theta,M}$ and $\mathcal{P}_{v,\phi}$, defined in Definitions~\ref{def-galphak} and \ref{def-palphak}, respectively.  Consider the class of change point estimators $\mathcal{T}(\alpha)$ defined in \eqref{mathcaltmean}.  We have the following.
\begin{enumerate}[label=(\textbf{\alph*})]
    \item (Regime 1a) For the $\mathcal{G}_{\theta,M}$ class, suppose that $\kappa/M\leq {\varepsilon} \{\log[1/(2\varepsilon)]\}^{1/\theta}$  for some $\varepsilon\leq C_\theta$, where $C_\theta>0$ is an absolute constant depending on $\theta$. Then, it holds that
    \begin{align*}
        \underset{\hat{t} \in \mathcal{T}(\alpha)}{\inf} \ \sup_{P \in \Theta(\Delta, \kappa,\mathcal{G}_{\theta,M})} \Pb_{P} \left\{ \hat{t}=\infty
        \right\} \geq 1- \alpha.
    \end{align*}

    For the $\mathcal{P}_{v,\phi}$ class, suppose that $\kappa/\phi \leq2^{-1/v} \varepsilon^{1-1/v}$. Then, it holds that
    \begin{align*}
        \underset{\hat{t} \in \mathcal{T}(\alpha)}{\inf} \ \sup_{P \in \Theta(\Delta, \kappa,\mathcal{P}_{v,\phi})} \Pb_{P} \left\{ \hat{t}=\infty
        \right\} \geq 1- \alpha.
    \end{align*}
    
    \item (Regime 1b) Let $n\in\mathbb{N}$ and $\omega \in (0, 1-\alpha)$. Suppose that either
$$\mathcal{D}=\mathcal{G}_{\theta,M}
\quad\text{with}\quad M\geq C_1\sigma,
\qquad\text{or}\qquad
\mathcal{D}=\mathcal{P}_{v,\phi}
\quad\text{with}\quad \phi\geq C_2\sigma,$$
for some constant $C_1$ depending only on $\theta$ and $C_2$ depending only on $v$.
    Then there exists a constant $c>0$ depending only on $\omega$ such that if $\kappa^2 \Delta \leq c\sigma^2$, it holds that
    \begin{align*}
        \underset{\hat{t} \in \mathcal{T}(\alpha)}{\inf} \ \sup_{P \in \Theta(\Delta, \kappa,\mathcal{D})} \Pb_{P} \left\{ \hat{t} -\Delta  >  n 
        \right\} \geq 1- \alpha- \omega.
    \end{align*}
    \item (Regime 2) Suppose that either
$$\mathcal{D}=\mathcal{G}_{\theta,M}
\quad\text{with}\quad M\geq C_1\sigma,
\qquad\text{or}\qquad
\mathcal{D}=\mathcal{P}_{v,\phi}
\quad\text{with}\quad \phi\geq C_2\sigma,$$
and that
     \begin{equation*}
         \alpha+2\alpha^{1/4}<\frac{1}{2} \quad \text{and} \quad \alpha^{5/4}\log\left(\frac{1}{\alpha}\right)\leq 4\kappa^{2}\sigma^{-2}.
     \end{equation*}
     Then, it holds that
     \begin{equation*}
         \inf_{\hat{t}\in \mathcal{T}(\alpha)} \sup_{P \in \Theta(\Delta, \kappa,\mathcal{D})}\E_P\{(\hat{t}-\Delta)^+\}\geq \frac{\sigma^2}{4\kappa^{2}}\log\left(\frac{1}{\alpha}\right).
     \end{equation*}
     \item (Regime 3) For the $\mathcal{G}_{\theta,M}$ class, suppose that $\kappa\geq 4M(2\log(4e))^{1/\theta}$, 
     \begin{equation*}
         \alpha+2\alpha^{1/4}<\frac{1}{2} \quad \text{and} \quad \alpha^{-1/4}\log\left(\frac{1}{\alpha}\right)\leq (2\kappa)^{\theta}\exp(4^\theta\log(4e)).
     \end{equation*}
     Then, it holds that
     \begin{equation*}
         \inf_{\hat{t}\in \mathcal{T}(\alpha)} \sup_{P \in \Theta(\Delta, \kappa,\mathcal{G}_{\theta,M})}\E_P\{(\hat{t}-\Delta)^+\}\geq \frac{1}{4(2\kappa)^{\theta}}\log\left(\frac{1}{\alpha}\right).
     \end{equation*}

     For the $\mathcal{P}_{v,\phi}$ class, suppose that
     \begin{equation*}\label{thm_assum:ak2}
         \alpha+2\alpha^{1/4}<\frac{1}{2} \quad \text{and} \quad \alpha^{-1/4}\log^{2.02}\left(\frac{1}{\alpha}\right)\leq (2\kappa\log(\kappa)-1)^{v+0.01}(4(v+1.01)\log(\kappa))^{2.02}.
     \end{equation*}
     Then, it holds that
     \begin{equation*}
         \inf_{\hat{t}\in \mathcal{T}(\alpha)} \sup_{P \in \Theta(\Delta, \kappa,\mathcal{P}_{v,\phi})}\E_P\{(\hat{t}-\Delta)^+\}\geq \frac{1}{16v\log(\kappa)}\log\left(\frac{1}{\alpha}\right).
     \end{equation*}
     \item (Regime 4) Suppose that either
$$\mathcal{D}=\mathcal{G}_{\theta,M}
\qquad\text{or}\qquad
\mathcal{D}=\mathcal{P}_{v,\phi},$$
     and that $\alpha+\alpha^{1/4}<1/2$. Then, it holds that
     \begin{equation*}
         \inf_{\hat{t}\in \mathcal{T}(\alpha)} \sup_{P \in \Theta(\Delta, \kappa,\mathcal{D})}\E_P\{(\hat{t}-\Delta)^+\}\geq \frac{3}{16}\frac{\log(1/\alpha)}{\log((1-\varepsilon)/\varepsilon)}.
     \end{equation*}
\end{enumerate}
\end{theorem}

Theorem \ref{thm:univ_lower} identifies a series of phase transitions in the robust univariate online change point detection problems.  As the signal strength increases, the principal statistical bottleneck shifts successively from the change point location, to the tail behaviour, and finally to the contamination level. 

In Regime 1, no change point estimator is guaranteed to detect the change point.  More precisely, when the jump size is too small relative either to the contamination level, or to the change point location, no procedure can reliably detect the change point.  Regime 1a reflects a contamination-driven impossibility.  The adversarial component can mask the mean shift effectively, such that consistent detection is impossible.  Regime 1b captures a different obstruction.  Even without contamination, if the change occurs too early relative to the signal size, then the available pre-change sample is insufficient to localise the change with non-trivial accuracy; see \citet[][Proposition 8]{moen_2025}. 

Once the signal exceeds this detectability threshold, the problem enters Regime 2, where the detection delay exhibits the familiar sub-Gaussian-type scaling $\sigma^2\kappa^{-2} \log(1/\alpha)$ \citep{Yu_2023}. In this regime, the signal is strong enough to be detectable, but not yet so that robustness or tail effects dominate.  Consequently, the delay is governed primarily by the number of observations required to average out the intrinsic noise in the inlier distribution. 

Regime 3 marks the onset of genuinely heavy-tail behaviour.  Further increases in the signal no longer translate into the Gaussian rate of improvement as in Regime 2.  Instead, the detection delay is controlled by the tail decay of the inlier distribution.  For sub-Weibull inliers, the lower bound scales as $\kappa^{-\theta} \log(1/\alpha)$, and $(v\log(\kappa))^{-1} \log(1/\alpha)$ for the finite-$v$th-moment class.  Intuitively, although the signal is now large, extreme inlier observations can dominate the detection delay.

In Regime 4, the contamination level becomes the dominant constraint.  Beyond this point, increasing the signal strength no longer yields a corresponding reduction in the detection delay.  The delay is lower bounded by a quantity of order $\log(1/\alpha)/\log((1-\varepsilon)/\varepsilon)$, reflecting the intrinsic difficulty of separating contaminated from uncontaminated observations in an online setting.  This regime may therefore be viewed as the contamination-limited phase.

Taken together, these four regimes provide a unified picture of the minimax difficulty of the univariate problem by precisely characterising the qualitatively distinct phenomena, depending on the relative magnitudes of the jump size, contamination proportion and tail parameters.

\subsection{A robust online change point detection procedure} \label{sect:rume_cpt}

In this subsection, we develop a robust online change point detection procedure for univariate data. This procedure is designed to retain power under both Huber contamination and heavy-tailed inlier noise.  It follows the standard online mean change point detection routine.  For a sequence $\{X_t\}_{t \in \mathbb{N}} \subseteq \mathbb{R}$ and for a mean estimator $\widehat{\mu}_{s:t}$ based on data $X_{s:t}$, we declare a change point at $t \in \mathbb{N}$, if 
\[
    \max_{1 \leq s < \lfloor t/2 \rfloor} \big| \widehat{\mu}_{1:s} - \widehat{\mu}_{(s+1):t}\big|
\]
exceeds a prescribed threshold. This type of procedure has been incorporated in many recent algorithms; see e.g.~\cite{Cho2016}, \cite{Liu_Gao_Samworth_2021}, \cite{Follain2022}, \cite{Yu_2023},  \cite{Zhang2023} and \cite{gong2024}. In view of the contamination and heavy-tailedness, we summon some robust mean estimators as choices for~$\widehat{\mu}$.  In particular, we consider both medians and the robust univariate mean estimator \citep[RUME,][]{pmlr-v108-prasad20a}, depending on the available sample size.  The algorithm is detailed in \Cref{alg:cusum_rume}.  

To be specific, for a pair of integers $s, t$ satisfying $h_t\leq s \leq \lfloor t/2 \rfloor$ and $2 | s$, where $h_t$ is some threshold (to be specified) that depends on $t$, we consider $\widehat{\mu}_{s:t}= \mathrm{RUME}(X_{s:t})$, otherwise $\widehat{\mu}_{s:t} = \med(X_{s:t})$.  This means that employing the RUME estimator only when the available sample is sufficient, as required by the theoretical guarantees to be shown in Lemmas~\ref{lem:rume_moment} and \ref{prop2}.

The thresholds $\{\chi_{s,t}\}$ and $\{\zeta_{s,t}\}$ are parametrised by the contamination level and the tail behaviour of the inliers.  The theoretical choices for these thresholds will be presented in Theorems \ref{thm:rume_cpt_subW} and \ref{thm:rume_cpt} below, with practical guidance discussed in \Cref{subsect:univ_sim}. 

We note that the algorithm requires the contamination level $\varepsilon$ and the tail behaviour of the inlier distribution as inputs.  These are commonly seen in the robust statistics literature; see e.g.~\cite{Cheng_Diakonikolas_Ge_2019}, \cite{pmlr-v108-prasad20a} and \cite{Li_2021} for the former, and e.g.~\cite{Comminges2021} and \cite{li_2025} for the latter.  In practice, tail parameters may be estimated from pilot data when available \citep[see, e.g.][]{Hill1975, Vladimirova_Girard_Nguyen_Arbel_2020}, or calibrated under a conservative tail assumption otherwise, for example by assuming only finite variance.  In \Cref{sect:discuss} we provide some further discussions on relaxation of the requirement on $\varepsilon$.

\begin{algorithm}[H]
\caption{Online univariate change point detection via RUME and medians}\label{alg:cusum_rume}
\begin{algorithmic}
\INPUT{Dataset $\{X_i\}_{i\in\mathbb{N}}$, Minimum sample size thresholds $\{h_{t}\}_{t\in\mathbb{N}}$, Contamination level $\varepsilon$, Median detection thresholds~$\{\chi_{s,t}\}_{(s,t)\in\mathbb{N}^2, 1\leq s<h_t}$, RUME detection thresholds $\{\zeta_{s,t}\}_{(s,t)\in\mathbb{N}^2, h_t\leq s<\lfloor t/2\rfloor}$} 
\State $t \gets 1, \mathrm{FLAG} \gets 0$
\While{$\mathrm{FLAG} = 0$}
    \State $t \gets t + 1$
    \For{$s \in \{1, \ldots,  \lfloor t/2 \rfloor\}$}
        \If{$s < h_t$ \textbf{and} $\frac{1}{2e} \left[ \frac{4\alpha}{3(t^3-t)} \right]^{2/s} > \varepsilon$}
            \State $\mathrm{FLAG} \gets \mathbbm{1}\{|\med(X_{1:s}) - \med(X_{(t-s+1):t})| > \chi_{s,t}\}$
        \ElsIf{$s \geq h_t$ \textbf{and} $2 \mid s$}
            \State $\mathrm{FLAG} \gets  \mathbbm{1}\{|\RUME(X_{1:s})-\RUME(X_{(t-s+1):t})| > \zeta_{s,t}\}$ \Comment{See \Cref{algo:RUME}}
        \EndIf        
    \EndFor
\EndWhile
\OUTPUT{$t$}
\end{algorithmic}
\end{algorithm}

\subsubsection{Detection guarantees under sub-Weibull assumptions}

To understand the performance of \Cref{alg:cusum_rume}, we first start with signal strength conditions, corresponding to Regimes 2 and 3 in \Cref{table:delayG}, respectively.

\begin{assumption} \label{assume-2}
We assume that one of the following holds.
\begin{enumerate}[label=(\textbf{\alph*})]
    \item \label{assum: rumecpt_subW} Assume $\kappa\gtrsim \varepsilon\bigl(\log(1/\varepsilon)\bigr)^{1+1/\theta}$ and $\Delta\gtrsim \max\left({\kappa^{-2}},1\right)\log({\Delta}/{\alpha})$.
    \item \label{assum: medcpt_subW}  Assume $\kappa\gtrsim 1$ and $\Delta\gtrsim \max\bigl(\kappa^{-\theta}, \log^{-1}(1/\varepsilon) \bigr) \log({\Delta}/{\alpha})$.
\end{enumerate}
\end{assumption}

\begin{theorem}[Change point detection for class $\mathcal{G}_{\theta,M}$]\label{thm:rume_cpt_subW}
Suppose $\{X_i\}_{i\in\mathbb{N}}$ satisfies Assumption \ref{assum:kappa}, where $\mathcal{D}=\mathcal{G}_{\theta,M}$ for some $\theta\in(0,2]$ and $M>0$. Let $\varepsilon\in[0,0.09)$, $\alpha\in (0,1)$ and $\hat{t}$ be the stopping time returned by Algorithm \ref{alg:cusum_rume} with sample size threshold
\begin{align*}
    h_t= \left\lceil20\log\left\{\frac{3(t^3-t)}{8\alpha}\right\}\right\rceil
\end{align*} and detection thresholds
    \begin{align}
        \chi_{s,t}&=2M\log^{1/\theta} \left[\frac{4e}{\left(\frac{4\alpha}{3(t^3-t)}\right)^{2/s}-2e\varepsilon}\right], \label{eqn:subW_thres1} \\
        \zeta_{s,t}&=2C_1\varepsilon_{s,t}'\log^{1+1/\theta} \left(\frac{1}{\varepsilon_{s,t}'}\right)+2C_2\sqrt{\frac{2}{s}\log \left[\frac{3(t^3-t)}{8\alpha}\right]}, \label{eqn:subW_thres2}
    \end{align}
    where $\varepsilon'_{s,t}=\max\left\{\varepsilon, 2s^{-1}\log\left[\frac{3(t^3-t)}{8\alpha}\right]\right\}$ and $C_1,C_2 > 0$ are absolute constants defined in Lemma \ref{lem:rume_moment}(b). We then have the following.
    \begin{enumerate}[label=(\textbf{\alph*})]
        \item \label{subthm:subWcpt1} $\mathbb{P}_{\infty}(\hat{t}<\infty)\leq \alpha$.
        \item \label{subthm:subWcpt2} $\mathbb{P}_{\Delta}(\hat{t}\leq\Delta)\leq \alpha$ for any $\Delta\geq1$.
        \item \label{subthm:subWcpt3} Under \Cref{assume-2}\ref{assum: rumecpt_subW}, it holds that 
        \begin{equation*}
        \mathbb{P}_\Delta\left(\Delta<\hat{t}\leq\Delta+\left\lceil C \max\left(\frac{1}{\kappa^2},1\right)\log\left(\frac{\Delta}{\alpha}\right) \right\rceil\right)\geq 1-\alpha,
        \end{equation*}
        for some constant $C > 0$ depending only on $\theta$ and $M$.
        \item \label{subthm:subWcpt4} Under \Cref{assume-2}\ref{assum: medcpt_subW}, it holds that
        \begin{equation*}
            \mathbb{P}_\Delta\left(\Delta<\hat{t}\leq\Delta+\left\lceil C'\max\left(\frac{1}{\kappa^\theta},\frac{1}{\log(1/\varepsilon)}\right)\log\left(\frac{\Delta}{\alpha}\right)\right\rceil\right)\geq 1-\alpha,
        \end{equation*}
        for some constant $C' > 0$ depending only on $\theta$ and $M$.
    \end{enumerate}
\end{theorem}
In the no-change regime $\Delta = \infty$, \Cref{thm:rume_cpt_subW}\ref{subthm:subWcpt1} shows that the false alarm probability is controlled at level $\alpha$. Moreover, \Cref{thm:rume_cpt_subW}\ref{subthm:subWcpt2} ensures that, even when a change is present, the procedure does not stop before the change point with probability at least $1 - \alpha$.  \Cref{alg:cusum_rume} thus controls false alarm uniformly over both the null and pre-change segments. 

More importantly, \Cref{thm:rume_cpt_subW} provides high-probability upper bounds on the detection delay under exponentially decaying tails, and these bounds match the lower bounds in \Cref{thm:univ_lower}, up to a factor of $\log(\Delta)$. Under \Cref{assume-2}\ref{assum: rumecpt_subW}, corresponding to the moderate-signal regime, the delay is of order $\kappa^{-2}\log(\Delta/\alpha)$ when the contamination level is not too severe.  In this range, the RUME-based component of the procedure drives detection, and the rate is the same Gaussian-type scaling in Regime~2 of the lower bound. 
With a larger signal strength, the delay in \Cref{thm:rume_cpt_subW}\ref{subthm:subWcpt3} is instead of the order $\log(\Delta/\alpha)$, reflecting the sample size required for RUME to operate reliably under contamination. This illustrates that RUME recovers optimal rates only when detecting changes of moderate signal sizes. 

Under \Cref{assume-2}\ref{assum: medcpt}, the theorem enters the large-signal regime. In this range, the median-based component of the procedure drives detection. When $\kappa$ is of order at least $O(1)$ but still satisfies $\kappa \lesssim \log^{1/\theta}(1/\varepsilon)$, \Cref{thm:rume_cpt_subW}\ref{subthm:subWcpt4} shows that the detection delay is of the order $\kappa^{-\theta}\log(\Delta/\alpha)$, corresponding to Regime 3 of the lower bound. This reveals explicitly how the sub-Weibull parameter $\theta$ influences the detection problem: fixing the signal strength, the delay increases as $\theta$ decreases. Once the signal is sufficiently strong, the remaining difficulty is determined by the heaviness of the inlier tails rather than by Gaussian-type averaging. 

When $\kappa \gtrsim \log^{1/\theta}(1/\varepsilon)$, the delay no longer improves with the jump size and is instead of order $(\log (1/\varepsilon))^{-1}\log(\Delta/\alpha)$. This behaviour is consistent with Regime 4 of the lower bound. At this point, detection is limited by the sample needed to observe a sufficient number of uncontaminated samples, rather than by any lack of separations between the pre- and post-change means.

\subsubsection{Detection guarantees under finite-moment assumptions}

We now present a similar result under the finite-moment assumptions. We again start with signal strength conditions, corresponding to Regimes 2 and 3 in \Cref{table:delayP}, respectively.

\begin{assumption} \label{assume-1}
We assume that one of the following holds.
\begin{enumerate}[label=(\textbf{\alph*})]
    \item \label{assum: rumecpt} Assume $\kappa\gtrsim {\varepsilon}^{1-1/v}$ and $\Delta\gtrsim \max\left({\kappa^{-2}},1\right)\log({\Delta}/{\alpha})$.
    \item \label{assum: medcpt} Assume $\kappa\gtrsim 1$ and $\Delta\gtrsim \max\bigl( (v\log(\kappa))^{-1}, \log^{-1}(1/\varepsilon) \bigr) \log({\Delta}/{\alpha})$.
\end{enumerate}
\end{assumption}

\begin{theorem}[Change point detection for class $\mathcal{P}_{v,\phi}$]\label{thm:rume_cpt}
    Suppose $\{X_i\}_{i\in\mathbb{N}}$ satisfies \Cref{assum:kappa}, where $\mathcal{D}=\mathcal{P}_{v,\phi}$ for some $v\geq2$ and $\phi > 0$. Let $\varepsilon\in[0,0.09)$, $\alpha\in (0,1)$ and $\hat{t}$ be the stopping time returned by  \Cref{alg:cusum_rume} with sample size threshold \begin{align*}
    h_t=\left\lceil20\log\left\{\frac{3(t^3-t)}{8\alpha}\right\}\right\rceil
\end{align*} and detection thresholds
    \begin{align}
        \chi_{s,t}&=2\phi\left\{{\frac{1}{2e}\left[\frac{4\alpha}{3(t^3-t)}\right]^{2/s}-\varepsilon}\right\}^{-1/v}, \label{eqn:fm_thres1} \\
        \zeta_{s,t}&=2\phi\left\{C_3\varepsilon'^{1-1/v}_{s,t}+C_4\sqrt{\frac{2}{s}\log\left[\frac{3(t^3-t)}{8\alpha}\right]}\right\}, \label{eqn:fm_thres2} 
    \end{align}
    where $\varepsilon'_{s,t}=\max\left\{\varepsilon, 2s^{-1} \log\left[\frac{3(t^3-t)}{8\alpha}\right]\right\}$ and $C_3, C_4 > 0$ are absolute constants defined in Lemma \ref{lem:rume_moment}(a). We then have the following.
    \begin{enumerate}[label=(\textbf{\alph*})]
        \item $\mathbb{P}_{\infty}(\hat{t}<\infty)\leq \alpha$.
        \item $\mathbb{P}_{\Delta}(\hat{t}\leq\Delta)\leq \alpha$ for any $\Delta\geq1$.
        \item Under \Cref{assume-1}\ref{assum: rumecpt}, it holds that
        \begin{equation*}
            \mathbb{P}_\Delta\left(\Delta<\hat{t}\leq \Delta+\left\lceil C \max\left(\frac{1}{\kappa^2},1\right)\log\left(\frac{\Delta}{\alpha}\right)\right\rceil\right)\geq 1-\alpha,
        \end{equation*}
        for some constant $C>0$ depending only on $\phi$.
        \item Under \Cref{assume-1}\ref{assum: medcpt}, it holds that
        \begin{equation*}
            \mathbb{P}_\Delta\left(\Delta<\hat{t}\leq \Delta+\left\lceil C'\max\left(\frac{1}{v\log(\kappa)},\frac{1}{\log(1/\varepsilon)}\right)\log\bigg(\frac{\Delta}{\alpha}\bigg)\right\rceil\right)\geq 1-\alpha,
        \end{equation*}
        for some absolute constant $C'>0$.
    \end{enumerate}
\end{theorem}
By choosing thresholds $\chi_{s,t}$ and $\zeta_{s,t}$ adapted to the finite-moment class, Algorithm \ref{alg:cusum_rume} again achieves the same false alarm guarantees as in \Cref{thm:rume_cpt_subW}. When $\Delta=\infty$, the false alarm probability is controlled at level $\alpha$; when a change is present, the procedure does not stop before the change point with probability at least $1-\alpha$.

\Cref{thm:rume_cpt} also provides high-probability upper bounds on the detection delay under finite-moment assumptions, and these bounds match the lower bounds of \Cref{thm:univ_lower}, up to a factor of $\log(\Delta)$. The main distinction compared to the results under sub-Weibull assumptions lies in how the tail behaviour influences the rates. Under \Cref{assume-1}\ref{assum: rumecpt}, corresponding to the moderate-signal regime, the delay retains the Gaussian-type scaling $\kappa^{-2}\log(\Delta/\alpha)$ under mild contamination. However, under \Cref{assume-1}\ref{assum: medcpt}, corresponding to the large-signal regime, the dependence on the number of moments $v$ becomes explicit. Specifically, when $\kappa=O(1)$ and $\kappa \lesssim \varepsilon^{-1/v}$, the delay scales as $(v\log(\kappa))^{-1}\log(\Delta/\alpha)$, showing that heavier tails slow detection relative to the Gaussian case. Once $\kappa \gtrsim \varepsilon^{-1/v}$, the delay saturates at order $(\log(1/\varepsilon))^{-1}\log(\Delta/\alpha)$, reflecting the need to accumulate sufficiently many uncontaminated observations.

Putting Theorems~\ref{thm:univ_lower}, \ref{thm:rume_cpt_subW} and \ref{thm:rume_cpt} together, we obtain a complete characterisation of the rates of detection delay for the univariate change point problem, up to logarithmic factors. We note that the lower bound in \Cref{thm:univ_lower} is stated in terms of the expected detection delay, while Theorems~\ref{thm:rume_cpt_subW} and~\ref{thm:rume_cpt} provide high-probability upper bounds.  This discrepancy is commonly seen in the literature, and recent works \citep[e.g.][]{chhor2025generalized, MaVerchandSamworth2026} have shown remedies for unification under weak conditions.

\section{Robust multivariate mean testing algorithms} \label{sect:hdtest}

In this section, we develop a robust multivariate mean testing procedure, which will serve as the main subroutine in our change point detection algorithm. This strategy is motivated by viewing change point detection as a sequence of tests for equality of mean vectors between two subsamples. Rather than estimating the pre- and post-change means and then comparing the resulting estimates, as in our univariate procedure, we directly test whether the corresponding mean vectors are equal. This distinction is relevant in the multivariate setting, where testing and estimation may exhibit different sample-size requirements; see, e.g.~the introduction sections of \cite{Diakonikolas2017} and \cite{Canonne2023} for further discussion.

We consider the robust mean testing problem under the dynamic Huber $\varepsilon$-contamination model~\eqref{huber}. Given a sample $\mathcal{Y}=\{Y_1,\ldots,Y_n\}\subseteq \R^p$, we assume that the data are generated independently according to
\begin{equation}\label{eqn:huberhd}
Y_i \sim (1-\varepsilon)F + \varepsilon H_i,
\end{equation}
where $F \in \mathcal{D}$ represents the inlier distribution with mean $\mu$, $\norm{\mu}_2=\kappa$, and identity covariance $I$. The contamination distributions $H_i$'s are arbitrary and may vary with $i \in [n]$. In this section, we consider the hypothesis testing problem
\begin{equation}\label{eqn:hypotheses}
    \mathrm{H}_0: \mu=\mathbf{0} \quad \mbox{vs.}\quad \mathrm{H}_1: \mu\neq \mathbf{0}.
\end{equation}
This problem is well understood in the Gaussian context. In particular, suppose that $F = \mathcal{N}_p(\mu, I)$ and the alternative is written as $\|\mu\|_2 \geq \kappa_{\min}$. In the absence of contamination, i.e.\ when $\varepsilon = 0$, the minimax testing rate satisfies $\kappa_{\min}^2\asymp_{\log} \sqrt{p}/n$ \cite[e.g.][]{Baraud2002,IngsterSuslina2003}. Under contamination, \cite{Canonne2023} established minimax lower and upper bounds of the form
\begin{equation*}
    \frac{\sqrt{p}}{n}+\varepsilon^{3/2}\sqrt{\frac{p}{n}}\lesssim \kappa_{\min}^2
    \lesssim
    \frac{\sqrt{p}}{n}+\varepsilon\sqrt{\frac{p}{n}}+\varepsilon^2\log(1/\varepsilon).
\end{equation*}
To the best of our knowledge, finite-sample analyses have yet been established for the contaminated setting in \eqref{eqn:huberhd}, where the inlier class $\mathcal{D}$ may contain heavy-tailed distributions.

To address the testing problem in \eqref{eqn:hypotheses}, we propose robust procedures for two classes of heavy-tailed inlier distributions: the sub-Weibull class $\mathcal{G}^p_{\theta,M}$ and the finite-moment class $\mathcal{P}^p_{v,\phi}$. The algorithms and theoretical guarantees for each class are presented in Sections~\ref{subsect:testalgo} and~\ref{subsect:testalgo2}, respectively. The procedures build upon the soft-filter framework developed in Section~8 of \citet{Canonne2023}. While \citet{Canonne2023} analyse this framework for $\mathcal{N}_p(\mu,I)$ inliers, we extend the analysis to accommodate heavy-tailed inliers. 

\subsection{Robust testing for the sub-Weibull class}\label{subsect:testalgo}
We first develop a robust mean testing procedure, \Cref{alg:efficient-mean-tester}, for the sub-Weibull inlier class $\mathcal{G}_{\theta,M}^p$. In a nutshell, \Cref{alg:efficient-mean-tester} uses filtering-based algorithms to downweight the suspected outliers, and then rejects or accepts the null hypothesis based on the weighted sum of the observations. 

The filtering mechanism here consists of two steps. First, a spectral filtering procedure (\Cref{alg:spectralfilter1} or~\ref{alg:spectralfilter2}, depending on whether the sample size is smaller than the number of dimensions) iteratively downweights observations by examining the leading eigenvector of $M(w,\mathcal{Y})-nI$ or $\Gram(w,\mathcal{Y})-nI$, where 
\[
    M (w, \mathcal{Y}) = \sum_{i \in \mathcal{Y}} w_i Y_i Y_i^\top \in \R^{p\times p}\quad \mbox{and} \quad \Gram (w,\mathcal{Y})_{ij} = \sqrt{w_i w_j} \iprod{Y_i, Y_j} 
\]
are the second moment matrix and Gram matrix respectively. Observations with more extreme projections along this leading eigenvector are more likely to be outliers or possess heavy tails, and are therefore downweighted more. Second, the weights are further refined using \Cref{alg:rowsumfilter}, which removes the $un$ observations with the largest deviation scores $\tau_k = \left| \left\langle \sqrt{w_k} Y_k, \Sum(w,\mathcal{Y})\right\rangle - w_k p \right|$,  where $\Sum(w,\mathcal{A})=\sum_{k\in\mathcal{A}}\sqrt{w_k}Y_k$ for any subset \(\mathcal{A}\subseteq \mathcal{Y}\). This step helps bound the cross term $\iprod{\Sum(w,\mathcal{Y}),\Sum(w,\mathcal{B})}$ that arises from expanding the test statistic in \eqref{proof_eqn:expansion}, where $\mathcal{B}$ is the collection of contaminated samples or those with extreme behaviour.

This procedure is an instance of the filtering-based methods commonly used in robust statistics.  The objective is to construct weights that preserve the signal contributed by the inliers while reducing the influence of contaminated observations \citep[see, e.g.][]{Lai_2016, Diakonikolas_2016, Diakonikolas2017, Cheng_Diakonikolas_Ge_2019}. Further details of \Cref{alg:efficient-mean-tester} are collected in \Cref{subsect:detail_routine}. The theoretical choices for the algorithmic thresholds are discussed next in \Cref{thm:poly-time-main}.

\begin{algorithm}[H]
\caption{Robust mean testing ($\mathrm{RobustMeanTest}(\{Y_{i}\}_{i=1}^n, \mathcal{D}; \kappa_0, \delta,\varepsilon, C_\gamma, T_u)$)}
\label{alg:efficient-mean-tester}
    \begin{algorithmic}
        \INPUT{Dataset $Y_1,\ldots,Y_n \in \R^p$, Class of inlier distributions $\mathcal{D}\in\{\mathcal{G}_{\theta,M}^p,\mathcal{P}^p_{v,\phi}\}$, Signal size input $\kappa_0$, Significance level $\delta$, Contamination level $\varepsilon$, Filter strength constant $C_\gamma$,  Detection sensitivity factor~$T_u$}
        \If{$\mathcal{D}=\mathcal{G}_{\theta,M}^p$}
            \State $u\leftarrow \varepsilon+n^{-1}+\sqrt{\frac{2(\varepsilon+n^{-1})\log(4/\delta)}{n}}+\frac{2\log(4/\delta)}{3n}$
        \Else
        \State {$u\leftarrow\varepsilon+\frac{1}{20}\min\left(1,(\frac{p}{n})^{v/4}\right)+\sqrt{\frac{2(\varepsilon+0.05\min\left(1,({p}/{n})^{v/4}\right))\log(4/\delta)}{n}}+\frac{2\log(4/\delta)}{3n}$}
        \EndIf
        \State $R_{f}\leftarrow C_{\gamma} \left(unp\log(1/u)+(\sqrt{np}+p)\log(2p/\delta)+n \kappa^2_0 \right)$
         \If{$n \leq p$}
            \State $w\leftarrow \mathrm{GramFilter}(\{Y_i\}_{i=1}^n, R_{f})$  \Comment{See Algorithm~\ref{alg:spectralfilter1}}
        \Else 
            \State $w\leftarrow \mathrm{MomentFilter}(\{Y_i\}_{i=1}^n, R_{f},u)$ \Comment{See Algorithm~\ref{alg:spectralfilter2}}
        \EndIf 
        \State $w'\leftarrow \mathrm{RowSumFilter}(\{Y_i\}_{i=1}^n, w,u)$ \Comment{See Algorithm~\ref{alg:rowsumfilter}}
         \If{$|\norm{\Sum (w', \mathcal{Y})}^2 - p \norm{w'}_1| > 0.5T_u \kappa_0^2 n^2$} 
            \State \textbf{return} $1$ 
        \Else
            \State \textbf{return} $0$
        \EndIf
    \end{algorithmic}
\end{algorithm}

\begin{proposition}[Mean testing with inlier in $\mathcal{G}^{p}_{\theta,M}$]
\label{thm:poly-time-main}
   Let $n, p \in \mathbb{N}$, $\delta > 0$ and $\varepsilon\in[0,0.08)$. Let $\{Y_i\}_{i \in [n]}$ be independently generated according to \eqref{eqn:huberhd}, with $\mathcal{D}=\mathcal{G}^{p}_{\theta,M}$ for some  $\theta\in(0,2]$ and $M>0$. Assume
\begin{equation}\label{eqn:u_val}
u=\varepsilon+n^{-1}+\sqrt{\frac{2(\varepsilon+n^{-1})\log(4/\delta)}{n}}+\frac{2\log(4/\delta)}{3n}\leq 0.08.
\end{equation}
Then, Algorithm~\ref{alg:efficient-mean-tester} with the detection sensitivity factor input $T_u=(1-7u)^2-2u$ has the following guarantees.
\begin{enumerate}[label=(\textbf{\alph*})]
    \item The algorithm has runtime $O((\varepsilon+n^{-1}) n^2 p \min(n, p) + np)$.
    \item If $\mu=\mathbf{0}$ and the input $\kappa_0$ satisfies
    \[
        \kappa_0^2\gtrsim \left\{\frac{\sqrt{p}}{n}+\varepsilon^2 p\log(1/\varepsilon)+\frac{p}{n^2}\right\}\cdot \mathrm{polylog}(\theta, n, p, 1/\delta),
    \]    
    then the algorithm outputs 0 with probability at least $1-\delta$.
      \item If $\kappa = \|\mu\|_2$ satisfies $0 < 2/C_\gamma\leq \kappa_0^2/\kappa^2\leq \bar{c}$, with $\bar{c} > 0$ being an absolute constant, and  
      \begin{equation}\label{eqn:test_t2e}
          \kappa^2\gtrsim \left\{ \frac{\sqrt{p}}{n}+\varepsilon^2 p\log(1/\varepsilon)+\frac{p}{n^2}\right\} \cdot \mathrm{polylog}(\theta, n, p, 1/\delta),
      \end{equation} 
      then the algorithm outputs 1 with probability at least $1-\delta$.
  \end{enumerate}
\end{proposition}

\Cref{thm:poly-time-main} provides finite-sample guarantees for \Cref{alg:efficient-mean-tester}, under sub-Weibull inliers and Huber contamination.  It shows that \Cref{alg:efficient-mean-tester}, which runs in polynomial time, controls both type I and II errors, provided the input signal size $\kappa_0$ is appropriately chosen.  The input $\kappa_0$ is the signal size used in the testing procedure to set the rejection threshold.  Under the null hypothesis, $\kappa_0$ should be large enough so the threshold is above the statistical fluctuation of the test statistics.  Under the alternative, the true signal strength $\kappa$ should satisfy the lower bound in \eqref{eqn:test_t2e}, and $\kappa_0$ is required to be of the same order as~$\kappa$. 

The condition $2/C_\gamma \leq \kappa_0^2/\kappa^2 \leq \bar{c}$ has two purposes.  The upper bound ensures that the rejection threshold is not too large compared with the true signal, so that the signal can still be detected. Similar upper bound requirements are also observed in the change point literature \citep[e.g.][]{Chen_2022, Padilla2022}. The lower bound ensures that the filtering radius $R_{f}$ in \Cref{alg:efficient-mean-tester} is large enough to contain the typical fluctuations in the operator norm of the empirical second-moment matrix or Gram matrix.  This is an artefact of the proof and a potential relaxation is discussed in \Cref{rmk:symmetrize}. 

The condition~\eqref{eqn:u_val} on the \emph{de facto} outlier level $u$ ensures that the combined proportion of contaminated observations and extreme inlier observations is controlled with high probability. Similar small-contamination conditions are used in robust estimation and testing;
see, e.g.~\cite{pmlr-v108-prasad20a} and \cite{Lugosi_Mendelson_2021}.  

The theoretical choice of the detection sensitivity factor is $T_u = (1-7u)^2 - 2u$. This factor appears in the rejection threshold and corresponds to a worst-case lower bound on the total weight retained on the set of inlier points close to the mean after filtering under the alternative. In practice, the filtering subroutines, GramFilter or MomentFilter, typically retain a much larger fraction of these weights. Thus, in \Cref{sect:sim_study}, we use a calibrated form $T_u = (1-ku)^2$, where $k$ is chosen empirically to ensure that the observed false alarm rate does not exceed $\alpha$. 

The minimum signal strength requirement in \eqref{eqn:test_t2e} consists of three terms. The first term, $\sqrt{p}/n$, is the classical Gaussian mean testing rate. The second term, $\varepsilon^2 p\log(1/\varepsilon)$, reflects the effect of Huber contamination on the filtering step.  The third term $p/n^2$ arises from controlling the second-moment matrix under sub-Weibull inliers. 

It is useful to compare this with the Gaussian strong contamination setting studied by \citet{Canonne2023}. Their results imply a minimum signal strength requirement of the form
\begin{equation}\label{eqn:gaussiantestrate}
    \kappa^2
    \gtrsim
    \left[\frac{\sqrt{p}}{n}+\varepsilon^2\log(1/\varepsilon)+\varepsilon\sqrt{\frac{p}{n}}\right] \cdot \mathrm{polylog}(p,n, 1/\varepsilon,1/\delta).
\end{equation}
When $\sqrt{p}/n$ is the dominating term in both \eqref{eqn:test_t2e} and \eqref{eqn:gaussiantestrate}, \Cref{thm:poly-time-main} recovers the Gaussian result.  In other regimes, our bound has an additional factor of \(p\) in the contamination-related term, as well as an extra $p/n^2$ term.
This discrepancy arises from our use of the matrix Bernstein inequality (\Cref{thm:matrix_bernstein}) to control the truncated second-moment matrix.
The analysis of \cite{Canonne2023}, who considered only the Gaussian setting, uses an $\varepsilon$-net argument to obtain sharper concentration of the second-moment matrix. This approach, although leading to optimal rates in the Gaussian setting, would incur an additional factor of~$p^{1/\theta}$ in the final rates under the sub-Weibull setting.

\subsection{Robust testing for the finite-moment class}\label{subsect:testalgo2}
When \Cref{alg:efficient-mean-tester} is used under the sub-Weibull assumption, i.e.~$\mathcal{D}=\mathcal{G}^p_{\theta,M}$, the minimum required signal strength to control type I and II errors depends only logarithmically on the failure probability. 
In contrast, when the same procedure is used under finite-moment assumptions, i.e.~$\mathcal{D}=\mathcal{P}^p_{v,\phi}$, the corresponding signal strength requirement depends polynomially on the failure probability; see \Cref{thm:poly-time-main_fm}. 
While this is not a major issue for the stand-alone mean testing problem when the failure probability is treated as a fixed constant, it becomes problematic when \Cref{alg:efficient-mean-tester} is applied repeatedly for online change point detection. Distributing the type~I error budget across many tests would make the minimum detection threshold needed for non-trivial power grow quickly over time.

To address this issue, in \Cref{alg:efficient-mean-tester-mom}, we use a median-of-means-based approach.  We note that median-of-means-type statistics have been applied in a wide range of statistical problems \citep[see, e.g.][]{lerasle2011, Lugosi2019subG, hopkins_2021, lecue2020, li_2025}. To be specific, we split the sample into $K$ blocks, run \Cref{alg:efficient-mean-tester} on each block and aggregate the resulting statistics by their median.  This turns constant-probability block-level control into high-probability control with only logarithmic dependence on the failure probability.  We reject the null when this median is large.  The full procedure is presented in \Cref{alg:efficient-mean-tester-mom}.

\begin{algorithm}[H]
\caption{Robust mean testing via medians ($\mathrm{RobustMeanTest}_{\text{MoM}}(\{Y_{i}\}_{i=1}^n,\mathcal{P}_{v,\phi}^p; \kappa_0,K,\varepsilon, C_\gamma, T_u)$)}
\label{alg:efficient-mean-tester-mom}
    \begin{algorithmic}
        \INPUT{Dataset $Y_1,\ldots,Y_n \in \R^p$, Class of inlier distribution $\mathcal{P}_{v,\phi}^p$, Signal size input $\kappa_0$, Group number $K$, Contamination level $\varepsilon$, Filtering strength $C_\gamma$, Detection sensitivity factor $T_u$}
        \State Partition the dataset uniformly at random into $K$ disjoint blocks $\mathcal{Y}_1,\dots,\mathcal{Y}_K$ of equal size $n_0 = \lfloor n/K \rfloor$.
        \State {$u\leftarrow\varepsilon+\frac{1}{20}\min\left\{1, \left(\frac{p}{n_0}\right)^{v/4}\right\}+\sqrt{\frac{2\left[\varepsilon+0.05\min\{1, (p/n_0)^{v/4}\}\right]\log(16K)}{n_0}}+\frac{2\log(16K)}{3n_0}$}
        \State $R_{f}\leftarrow C_{\gamma} \left(un_0p\log(1/u)+(\sqrt{n_0p}+p)\log(8Kp)+ n_0 \kappa^2_0\right) $
        \For{$i\in\{1,\ldots,K\}$}
         \If{$n_0 \leq p$} 
            \State $w\leftarrow \mathrm{GramFilter}(\{Y_i\}_{i=1}^{n_0}, R_{f})$  \Comment{See Algorithm~\ref{alg:spectralfilter1}}
        \Else 
            \State $w\leftarrow \mathrm{MomentFilter}(\{Y_i\}_{i=1}^{n_0}, R_{f}, u)$ \Comment{See Algorithm~\ref{alg:spectralfilter2}}
        \EndIf 
        \State $w'\leftarrow \mathrm{RowSumFilter}(\{Y_i\}_{i=1}^{n_0}, w,u)$ \Comment{See Algorithm~\ref{alg:rowsumfilter}}
        \State Form the statistic $U_i \leftarrow \Big| \norm{\Sum(w',\mathcal{Y}_i)}_2^2 - p \norm{w'_{\mathcal{Y}_i}}_1 \Big|$.
        \EndFor
         \If{$\med(\{U_{i}\}_{i=1}^K) > 0.5T_u \kappa_0^2 n_0^2$} 
            \State \textbf{return} $1$ 
        \Else
            \State \textbf{return} $0$
        \EndIf
    \end{algorithmic}
\end{algorithm}

\begin{proposition}[Mean testing with inlier in $\mathcal{P}^{p}_{v,\phi}$]
\label{thm:poly-time-main-mom}
Let $n, p \in \mathbb{N}$, $\omega > 0$, and $\varepsilon \in [0,0.08)$. Let $\{Y_i\}_{i \in [n]}$ be independently generated according to \eqref{eqn:huberhd}, with $\mathcal{D}=\mathcal{P}^{p}_{v,\phi}$ for some $v\geq 4$ and $\phi>0$. Assume
\begin{equation}\label{eqn:u_val-mom}
    u = \varepsilon + \frac{1}{20}\min\left(1, \left(\frac{p}{n_0}\right)^{v/4}\right)
+ \sqrt{\frac{2\left[\varepsilon + 0.05\min\left\{1, \left({p}/{n_0}\right)^{v/4}\right\}\right]\log(16K)}{n_0}}
+ \frac{2\log(16K)}{3n_0}\leq 0.08.
\end{equation}
Then, \Cref{alg:efficient-mean-tester-mom} with the detection sensitivity factor $T_u=(1-7u)^2-2u$ and group number $K=\lceil 8 \log(1/\omega) \rceil$ has the following guarantees.
  \begin{enumerate}[label=(\textbf{\alph*})]
      \item The algorithm has runtime $O(K[\varepsilon +\min\{1,(p/n_0)^{v/4}\}] n_0^2 p \min(n_0, p) + np)$.
      \item If $\mu=\mathbf{0}$ and the input $\kappa_0$ satisfies
      \begin{equation*}
            \kappa_0^2\gtrsim\begin{cases}
                \left[\frac{\sqrt{p}}{n}+\varepsilon^2p\log(1/\varepsilon)+\left(\frac{p}{n}\right)^{2-2/v}+\frac{p^{1+v/2}}{n^{v/2}}\right]\cdot\mathrm{polylog}(v, n, p, 1/\omega),& n>pK,\\
                p \cdot\mathrm{polylog}(v, n, p, 1/\omega),  & n\leq pK,
            \end{cases} 
        \end{equation*}
      then the algorithm outputs 0 with probability at least $1-\omega$.
      \item If $\kappa = \|\mu\|_2$ satisfies $0 < 2/C_\gamma\leq \kappa_0^2/\kappa^2\leq \bar{c}$, with $\bar{c} > 0$ being an absolute constant, and 
\begin{equation}\label{eqn:test_t2e_mom}
    \kappa^2\gtrsim\begin{cases}
                \left[\frac{\sqrt{p}}{n}+\varepsilon^2p\log(1/\varepsilon)+\left(\frac{p}{n}\right)^{2-2/v}+\frac{p^{1+v/2}}{n^{v/2}}\right]\cdot\mathrm{polylog}(v, n, p, 1/\omega),& n>pK,\\
                p \cdot\mathrm{polylog}(v, n, p, 1/\omega),  & n\leq pK,
            \end{cases} 
\end{equation}
then the algorithm outputs 1 with probability at least $1-\omega$.
  \end{enumerate}
\end{proposition}

\Cref{thm:poly-time-main-mom} gives finite-sample guarantees on type I and II error rates for \Cref{alg:efficient-mean-tester-mom} under finite-moment inliers. Similar to \Cref{thm:poly-time-main}, a large enough $\kappa_0$ is needed to control the type I error. To control type~II error, $\kappa_0$ must be of the same order as the true signal strength $\kappa$, with $\kappa$ satisfying \eqref{eqn:test_t2e_mom}. The replacement of condition \eqref{eqn:u_val} by \eqref{eqn:u_val-mom} reflects the need to control the proportion of contaminated observations and extreme-valued inliers for each group for the median-of-means construction.

The median-of-means step is used to obtain high-probability control.  Each block only needs a constant-probability guarantee, and taking the median across $K = \lceil 8\log(1/\omega) \rceil$ blocks amplifies this to a failure probability of at most $\omega$.  As a result, the signal requirement depends only logarithmically on $1/\omega$, which is essential for the online change point settings in \Cref{sect:high_dim_cpt}.

When $n>pK$, the minimum signal strength requirement in \eqref{eqn:test_t2e_mom} consists of the Gaussian testing term $\sqrt{p}/n$, the contamination term $\varepsilon^2 p \log(1/\varepsilon)$, and two finite moment terms $(p/n)^{2-2/v}+ p^{1+v/2}/n^{v/2}$.  The latter two terms reflect the cost of controlling second-moment fluctuations under only a finite $v$-th moment assumption, decreasing with increasing $v$.
Compared with the Gaussian result of \citet{Canonne2023},  \Cref{thm:poly-time-main-mom} recovers the same leading rate when $\sqrt{p}/n$ dominates. The additional terms quantify the price of finite-moment inliers and the matrix concentration argument used after truncation. When $n\leq pK$, the minimum signal strength requirement is of order $p$, independent of $n$. This could be an artefact of the proof; see \Cref{rem:contam_rate_fm} for detailed explanations.

The runtime of \Cref{alg:efficient-mean-tester-mom} is essentially $K = \lceil 8 \log(1/\omega) \rceil$ times that of \Cref{alg:efficient-mean-tester} applied to a block of size~$n_0$. Thus, the algorithm still runs in polynomial time.

\section{Robust multivariate online change point detection} \label{sect:high_dim_cpt}

In this section, we apply the robust mean testing algorithms developed in \Cref{sect:hdtest} to the online change point detection problem.  The resulting procedure is given in \Cref{alg:cpd_highd}. Its theoretical properties are studied in Theorems~\ref{thm:cpt_hd_subW} and \ref{thm:cpt_hd_fm}, for the sub-Weibull class $\mathcal{G}_{\theta,M}^p$ and the finite-moment class $\mathcal{P}_{v,\phi}^p$, respectively.

\begin{algorithm}[ht]
\caption{Online change point detection via robust mean testing}\label{alg:cpd_highd}
\begin{algorithmic}
\INPUT Dataset $\{X_u\}_{u\in\mathbb{N}}$, Class of inlier distributions $\mathcal{D}=\{\mathcal{G}_{\theta,M}^p,\mathcal{P}_{v,\phi}^p\}$, Signal size input $\kappa_0>0$, Standard deviation $\sigma>0$, False alarm probability $\alpha\in(0,1)$, Minimum sample size $h_t$, Outlier control threshold $\Omega$, Contamination level $\varepsilon\in [0,\Omega)$, Filtering strength $C_\gamma$, Detection sensitivity factor $T_u$, Group number constant $K_c$ (for $\mathcal{D}=\mathcal{P}_{v,\phi}^p$ only) 
\State $t \gets 2$
\State $\mathrm{FLAG} \gets 0$
\While{$\mathrm{FLAG} = 0$}
    \State $t \gets t + 1$
    \State $\delta_t\gets \frac{4\alpha}{t^2(t+1)}$
    \State $Y_{r,t} \gets (X_{t-r+1} - X_s)/(\sqrt{2}\sigma), \forall 1\leq r\leq \lfloor t/2 \rfloor$ 
    \For{$s \in\{ h_t, \ldots, \lfloor t/2 \rfloor\}$}
    \If{$\mathcal{D}=\mathcal{G}_{\theta,M}^p$}
    \State $u \gets 2\varepsilon + \frac{1}{n}
+ \sqrt{\frac{2(2\varepsilon + 1/n)\log(4/\delta_t)}{n}}
+ \frac{2\log(4/\delta_t)}{3n}$
        \If{$u\leq \Omega$}
        \State $\mathrm{FLAG} \leftarrow \mathrm{RobustMeanTest}(\{Y_{i,t}\}_{i=1}^s,\mathcal{G}_{\theta,M}^p;\tfrac{\kappa_0}{\sqrt{2}}, \delta_t, 2\varepsilon, C_\gamma, T_u)$ \Comment{See \Cref{alg:efficient-mean-tester}}
        \EndIf
        \Else
        \State $K\leftarrow \lceil K_c \log(1/\delta_t) \rceil$
        \State $n_0\gets \lfloor n/K \rfloor$ 
        \State $u \gets 2\varepsilon + \frac{1}{20}\min\left(1,\left(\frac{p}{n_0}\right)^{v/4}\right)
+ \sqrt{\frac{2(2\varepsilon + 0.05\min((p/n_0)^{v/4},1))\log(16K)}{n_0}}
+ \frac{2\log(16K)}{3n_0}$
        \If{$u\leq \Omega$}
        \State $\mathrm{FLAG}  \leftarrow \mathrm{RobustMeanTest}_{\text{MoM}}(\{Y_{i,t}\}_{i=1}^s,\mathcal{P}_{v,\phi}^p; \tfrac{\kappa_0}{\sqrt{2}}, K, 2\varepsilon, C_\gamma, T_u)$ \Comment{See \Cref{alg:efficient-mean-tester-mom}}
        \EndIf
        \EndIf
        \If{$\mathrm{FLAG} = 1$} \textbf{break} \EndIf
    \EndFor
\EndWhile
\State \Return $t$
\end{algorithmic}
\end{algorithm}

The main idea of Algorithm~\ref{alg:cpd_highd} is to reduce the online change point detection problem to a sequence of robust mean testing problems. At each $t \geq 2$, we compare observations from the beginning and the end of the current data stream by forming pairwise differences
\[
Y_{r,t} = \frac{X_{t-r+1} - X_r}{\sqrt{2}\sigma}, \quad r = 1, \ldots, \lfloor t/2 \rfloor.
\]
This construction, also used in \citet{li_2025}, transforms a possible change in the mean of $\{X_t\}$ into a non-zero mean in the derived variables $\{Y_{r, t}\}$.  Under the no-change model, these variables have zero inlier means. Under the alternative, the inlier mean is non-zero for a range of $r$; see \Cref{prop:properties_Ds} for details.

\Cref{alg:cpd_highd} scans over time $t$ and candidate window sizes $s$.  For each pair $(s, t)$ satisfying the minimum sample size requirement, the procedure applies the appropriate robust mean testing algorithm to the derived sample $\{Y_{i, t}\}_{i = 1}^s$.  Depending on the tail behaviour of the data, we apply either Algorithm~\ref{alg:efficient-mean-tester} or its median-of-means variant, Algorithm~\ref{alg:efficient-mean-tester-mom}. We note that this choice depends on the knowledge of the tail properties of data points.  The procedure stops as soon as one of these tests rejects the null hypothesis.

We now state the signal size conditions under which \Cref{alg:cpd_highd} achieves false alarm control and detection delay guarantees.

\begin{assumption} \label{assum: hd_cpt}
We assume that one of the following holds.
\begin{enumerate}[label=(\textbf{\alph*})]
    \item \label{assum: hd_cpt1} Assume that $\kappa$ satisfies
    \[
        \kappa^2\gtrsim \left\{ \frac{\sqrt{p}}{\Delta}+\frac{p}{\Delta^2}+\varepsilon^2 p\log(1/\varepsilon)\right\} \cdot \mathrm{polylog}(\theta, \Delta, p, 1/\alpha).
    \]

    \item \label{assum: hd_cpt2} Assume that $\kappa$ satisfies
    \begin{equation*}
    \kappa^2\gtrsim\begin{cases}
                \left[\frac{\sqrt{p}}{\Delta}+\varepsilon^2p \log(1/\varepsilon)+\left(\frac{p}{\Delta}\right)^{2-2/v}+\frac{p^{1+v/2}}{\Delta^{v/2}}\right]\cdot\mathrm{polylog}(v, \Delta, p, 1/\alpha),& \Delta\gtrsim p \log(\Delta/\alpha),\\
                p \cdot\mathrm{polylog}(v, \Delta, p, 1/\alpha),  & \Delta\lesssim p \log(\Delta/\alpha).
    \end{cases} 
\end{equation*}
\end{enumerate}
\end{assumption}

\Cref{assum: hd_cpt} is the change point analogue of the mean testing signal conditions in \Cref{sect:hdtest}. In particular, \Cref{assum: hd_cpt}\ref{assum: hd_cpt1} corresponds to \eqref{eqn:test_t2e} for sub-Weibull inliers, and \Cref{assum: hd_cpt}\ref{assum: hd_cpt2} to \eqref{eqn:test_t2e_mom} for finite-moment inliers. This connection arises because, after a change occurs at time \(\Delta\), the pairwise differences used in \Cref{assum: hd_cpt} form a mean testing problem with signal size proportional to \(\kappa\) and sample size of order $\Delta$.  The extra logarithmic factors come from repeated testing over many times and window sizes, since the total false alarm probability $\alpha$ is split across all tests.

When the term \(\sqrt{p}/\Delta\) dominates, \Cref{assum: hd_cpt} reduces, up to logarithmic factors, to $\kappa^2 \Delta \gtrsim_{\log} \sqrt{p}$. This matches the rate-optimal signal strength requirements in multivariate online change point detection \citep[e.g.~Proposition~8 in][]{moen_2025}.
The remaining terms reflect the additional cost of contamination and heavy-tailedness.

\begin{theorem}[Multivariate change point detection for class $\mathcal{G}_{\theta,M}^p$] \label{thm:cpt_hd_subW}
    Suppose $\{X_i\}_{i\in\mathbb{N}}$ satisfies \Cref{assum:kappa} with $\mathcal{D}=\mathcal{G}^p_{\theta,M}$ for some $\theta\in (0,2]$ and $M>0$. Let $\varepsilon\in[0,0.08)$, $\alpha\in (0,1)$ and $\hat{t}$ be the stopping time returned by \Cref{alg:cpd_highd}, with signal size input $\kappa_0\asymp \kappa$, $\Omega=0.08$, minimum sample size 
    \begin{equation}\label{eqn:min_h_t1}
        h_t=C\left(1+\frac{\sqrt{p}}{\kappa^2}+\frac{\sqrt{p}}{\kappa}\right) \cdot\mathrm{polylog} (p, 1 / \varepsilon, 1 / \kappa, {t}/{\alpha}) 
    \end{equation} and detection sensitivity factor $T_u=(1-7u)^2-2u$. We then have the following.
    \begin{enumerate}[label=(\textbf{\alph*})]
        \item $\mathbb{P}_{\infty}(\hat{t}<\infty) \leq  \alpha$.
        \item $\mathbb{P}_{\Delta}(\hat{t}\leq\Delta) \leq  \alpha$ for any $\Delta\geq1$.
        \item Under \Cref{assum: hd_cpt}\ref{assum: hd_cpt1}, it holds that
        \begin{equation*}
            \mathbb{P}_\Delta\left\{\Delta<\hat{t} \leq \Delta + C\left(1+\frac{\sqrt{p}}{\kappa^2}+\frac{\sqrt{p}}{\kappa}\right)\cdot\mathrm{polylog} (p, 1 / \varepsilon, 1 / \kappa, {\Delta}/{\alpha})\right\}\geq 1-\alpha,
        \end{equation*}
        for some constant $C>0$ depending only on $\theta$ and $M$.
    \end{enumerate}
\end{theorem}

\begin{theorem}[Multivariate change point detection for class $\mathcal{P}^p_{v,\phi}$] \label{thm:cpt_hd_fm}
    Suppose $\{X_i\}_{i\in\mathbb{N}}$ satisfies \Cref{assum:kappa} with $\mathcal{D}= \mathcal{P}^p_{v,\phi}$ for some $v\geq 4$, $\phi>0$. Let $\varepsilon\in[0,0.04)$, $\alpha\in (0,1)$ and $\hat{t}$ be the stopping time returned by \Cref{alg:cpd_highd}, with signal size input $\kappa_0\asymp \kappa$, $\Omega=0.08$, minimum sample size 
    \begin{equation}\label{eqn:min_h_t2}
        h_t=Cg(p,\kappa)\mathrm{poly}\log(p,1/\varepsilon,1/\kappa, t/\alpha),    
    \end{equation}
    group number constant $K_c=8$ and detection sensitivity factor $T_u=(1-7u)^2-2u$, where
    \[
        g(p,\kappa)=\begin{cases}
            \sqrt{p}/\kappa^2, & \varepsilon\sqrt{p}\log(1/\varepsilon)\lesssim \kappa \lesssim p^{-\frac{(2v-2)\vee (v+4)}{4v-8}}, \\
            {p^{1+2/v}}/{\kappa^{4/v}}+{p}/{\kappa^{v/(v-1)}}, & p^{-\frac{(2v-2)\vee (v+4)}{4v-8}}\lesssim \kappa \lesssim \sqrt{p}, \\
            1, & \kappa \gtrsim \sqrt{p}.
        \end{cases}
    \]
    We then have the following.
    \begin{enumerate}[label=(\textbf{\alph*})]
        \item $\mathbb{P}_{\infty}(\hat{t}<\infty) \leq  \alpha$.
        \item $\mathbb{P}_{\Delta}(\hat{t}\leq\Delta) \leq  \alpha$ for any $\Delta \geq 1$.
        \item Under \Cref{assum: hd_cpt}\ref{assum: hd_cpt2}, it holds that
        \begin{equation*}
            \mathbb{P}_\Delta\left(\Delta < \hat{t} \leq \Delta + Cg(p,\kappa)\mathrm{poly} \log(p,1/\varepsilon,1/\kappa,\Delta/\alpha)\right)\geq 1-\alpha,
        \end{equation*}
        for some constant $C>0$ depending only on $v$ and $\phi$.
    \end{enumerate}
\end{theorem}

Theorems \ref{thm:cpt_hd_subW} and \ref{thm:cpt_hd_fm} provide theoretical guarantees of \Cref{alg:cpd_highd} for inlier distributions $\mathcal{G}_{\theta,M}^p$ and $\mathcal{P}^p_{v,\phi}$, respectively. Firstly, \Cref{alg:cpd_highd} achieves false alarm control analogous to Theorems \ref{thm:rume_cpt_subW} and \ref{thm:rume_cpt}. When there is no change point ($\Delta=\infty$), the false alarm probability is at most $\alpha$, while in the presence of a change point, the procedure does not stop before the change with probability at least $1-\alpha$. 

Secondly, when there is a change and \Cref{assum: hd_cpt} holds, the detection delay of \Cref{alg:cpd_highd} can be controlled. Similar to Regimes~2--4 in Tables~\ref{table:delayG} and \ref{table:delayP} for the univariate setting, three distinct detection delay regimes emerge in the multivariate case once the change becomes detectable. For relatively small signal strengths satisfying $\kappa \gtrsim \varepsilon \sqrt{p\log(1/\varepsilon)}$, the detection delay scales as $\sqrt{p}/\kappa^2$. This rate corresponds to Regime~2 in the univariate setting, which is recovered by setting $p=1$. It also matches the expected detection delay in the online setting under Gaussian assumptions for dense changes obtained in \cite{Chen_2022}. As $\kappa$ increases further, the improvement in detection delay slows down: for sub-Weibull inliers, the rate becomes $\sqrt{p}/\kappa$, while for inliers with only finite $v$-th moment, the rate becomes $$\frac{p^{1+2/v}}{\kappa^{4/v}}+\frac{p}{\kappa^{v/(v-1)}}.$$ We conjecture that this slower decay rate reflects the increased difficulty of robust detection under heavy-tailed distributions, motivated by the observations from Regime~3 of the univariate setting. Finally, in the strong-signal regime where $\kappa\gtrsim \sqrt{p}$, the detection delay becomes $\tilde{O}(1)$. This saturation is not caused by a lack of signal strength, but rather by the minimum sample requirements needed to guarantee the statistical accuracy of the mean tests under contamination and heavy-tailed noise; see \eqref{eqn:u_val} and \eqref{eqn:u_val-mom}. We conjecture that this corresponds to Regime~4 in the univariate setting, where the detection delay is also of constant order up to logarithmic factors. The rates for the intermediate and large $\kappa$ regimes may not be tight, as they follow from the mean testing analysis. 

To further illustrate the above result, we briefly compare with the relevant literature, although scarce. \cite{li_2025} studied the minimax testing problem for offline high-dimensional mean change point detection. Their main results suggest that when the change is dense the minimax testing rates are essentially unaffected by the tail behaviour of the underlying distribution. This message is consistent with our detection-delay guarantees in the small \(\kappa\) regime, where the delay is not affected by the heavy-tailed noise compared to the Gaussian case.

We note that the parameter choices stated in Theorems \ref{thm:cpt_hd_subW} and \ref{thm:cpt_hd_fm} are designed to transfer the type I error and power guarantees
of the robust mean testing subroutines to the change point procedure. In particular, the minimum sample size threshold $(h_t)_{t\in\mathbb{N}}$ is chosen according to \eqref{eqn:min_h_t1} and \eqref{eqn:min_h_t2}, which are the minimum sample sizes for Algorithms~\ref{alg:efficient-mean-tester} and \ref{alg:efficient-mean-tester-mom} to detect a mean change of magnitude $\kappa$ with type I and II error control under $\mathcal{G}_{\theta,M}^p$ and $\mathcal{P}_{v,\phi}^p$, respectively. Thus, whenever a test is performed, the corresponding robust mean testing subroutine is applied with a sufficient sample size.
Practical guidance on choosing the thresholds in \eqref{eqn:min_h_t1} and \eqref{eqn:min_h_t2} is given in \Cref{subsect:sim_hdcpt}.

\section{Simulation studies}\label{sect:sim_study}
We present simulation results to empirically validate our theoretical predictions. Specifically, we investigate the empirical performance of \Cref{alg:cusum_rume} for univariate online change point detection (\Cref{subsect:univ_sim}), \Cref{alg:cpd_highd} with a slight variant for multivariate online change point detection (\Cref{subsect:sim_hdcpt}) and \Cref{alg:efficient-mean-tester} for multivariate mean testing (\Cref{subsect:hdtest}). The code for reproducing our experiments can be found at \url{https://github.com/edwintang903/online-robustcpt}.

\subsection{Univariate online change point detection}\label{subsect:univ_sim}
In this subsection, we run simulations to illustrate the four regimes of detection delay as presented in Tables \ref{table:delayG} and \ref{table:delayP}, using the proposed online change point detection procedure \Cref{alg:cusum_rume}. We consider $n = 2400$ independent observations $\{X_i\}_{i=1}^n$ in $\mathbb{R}$ arriving sequentially. For each $i \in \{1, \dots, n\}$, the observation is drawn from the distribution $Q_i$ defined in~\eqref{huber}, with the contamination proportion set to be $\varepsilon=0.1$. The outlier distribution $H_i$ is chosen to be a high-variance normal distribution $\mathcal{N}(0,100)$ and we consider the following two choices for the inlier distribution $F_i$:
\begin{itemize}
    \item $\text{Laplace}(\mu, b=\sqrt{10.5})$: a Laplace distribution with location parameter $\mu$ and scale parameter $b = \sqrt{10.5}$, so that the variance is $\sigma^2 = 2b^2 = 21$;
    \item $t(\mu,\sigma=\sqrt{21}, \nu=2.1)$: a rescaled Student’s $t$-distribution with $\nu = 2.1$ degrees of freedom, mean $\mu$ and variance $\sigma^2=21$. Specifically, if $Z \sim t_{\nu}$, we have $X = \sigma\sqrt{\frac{\nu-2}{\nu}}Z + \mu\sim t(\mu,\sigma, \nu)$.
\end{itemize}
These two distributions represent distinct heavy-tailed classes with the same variance: specifically, $\text{Laplace}(\mu, \sqrt{10.5}) \in \mathcal{G}_{1, \sqrt{10.5}}$ and $t(\mu, \sqrt{21}, 2.1) \in \mathcal{P}_{2,21}$. Using these distributions allows us to examine how detection delay varies under different tail behaviours.

We set the desired false alarm probability to be $\alpha=0.2$. In Algorithm~\ref{alg:cusum_rume}, we use the detection thresholds $\chi_{s,t}$ and $\zeta_{s,t}$ specified in~\eqref{eqn:subW_thres1} and~\eqref{eqn:subW_thres2}, respectively, for $t(\mu, \sqrt{21}, 2.1) \in \mathcal{P}_{2,21}$ and the thresholds specified in~\eqref{eqn:fm_thres1} and~\eqref{eqn:fm_thres2} for $\text{Laplace}(\mu, \sqrt{10.5}) \in \mathcal{G}_{1, \sqrt{10.5}}$. Note that for $\chi_{s,t}$, in~\eqref{eqn:subW_thres1} and~\eqref{eqn:fm_thres1}, we replace $M$ and $\phi$, respectively, by tunable constants to reduce the conservativeness of the procedure in practice. We then calibrate all constants appearing in these thresholds under the no-change setting so that the empirical false alarm rate is bounded by $\alpha = 0.2$.

To empirically evaluate the performance of Algorithm~\ref{alg:cusum_rume} under a change, we set a single change point at $\Delta = 600$, where the mean $f_i$ of the inlier distribution $F_i$ undergoes a jump of size $\kappa$:
$$
f_i=
\begin{cases}
    0 & \text{if } i \leq \Delta,\\
    \kappa & \text{if } i > \Delta.
\end{cases}
$$
For each jump size $\kappa$ in a suitably chosen grid, we conduct 2000 simulation runs to estimate the detection rate and the average detection delay. The detection rate denotes the probability that the detector raises an alarm after $\Delta$ but before the arrival of the last observation. For the Laplace distribution, the regime transitions are illustrated in \Cref{fig:th1plot}: panel (a) shows the detection rate as a function of $\kappa$ for small jump sizes, while panels (b) and (c) plot the mean detection delay against $\kappa$ for moderate and large jump sizes, respectively. Analogous plots for the $t$-distribution are presented in \Cref{fig:v2plot}. These figures empirically verify our theoretical results and reveal all four regimes for detection delay as shown in Tables~\ref{table:delayG} and~\ref{table:delayP}, for inlier distributions with exponentially- and polynomially-decaying tails, as we discuss in the following.

As shown in Figures~\ref{fig:th1(a)} and~\ref{fig:v2(a)}, when $\kappa = 0$, the empirical false alarm rate is controlled at the desired level $0.2$. For small signal-to-noise ratios, specifically $\kappa/\sigma \leq  0.03$ for $\text{Laplace}(\mu, \sqrt{10.5})$ and $\kappa/\sigma \leq 0.01$ for $t(\mu, \sqrt{21}, 2.1)$, the empirical probability of raising an alarm during the monitoring period is the same as the false alarm rate under no change. This indicates that, in this regime, the change point is effectively undetectable using our procedure, in agreement with Regime~1 of Tables~\ref{table:delayG} and~\ref{table:delayP}. As $\kappa/\sigma$ increases, the detection problem becomes easier, and the empirical non-detection probability decreases towards zero. Although the theoretical transition between Regimes~1 and~2 is sharp, the finite monitoring horizon results in a more gradual transition in the non-detection rate being observed in practice. Nevertheless, the phase transition behaviour remains clearly visible in the simulations.

Figures~\ref{fig:th1(b)} and~\ref{fig:v2(b)} illustrate a transition in the detection decay dynamics between Regimes 2 and 3. The key observation is that no single relationship between the mean detection delay and the jump size provides a good fit across the entire range of medium jump strength. For smaller values of $\kappa/\sigma$, an approximate inverse square law fits the data points, but this is no longer the case when the signal becomes larger. Towards the right part of both figures, we observe that the fit depends on the noise distribution: an inverse relationship $d \propto 1/\kappa$ for the Laplace distribution, and an inverse-logarithmic trend $d \propto 1/\log(\kappa)$ for the $t$-distribution. These observations agree with the order of the detection delay for Regimes 2 and 3 in Tables~\ref{table:delayG} and~\ref{table:delayP}.

As we increase $\kappa/\sigma$ further above 1, both Figures~\ref{fig:th1(c)} and~\ref{fig:v2(c)} demonstrate that the detection delay of our procedure eventually plateaus and no longer depends on the jump size, which agrees with Regime 4 in both tables.

\begin{figure}[htbp]
     \centering
     \begin{subfigure}[b]{0.48\textwidth}
         \centering
         \includegraphics[width=\textwidth]{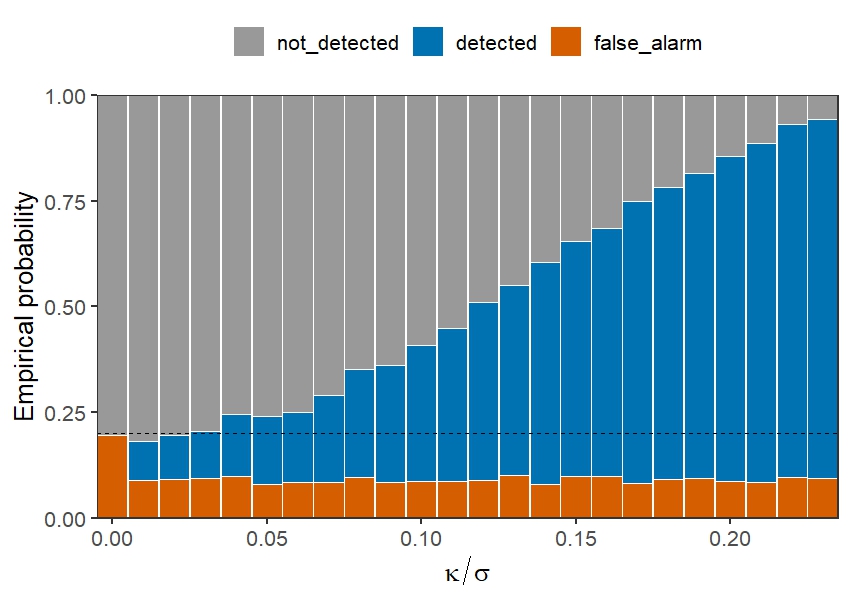}
         \caption{Empirical false alarm rates and detection rates against $\kappa/\sigma\in[0,0.23]$. For each value of $\kappa$, we record the proportion of simulation runs in which the algorithm stopped at or before time $\Delta=600$ (false\_alarm), between times $601$ and $2400$ (detected), or failed to stop by time $2400$ (not\_detected). The black dotted line represents the nominal false alarm rate control ($\alpha = 0.2$).}
         \label{fig:th1(a)}
     \end{subfigure}\\
     \begin{subfigure}[b]{0.48\textwidth}
         \centering
         \includegraphics[width=\textwidth]{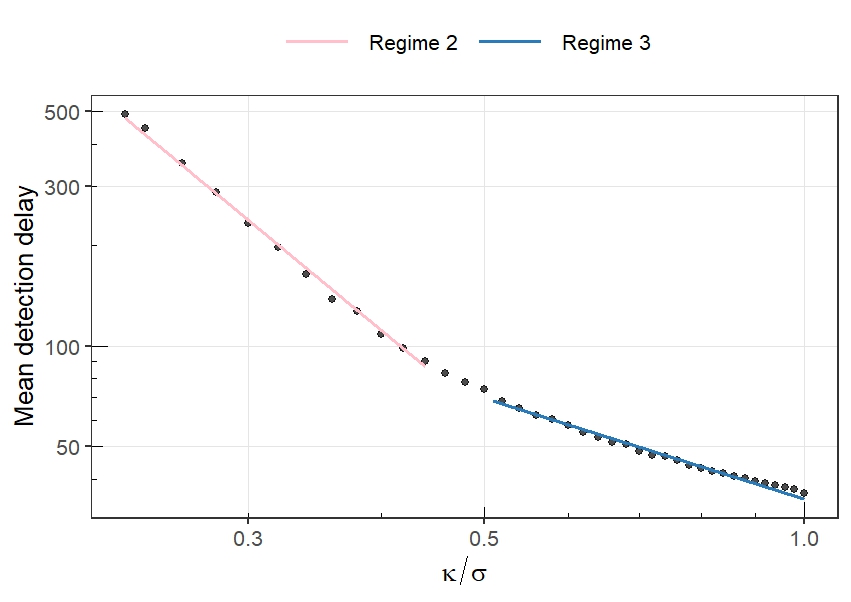}
         \caption{Mean detection delay ($d$) against $\kappa/\sigma\in[0.23,1]$. The pink curve on the left corresponds to a weighted least-squares fit over the small-$\kappa$ regime of the form $d \propto \kappa^{-2.63}$, while the blue curve on the right corresponds to a weighted least-squares fit over the moderate-$\kappa$ regime of the form $d \propto \kappa^{-1}$. In both cases, the weights are given by the inverse empirical standard deviation of the detection delay at each $\kappa$.}
         \label{fig:th1(b)}
     \end{subfigure}
     \hfill
     \begin{subfigure}[b]{0.48\textwidth}
         \centering
         \includegraphics[width=\textwidth]{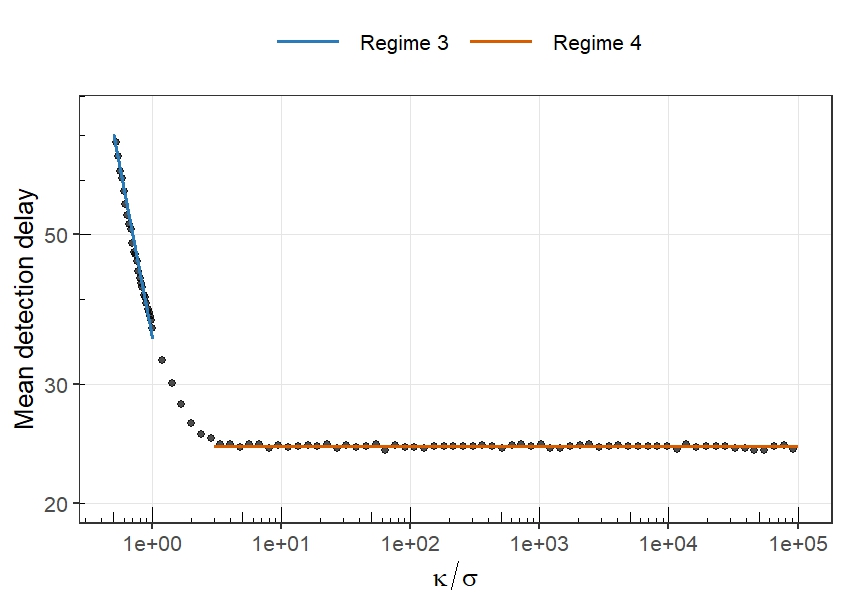}
         \caption{Mean detection delay ($d$) against $\kappa/\sigma\in[0.5,10^5]$. The blue curve on the left corresponds to a weighted least-squares fit over the moderate-$\kappa$ regime of the form $d \propto \kappa^{-1}$, while the blue curve on the right corresponds to a weighted least-squares fit over the larger-$\kappa$ regime of the form $d \propto 1$. In both cases, the weights are given by the inverse empirical standard deviation of the detection delay at each $\kappa$.}
         \label{fig:th1(c)}
     \end{subfigure}
     \caption{Illustration of different regimes (\Cref{table:delayP}) with the Laplace distribution as the inlier. Results are based on running \Cref{alg:cusum_rume} on 2000 simulated datasets for each value of $\kappa$.}
     \label{fig:th1plot}
\end{figure}

\begin{figure}[!htbp]
     \centering
     \begin{subfigure}[b]{0.48\textwidth}
         \centering
         \includegraphics[width=\textwidth]{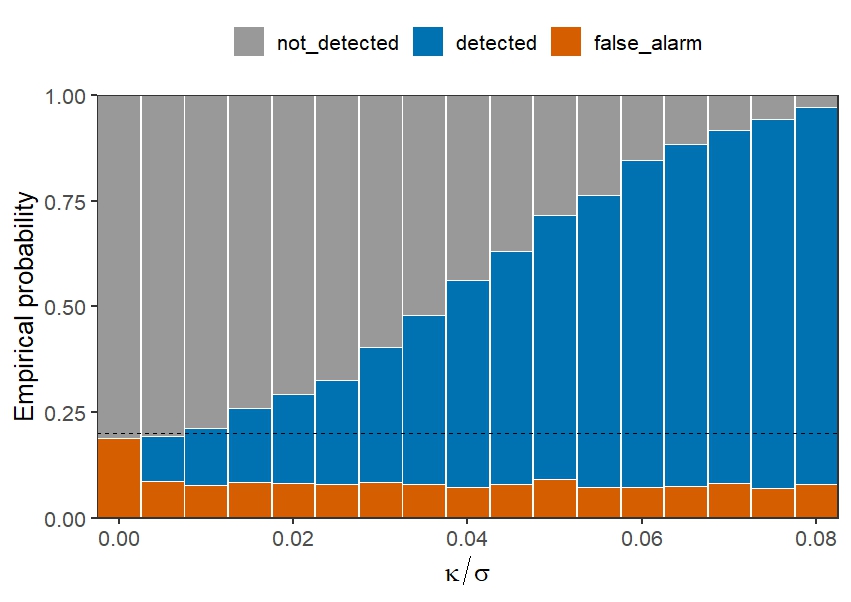}
         \caption{Empirical false alarm rates and detection rates against $\kappa/\sigma\in [0, 0.08]$. For each value of $\kappa$, we record the proportion of simulation runs in which the algorithm stopped at or before time $\Delta=600$ (false\_alarm), between times $601$ and $2400$ (detected), or failed to stop by time $2400$ (not\_detected). The black dotted line represents the nominal false alarm rate control ($\alpha = 0.2$).}
         \label{fig:v2(a)}
     \end{subfigure}\\
     \begin{subfigure}[b]{0.48\textwidth}
         \centering
         \includegraphics[width=\textwidth]{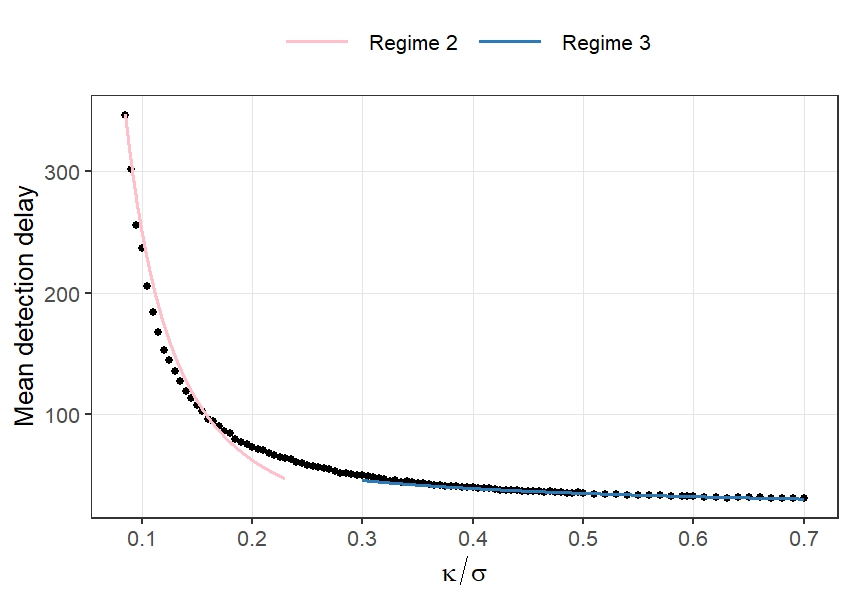}
         \caption{Mean detection delay ($d$) against $\kappa/\sigma\in [0.07, 0.7]$. The pink curve on the left corresponds to a weighted least-squares fit over the small-$\kappa$ regime of the form $d \propto \kappa^{-2}$, while the blue curve on the right corresponds to a weighted least-squares fit over the moderate-$\kappa$ regime of the form $d \propto (\log(4\kappa))^{-1}$. In both cases, the weights are given by the inverse empirical standard deviation of the detection delay at each $\kappa$.}
         \label{fig:v2(b)}
     \end{subfigure}
     \hfill
     \begin{subfigure}[b]{0.48\textwidth}
         \centering
         \includegraphics[width=\textwidth]{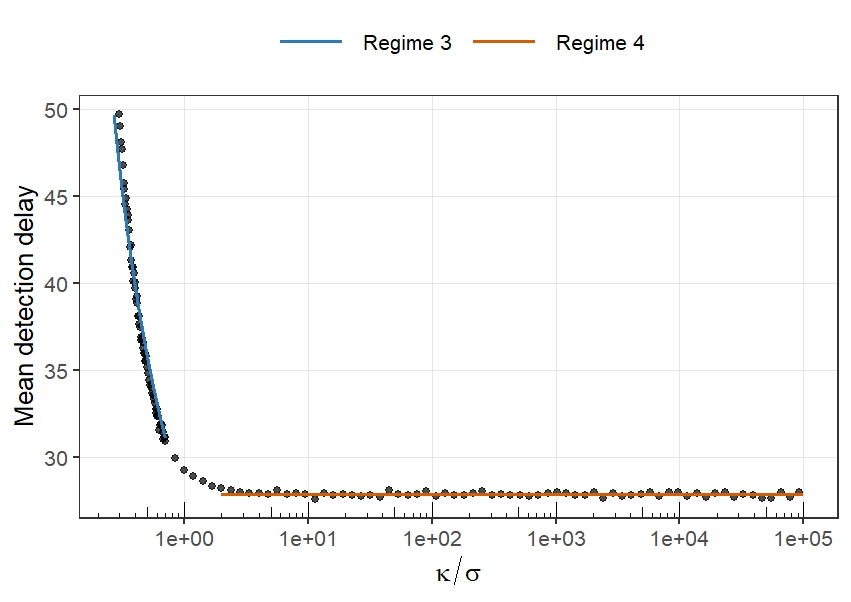}
         \caption{Mean detection delay ($d$) against $\kappa/\sigma\in[0.3,10^5]$. The blue curve on the left corresponds to a weighted least-squares fit over the moderate-$\kappa$ regime of the form $d \propto (\log(4\kappa))^{-1}$, while the blue curve on the right corresponds to a weighted least-squares fit over the larger-$\kappa$ regime of the form $d \propto 1$. In both cases, the weights are given by the inverse empirical standard deviation of the detection delay at each $\kappa$.}
         \label{fig:v2(c)}
     \end{subfigure}
     \caption{Illustration of the different regimes (\Cref{table:delayP}) with $t$-distribution as the inlier. Results are based on running \Cref{alg:cusum_rume} on 2000 simulated datasets for each value of $\kappa$.}
     \label{fig:v2plot}
\end{figure}

\subsection{Multivariate mean testing}\label{subsect:hdtest}

In this subsection, we empirically evaluate the robust mean testing algorithm 
(\Cref{alg:efficient-mean-tester}). In our simulations, we consider $n$ independent observations $\{X_i\}_{i=1}^n$ in $\mathbb{R}^p$. For each $i \in \{1, \dots, n\}$, the observation is drawn from the distribution $Q_i$ defined in~\eqref{eqn:huberhd}, with the contamination level set to be $\varepsilon=0.01$. The outlier distribution $H_i$ is chosen to be $\mathcal{N}_p(-\mathbf{1}_p, I)$ and we consider the following two choices for the inlier distribution $F_i$:
\begin{itemize}
    \item $(\text{Laplace}(\mu, 1/\sqrt{2}))^{\otimes p}$,
    \item $(t(\mu,\sigma=1,\nu=4.1))^{\otimes p}$.
\end{itemize}
These two distributions represent distinct heavy-tailed classes: specifically, $t(\mu, 1, 4.1) \in \mathcal{P}_{4,63}$ and $\text{Laplace}(\mu, 1/\sqrt{2}) \in \mathcal{G}_{1, \sqrt{0.5}}$. By using these distributions, we can quantify the difficulty of mean testing with respect to varying tail behaviours. Parameters for the distributions were chosen so that both distributions have variance 1.

The remaining algorithmic inputs are chosen as follows. The desired type~I error level is set to $\delta = 0.1$, and the detection sensitivity factor is $T_u = (1 - 6u)^2$ with $u$ chosen as specified in \Cref{alg:efficient-mean-tester}. Other parameters, including the sample size $n$, dimension $p$, signal size input $\kappa_0$, and filtering strength $C_\gamma$, are varied in the subsequent simulations. Since $\delta=0.1$ is large, we use \Cref{alg:efficient-mean-tester} rather than \Cref{alg:efficient-mean-tester-mom} here for distributions under finite-moment assumptions, which gives a lower signal strength threshold.

\subsubsection{Sensitivity analysis}\label{subsubsect:sensitivity}
We first evaluate the robustness of the algorithm to different choices of the signal size input $\kappa_0$ and filtering strength $C_\gamma$. We consider $n=500$ samples with dimension $p=\{10,600\}$. Recall that the filtering subroutine used in \Cref{alg:efficient-mean-tester} depends on whether $n > p$. Specifically, when $p=10$, we use \Cref{alg:spectralfilter1}, whereas when $p=600$, we use \Cref{alg:spectralfilter2}. For each value of $\kappa_0\in\{0.1,0.2,\ldots,2.5\}$ and $C_\gamma=\{0.01, 0.03, 0.05, 0.07, 0.1, 0.15, 0.2\}$, we run \Cref{alg:efficient-mean-tester} on 1000 simulated datasets to obtain the empirical type I error rate and power. To compute the type~I error, we set $\mu = \mathbf{0}$ and record the proportion of false rejections. To evaluate the power, we fix $\mu = (1/\sqrt{p}) \mathbf{1}_p$, corresponding to a true signal strength $\kappa = 1$, and compute the proportion of correct rejections. In \Cref{fig:hd_misspec}, we overlay the type~I error and power profiles for different values of $C_\gamma$ in a single panel to aid comparison. The corresponding faceted plots are provided in \Cref{fig:individual-cgamma} of \Cref{append:facet_plot}, where the same trends are shown separately for each value of $C_\gamma$.

Across most settings in \Cref{fig:hd_misspec}, the type I error rate decreases rapidly to zero as $\kappa_0$ increases from zero, whereas the power remains close to one over a wider range of $\kappa_0$ and begins to decrease only when $\kappa_0$ is close to, or exceeds, the true signal size $\kappa=1$. This contrast is primarily driven by the role of $\kappa_0$ in the rejection threshold. When $\kappa_0$ is small, the threshold is low, so rejection is easy under both the null and the alternative, which leads to both high type I error and high power. As $\kappa_0$ increases, the threshold becomes more stringent, reducing false rejections and hence lowering the type I error. Under the alternative, however, the non-zero mean inflates the test statistic, so the power remains high until the threshold becomes sufficiently large.

For small values of $C_\gamma$, namely $C_\gamma\in\{0.01,0.03\}$ in \Cref{fig:misspec(b)} and \Cref{fig:misspec(c)}, the behaviour differs from the general trend described above. When $C_\gamma=0.01$, both the type I error rate and the power are essentially zero throughout the range $\kappa_0\in[0.1,2.5]$. This is because a smaller $C_\gamma$ makes the filtering step too aggressive, causing nearly all weights to be set to zero. As a result, the test statistic is unable to trigger rejection. When $C_\gamma=0.03$, the power curve exhibits a bell-shaped pattern centred near the true signal size $\kappa=1$. This behaviour is due to the term $C_\gamma n\kappa_0^2$ in the filtering radius $R_{f}$ in \Cref{alg:efficient-mean-tester}. For small $\kappa_0$, the filtering radius is too small, so signal-bearing observations are filtered out too aggressively. As $\kappa_0$ increases, the radius becomes large enough to retain enough of these observations, and the power increases. For even larger $\kappa_0$, the rejection threshold becomes too stringent, and the power decreases for the same reason discussed above.

The effect of $C_\gamma$ itself is also clear from the plots. Increasing $C_\gamma$ shifts both the type I error and power curves to the right. This is because a larger $C_\gamma$ widens the filtering radius, allowing data points to retain larger weights and thereby inflating the test statistic under both the null and the alternative. Consequently, a larger $\kappa_0$ is required both to control the type I error under the null and to induce a decay in power under the alternative. Once $C_\gamma$ is sufficiently large, for example $C_\gamma\geq 0.07$, the type I error and power profiles effectively overlap. This suggests that the filtering radius is then large enough for the spectral filtering step to become effectively inactive, so further increases in $C_\gamma$ have little impact on the empirical behaviour of the test.

These observations are consistent with Propositions \ref{thm:poly-time-main} and \ref{thm:poly-time-main_fm}, which impose two requirements on the choice of $\kappa_0$: it must be sufficiently large for type I error control, and it must be of the same scale as the true signal size $\kappa$ for power control. Under $C_\gamma=0.1$, panels (a), (b), and (c) suggest that choosing $\kappa_0\in[0.8,1.4]$ achieves both type I error control and high power, while panel (d) suggests that choosing $\kappa_0\in[1,2.5]$ also achieves both. Motivated by these findings, we set $C_\gamma=0.1$ for the remaining simulations.

\begin{figure}[htbp]
     \centering
     \begin{subfigure}[b]{0.48\textwidth}
         \centering
         \includegraphics[width=0.86\textwidth]{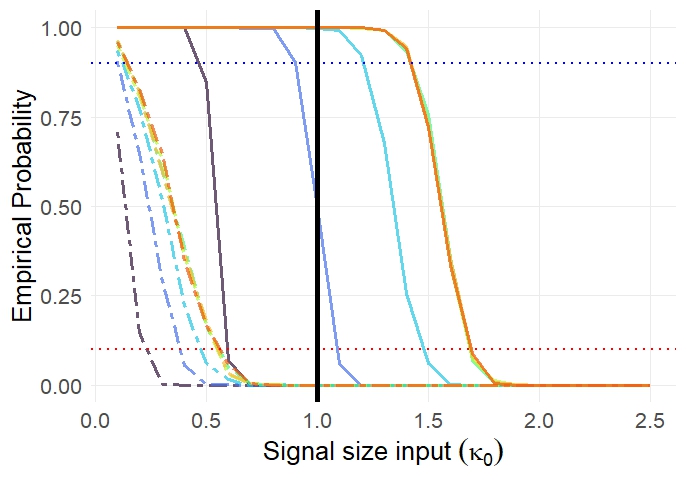}
         \caption{Laplace distribution, $p=10, \varepsilon=0.01$.}
         \label{fig:misspec(a)}
     \end{subfigure}
     \hfill
     \begin{subfigure}[b]{0.48\textwidth}
         \centering
         \includegraphics[width=\textwidth]{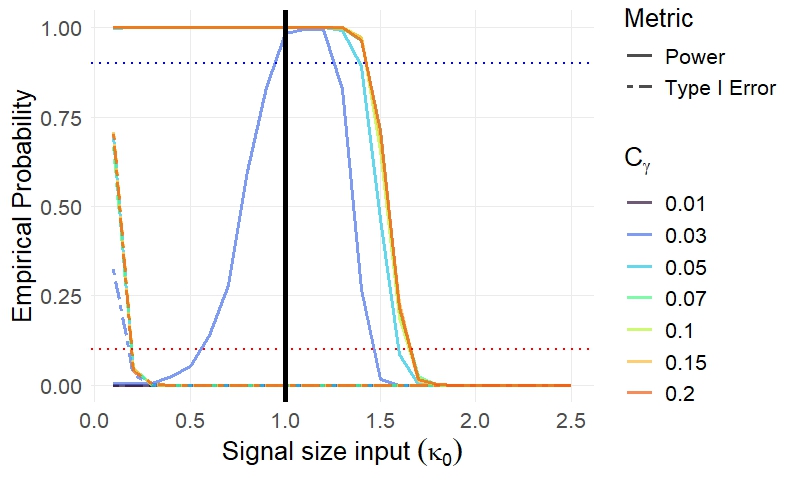}
         \caption{$t$-distribution, $p=10, \varepsilon=0.01$.}
         \label{fig:misspec(b)}
     \end{subfigure}
     \begin{subfigure}[b]{0.48\textwidth}
         \centering
         \includegraphics[width=0.86\textwidth]{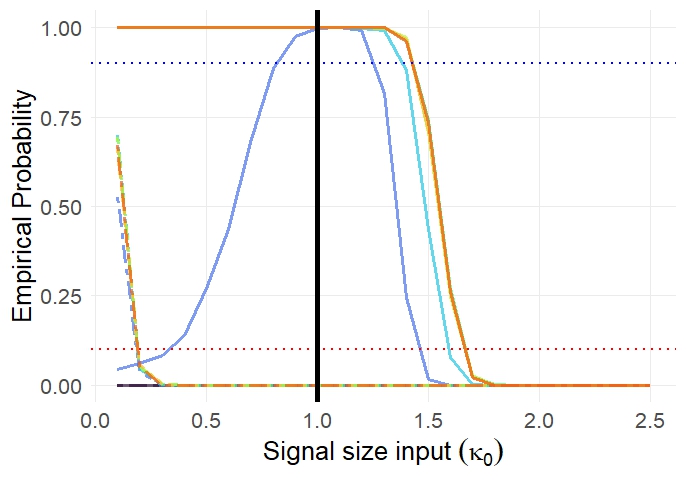}
         \caption{Laplace distribution, $p=600, \varepsilon=0.01$.}
         \label{fig:misspec(c)}
     \end{subfigure}
     \hfill
     \begin{subfigure}[b]{0.48\textwidth}
         \centering
         \includegraphics[width=\textwidth]{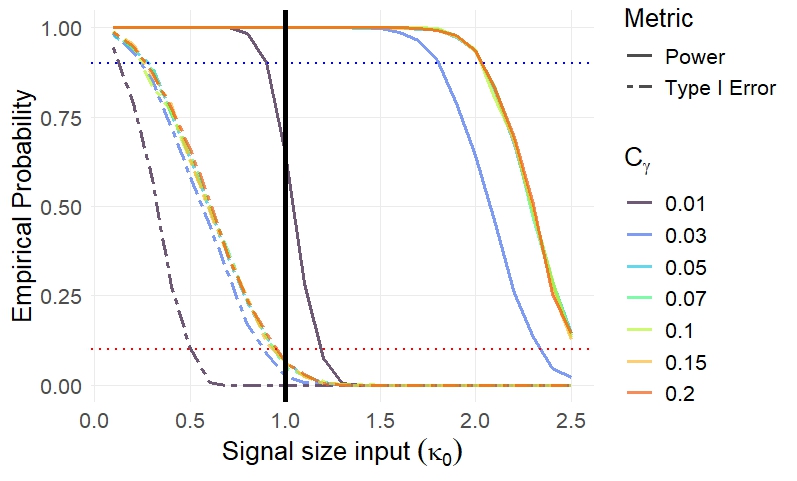}
         \caption{$t$-distribution, $p=600, \varepsilon=0.01$.}
         \label{fig:misspec(d)}
     \end{subfigure}
     \caption{Empirical type I error and power in testing between hypotheses in \eqref{eqn:hypotheses} using \Cref{alg:efficient-mean-tester} under different specifications of $\kappa_0$. True signal size is $\kappa=1$, indicated by a vertical black line. For each value of $\kappa_0$ and $C_\gamma$, we report the proportion of simulation runs which resulted in false rejections (type I error) and correct rejections (Power). The horizontal red line indicates the 10\% threshold for type I error control and the horizontal blue line indicates the 90\% threshold for power.}
     \label{fig:hd_misspec}
\end{figure}

\subsubsection{Detectability regimes}\label{subsubsect:detectable}

We empirically investigate the detectability of the signal strength $\kappa$ as a function of the sample size $n$ and the dimension $p$. We set $\kappa_0 = \kappa$, matching the true signal strength at each value of $\kappa$, and fix $C_\gamma = 0.1$. We define detection as reliable for $(n,\kappa,\delta)$ when both Type~I and Type~II error rates remain below $\delta$ for a signal of size $\kappa$ based on $n$ samples.

To characterise the detection threshold, we run \Cref{alg:efficient-mean-tester} over two parameter grids. In \Cref{fig:hd_detectable}, the top row varies $n\in \{400,500,\ldots,3000 \}$ and $\kappa\in\{0.15,0.16,\ldots,0.5\}$ with $p=100$ fixed, whereas the bottom row varies $p\in \{10,20,\ldots,500\}$ and $\kappa\in\{0.2,0.22,\ldots,0.7\}$ with $n=500$ fixed. The left and right columns correspond to the Laplace and $t$ distributions, respectively.

In all panels of \Cref{fig:hd_detectable}, the transition between the red (unreliable) and blue (reliable) regions provides an empirical estimate of the detection threshold, namely the minimum signal strength required for reliable detection as a function of $n$ or $p$. Holding $p$ fixed, Figures~\ref{subfig:detectable(a)} and~\ref{subfig:detectable(b)} show that this threshold decreases as $n$ increases. In particular, it exhibits a power-law scaling in $n$, with $\kappa\propto n^{-0.45}$ for both the Laplace and $t$-distribution settings. Meanwhile, holding $n$ fixed, Figures~\ref{subfig:detectable(c)} and~\ref{subfig:detectable(d)} show that the detection threshold increases with $p$. The empirical relationship is well approximated by $\kappa \propto p^{0.20}$ for the Laplace distribution and $\kappa \propto p^{0.18}$ for the $t$-distribution. These empirical findings are broadly consistent with the rate $\kappa \propto p^{0.25}n^{-0.5}$ established in \Cref{thm:poly-time-main,thm:poly-time-main_fm}. Although the full theoretical bounds contain additional terms, we conjecture that the parameter ranges considered in \Cref{fig:hd_detectable} lie in the regime $n\gtrsim \sqrt{p}$, where the term $\kappa \propto p^{0.25}n^{-0.5}$ dominates in the Laplace setting.

\begin{figure}[htbp!]
     \centering
     \begin{subfigure}[b]{0.48\textwidth}
         \centering
         \includegraphics[width=\textwidth]{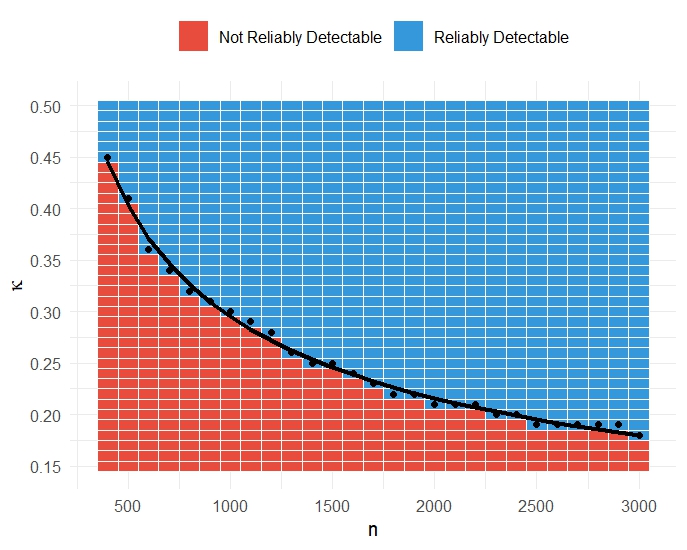}
         \caption{Laplace distribution inlier with $p=100$ fixed. The curve fitted to the boundary points follows $\kappa\propto n^{-0.45}$.}
         \label{subfig:detectable(a)}
     \end{subfigure}
     \hfill
     \begin{subfigure}[b]{0.48\textwidth}
         \centering
         \includegraphics[width=\textwidth]{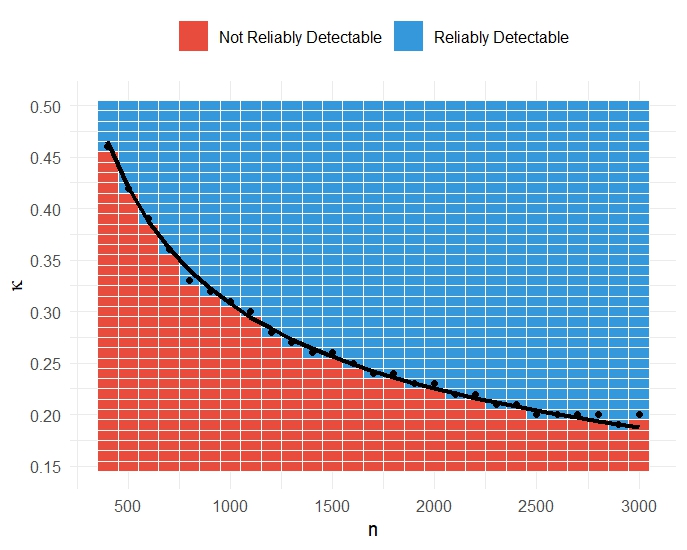}
         \caption{$t$-distribution inlier with $p=100$ fixed. The curve fitted to the boundary points follows $\kappa\propto n^{-0.45}$.}
         \label{subfig:detectable(b)}
     \end{subfigure}
     \begin{subfigure}[b]{0.48\textwidth}
         \centering
         \includegraphics[width=\textwidth]{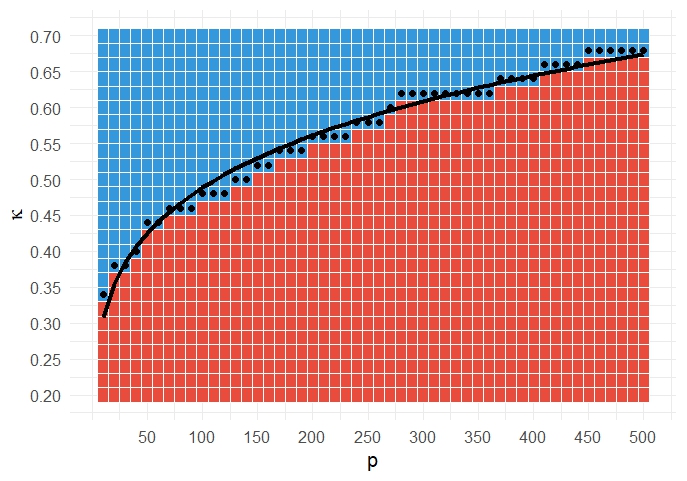}
         \caption{Laplace distribution inlier with $n=500$ fixed. The curve fitted to the boundary points follows $\kappa\propto p^{0.2}$.}
         \label{subfig:detectable(c)}
     \end{subfigure}
     \hfill
     \begin{subfigure}[b]{0.48\textwidth}
         \centering
         \includegraphics[width=\textwidth]{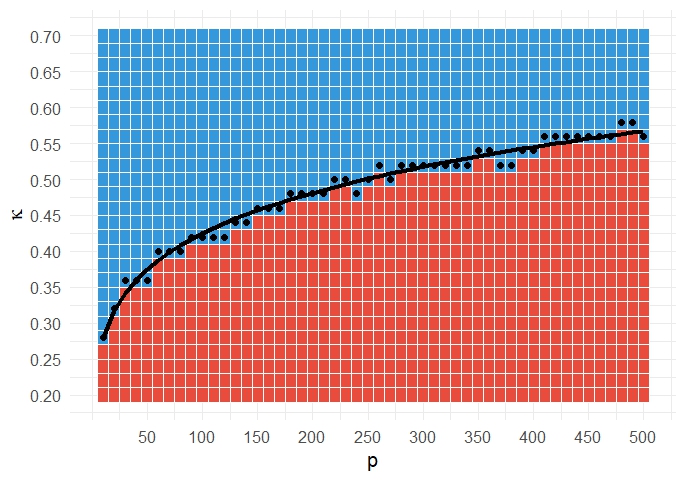}
         \caption{$t$-distribution inlier with $n=500$ fixed. The curve fitted to the boundary points follows $\kappa\propto p^{0.18}$.}
         \label{subfig:detectable(d)}
     \end{subfigure}
     \caption{Empirical minimum signal size requirement of \Cref{alg:efficient-mean-tester} under two inlier distributions: the Laplace distribution and the $t$-distribution. Panels (a)--(b) display detection reliability over a grid of pairs $(n,\kappa)$, with $n\in\{400,500,\ldots,3000\}$ and $\kappa\in\{0.15,0.16,\ldots,0.5\}$, while $p$ is fixed. Panels (c)--(d) display detection reliability over a grid of pairs $(p,\kappa)$, with $p\in\{10,20,\ldots,500\}$ and $\kappa\in\{0.2,0.22,\ldots,0.7\}$, while $n$ is fixed. Each cell is based on 1000 simulation runs: blue indicates reliable detection, defined by empirical Type~I and Type~II error rates at most $0.1$, while red indicates otherwise. A best-fit line is fitted to the blue cells adjacent to red cells, providing an empirical approximation of the detectability relationship with respect to $n$ or $p$.}
     \label{fig:hd_detectable}
\end{figure}

\subsection{Multivariate online change point detection}\label{subsect:sim_hdcpt}

In this subsection, we illustrate the four regimes of detection delay using a slight variant of the proposed multivariate change point detector, \Cref{alg:cpd_highd}; see \Cref{alg:cpd_highd_sim} in \Cref{append:sim_study} for the full modified algorithm. 
The main modification concerns the allocation of the false alarm probability. In \Cref{alg:cpd_highd}, the level $\alpha$ is distributed over all potential tests by a union bound, leading to a per-test level of order $t^{-2}(t+1)^{-1}$. This can be overly conservative since tests are only performed at times for which the outlier-control condition $u \leq \Omega$ holds. We therefore allocate the error budget only over executed tests. If $t_r$ denotes the
$r$-th time at which testing is performed, then for each potential change location $s \leq t_r/2$ considered at time $t_r$, we assign a significance level of order $\alpha t_r^{-1}r^{-2}$. Since there are at most $t_r$ such locations at time $t_r$, the total budget at the $r$-th executed testing time is $O(\alpha r^{-2})$, and hence the overall false alarm probability remains controlled by a union bound.

In our simulations, we consider $n$ independent observations $\{X_i\}_{i=1}^n$ in $\mathbb{R}^p$ arriving sequentially, where $p=10$. For each $i \in \{1, \dots, n\}$, the observation is drawn from the distribution $Q_i$ defined in~\eqref{huber}, with contamination proportion set to be $\varepsilon=0.01$. The outlier distribution $H_i$ is chosen to be $\mathcal{N}(-\mathbf{1}_{p}, I)$ and we consider the following two choices for the inlier distribution $F_i$:
\begin{itemize}
    \item $(\text{Laplace}(f_i, 1/\sqrt{2}))^{\otimes p}$;
    \item $(t(f_i,\sigma=1, \nu=4.1))^{\otimes p}$.
\end{itemize}

The inputs of \Cref{alg:cpd_highd_sim} are chosen as follows: the signal size parameter is set to $\kappa_0 = 0.5$, the filter strength to $C_\gamma = 0.1$, and the minimum sample size to $h_t = 1$. We set the desired false alarm probability to $\alpha = 0.1$. To achieve this level of control, we calibrate the remaining parameters (namely, the detection sensitivity factor $T_u$, the outlier control threshold $\Omega$ and the group number constant $K_c$) under the no-change setting so that the empirical false alarm rate does not exceed $\alpha$.

To empirically evaluate the performance of Algorithm~\ref{alg:cpd_highd_sim} under a change, we set a single change point at $\Delta$, where the mean $f_i$ of the inlier distribution $F_i$ undergoes a jump of magnitude $\kappa$:
$$f_i=\begin{cases}
    0 &\text{if $i\leq \Delta$,}\\
    (\kappa/\sqrt{p})\mathbf{1}_{p}  &\text{if $i > \Delta$.}
\end{cases}$$ 
For each jump size $\kappa$ in a suitably chosen grid, we conduct 1000 simulation runs to estimate the detection rate (defined in \Cref{subsect:univ_sim}) and the average detection delay. 

For the Laplace distribution, we set $n=2000$ and $\Delta=500$, with algorithmic inputs $T_u=(1-2.8u)^2$ and $\Omega=0.09$. The regime transitions are illustrated in \Cref{fig:hdcpt_detectable}: \Cref{fig:hd_th1(a)} shows the detection rate as a function of $\kappa$ for small jump sizes, while \Cref{fig:hd_th1(b)} plots the mean detection delay against $\kappa$ for larger jump sizes. Analogous plots for the $t$-distribution are presented in \Cref{fig:hdcpt_detectable_t}, where we set $n=7500$ and $\Delta=3000$, with inputs $T_u=(1-0.4u)^2$, $K_c=1$, and $\Omega=0.15$. The larger values of $n$ and $\Delta$ reflect the increased difficulty of robust detection under the $t$-distribution, which has only a finite fourth moment.

Firstly, Figures~\ref{fig:hd_th1(a)} and~\ref{fig:hd_v4(a)} show that the calibrated parameters control the false alarm probability at level $\alpha\leq 0.1$. As in Regime~1 of the univariate setting, the change point is effectively undetectable by our procedure when $\kappa$ is small. For example, when $\kappa\leq 0.05$ in \Cref{fig:hd_th1(a)} or $\kappa\leq 0.15$ in \Cref{fig:hd_v4(a)}, the empirical probability of raising an alarm during the monitoring period is essentially the same as the false alarm rate under no change. This illustrates the necessity of \Cref{assum: hd_cpt}: effective detection is only possible once the signal strength exceeds a minimum threshold. By contrast, once $\kappa\geq 0.5$, the change is detected with probability at least $0.9$ for both settings. This mirrors the transition between Regimes~1 and~2 observed in the univariate setting.

Secondly, within the detectable regime, Figures~\ref{fig:hd_th1(b)} and~\ref{fig:hd_v4(b)} show that the mean detection delay initially decreases rapidly as $\kappa$ increases, before the rate of decrease slows and eventually plateaus. This behaviour is consistent with the three regimes predicted by Theorems~\ref{thm:cpt_hd_subW} and~\ref{thm:cpt_hd_fm}, and mirrors Regimes~2--4 in the univariate setting. For moderate $\kappa$, the Laplace setting exhibits an intermediate decay that is well described by $d\propto \kappa^{-1}$. For larger $\kappa$, the detection delay stabilises at order one, as predicted by our theory in the strong-signal regime. 

\begin{figure}[htbp]
     \centering
     \begin{subfigure}[b]{0.48\textwidth}
         \centering
         \includegraphics[width=\textwidth]{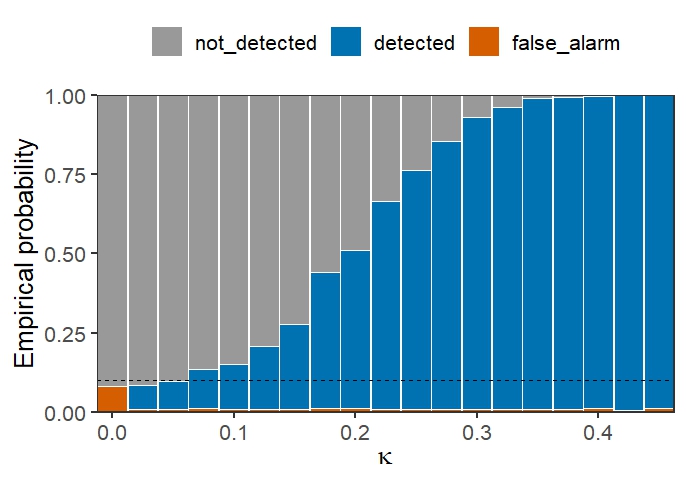}
         \caption{Empirical false alarm rates and detection rates against $\kappa\in [0, 0.45]$. For each value of $\kappa$, we record the proportion of simulation runs in which the algorithm stopped at or before time $\Delta=500$ (false\_alarm), between times $501$ and $2000$ (detected), or failed to stop by time $2000$ (not\_detected). The black dotted line represents the nominal false alarm rate control ($\alpha = 0.1$).}
         \label{fig:hd_th1(a)}
     \end{subfigure}
     \hfill
     \begin{subfigure}[b]{0.48\textwidth}
         \centering
         \includegraphics[width=\textwidth]{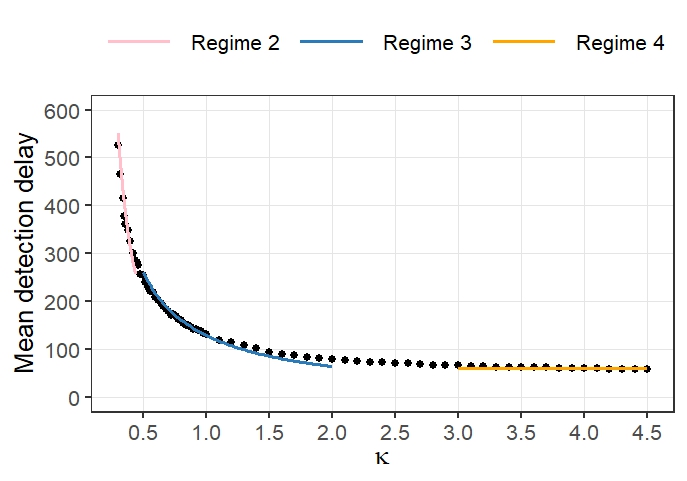}
         \caption{Mean detection delay $d$ as a function of \mbox{$\kappa \in [0.35,4.5]$}. The fitted curves illustrate the three predicted regimes: $d \propto \kappa^{-2}$ for small $\kappa$ (pink), $d \propto \kappa^{-1}$ for intermediate $\kappa$ (blue), and $d \propto 1$ for large $\kappa$ (orange). The fits are obtained by weighted least squares, with weights equal to the inverse empirical standard deviation of the detection delay at each value of $\kappa$.}
         \label{fig:hd_th1(b)}
     \end{subfigure}
    \caption{Illustration of the different regimes with inlier distribution $(\text{Laplace}(f_i, 1/\sqrt{2}))^{\otimes p}$ and $p=10$. Results are based on running \Cref{alg:cpd_highd_sim} on 1000 simulated datasets for each value of $\kappa$.}
    \label{fig:hdcpt_detectable}
\end{figure}

\begin{figure}[htbp]
     \centering
     \begin{subfigure}[b]{0.48\textwidth}
         \centering
         \includegraphics[width=\textwidth]{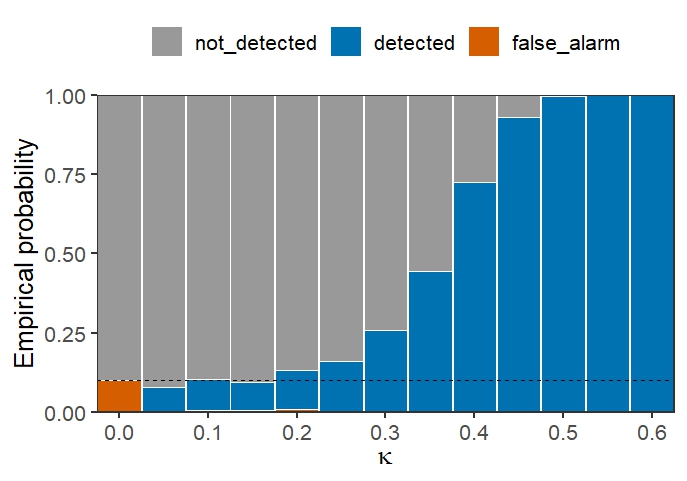}
         \caption{Empirical false alarm rates and detection rates against $\kappa\in[0,0.6]$. For each $\kappa$, we record the proportion of simulation runs in which the algorithm stopped at or before time $\Delta=3000$ (false\_alarm), between times $3001$ and $7500$ (detected), or failed to stop by time $7500$ (not\_detected). The black dotted line represents the nominal false alarm rate control ($\alpha = 0.1$).}
         \label{fig:hd_v4(a)}
     \end{subfigure}
     \hfill
     \begin{subfigure}[b]{0.48\textwidth}
         \centering
         \includegraphics[width=\textwidth]{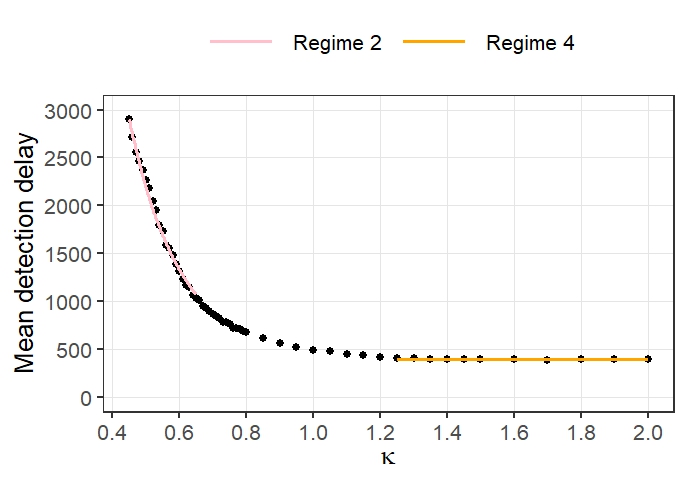}
         \caption{Mean detection delay $d$ as a function of \mbox{$\kappa \in [0.45,2]$}. The fitted curves illustrate the three predicted regimes: $d \propto \kappa^{-2.7}$ for small $\kappa$ (pink) and $d \propto 1$ for large $\kappa$ (orange). The fits are obtained by weighted least squares, with weights equal to the inverse empirical standard deviation of the detection delay at each value of $\kappa$.}
         \label{fig:hd_v4(b)}
     \end{subfigure}
     \caption{Illustration of the different regimes with inlier distribution $(t(f_i,1,4.1))^{\otimes p}$ and $p=10$. Results are based on running \Cref{alg:cpd_highd_sim} on 1000 simulated datasets for each value of $\kappa$.}
     \label{fig:hdcpt_detectable_t}
\end{figure}

\section{Discussion}\label{sect:discuss}
In this work, we studied online mean change point detection under both Huber contamination and heavy-tailed inlier distributions.  For univariate data, we characterised the detection delay across different regimes of signal strength, contamination level and tail behaviour. For multivariate data, we developed robust mean testing procedures and used them as subroutines for online change point detection. Several questions remain open.

First, our procedures require prior knowledge of the contamination level $\varepsilon$.  A conservative upper bound can be used in practice, but this may lead to unnecessarily large thresholds and hence larger detection delays.  One possible direction is to adapt results on robust mean estimation with unknown contamination level without sacrificing the order of the estimation error \citep[e.g.][]{jain2022}.  For example, in the univariate setting, one could run the median and RUME procedures over a geometric grid of candidate contamination levels, and then select a mean estimator by intersecting the corresponding confidence sets. This approach, however, is currently better suited to estimation than testing.  Extending it to robust testing and subsequently to multivariate change point detection is an interesting direction for future work. 
Second, the multivariate rates obtained in the moderate- and large-signal regimes may not be sharp. This is largely a consequence of relying on the matrix Bernstein inequality, which may not have fully exploited the tail structure of sub-Weibull or finite-moment inliers.  Sharper bounds would likely require a more refined  understanding of the spectral properties of empirical covariance matrices under heavy-tailed assumptions.
Third, our multivariate theory assumes coordinate-wise independence.  This assumption allows us to exploit existing concentration properties of univariate sub-Weibull random variables.  Removing it would require new concentration tools for dependent heavy-tailed random vectors, especially under contamination. Developing such tools would broaden the applicability of the proposed methods.
Finally, while the present paper focuses primarily on statistical optimality and minimax detection delay, it would be useful to design robust online procedures with lower update and storage costs.

\section*{Acknowledgements}

Tang is supported by the Chancellors' Scholarship scheme and the Statistics Centre for Doctoral
Training at the University of Warwick.  Yu is partially supported by the Philip Leverhulme Prize
and EPSRC programme grant EP/Z531327/1. The authors would like to thank Tengyao Wang for helpful discussions.

\bibliographystyle{apalike}
\bibliography{ref}

\newpage
\appendix

\setcounter{section}{0}
\setcounter{theorem}{0}
\setcounter{definition}{0}
\setcounter{equation}{0}
\setcounter{figure}{0}
\setcounter{table}{0}
\setcounter{algorithm}{0}
\setcounter{remark}{0}

\renewcommand{\thesection}{S\arabic{section}}   
\renewcommand{\thetheorem}{S\arabic{theorem}}
\renewcommand{\thedefinition}{S\arabic{definition}}
\renewcommand{\thelemma}{S\arabic{lemma}}
\renewcommand{\theproposition}{S\arabic{proposition}}
\renewcommand{\thecorollary}{S\arabic{corollary}}
\renewcommand{\theequation}{S\arabic{equation}}  
\renewcommand{\thefigure}{S\arabic{figure}}
\renewcommand{\thetable}{S\arabic{table}}
\renewcommand{\thealgorithm}{S\arabic{algorithm}}
\renewcommand{\theremark}{S\arabic{remark}}

\section*{Appendices}
The proofs of all theoretical results are presented in the Appendices. A summary of proof dependencies is presented in \Cref{sect:roadmap}. \Cref{sect:univproof} contains proofs of the univariate change point results, namely  Theorems \ref{thm:univ_lower}--\ref{thm:rume_cpt}. \Cref{sect:hdtestproof} contains technical details of the subroutines of \Cref{alg:efficient-mean-tester}, as well as proofs of the multivariate testing results, namely Propositions \ref{thm:poly-time-main}--\ref{thm:poly-time-main-mom}. \Cref{sect:hdcptproof} collects the proofs of Theorems \ref{thm:cpt_hd_subW}--\ref{thm:cpt_hd_fm} for the multivariate change point problem. The full pseudocode for the practical multivariate change point algorithm and supplementary plots for \Cref{subsect:hdtest} can be found in \Cref{sect:hdcptalgo}. \Cref{sect:auxillarylem} contains auxiliary results.

\addcontentsline{toc}{section}{Appendices}
\renewcommand{\contentsname}{Content of Appendices}
\tableofcontents
\addtocontents{toc}{\protect\setcounter{tocdepth}{7}}

\newpage
\section{Summary of proof dependencies}\label{sect:roadmap}
We outline the roadmap of our proofs using directed acyclic graphs. An arrow from Result A to Result B (denoted $A \to B$) indicates that Result A is a prerequisite for the proof of Result B.

\subsection{Proofs for Section \ref{sect:univariate}}
\begin{figure}[H]
\centering
\begin{tikzpicture}[
    node distance=1.2cm and 1cm,
    thm_box/.style={rectangle, draw=blue!80, fill=blue!5, very thick, minimum width=2.5cm, font=\bfseries},
    pre_box/.style={rectangle, draw=black!60, fill=gray!5, rounded corners=2pt, minimum width=2.5cm}
]

\node (univ_low) [thm_box] {\Cref{thm:univ_lower}};
\node (p1) [pre_box, left=of univ_low] {\Cref{thm:low}};
\node (p2) [pre_box, above left=of univ_low] {\Cref{prop:meanminimax}};
\node (p3) [pre_box, above=of univ_low] {\Cref{thm:yu}};
\node (p4) [pre_box, above right=of univ_low] {\Cref{prop:fm_lowbd}};
\node (p5) [pre_box, right=of univ_low] {\Cref{prop:subW_lowbd}};
\node (p6) [pre_box, below=of univ_low] {\Cref{thm: regime4_lowerbd}};

\foreach \n in {p1,p2,p3,p4,p5,p6} \draw [->] (\n) -- (univ_low);

\end{tikzpicture}
\caption{Dependencies of the proof of \Cref{thm:univ_lower}}
\end{figure}

\begin{figure}[H]
    \centering
\begin{tikzpicture}[
    node distance=1.2cm and 1cm,
    thm_box/.style={rectangle, draw=blue!80, fill=blue!5, very thick, minimum width=2.5cm, font=\bfseries},
    pre_box/.style={rectangle, draw=black!60, fill=gray!5, rounded corners=2pt, minimum width=2.5cm, font=\small}
]

\coordinate (center_axis) at (0,0);

\node (rume_cpt) [thm_box, left=0.75cm of center_axis] {\Cref{thm:rume_cpt}};
\node (rume_subw) [thm_box, right=0.75cm of center_axis] {\Cref{thm:rume_cpt_subW}};

\node (s1) [pre_box, below=of center_axis] {\Cref{prop2}};
\node (s2) [pre_box, above=of center_axis] {\Cref{lem:rume_moment}};
\node (var_shift) [pre_box, left=of s2] {\Cref{lem:variance_shift}};

\node (s3) [pre_box, right=of rume_subw] {\Cref{thm:median_subWeib}};
\node (subW_prop) [pre_box, right=of s2] {\Cref{thm:subWeib}}; 
\node (s4) [pre_box, left=of rume_cpt] {\Cref{thm:median_moment}};

\draw [->] (var_shift) -- (s2);
\draw [->] (subW_prop) -- (s2);

\draw [->] (s1) -- (rume_cpt.south);
\draw [->] (s1) -- (rume_subw.south);

\draw [->] (s2) -- (rume_cpt.north);
\draw [->] (s2) -- (rume_subw.north);

\draw [->] (subW_prop) -- (s3);
\draw [->] (s3) -- (rume_subw);
\draw [->] (s4) -- (rume_cpt);

\end{tikzpicture}
\caption{Dependencies of the proofs for Theorems \ref{thm:rume_cpt_subW} and \ref{thm:rume_cpt}}
\end{figure}
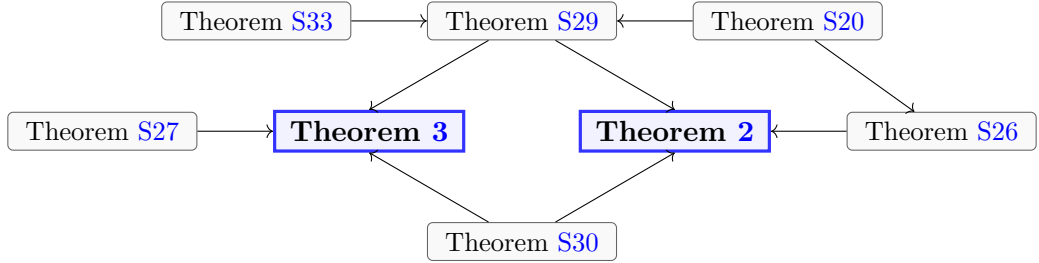

\subsection{Proofs for Section \ref{sect:hdtest} and \ref{sect:high_dim_cpt}}
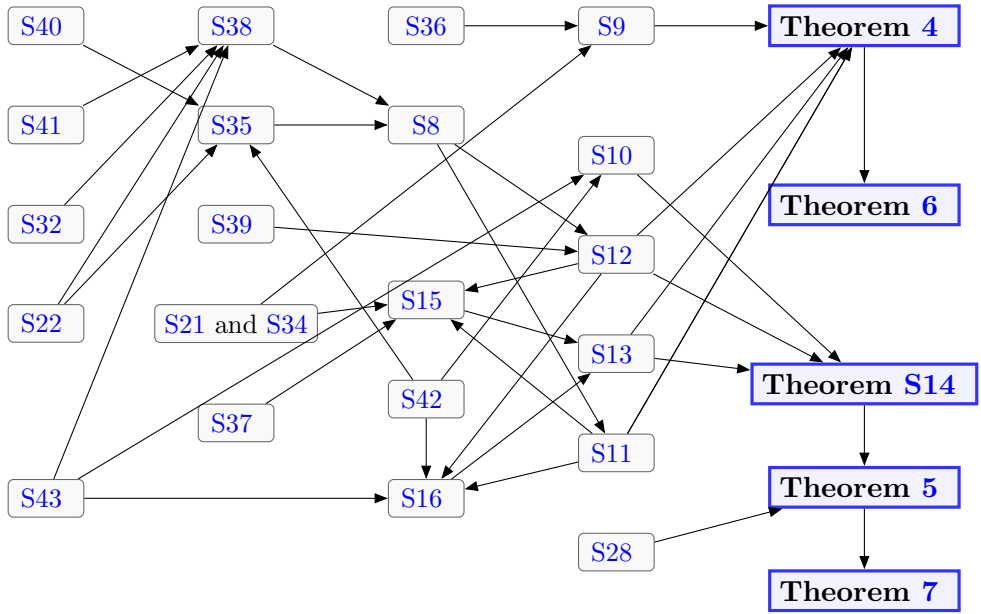
\begin{figure}[H]
    \centering
\begin{tikzpicture}[
    node distance=0.8cm and 1.5cm,
    thm_box/.style={rectangle, draw=blue!80, fill=blue!5, very thick, minimum width=2.5cm, font=\bfseries},
    pre_box/.style={rectangle, draw=black!60, fill=gray!5, rounded corners=2pt, minimum width=1cm, font=\small},
    align=center,
    >={Stealth[inset=0pt, length=5pt]}
]

\node (vec_bern)  [pre_box] { \ref{thm:vector_berstein} };
\node (mat_bern)  [pre_box, below=of vec_bern]  { \ref{thm:matrix_bernstein} };
\node (rosenthal) [pre_box, below=of mat_bern] { \ref{lemma:rosenthal} };
\node (quad_weib) [pre_box, below=of rosenthal] { \ref{prop:quadraticweibulltail} };
\node (app_cov)    [pre_box, below=of quad_weib, yshift=-1cm] { \ref{lem:approx_cov} };

\node (cov_conc)   [pre_box, right=of vec_bern] { \ref{lem:cov_matrix_conc} };
\node (normbound)  [pre_box, below=of cov_conc] { \ref{lem:normbound} };
\node (fact_small) [pre_box, below=of normbound] { \ref{fact:small-subset-deviations} };
\node (cond_red)   [pre_box, below=of fact_small] {\ref{lemma:weib_concentration} and \ref{lem:conditioning_reduction}};
\node (crossterm) [pre_box, below=of cond_red] { \ref{lem:crossterm} };

\node (norm_subw)  [pre_box, right=of cov_conc] { \ref{lem:normbound_subW}};
\node (cov_matrix0) [pre_box, below=of norm_subw] { \ref{lem:cov_matrix_conc0}};
\node (rowsum1)    [pre_box, below=of cov_matrix0, yshift=-1cm] { \ref{lem:rowsumfilter-1} };
\node (app_mean)   [pre_box, below=of rowsum1] { \ref{lem:approx_mean} };
\node (rowsum1b)   [pre_box, below= of app_mean] { \ref{lem:rowsumfilter-1b} };

\node (reg_weights)  [pre_box, right=of norm_subw] { \ref{lem:regular-weights-imply-tester} };
\node (reg_weights4) [pre_box, below=of reg_weights, yshift=-0.4cm] { \ref{lem:regular-weights-imply-tester4mom} };
\node (spectral)     [pre_box, below=of reg_weights4] { \ref{lem:spectralfilter2} };
\node (rowsum)       [pre_box, below=of spectral] { \ref{lem:rowsumfilter} };
\node (spec1)        [pre_box, below=of rowsum] { \ref{lem:spectralfilter1} };
\node (prop_med)     [pre_box, below=of spec1] { \ref{prop:med} };

\node (main_main) [thm_box, right=of reg_weights]  { \Cref{thm:poly-time-main} };
\node (cpt_subW) [thm_box, below=of main_main, yshift=-1cm]  { \Cref{thm:cpt_hd_subW} };
\node (main_fm)   [thm_box, below=of cpt_subW, yshift=-1cm] { \Cref{thm:poly-time-main_fm} };
\node (main_mom)  [thm_box, below=of main_fm]  { \Cref{thm:poly-time-main-mom} };
\node (cpt_fm) [thm_box, below=of main_mom]  { \Cref{thm:cpt_hd_fm} };

\draw [->] (vec_bern)  -- (normbound);
\draw [->] (app_mean)  -- (normbound);
\draw [->] (quad_weib) -- (normbound);

\draw [->] (mat_bern)  -- (cov_conc);
\draw [->] (app_cov)   -- (cov_conc);
\draw [->] (rosenthal) -- (cov_conc);
\draw [->] (quad_weib) -- (cov_conc);

\draw [->] (main_main) -- (cpt_subW);
\draw [->] (main_mom) -- (cpt_fm);

\draw [->] (cov_conc) -- (cov_matrix0);
\draw [->] (normbound) -- (cov_matrix0);

\draw [->] (cov_matrix0) -- (spec1);
\draw [->] (cov_matrix0) -- (spectral);
\draw [->] (fact_small)  -- (spectral);
\draw [->] (rowsum1)     -- (rowsum);
\draw [->] (rowsum1b)    -- (rowsum);

\draw [->] (reg_weights) -- (main_main);
\draw [->] (spectral)    -- (main_main);
\draw [->] (rowsum)      -- (main_main);
\draw [->] (spec1)       -- (main_main);

\draw [->] (reg_weights4) -- (main_fm);
\draw [->] (spectral)     -- (main_fm);
\draw [->] (rowsum)       -- (main_fm);
\draw [->] (spec1)       -- (main_main);

\draw [->] (main_fm)  -- (main_mom);
\draw [->] (prop_med) -- (main_mom);

\draw [->] (cond_red)  -- (reg_weights);
\draw [->] (norm_subw) -- (reg_weights);

\draw [->] (app_mean) -- (reg_weights4);
\draw [->] (app_cov)  -- (reg_weights4);

\draw [->] (cond_red) -- (rowsum1);
\draw [->] (spec1) -- (rowsum1);
\draw [->] (spectral) -- (rowsum1);
\draw [->] (crossterm) -- (rowsum1);

\draw [->] (app_mean) -- (rowsum1b);
\draw [->] (app_cov) -- (rowsum1b);
\draw [->] (spec1) -- (rowsum1b);
\draw [->] (spectral) -- (rowsum1b);
\end{tikzpicture}
\caption{Dependencies of the proofs for Propositions \ref{thm:poly-time-main}-- \ref{thm:poly-time-main-mom} and Theorems \ref{thm:cpt_hd_subW}--\ref{thm:cpt_hd_fm}.}
\end{figure}

\newpage
\section{Proofs of results in Section \ref{sect:univariate}}\label{sect:univproof}
\subsection{Proof of Theorem \ref{thm:univ_lower}}
The lower bounds in the five cases can be shown via the following propositions. 

\subsubsection*{Regime 1a}
Regime 1 characterises the information-theoretic lower bounds of the problem. It addresses the scenarios where Assumptions~\ref{assume-2} or~\ref{assume-1} are violated. To illustrate Regime~1a, \Cref{thm:low} shows that when contamination level is large relative to the signal strength, there exist inlier and contamination distributions such that any change point estimator is inconsistent.

\begin{proposition} \label{thm:low}
Suppose $\{X_i\}_{i\in\mathbb{N}}$ satisfies Assumption \ref{assum:kappa}, where $\varepsilon>0$ and
\begin{align*}
    Q_i=\begin{cases}
        (1-\varepsilon)\delta(0)+\varepsilon \delta(\kappa /\varepsilon)=q_a, &\text{if $i\leq\Delta$,} \\
        (1-\varepsilon)q_a+\varepsilon q_a, &\text{if $i>\Delta$,}
    \end{cases}
\end{align*}
where $\kappa\geq 0$. Denote inlier distributions as $F_1=\delta(0)$ and $F_{\Delta+1}=q_a$.
\begin{enumerate}[label=(\textbf{\alph*})]
    \item Assume that $\kappa\leq2^{-1/v}\phi\varepsilon^{1-1/v}$ for some $\phi>0$ and $v\in\mathbb{N}$. Then $F_i\in \mathcal{P}_{v,\phi}$, which implies that~$Q\in \Theta(\Delta,\kappa,\mathcal{P}_{v,\phi})$. 
    \item Assume that $\kappa/M\leq {\varepsilon} \{\log[1/(2\varepsilon)]\}^{1/\theta}$ for some $M>0$, $\theta>0$ and $\varepsilon\leq C_\theta$, where $C_\theta>0$ is a constant depending on $\theta$. Then $F_i\in\mathcal{G}_{\theta,M}$, which implies that $Q\in \Theta(\Delta,\kappa,\mathcal{G}_{\theta,M})$.
\end{enumerate}
With this choice of Q, we have
\begin{align*}
        \underset{\hat{t} \in \mathcal{T}(\alpha)}{\inf} \Pb_{Q} \left\{ \hat{t}=\infty
        \right\} \geq 1- \alpha.
    \end{align*}
\end{proposition}
The idea of the proof is that when signal-to-noise ratio is very low, contamination distributions for the pre-change and post-change distributions can be constructed such that $Q_i$ are identical for all $i\in\mathbb{N}$. Thus, we are unable to tell when a change in mean has occurred.
 
\begin{proof}
    
    (a) We first note that the difference in means before and after the change point satisfies
    $$|\E_{X\sim F_1}[X]-\E_{X\sim F_{\Delta+1}}[X]|=\kappa.$$
    Then, we show that $F_1$ and $F_{\Delta+1}$ have a finite absolute $v$th central moment upper bounded by $\phi^{v}$. It is straightforward to show that $\E_{X\sim F_1}[|X-\E[X]|^v]=0$. For $F_{\Delta+1}$, we have
    \begin{align*}
    \E_{X\sim F_{\Delta+1}}[|X-\E[X]|^v]&=(1-\varepsilon)|0-\kappa|^v+\varepsilon|\kappa/\varepsilon-\kappa|^v\\
    &\leq \kappa^v(1+\varepsilon^{1-v})\leq2\kappa^v\varepsilon^{1-v}\leq\phi^{v},
    \end{align*}
    where in the second inequality, we used the fact that $\varepsilon^{1-v}\geq(1/2)^{1-v}\geq 1$ and the final inequality used the assumption $\kappa\leq2^{-1/v}\phi\varepsilon^{1-1/v}$.

    (b) Next, we show that $F_1$ and $F_{\Delta+1}$ are sub-weibull$(\theta)$ with the Orlicz norm $M$.
    It is straightforward to show that for $X\sim F_1$, $\norm{X-\E[X]}_{\psi_\theta}=0$. For $X\sim F_{\Delta+1}$, 
    \begin{align*}
    \E_{X\sim F_{\Delta+1}}[\exp(|X-\E[X]|/M)]&=(1-\varepsilon)e^{(|0-\kappa|/M)^\theta}+\varepsilon e^{({|\kappa/\varepsilon-\kappa|}/{M})^\theta}\leq3/2+1/2=2, 
    \end{align*} 
    allowing us to conclude that $\norm{X-\E[X]}_{\psi_\theta}\leq M$. For the calculation of Orlicz norm above, we used the inequality
    $$(\kappa/M)^\theta\leq \left(\frac{\varepsilon}{1-\varepsilon}\right)^{\theta} \log \frac{1}{2\varepsilon}\leq \log \frac{3}{2},$$
    where the first inequality holds by the assumption $\kappa/M\leq {\varepsilon} (\log(1/(2\varepsilon)))^{1/\theta}$ and the second inequality holds for $0\leq\varepsilon<C_\theta$ where $C_\theta$ is a constant that depends on $\theta$.
    
    (c) Since $Q_1=Q_{\Delta+1}$, this implies $\{X_i\}_{i\in\mathbb{N}}$ is i.i.d.~under the above restriction on $\kappa$ regardless of the value of $\Delta$. Thus, for any $T\in \mathcal{T}(\alpha)$, we must have
   $$\Pb_{\Delta}(T<\infty)\leq \Pb_{\infty}(T<\infty)<\alpha.$$
\end{proof}

\subsubsection*{Regime 1b}
At the same time, \Cref{prop:meanminimax} establishes the impossibility of detection for the regime where $\kappa \lesssim \Delta^{-1/2}$. The proof proceeds by restricting attention to Gaussian inliers without contamination $(\varepsilon=0)$, and showing that any procedure satisfying the false-alarm constraint can have arbitrarily large detection delay with non-zero probability. To be precise, we consider the family of normal distributions with fixed variance $\sigma^2$, and denote the class of distributions by 
$$\mathcal{N}_\sigma = \{ \mathcal{N}(\mu, \sigma^2) : \mu \in \mathbb{R} \}.$$ 
Since $\mathcal{N}_{\sigma}\subset\mathcal{G}_{\theta,M}$ whenever $M\geq C_1\sigma$, and $\mathcal{N}_{\sigma}\subset\mathcal{P}_{v,\phi}$ whenever $\phi\geq C_2\sigma$, for large enough constants constants $C_1, C_2>0$, the same lower bound extends to these larger distribution classes.

We define the following notation for the proof. Since any measure $P \in \Theta(\Delta, \kappa, \mathcal{N}_\sigma)$ is uniquely determined by its mean sequence $f \in \mathcal{S}(\Delta, \kappa)$, we denote the probability measure associated with a specific sequence $f$ as $\mathbb{P}_f$. This identification similarly holds for the null space, where each $P_0 \in \Theta_0(\mathcal{N}_\sigma)$ is uniquely associated with a constant mean sequence $f \in \mathcal{S}_0$.

\begin{proposition}[\citealp{moen_2025}, Proposition 8]\label{prop:meanminimax}
    For any $n,\Delta \in \mathbb{N}$, $\sigma>0$, $\alpha \in (0,1)$ and $\omega \in (0, 1-\alpha)$, there exist a constant $c>0$ depending only on $\omega$ such that if $\kappa^2 \Delta \leq c\sigma^2$, we have
    \begin{align*}
        \underset{\hat{t} \in \mathcal{T}(\alpha)}{\inf} \ \underset{P \in \Theta(\Delta, \kappa, \mathcal{N}_\sigma)}{\sup} \Pb_{P} \left\{ \hat{t} -\Delta  >  n 
        \right\} \geq 1- \alpha - \omega.
    \end{align*}
\end{proposition}

\begin{proof}
To begin, note first that for any $f\in \mathcal{S}(\Delta,\kappa)$, $f_0\in \mathcal{S}_{0}$, $T \in \mathcal{T}(\alpha)$ from \eqref{mathcaltmean}, we have
\begin{align*}
     \Pb_{f}(\hat{t} > \Delta + n) &= \Pb_{f}(\hat{t} > \Delta + n) + \Pb_{f_0}(\hat{t} \leq \Delta + n) - \Pb_{f_0}(\hat{t} \leq  \Delta + n).
\end{align*}
Since $\hat{t} \in \mathcal{T}(\alpha)$, we have $\Pb_{f_0}(\hat{t} \leq \Delta + n) \leq \Pb_{f_0}(\hat{t} < \infty) \leq \alpha$ for any such $f_0$, and thus
\begin{align*}
     \Pb_{f}(\hat{t} > \Delta + n) &\geq \Pb_{f}(\hat{t} > \Delta + n) + \underset{f_0 \in \mathcal{S}_{0}}{\sup} \Pb_{f_0}(\hat{t} \leq \Delta + n) -\alpha.
\end{align*}
Write $l = \Delta + n$.  Since $\hat{t}$ is an extended stopping time with respect to the filtration $(\mathcal{F}_t)_{t \in \mathbb{N}}$ generated by the $Y_i$, there exists a measurable function $\psi : \R^{l} \mapsto \{0,1\}$ such that we may write $\mathbbm{1}\{\hat{t} \leq l\} = \psi(Y_1, \ldots, Y_l)$. Since $\hat{t} \in \mathcal{T}(\alpha)$ was arbitrary, it therefore follows that   
\begin{align*}
     &\underset{\hat{t} \in \mathcal{T}(\alpha)}{\inf} \ \underset{P \in \Theta(\Delta, \kappa, \mathcal{N}_{\sigma})}{\sup}  \Pb_{P}(\hat{t} - \Delta  > n)=\underset{\hat{t} \in \mathcal{T}(\alpha)}{\inf} \ \underset{f\in\mathcal{S}(\Delta,\kappa)}{\sup}  \Pb_{f}(\hat{t} - \Delta  > n) \\
     \geq & \underset{\psi \in \Psi(l)}{\inf} \ \left[ \underset{f\in\mathcal{S}(\Delta,\kappa)}{\sup}  \Pb_{f}\{\psi(Y^{(l)}) = 0\} + \underset{f_0\in \mathcal{S}_0}{\sup}  \Pb_{f_0}\{\psi(Y^{(l)}) = 1\} \right] - \alpha,
\end{align*}
where $Y^{(l)} = (Y_1, \ldots, Y_l)$ and $\Psi(l)$ is the set of all measurable functions $\psi : \R^{l} \mapsto \{0,1\}$. 
To prove Proposition \ref{prop:meanminimax}, it suffices to choose a sufficiently small value of $c>0$ (depending only on $\omega$), such that 
\begin{align}
    \underset{\psi \in \Psi(l)}{\inf} \ \left[ \underset{f\in\mathcal{S}(\Delta,\kappa)}{\sup}  \Pb_{f}\{\psi(Y^{(l)}) = 0\} + \underset{f_0\in \mathcal{S}_0}{\sup}  \Pb_{f_0}\{\psi(Y^{(l)}) = 1\}  \right] \geq 1 - \omega, \label{minimaxproofclaim}
\end{align}
for any $\kappa>0$. 

To prove \eqref{minimaxproofclaim},  
we will argue in a very similar fashion as in the proof of Proposition 3 in \cite{Liu_Gao_Samworth_2021}. Due to Lemmas 21 and 23 in \cite{Liu_Gao_Samworth_2021}, given any $\alpha>0$ it suffices to find a value of $c$ depending only on $\omega$ and a prior distribution $\nu$ with support on $\mathcal{S}(\Delta, \kappa)$ such that
    \begin{equation*}
    \E_{(f^{(1)}, f^{(2)}) \sim \nu\otimes\nu} \exp\left( \frac{1}{\sigma^2} \sum_{i \in [l]}  f_{i}^{(1)} f_{i}^{(2)} \right) \leq 1+ 2\omega^2.
    \end{equation*}
    Define the prior distribution $\nu$ to be the distribution of $f \in \mathcal{S}(\Delta, \kappa)$ generated according to the following process. $f_i = u \kappa \mathbbm{1}_{\{i\leq \Delta\}}$ where $u\overset{\mathrm{i.i.d.}}{\sim} \mathrm{Unif}(\{-1, 1\})$.

Now, let $f^{(1)}_i=u\kappa\mathbbm{1}_{\{i\leq \Delta\}}$ and $f^{(2)}_i=m\kappa\mathbbm{1}_{\{i\leq \Delta\}}$, where $m,u\overset{\mathrm{i.i.d.}}{\sim} \mathrm{Unif}(\{-1, 1\})$. Thus we have 
    \begin{align*}
        \sigma^{-2} \sum_{i \in [l]} f_{i}^{(1)} f_{i}^{(2)}
        &=\Delta \kappa^2\sigma^{-2}mu\leq c mu.
    \end{align*}
Thus, for any value of $\kappa$, it holds that 
    \begin{align*}
         \E_{(f^{(1)}, f^{(2)}) \sim \nu\otimes\nu} \exp\left( \sigma^{-2} \sum_{i \in [l]} f_{i}^{(1)} f_{i}^{(2)}\right) \leq \E_{m,u} [e^{cmu}]= \frac{1}{2}(e^c+e^{-c})=\cosh(c). 
    \end{align*}
    where the second inequality follows from the fact that $mu \sim \mathrm{Unif}(\{-1,1\})$. Thus, we may take $c=\cosh^{-1}(1+2\omega^2)$.
\end{proof}

\subsubsection*{Regime 2}
For Regime 2, \cite{Yu_2023} established a minimax lower bound for the normal family in the absence of contamination; this result is recalled in \Cref{thm:yu}. By the same inclusion argument as in Regime 1b, this lower bound extends to the larger distribution classes $\mathcal{G}_{\theta,M}$ and $\mathcal{P}_{v,\phi}$, provided that $M$ and $\phi$ are sufficiently large.

\begin{proposition}[\citealp{Yu_2023}, Proposition 7] \label{thm:yu}
    Suppose $\{X_i\}_{i\in\mathbb{N}}$ satisfies Assumption \ref{assum:kappa}, where $\varepsilon=0$ and
    \begin{align*}
    Q_i=\begin{cases}
        \mathcal{N}(0,\sigma^2) &\text{if $i\leq\Delta$,} \\
        \mathcal{N}(\kappa,\sigma^2) &\text{if $i>\Delta$,}
    \end{cases}
    \end{align*}
    for some $\kappa>0$ and $\sigma^2>0$. Then, we have $Q\in\Theta(\Delta,\kappa,\mathcal{N}_\sigma)$. Moreover, for any $\alpha\in(0,1)$ satisfying
     \begin{equation*}
         \alpha+2\alpha^{1/4}<\frac{1}{2} \quad \text{and} \quad \alpha^{5/4}\log\left(\frac{1}{\alpha}\right)\leq (2\kappa)^{2}\sigma^{-2},
     \end{equation*}
     it holds that
     \begin{equation*}
         \inf_{\hat{t}\in \mathcal{T}(\alpha)} \E_Q\{(\hat{t}-\Delta)^+\}\geq \frac{\sigma^2}{4\kappa^{2}}\log\left(\frac{1}{\alpha}\right).
     \end{equation*}
\end{proposition}

\subsubsection*{Regime 3}
We now derive a minimax lower bound on the expected detection delay over the inlier distributions belonging to the class of distributions with finite $v$-th moment, $\mathcal{P}_{v,\phi}$. Rather than considering the whole class of distributions, it is sufficient to consider the case where the inlier distributions belong to the symmetric type II Pareto family and there is no contamination ($\varepsilon = 0$). To be specific, let $\mathfrak{P}_s = \{ \mathrm{Par}(\mu,s) : \mu \in \mathbb{R} \}$ denote the set of symmetric type II Pareto distributions with shape parameter $s > 2$, where the density is given by
$$p(x; \mu, s) = \frac{s-1}{2}(1 + |x - \mu|)^{-s}, \quad x \in \mathbb{R}.$$
The proof adapts arguments used in Theorem 2 of \cite{Lai_1998}. 

\begin{proposition}[$\mathcal{P}_{v,\phi}$ lower bound]\label{prop:fm_lowbd} 
Suppose $\{X_i\}_{i\in\mathbb{N}}$ satisfies Assumption \ref{assum:kappa}, where $\varepsilon=0$ and
\begin{align*}
    Q_i=\begin{cases}
        \mathrm{Par}(0,v+1.01) &\text{if $i\leq\Delta$,} \\
        \mathrm{Par}(\kappa,v+1.01) &\text{if $i>\Delta$,}
    \end{cases}
\end{align*}
for some $v\geq2$ and $\kappa>0$. Then, we have $Q\in\Theta(\Delta,\kappa,\mathcal{P}_{v,\phi})$. Moreover, for any $\alpha\in(0,1)$ satisfying
    \begin{equation}\label{v2_assum:ak2}
         \alpha+2\alpha^{1/4}<\frac{1}{2} \quad \text{and} \quad \alpha^{-1/4}\log^{2.02}\left(\frac{1}{\alpha}\right)\leq (2\kappa\log\kappa-1)^{v+0.01}(4v\log\kappa)^{2.02},
     \end{equation}
     and for any change point time $\Delta$, it holds that
     \begin{equation*}
         \inf_{\hat{t}\in \mathcal{T}(\alpha)} \E_Q\{(\hat{t}-\Delta)^+\}\geq \frac{1}{16v\log{\kappa}}\log\left(\frac{1}{\alpha}\right).
     \end{equation*}
\end{proposition}

\begin{proof}
To show $Q\in\Theta(\Delta,\kappa,\mathcal{P}_{v,\phi})$, it suffices to show that the distribution $\mathrm{Par}(0,v+1.01)$ has $v$ finite central moments. Indeed, if $X\sim\mathrm{Par}(0,v+1.01)$ and $k\leq v$, we have
\begin{align*}
    \E[|X|^k]&=\frac{v+0.01}{2}\int_{-\infty}^\infty |x|^k(1 + |x|)^{-v-1.01} dx\\
    &=(v+0.01)\int_{0}^\infty x^k(1 + x)^{-v-1.01} dx\\
    &\leq(v+0.01)\left[\int_{0}^1 x^k dx+\int_{1}^\infty (1+x)^{k-v-1.01} dx\right]<\infty.
\end{align*}
The first inequality follows from the fact that $(1+x)^{-v-1.01}\leq 1$ for $x\in [0,1]$ and the second inequality follows from the fact that $k-v-1.01<-1$.

    We now prove the lower bound on the infimum of expected detection delay.
     
     Step 1: Denote the joint distribution of $\{X_i\}_{i=1,2,\ldots}$ by $\mathbb{P}_{\Delta}=\bigotimes_{i=1}^\infty Q_i$. For any $n$, let $\mathbb{P}^n_\Delta$ be the restriction of a distribution $\mathbb{P}_\Delta$ to $\mathcal{F}_n$, i.e.~the $\sigma$-algebra generated by the observations $\{X_i\}^n_{i=1}$. For any $\varphi\geq1$ and $n\geq \varphi$, we have that for any $n\geq \Delta$, it holds that
     $$\frac{\mathrm{d}\mathbb{P}^n_{\varphi}}{\mathrm{d}\mathbb{P}^n_{\infty}}=\exp\left(\sum_{i=\varphi+1}^{n} Z_i\right),$$
     where 
     $$Z_i=(v+1.01)\log\left(\frac{1+|X_i-0|}{1+|X_i-\kappa|}\right).$$
     For any $\varphi \geq 1$, define the event
     $$E_{\varphi}=\left\{\varphi<T<\varphi+\frac{1}{4(v+1.01)\log{\kappa}}\log\left(\frac{1}{\alpha}\right),\quad\sum_{i=\varphi+1}^TZ_i<\frac{3}{4}\log\left(\frac{1}{\alpha}\right)\right\}.$$
     Define
     $$d=\frac{\log(1/\alpha)}{4(v+1.01)\log{\kappa}}.$$
    Then we have
    $$\Pb_{\varphi}(E_{\varphi})\leq \int_{E_c} \exp \left(\sum_{i=\varphi+1}^{n} Z_i\right) dP_{\infty}\leq \alpha^{-3/4}\Pb_{\infty}(E_{\varphi})\leq\alpha^{-3/4}\alpha=\alpha^{1/4},$$
    where the first two inequalities follow from the definition of $E_\varphi$, and the last inequality follows from the definition of $\mathcal{T}(\alpha)$.
    
    Step 2: We first show two useful results.
    Firstly, we note that $X_i$ has polynomially-decaying tails $\forall i\geq\varphi+1$, i.e.
    \begin{equation}\label{eqn_proof:polydec}
        \Pb_{\varphi}(|X_i-\kappa|\geq s)=\frac{1}{(1+s)^{v+0.01}}\leq\frac{1}{s^{v+0.01}}, \quad \forall s>0, \,\forall i\geq\varphi+1,
    \end{equation}
    where the equality follows directly from evaluating the cumulative density function. A corollary of \eqref{eqn_proof:polydec} is the following concentration inequality for $\sum_{i=\varphi+1}^{\varphi+t}|X_i-\kappa|$, where $t\in \mathbb{N}$, given by
    \begin{equation}\label{eqn_proof:conc_poly}
        \Pb_{\varphi}\left(\sum_{i=\varphi+1}^{\varphi+t}|X_i-\kappa|>s\right)\leq \frac{t^{v+1.01}}{s^{v+0.01}} \quad \forall s>0.
    \end{equation}
    This can be shown by a union bound argument as follows.
    \begin{align*}
        \Pb_{\varphi}\left(\sum_{i=\varphi+1}^{\varphi+t}|X_i-\kappa|>s\right)&\leq \Pb_{\varphi}\left(\bigcup_{i=\varphi+1}^{\varphi+t}\left\{|X_i-\kappa|>\frac{s}{t}\right\}\right)\\
        & \leq \sum_{i=\varphi+1}^{\varphi+t}\Pb_{\varphi}\left(|X_{i}-\kappa|>\frac{s}{t}\right) \\
        &\leq t\left(\frac{s}{t}\right)^{-(v+0.01)} =\frac{t^{v+1.01}}{s^{v+0.01}}.
    \end{align*}
    Secondly, we can show that
    \begin{equation}\label{eqn:ZdomY2}
        Z_i\leq (v+1.01)\left(\log\kappa+\frac{1+|X_{i}-\kappa|}{\kappa}\right).
    \end{equation}
    Indeed, \eqref{eqn:ZdomY2} follows directly from the definition of $\kappa$ and the mean value theorem. Firstly, by triangle inequality, we have that 
    \begin{align*}
        (v+1.01)^{-1}Z_i=\log\left(1+\frac{|X_i-0|-|X_i-\kappa|}{1+|X_i-\kappa|}\right)\leq\log\left(1+\frac{\kappa}{1+|X_i-\kappa|}\right).
    \end{align*}
    Meanwhile, given for any $a>b>0$, by the mean value theorem, we have that for some $c\in(a,b)$, 
    \begin{align*}
    \frac{\log(a)-\log(b)}{a-b}=\frac{1}{c} \leq \frac{1}{b}.
    \end{align*}
    Substituting $a=\kappa/(1+|X_i-\kappa|)+1$ and $b=\kappa/(1+|X_i-\kappa|)$, we get
    \begin{align*}
    \log\left(1+\frac{\kappa}{1+|X_i-\kappa|}\right)\leq\log\left(\frac{\kappa}{1+|X_i-\kappa|}\right)+\frac{1+|X_i-\kappa|}{\kappa}.
    \end{align*}
    The desired result is obtained by noting that $\kappa/(1+|X_i-\kappa|)\leq\kappa$. 
    
    Step 3: For any $\varphi\geq1$ and $T\in \mathcal{T}(\alpha)$, since $\{T\geq \varphi\}\in\mathcal{F}_{\varphi-1},$ 
    \begin{align*}
        &\Pb_{\varphi}\left(\varphi<T<\varphi+d,\sum_{i=\varphi+1}^TZ_i\geq \frac{3}{4}\log\left(\frac{1}{\alpha}\right)| T>\varphi\right)\\
        &\leq \esssup \Pb_{\varphi}\left(\max_{1\leq t\leq d-1}\sum_{i=\varphi+1}^{\varphi+t}Z_i\geq \frac{3}{4}\log\left(\frac{1}{\alpha}\right)\right)\\
        &\leq \esssup d \max_{1\leq t\leq d-1} \Pb_{\varphi}\left(\sum_{i=\varphi+1}^{\varphi+t}Z_i\geq \frac{3}{4}\log\left(\frac{1}{\alpha}\right)\right)\\
       & \leq \esssup d \max_{1\leq t\leq d-1} \Pb_{\varphi}\left(\sum_{i=\varphi+1}^{\varphi+t}\left(\log\kappa+\frac{1+|X_{i}-\kappa|}{\kappa}\right)\geq \frac{3}{4(v+1.01)}\log\left(\frac{1}{\alpha}\right)\right) \\
       &= \esssup d \max_{1\leq t\leq d-1} \Pb_{\varphi}\left(\sum_{i=\varphi+1}^{\varphi+t}{|X_{i}-\kappa|}\geq \kappa\left(\frac{3}{4(v+1.01)}\log\left(\frac{1}{\alpha}\right) - t \log\kappa\right)-t\right)\\
      & \leq \esssup d \max_{1\leq t\leq d-1} \Pb_{\varphi}\left(\sum_{i=\varphi+1}^{\varphi+t}{|X_{i}-\kappa|}\geq \kappa\left(\frac{3}{4(v+1.01)}\log\left(\frac{1}{\alpha}\right) - d \log\kappa\right)-d\right)\\
      & = \esssup d \max_{1\leq t\leq d-1} \Pb_{\varphi}\left(\sum_{i=\varphi+1}^{\varphi+t}{|X_{i}-\kappa|}\geq \frac{2\kappa\log{\kappa}-1}{4(v+1.01)\log{\kappa}}\log\left(\frac{1}{\alpha}\right)\right)\\
      &\leq d \Pb_{\varphi}\left(\sum_{i=\varphi+1}^{\varphi+d}{|X_{i}-\kappa|}\geq \frac{2\kappa\log{\kappa}-1}{4(v+1.01)\log{\kappa}}\log\left(\frac{1}{\alpha}\right)\right)\\
      & \leq \left(\frac{\log(1/\alpha)}{4(v+1.01)\log{\kappa}}\right)^{v+2.01}\left(\frac{2\kappa\log{\kappa}-1}{4(v+1.01)\log{\kappa}}\log\left(\frac{1}{\alpha}\right)\right)^{-v-0.01}\\
      & = \left(\frac{\log(1/\alpha)}{4(v+1.01)\log{\kappa}}\right)^{2.02}\left(2\kappa\log{\kappa}-1\right)^{-v-0.01} \leq \alpha^{1/4},
    \end{align*}
    where the third inequality follows from \eqref{eqn:ZdomY2}; the sixth inequality is due to \eqref{eqn_proof:conc_poly}; and the seventh inequality is due to the assumption stated in \eqref{v2_assum:ak2}. Combining the results above, we have
    \begin{equation*}
        \sup_{\varphi\geq1} \Pb_{\varphi}\left(\varphi<T<\varphi+\frac{\log(1/\alpha)}{4(v+1.01)\log{\kappa}}\right)<2\alpha^{1/4}.
    \end{equation*}
    
    Step 4: We now have for any change point time $\Delta$,
    \begin{align*}
        \E_{\Delta}[(T-\Delta)^+]&\geq \frac{1}{4(v+1.01)\log{\kappa}}\log\left(\frac{1}{\alpha}\right)\Pb_{\Delta}\left(T-\Delta\geq \frac{1}{4(v+1.01)\log{\kappa}}\log\left(\frac{1}{\alpha}\right)\right)\\
        &\geq \frac{1}{8v\log{\kappa}}\log\left(\frac{1}{\alpha}\right)\left[\Pb_{\Delta}(T>\Delta)-\Pb_{\Delta}\left(\Delta<T<\Delta+\frac{1}{8v\log{\kappa}}\log\left(\frac{1}{\alpha}\right)\right)\right]\\
        &\geq \frac{1}{8v\log{\kappa}}\log\left(\frac{1}{\alpha}\right)[1-\alpha-2\alpha^{1/4}]\geq \frac{1}{16v\log{\kappa}}\log\left(\frac{1}{\alpha}\right).
    \end{align*}
    The second inequality holds since $v+1.01\leq 2v$ for $v\geq 2$, and the last inequality holds since we assumed $\alpha+2\alpha^{1/4}<1/2$.
\end{proof}

Similarly, to give a minimax lower bound on the  expected detection delay over the sub-Weibull class $\mathcal{G}_{\theta,M}$, it suffices to consider the Generalised Gaussian distribution in the absence of contamination ($\varepsilon = 0$). Let $\mathcal{GG}_{\theta,M} = \{ GG(\mu, \theta, M) : \mu \in \mathbb{R} \}$ denote the family of Generalised Gaussian distributions with mean $\mu$, shape parameter $\theta > 0$ and scale parameter $M>0$, defined by the densities:
$$p(x; \mu, \theta, M) = \frac{\theta}{2M\Gamma(1/\theta)} \exp\left( -\left[ \frac{|x-\mu|}{M} \right]^\theta \right), \quad x \in \mathbb{R}.$$

\begin{proposition}[$\mathcal{G}_{\theta,M}$ lower bound]\label{prop:subW_lowbd}
     Suppose $\{X_i\}_{i\in\mathbb{N}}$ satisfies Assumption \ref{assum:kappa}, where $\varepsilon=0$ and
\begin{align*}
    Q_i=\begin{cases}
        \mathrm{GG}(0,\theta, M) &\text{if $i\leq\Delta$,} \\
        \mathrm{GG}(\kappa,\theta, M), &\text{if $i>\Delta$,}
    \end{cases}
\end{align*} 
for some $\theta\in(0,2]$, $\kappa>0$ and $M>0$. Then we have $Q \in \Theta(\Delta, \kappa, \mathcal{G}_{\theta,M})$. Moreover, suppose that $\kappa\geq 4M(2\log(4e))^{1/\theta}$. Then, for any $\alpha\in(0,1)$ satisfying
     \begin{equation}\label{assum:subw_low}
         \alpha+2\alpha^{1/4}<\frac{1}{2} \quad \text{and} \quad \alpha^{-1/4}\log\left(\frac{1}{\alpha}\right)\leq (2\kappa)^{\theta}\exp(4^\theta\log(4e)),
     \end{equation}
     and for any change point time $\Delta$, it holds that
     \begin{equation*}
         \inf_{\hat{t}\in \mathcal{T}(\alpha)} \E_Q\{(\hat{t}-\Delta)^+\}\geq \frac{1}{4(2\kappa)^{\theta}}\log\left(\frac{1}{\alpha}\right).
     \end{equation*}
\end{proposition}

\begin{proof}
     Step 1: Denote the joint distribution of $\{X_i\}_{i=1,2,\ldots}$ as $\mathbb{P}_{\Delta}=\bigotimes_{i=1}^\infty Q_i$. For any $n$, let $\mathbb{P}^n_\Delta$ be the restriction of a distribution $\mathbb{P}_{\Delta}$ to $\mathcal{F}_n$, i.e.~the $\sigma$-algebra generated by the observations $\{X_i\}^n_{i=1}$. For any $\varphi\geq1$ and $n\geq \varphi$, we have that for any $n\geq \Delta$, it holds that
     $$\frac{\mathrm{d}\Pb^n_{\varphi}}{\mathrm{d}\Pb^n_{\infty}}=\exp\left(\sum_{i=\varphi+1}^{n} Z_i\right),$$
     where 
     \begin{equation*}
         Z_i=M^{-1}(|X_i-0|^\theta-|X_i-\kappa|^\theta).
     \end{equation*}
     For any $\varphi \geq 1$, define the event
     $$E_{\varphi}=\left\{\varphi<T<\varphi+\frac{1}{2(2\kappa)^{\theta}}\log\left(\frac{1}{\alpha}\right),\quad\sum_{i=\varphi+1}^TZ_i<\frac{3}{4}\log\left(\frac{1}{\alpha}\right)\right\}.$$
     Define
     $$d=\frac{\log(1/\alpha)}{2(2\kappa)^{\theta}}.$$
    Then we have
    $$\Pb_{\varphi}(E_{\varphi})\leq \int_{E_c} \exp \left(\sum_{i=\varphi+1}^{n} Z_i\right) dP_{\infty}\leq \alpha^{-3/4}\Pb_{\infty}(E_{\varphi})\leq\alpha^{-3/4}\alpha=\alpha^{1/4},$$
    where the first two inequalities follow from the definition of $E_\varphi$, and the last inequality follows from the definition of $\mathcal{T}(\alpha)$.
    
    Step 2: We first show three useful results.
    Firstly, we show that 
    \begin{equation}\label{eqn_proof:orlicz}
        \norm{X_i-f_i}_{\psi_\theta}=(1-2^{-\theta})^{1/\theta}M\leq M, \quad \forall i\in\mathbb{N}.
    \end{equation}
    This can be shown by evaluating the expectation stated in Definition \ref{defn:orlicz}.
    \begin{align*}
        \E \exp\left\{\left(\frac{|X_i-f_i|}{(1-2^{-\theta})^{1/\theta}M}\right)^\theta\right\}&=\int_{-\infty}^\infty \frac{\theta}{2M\Gamma(1/\theta)} \exp\left\{\frac{|x-f_i|^\theta}{(1-2^{-\theta})M^\theta}-\frac{|x-f_i|^\theta}{M^\theta}\right\}dx\\
        &=\int_{-\infty}^\infty \frac{\theta}{2M\Gamma(1/\theta)} \exp\left\{\frac{|x-f_i|^\theta}{(2^\theta-1)M^\theta}\right\} dx\\
        &=(2^\theta-1)^{1/\theta}\leq2.
    \end{align*}
    Secondly, we show that $|X_i-\kappa|^\theta$ is sub-exponential for all $i\geq \varphi+1$. Indeed, for all $s>0$, we have
    \begin{equation}\label{eqn_proof:subexp}
        \mathbb{P}_{\varphi}(|X_i-\kappa|^\theta\geq s)=\mathbb{P}_{\varphi}(|X_i-\kappa|\geq s^{1/\theta})\leq 2\exp\{-(s^{1/\theta}/M)^\theta\}=2\exp(-s/M^\theta), 
    \end{equation}
    where the inequality follows directly from \eqref{eqn_proof:orlicz} and Lemma \ref{thm:subWeib}. Thirdly, we consider the random variable
    \begin{equation*}
        Y_i=2^\theta(|X_i-\kappa|^\theta+\kappa^\theta), \quad\forall i\in \mathbb{N},
    \end{equation*}
    and we can show that
    \begin{equation}\label{eqn:ZdomY1}
        Z_i\leq M^{-1}Y_i, \quad\forall \theta> 0.
    \end{equation}
    \eqref{eqn:ZdomY1} follows directly from the inequality
    \begin{align*}
    |a+b|^\theta\leq|2\max(a,b)|^\theta\leq2^\theta(|a|^\theta+|b|^\theta).
    \end{align*}
    Subtracting $|a|^\theta$ from both sides give
    \begin{equation*}
        |a+b|^\theta-|a|^\theta\leq2^\theta(|a|^\theta+|b|^\theta).
    \end{equation*}
    Substitute $a=X_i-\kappa$, $b=\kappa-0$ to get desired result.

    Step 3: For any $\varphi\geq1$ and $T\in \mathcal{T}(\alpha)$, since $\{T\geq \varphi\}\in\mathcal{F}_{\varphi-1},$ 
    \begin{align*}
        &\mathbb{P}_{\varphi}\left(\varphi<T<\varphi+\frac{1}{2(2\kappa)^{\theta}}\log\left(\frac{1}{\alpha}\right),\sum_{i=\varphi+1}^TZ_i\geq \frac{3}{4}\log\left(\frac{1}{\alpha}\right)| T>\varphi\right)\\
        &\leq \esssup \mathbb{P}_{\varphi}\left(\max_{1\leq t\leq d-1}\sum_{i=\varphi+1}^{\varphi+t}Z_i\geq \frac{3}{4}\log\left(\frac{1}{\alpha}\right)\right)\\
        &\leq \esssup d \max_{1\leq t\leq d-1} \mathbb{P}_{\varphi}\left(\sum_{i=\varphi+1}^{\varphi+t}Z_i\geq \frac{3}{4}\log\left(\frac{1}{\alpha}\right)\right)\\
       & \leq \esssup d \max_{1\leq t\leq d-1} \mathbb{P}_{\varphi}\left(\sum_{i=\varphi+1}^{\varphi+t}Y_i\geq \frac{3M}{4}\log\left(\frac{1}{\alpha}\right)\right) \\
       &= \esssup d \max_{1\leq t\leq d-1} \mathbb{P}_{\varphi}\left(\sum_{i=\varphi+1}^{\varphi+t}2^{-\theta}Y_i - \kappa^{\theta}\geq \frac{3M}{4\times2^{\theta}}\log\left(\frac{1}{\alpha}\right) - t \kappa^{\theta}\right)\\
      & \leq \esssup d \max_{1\leq t\leq d-1} \mathbb{P}_{\varphi}\left(\sum_{i=\varphi+1}^{\varphi+t}|X_i-\kappa|^\theta\geq \frac{3M}{4\times2^{\theta}}\log\left(\frac{1}{\alpha}\right) - \frac{\log(1/\alpha)}{2 (2\kappa)^{\theta}} \kappa^{\theta}\right)\\
       &  = \esssup d \max_{1\leq t\leq d-1} \mathbb{P}_{\varphi}\left(\sum_{i=\varphi+1}^{\varphi+t}|X_i-\kappa|^\theta \geq \left(3\times2^{-\theta-2}M - 2^{-\theta -1}\right)\log(1/\alpha)\right) \\
       &\leq \frac{2\log(1/\alpha)}{2(2\kappa)^{\theta}} \exp\left\{-\frac{\left(3\times2^{-\theta-2}M - 2^{-\theta -1}\right)\log(1/\alpha) M^{-\theta}}{\log(1/\alpha)/(2\kappa)^{\theta}}\right\}\\
        &\leq \frac{\exp\{-2^{-1}(\kappa/M)^{\theta}\}}{(2\kappa)^{\theta}}\log\left(\frac{1}{\alpha}\right)\\
        &\leq \frac{\exp\{-4^{\theta}\log(4e)\}}{(2\kappa)^{\theta}}\log\left(\frac{1}{\alpha}\right) \leq \alpha^{1/4},
    \end{align*}
    where the third inequality follows from \eqref{eqn:ZdomY1}; the fifth inequality is due to \eqref{eqn_proof:subexp} alongside the concentration property of sums of exponential random variables; the sixth inequality is due to the restriction $M\geq4/3$; the seventh inequality is given by the regime $\kappa\geq 4M(2\log(4e))^{1/\theta}$; and the eighth inequality follows from \eqref{assum:subw_low}.  Combining the results above, 
\begin{equation*}
    \sup_{\varphi\geq1} \Pb_{\varphi}\left(\varphi<T<\varphi+\frac{1}{2(2\kappa)^{\theta}}\log\left(\frac{1}{\alpha}\right)\right)<2\alpha^{1/4}.
\end{equation*}

Step 4: We now have for any change point time $\Delta$,
\begin{align*}
    \E_{\Delta}[(T-\Delta)^+]&\geq \frac{1}{2(2\kappa)^{\theta}}\log\left(\frac{1}{\alpha}\right)\Pb_{\Delta}\left\{T-\Delta\geq \frac{1}{2(2\kappa)^{\theta}}\log\left(\frac{1}{\alpha}\right)\right\}\\
    &\geq \frac{1}{2(2\kappa)^{\theta}}\log\left(\frac{1}{\alpha}\right)\left[\Pb_{\Delta}(T>\Delta)-\Pb_{\Delta}\left\{\Delta<T<\Delta+\frac{1}{2(2\kappa)^{\theta}}\log\left(\frac{1}{\alpha}\right)\right\}\right]\\
    &\geq \frac{1}{2(2\kappa)^{\theta}}\log\left(\frac{1}{\alpha}\right)[1-\alpha-2\alpha^{1/4}]\geq \frac{1}{4(2\kappa)^{\theta}}\log\left(\frac{1}{\alpha}\right).
\end{align*}
The last inequality holds since $\alpha+2\alpha^{1/4}<1/2$.
\end{proof}

\subsubsection*{Regime 4}
Finally, we discuss Regime 4, the case where the signal is very large. To show a minimax lower bound on the the expected detection delay, the idea here is to construct a distribution on two points $\{0,\kappa\}$ and derive a lower bound on the expected detection delay for this model.

\begin{proposition} \label{thm: regime4_lowerbd}
Suppose $\{X_i\}_{i\in\mathbb{N}}$ satisfies Assumption \ref{assum:kappa}, where $\varepsilon>0$, $\kappa>0$, and
\begin{align*}
    Q_i=\begin{cases}
        (1-\varepsilon)\delta(0)+\varepsilon \delta(\kappa), &\text{if $i\leq\Delta$,} \\
        (1-\varepsilon)\delta(\kappa)+\varepsilon \delta(0), &\text{if $i>\Delta$.}
    \end{cases}
\end{align*}  
\begin{enumerate}[label=(\textbf{\alph*})]
    \item Denote $F_1=\delta(0)$, $F_{\Delta+1}=\delta(\kappa)$. Then $F_i\in\mathcal{P}_{v,0}$ for all $v \geq1$ and $F_i\in\mathcal{G}_{\theta,0}$ for all $\theta\in\R^{+}$.
    \item 
    For any $\alpha\in(0,1)$ such that 
    \begin{equation*}
         \alpha+\alpha^{1/4}<1/2,
     \end{equation*}
     it holds that, for any change point time $\Delta$,
     \begin{equation*}
         \inf_{\hat{t}\in \mathcal{T}(\alpha)}\E_Q\{(\hat{t}-\Delta)^+\}\geq \frac{3}{16}\frac{\log(1/\alpha)}{\log((1-\varepsilon)/\varepsilon)}.
     \end{equation*}
\end{enumerate}
\end{proposition}

\begin{proof}
(a) We first show that the sequence of distributions $\{F_i\}_{i\in\mathbb{N}}$ is sub-Weibull$(\theta)$ for any $\theta\in\R^+$ and has finite $v$-th moment for all $v\in \mathbb{N}$. The finite-moment property is immediate since we just have a point mass.
    $$\E[|X_i-\E[X_i]|^v]=0,$$
    Similarly, since we have a point mass, we have that $\{F_i\}_{i\in\mathbb{N}}$ is sub-Weibull$(\theta)$ for any $\theta\in\R^+$ with Orlicz norm
    $$\inf \{t>0:\exp((0/t)^\theta)\leq2\}=0.$$
    
    (b) Now we prove the lower bound for the delay.
    
    Step 1. Denote the joint distribution of $\{X_i\}_{i=1,2,\ldots}$ as $\mathbb{P}_{\Delta}=\bigotimes_{i=1}^\infty Q_i$. For any $n$, let $\mathbb{P}^n_\Delta$ be the restriction of a distribution $\mathbb{P}_{\Delta}$ to $\mathcal{F}_n$, i.e.~the $\sigma$-algebra generated by the observations $\{X_i\}^n_{i=1}$. For any $\varphi\geq1$ and $n\geq \varphi$, we have that for any $n\geq \Delta$, it holds that
     $$\frac{\mathrm{d}\Pb^n_{\varphi}}{\mathrm{d}\Pb^n_{\infty}}=\exp\left(\sum_{i=\varphi+1}^{n} Z_i\right),$$
     where 
     $$Z_i=\left(\frac{2}{\kappa}X_i-1\right)\log\left(\frac{1-\varepsilon}{\varepsilon}\right).$$
     For any $\varphi \geq 1$, define the event
     $$E_{\varphi}=\left\{\varphi<T<\varphi+\left\lceil\frac{3}{8}\frac{\log(1/\alpha)}{\log((1-\varepsilon)/\varepsilon)}\right\rceil,\quad\sum_{i=\varphi+1}^TZ_i<\frac{3}{4}\log\left(\frac{1}{\alpha}\right)\right\}.$$
     Define
     $$d=\frac{3}{8}\frac{\log(1/\alpha)}{\log((1-\varepsilon)/\varepsilon)}.$$
    Then we have
    $$\Pb_{\varphi}(E_{\varphi})\leq \int_{E_c} \exp \left(\sum_{i=\varphi+1}^{n} Z_i\right) dP_{\infty}\leq \alpha^{-3/4}\Pb_{\infty}(E_{\varphi})\leq\alpha^{-3/4}\alpha=\alpha^{1/4},$$
    where the first two inequalities follow from the definition of $E_\varphi$, and the last inequality follows from the definition of $\mathcal{T}(\alpha)$.
    
    Step 2: Suppose $d>1$. For any $\varphi\geq1$ and $T\in \mathcal{T}(\alpha)$, since $\{T\geq \varphi\}\in\mathcal{F}_{\varphi-1},$ 
    \begin{align*}
        &\Pb_{\varphi}\left(\varphi<T<\varphi+\lceil d\rceil,\sum_{i=\varphi+1}^TZ_i\geq \frac{3}{4}\log\left(\frac{1}{\alpha}\right)| T>\varphi\right)\\
        &\leq \esssup \Pb_{\varphi}\left(\max_{1\leq t\leq \lceil d\rceil-1}\sum_{i=\varphi+1}^{\varphi+t}Z_i\geq \frac{3}{4}\log\left(\frac{1}{\alpha}\right)\right)\\
        &\leq \esssup  \Pb_{\varphi}\left(\max_{1\leq t\leq \lceil d\rceil-1}\sum_{i=\varphi+1}^{\varphi+t}(\frac{2X_i}{\kappa}-1)\geq \frac{3}{4}\log\left(\frac{1}{\alpha}\right)(\log(\frac{1-\varepsilon}{\varepsilon}))^{-1}\right)\\
        &= \esssup \Pb_{\varphi}\left(\max_{1\leq t\leq \lceil d\rceil-1}\frac{1}{t}\sum_{i=\varphi+1}^{\varphi+t}W_i\geq \frac{1}{2}+\frac{d}{t}\right)\\
        &=0,
    \end{align*}
    where in the first equality, $W_i=\frac{X_i}{\kappa}\sim Bern(1-\varepsilon)$ for all $i\in\{\varphi+1,\ldots,\varphi+t\}$, and we note that $d/t\geq 1$ in the second equality for $1\leq t\leq \lceil d\rceil-1$. Meanwhile, for $d\leq1$, since $\{T:\varphi<T<\varphi+\lceil d\rceil\}$ is empty, it immediately follows that 
    $$\Pb_{\varphi}\left(\varphi<T<\varphi+\lceil d\rceil,\sum_{i=\varphi+1}^TZ_i\geq \frac{3}{4}\log\left(\frac{1}{\alpha}\right)| T>\varphi\right)=0.$$
    Combining the results above, we have
    \begin{equation*}
        \sup_{\varphi\geq1} \Pb_{\varphi}\left(\varphi<T<\varphi+\lceil{d}\rceil\right)\leq\alpha^{1/4}.
    \end{equation*}
    
    Step 3: We now have for any change point time $\Delta$,
    \begin{align*}
        \E_{\Delta}[(T-\Delta)^+]&\geq \frac{3}{8}\frac{\log(1/\alpha)}{\log((1-\varepsilon)/\varepsilon)} \Pb_{\Delta}\left(T-\Delta\geq \frac{3}{8}\frac{\log(1/\alpha)}{\log((1-\varepsilon)/\varepsilon)}\right)\\
        &\geq \frac{3}{8}\frac{\log(1/\alpha)}{\log((1-\varepsilon)/\varepsilon)}\left[\Pb_{\Delta}(T>\Delta)-\Pb_{\Delta}\left(\Delta<T<\Delta+\lceil d\rceil\right)\right]\\
        &\geq \frac{3}{8}\frac{\log(1/\alpha)}{\log((1-\varepsilon)/\varepsilon)}[1-\alpha-\alpha^{1/4}]\geq \frac{3}{16}\frac{\log(1/\alpha)}{\log((1-\varepsilon)/\varepsilon)}.
    \end{align*}
    The last inequality holds since $\alpha+\alpha^{1/4}<1/2$.
\end{proof}

\subsection{Proof of Theorem \ref{thm:rume_cpt_subW}}
\begin{proof}
Define $\delta_t=8\alpha/(3(t^3-t))$. Lemma \ref{prop2}, which is a slight modification of Proposition D.1 of \cite{Li_2021}, shows that conditions \eqref{A1_moment} for Lemma \ref{lem:rume_moment} are satisfied by choosing this $h_t$. Thus by Lemma \ref{lem:rume_moment} and a union bound argument, $\forall f_1,f_{t}\in\mathbb{R}$ and $\forall s\in\mathbb{N}$ such that $s\geq h_t$ and $2|s$, the event 
    \begin{equation} \label{eqn:eventE1b}
    E_{s,t}=\{|\RUME(X_{1:s})-f_1|\leq a_{s,t}\}\cap\{|\RUME(X_{(t-s+1):t}) - f_{t}|\leq a_{s,t}\},
    \end{equation}
    happens with probability at least $1-\delta_{t}$, where 
    $$a_{s,t}=C_3 \varepsilon_{s,t}'\log^{1+1/\theta}\left( \frac{1}{\varepsilon_{s,t}'}\right)+C_4\sqrt{\frac{\log(1/\delta_{t})}{s}}.$$
    Similarly, by \Cref{thm:median_subWeib}, $\forall f_1,f_{s+1}\in\mathbb{R}$ and $\forall s,t\in\mathbb{N}$, the event 
    \begin{equation} \label{eqn:eventE2b}
    \mathcal{E}_{s,t}=\{|\med(X_{1:s})-f_1|\leq c_{s,t}\}\cap\{|\med(X_{(t-s+1):t}) - f_{t}|\leq c_{s,t}\},
    \end{equation}
    happens with probability at least $1-\delta_{t}$, where $$c_{s,t}=M\log^{1/\theta}\left( \frac{4e}{(\delta_t/2)^{2/s}-2e\varepsilon}\right).$$
    
    (a) We can write the event $\{\hat{t}=\infty\}$ as
   \begin{equation*}
   \begin{aligned}
   \{\hat{t}=\infty\}&=\left\{\forall s,t\in\mathbb{N},\ t\geq 2,\ s\in [h_t,\lfloor t/2\rfloor]:Z_{s,t}\leq \zeta_{s,t} \right\} \\
   &\quad{}\cap \left\{\forall s,t\in\mathbb{N},\ t\geq 2,\ s<h_t,\ s\leq\lfloor t/2\rfloor: Y_{s,t}\leq \chi_{s,t} \right\}.
\end{aligned}
\end{equation*}
    By a union bound argument, it holds that 
    \begin{align}
        \mathbb{P}_\infty(\hat{t}<\infty)&\leq \sum_{t=2}^{\infty} \sum_{h_t\leq s\leq\lfloor t/2\rfloor, 2|s} \mathbb{P}_\infty(Z_{s,t}\geq \zeta_{s,t})+\sum_{t=2}^{\infty}\sum_{s<h_t,s\leq\lfloor t/2\rfloor} \mathbb{P}_\infty(Y_{s,t}\geq \chi_{s,t}) \nonumber\\
        &\leq \sum_{t=2}^{\infty}\sum_{s=1}^{\lfloor t/2\rfloor} \delta_{t}\leq \frac{1}{2}\sum_{t=2}^{\infty} t\delta_{t}=\sum_{t=2}^{\infty} \frac{4\alpha}{3(t-1)(t+1)}=\frac{2\alpha}{3}\sum_{t=2}^{\infty} \left(\frac{1}{t-1}-\frac{1}{t+1} \right)=\alpha, \label{proof_eqn:t1e}
    \end{align}
    where the second inequality follows from considering the event $E_{s,t}$ in \eqref{eqn:eventE1b} and $\mathcal{E}_{s,t}$ in \eqref{eqn:eventE2b}, each of which happens with probability $1-\delta_{t}$.

(b) Note that $(X_1,\ldots, X_\Delta)$ have the same law under $\Pb_\infty$ and $\Pb_\Delta$. Thus,
    \begin{equation}\label{proof_eqn:t1e_subW_early}
\mathbb{P}_\Delta(\hat{t}\leq\Delta)=\mathbb{P}_\infty(\hat{t}<\Delta)\leq \sum_{t=2}^{\Delta}\sum_{s=1}^{\lfloor t/2\rfloor} \delta_{t}\leq \alpha, 
    \end{equation}
    where the first inequality follows from truncating the sum in \eqref{proof_eqn:t1e}.

(c) For notational simplicity, define
$$d_3=\left\lceil\max\left(\frac{16C_6^2}{\kappa^{2}},10\right)\log\bigg(\frac{3\Delta^3}{\alpha}\bigg)\right\rceil.$$
Our objective is to prove that the detection delay is bounded by $d_3$ with high probability:
\begin{equation*}
\mathbb{P}(\Delta < \hat{t} \leq \Delta + d_3)= \mathbb{P}(\hat{t} > \Delta) - \mathbb{P}(\hat{t} > \Delta + d_3) \geq 1 - \alpha.
\end{equation*}
From \eqref{proof_eqn:t1e_subW_early}, we have $\mathbb{P}(\hat{t} > \Delta) \geq 1 - \sum_{t=2}^{\Delta}\sum_{s=1}^{\lfloor t/2\rfloor} \delta_t$. Given that $\alpha$ is the infinite sum of $\delta_t$, the lower bound $1 - \alpha$ holds provided we show:
\begin{equation*}
\mathbb{P}(\hat{t} > \Delta + d_3) \leq \delta_{\Delta+d_3}.
\end{equation*}
By construction, $d_3$ satisfies the requirements of Lemma \ref{prop2} ($d_3= h_{2\Delta} \geq h_{\Delta+d_3}$). 
By considering the event $E_{d_3,\Delta+d_3}$, we have that with probability $1-\delta_{\Delta+d_3}$,
$$|\RUME(X_{1:d_3}) - f_1| + |\RUME(X_{(\Delta+1):(\Delta+d_3)}) - f_{\Delta + 1}|\leq a_{d_3,\Delta+d}+a_{d_3,\Delta+d}=\zeta_{d_3,\Delta+d_3}.$$
We can rewrite this as 
\begin{equation} \label{proof_eqn:rume_cpt_upperbd}
    \Pb(|\RUME(X_{1:d_3}) - f_1| + |\RUME(X_{(\Delta+1):(\Delta+d_3)}) - f_{\Delta + 1}|>\zeta_{d_3,\Delta+d_3})\leq \delta_{\Delta+d_3}.
\end{equation}
Denoting $q_{d,t}=\log(1/\delta_t)/d$, we have 
\begin{align} 
    \zeta_{d_3,\Delta+d_3}&= 2C_3\left(\varepsilon_{d_3,\Delta+d_3}'\log^{1+1/\theta}\left( \frac{1}{\varepsilon_{d_3,\Delta+d_3}'}\right)\right)+2C_4\sqrt{q_{d_3,\Delta+d_3}}. \nonumber\\
    &\leq 2C_3\left(\varepsilon\log^{1+1/\theta} \left(\frac{1}{\varepsilon}\right)+q_{d_3,\Delta+d_3}\left(\log \frac{1}{q_{d_3,\Delta+d_3}}\right)^{1+1/\theta}\right)+2C_4\sqrt{q_{d_3,\Delta+d_3}} \nonumber\\
    &\leq 2C_3\varepsilon\log^{1+1/\theta}\left( \frac{1}{\varepsilon}\right)+C_4'\sqrt{q_{d_3,\Delta+d_3}} \nonumber\\
    &\leq 2C_6\max\left(\varepsilon\log^{1+1/\theta} \left(\frac{1}{\varepsilon}\right), \sqrt{\frac{1}{d_3}\log\frac{3\Delta^3}{\alpha}}\right),  \label{proof_eqn:b_upperbd2}
\end{align}
where the second inequality is due to $a(\log \frac{1}{a})^{1+1/\theta}\leq \left(\frac{2+2/\theta}{e}\right)^{1+1/\theta} \sqrt{a}$ for any $a\in\mathbb{R}^+$, defining $C_4'=2C_4+2C_3\left(\frac{2+2/\theta}{e}\right)^{1+1/\theta}$ and $C_6=\max(2C_3, C_4')$. Meanwhile, we can relate $\kappa$ to $\zeta_{d_3,\Delta+d_3}$. Using the definition of $d_3$, we have
$$\frac{\kappa}{4C_6}\geq\sqrt{\frac{\log(1/\delta_{2\Delta})}{d_3}}.$$
Finally, we can combine this with $\kappa\geq4C_6\varepsilon\left(\log \frac{1}{\varepsilon}\right)^{1+1/\theta}$ in \Cref{assume-2}\ref{assum: rumecpt_subW} to get
\begin{equation}\label{proof_eqn:kappa_zeta}
    \kappa\geq 4C_6\max\left(\varepsilon\left(\log \frac{1}{\varepsilon}\right)^{1+1/\theta},\sqrt{\frac{\log(1/\delta_{2\Delta})}{d_3}}\right)\geq 2\zeta_{d_3,\Delta+d_3},
\end{equation}
where the final inequality follows from \eqref{proof_eqn:b_upperbd2}.

Finally, we can prove the upper bound for $\Pb(\hat{t}>\Delta+d_3)$:
    \begin{align*}
     &\Pb(\hat{t}>\Delta+d_3)\\
\leq& \Pb(Z_{d_3,\Delta+d_3}<\zeta_{d_3,\Delta+d_3}) \\
    =& \Pb(|\RUME(X_{1:d_3})-\RUME(X_{(\Delta+1):(\Delta+d_3)})|<\zeta_{d_3, \Delta+d_3}) \\
    \leq & \mathbb{P}(\kappa - |\RUME(X_{1:d_3})-f_1| - |\RUME(X_{(\Delta+1):(\Delta+d_3)}) - f_{\Delta + 1}| < \zeta_{d_3, \Delta+d_3}) \\
    = & \mathbb{P}(|\RUME(X_{1:d_3}) - f_1| + |\RUME(X_{(\Delta+1):(\Delta+d_3)}) - f_{\Delta + 1}| > \kappa - \zeta_{d_3, \Delta+d_3})\\
    \leq & \mathbb{P}(|\RUME(X_{1:d_3}) - f_1| + |\RUME(X_{(\Delta+1):(\Delta+d_3)}) - f_{\Delta + 1}| > \zeta_{d_3, \Delta+d_3})\leq \delta_{\Delta+d_3}.
\end{align*}
where the third inequality follows from \eqref{proof_eqn:kappa_zeta} and the fourth inequality follows from \eqref{proof_eqn:rume_cpt_upperbd}.

(d) For notational simplicity, define
$$d_4=\left\lceil\max\left(\frac{4}{(\kappa/4M)^\theta},\frac{2}{\log(1/(4e\varepsilon))}\right)\log\bigg(\frac{3\Delta^3}{\alpha}\bigg)\right\rceil.$$
Our objective is to prove that the detection delay is bounded by $d_4$ with high probability:
\begin{equation*}
\mathbb{P}(\Delta < \hat{t} \leq \Delta + d_4)= \mathbb{P}(\hat{t} > \Delta) - \mathbb{P}(\hat{t} > \Delta + d_4) \geq 1 - \alpha.
\end{equation*}
From \eqref{proof_eqn:t1e_subW_early}, we have $\mathbb{P}(\hat{t} > \Delta) \geq 1 - \sum_{t=2}^{\Delta}\sum_{s=1}^{\lfloor t/2\rfloor} \delta_t$. Given that $\alpha$ is the infinite sum of $\delta_t$, the lower bound $1 - \alpha$ holds provided we show:
\begin{equation*}
\mathbb{P}(\hat{t} > \Delta + d_4) \leq \delta_{\Delta+d_4}.
\end{equation*}
By construction, $d_4$ satisfies the requirements of Lemma \ref{prop2} ($d_4= h_{2\Delta} \geq h_{\Delta+d_4}$). 
By considering the event $\mathcal{E}_{d_4,\Delta+d_4}$, we have that with probability at least $1-\delta_{\Delta+d_4}$,
\begin{equation} \label{proof_eqn:med_cpt_upperbd3}
    |\med(X_{1:d_4}) - f_1| + |\med(X_{(\Delta+1):(\Delta+d_4)}) - f_{\Delta + 1}|> c_{d_4,\Delta+d_4}+c_{d_4,\Delta+d_4}=\chi_{d_4,\Delta+d_4}.
\end{equation}
We can relate $\kappa$ and $c_{d_4,\Delta+d_4}$ as follows. Under \Cref{assume-2}\ref{assum: medcpt_subW}, specifically $\kappa\geq 4M (2\log(8e))^{1/\theta}$, we have
\begin{equation}\label{eqn:kappa_cond3}
    (\kappa/4M)^{\theta}-\log(8e)\geq\frac{1}{2}(\kappa/4M)^{\theta}.
\end{equation}
Now using the definition of $d_4$, the following inequality holds.
\begin{equation}\label{eqn:dkappa3}
    d_4\geq \frac{4}{(\kappa/4M)^{\theta}}\log\frac{3\Delta^3}{\alpha}.
\end{equation}
Combining the last two inequalities \eqref{eqn:kappa_cond3}  and \eqref{eqn:dkappa3} give
\begin{equation} \label{proof:originalkappa3}
    \kappa\geq 4M\left[\log(8e)+\frac{2}{d_4}\log\left(\frac{3\Delta^3}{\alpha}\right)\right]^{1/\theta}=4M\left[\log(8e)+\log\left(\frac{1}{(\alpha/3\Delta^3)^{2/d_4}}\right)\right]^{1/\theta}.
\end{equation}
Next, using the definition of $d_4$, we also have 
\begin{align}
    (\alpha/(3\Delta^3))^{2/d_4}&\geq4e\varepsilon, \nonumber\\
    (\alpha/(3\Delta^3))^{2/d_4}-2e\varepsilon&\geq \frac{1}{2}(\alpha/(3\Delta^3))^{2/d_4}. \label{proof:getridepsilon3}
\end{align}
Combining the inequalities \eqref{proof:originalkappa3} and \eqref{proof:getridepsilon3} implies 
\begin{equation}\label{proof_eqn:kappabound4}
    \kappa\geq 4M\left(\log\frac{4e}{(\alpha/3\Delta^3)^{2/d_4}-2e\varepsilon}\right)^{1/\theta}=4c_{d_4,2\Delta}\geq4c_{d_4,\Delta+d_4}.
\end{equation}
Finally, we can upper bound $\Pb(\hat{t}>\Delta+d_4)$.
    \begin{align*}
     &\Pb(\hat{t}>\Delta+d_4)\\
\leq& \Pb(Y_{d_4,\Delta+d_4}<\chi_{d_4,\Delta+d_4}) \\
    =& \Pb(|\med(X_{1:d_4})-\med(X_{(\Delta+1):(\Delta+d_4)})|<2c_{d_4,\Delta+d_4}) \\
    \leq & \mathbb{P}(\kappa - |\med(X_{1:d_4})-f_1| - |\med(X_{(\Delta+1):(\Delta+d_4)}) - f_{\Delta + 1}| < 2c_{d_4,\Delta+d_4}) \\
    = & \mathbb{P}(|\med(X_{1:d_4}) - f_1| + |\med(X_{(\Delta+1):(\Delta+d_4)}) - f_{\Delta + 1}| > \kappa - 2c_{d_4,\Delta+d_4}) \\
    \leq & \mathbb{P}(|\med(X_{1:d_4}) - f_1| + |\med(X_{(\Delta+1):(\Delta+d_4)}) - f_{\Delta + 1}| > 2c_{d_4,\Delta+d_4}) \leq  \delta_{\Delta+d_4}.
\end{align*}
where the third inequality follows from \eqref{proof_eqn:kappabound4} and fourth inequality follows from \eqref{proof_eqn:med_cpt_upperbd3}.
\end{proof}

\subsection{Proof of Theorem \ref{thm:rume_cpt}}
\begin{proof}
Define $\delta_t=8\alpha/(3(t^3-t))$. Lemma \ref{prop2}, which is a slight modification of Proposition D.1 of \cite{Li_2021}, shows that conditions \eqref{A1_moment} for Lemma \ref{lem:rume_moment} are satisfied by choosing this $h_t$. Thus by Lemma \ref{lem:rume_moment} and a union bound argument, $\forall f_1,f_{t}\in\mathbb{R}$ and $\forall s\in\mathbb{N}$ such that $s\geq h_t$ and $2|s$, the event 
    \begin{equation} \label{eqn:eventE1}
    E_{s,t}=\{|\RUME(X_{1:s})-f_1|\leq a_{s,t}\}\cap\{|\RUME(X_{(t-s+1):t}) - f_{t}|\leq a_{s,t}\},
    \end{equation}
    happens with probability at least $1-\delta_{t}$, where 
    $$a_{s,t}=\phi\left(C_1 \varepsilon_{s,t}'^{1-1/v}+C_2\sqrt{\frac{2\log(1/\delta_{t})}{s}}\right).$$
    Similarly, by Lemma \ref{thm:median_moment}, $\forall f_1,f_{s+1}\in\mathbb{R}$ and $\forall s,t\in\mathbb{N}$, the event 
    \begin{equation} \label{eqn:eventE2}
    \mathcal{E}_{s,t}=\{|\med(X_{1:s})-f_1|\leq c_{s,t}\}\cap\{|\med(X_{(t-s+1):t}) - f_{t}|\leq c_{s,t}\},
    \end{equation}
    happens with probability at least $1-\delta_{t}$, where $$c_{s,t}=\left(\frac{2e\phi^{v}(1-\varepsilon)}{(\delta_t/2)^{2/s}-2e\varepsilon}\right)^{1/v}.$$
    
    (a) We can write the event $\{\hat{t}=\infty\}$ as
    \begin{equation*}
       \begin{aligned}
       \{\hat{t}=\infty\}&=\left\{\forall s,t\in\mathbb{N},\ t\geq 2,\ s\in [h_t,\lfloor t/2\rfloor]:Z_{s,t}\leq \zeta_{s,t} \right\} \\
       &\quad{}\cap \left\{\forall s,t\in\mathbb{N},\ t\geq 2,\ s<h_t,\ s\leq\lfloor t/2\rfloor: Y_{s,t}\leq \chi_{s,t} \right\}.
    \end{aligned}
\end{equation*}
    By a union bound argument, it holds that 
    \begin{align}
        \mathbb{P}_\infty(\hat{t}<\infty)&\leq \sum_{t=2}^{\infty} \sum_{h_t\leq s\leq\lfloor t/2\rfloor, 2|s} \mathbb{P}_\infty(Z_{s,t}\geq \zeta_{s,t})+\sum_{t=2}^{\infty}\sum_{s<h_t,s\leq\lfloor t/2\rfloor} \mathbb{P}_\infty(Y_{s,t}\geq \chi_{s,t}) \nonumber\\
        &\leq \sum_{t=2}^{\infty}\sum_{s=1}^{\lfloor t/2\rfloor} \delta_{t}\leq \frac{1}{2}\sum_{t=2}^{\infty} t\delta_{t}=\sum_{t=2}^{\infty} \frac{4\alpha}{3(t-1)(t+1)}=\frac{2\alpha}{3}\sum_{t=2}^{\infty} \left(\frac{1}{t-1}-\frac{1}{t+1} \right)=\alpha, \label{proof_eqn:t1e_fm}
    \end{align}
    where the second inequality follows from considering the event $E_{s,t}$ in \eqref{eqn:eventE1} and $\mathcal{E}_{s,t}$ in \eqref{eqn:eventE2}, each of which happens with probability $1-\delta_{t}$.

(b) Note that $(X_1,\ldots, X_\Delta)$ have the same law under $\Pb_\infty$ and $\Pb_\Delta$. Thus,
    \begin{equation}\label{proof_eqn:t1e_fm_early}
    \mathbb{P}_\Delta(\hat{t}\leq\Delta)=\mathbb{P}_\infty(\hat{t}<\Delta)\leq \sum_{t=2}^{\Delta}\sum_{s=1}^{\lfloor t/2\rfloor} \delta_{t}\leq \alpha, 
    \end{equation}
    where the first inequality follows from truncating the sum in \eqref{proof_eqn:t1e_fm}.
    
(c) For notational simplicity, define
$$d_1=\left\lceil\max\left\{\frac{16C_5^2\phi^{2}}{\kappa^2},\frac{2}{0.5-\sqrt{2\varepsilon(1-2\varepsilon)}},20\right\}\log\bigg(\frac{3\Delta^3}{\alpha}\bigg)\right\rceil\asymp \max\left\{\frac{1}{\kappa^2}, 1\right\}\log\bigg(\frac{\Delta}{\alpha}\bigg).$$
Our objective is to prove that the detection delay is bounded by $d_1$ with high probability:
\begin{equation*}
\mathbb{P}(\Delta < \hat{t} \leq \Delta + d_1)= \mathbb{P}(\hat{t} > \Delta) - \mathbb{P}(\hat{t} > \Delta + d_1) \geq 1 - \alpha.
\end{equation*}
From \eqref{proof_eqn:t1e_fm_early}, we have $\mathbb{P}(\hat{t} > \Delta) \geq 1 - \sum_{t=2}^{\Delta}\sum_{s=1}^{\lfloor t/2\rfloor} \delta_t$. Given that $\alpha$ is the infinite sum of $\delta_t$, the lower bound $1 - \alpha$ holds provided we show:
\begin{equation*}
\mathbb{P}(\hat{t} > \Delta + d_1) \leq \delta_{\Delta+d_1}.
\end{equation*}
By construction, $d_1$ satisfies the requirements of Lemma \ref{prop2} ($d_1= h_{2\Delta} \geq h_{\Delta+d_1}$). Therefore, by considering the event $E_{d_1,\Delta+d_1}$, we have that with probability $1-\delta_{\Delta+d_1}$,
$$|\RUME(X_{1:d_1}) - f_1| + |\RUME(X_{(\Delta+1):(\Delta+d_1)}) - f_{\Delta + 1}|\leq a_{d_1,\Delta+d_1}+a_{d_1,\Delta+d_1}=\zeta_{d_1,\Delta+d_1}.$$
By \Cref{assume-1}\ref{assum: rumecpt}, we have $\delta_{\Delta+d_1}\geq\delta_{2\Delta}$ and thus
\begin{align}
    \zeta_{d_1,\Delta+d_1}&= 2a_{d_1,\Delta+d_1}\leq 2\phi\left(C_1 \max\left(\varepsilon,\frac{2\log(1/\delta_{2\Delta})}{d_1} \right)^{1-1/v}+C_2\sqrt{\frac{2\log(1/\delta_{2\Delta})}{d_1}}\right) \nonumber\\
    &\leq 2\phi C_1 \varepsilon^{1-1/v}+2\phi\max(C_1,C_2)\sqrt{\frac{2\log(1/\delta_{2\Delta})}{d_1}} \nonumber\\
    &\leq 2C_5\phi \max\left(\varepsilon^{1-1/v},\sqrt{\frac{2\log(1/\delta_{2\Delta})}{d_1}}\right), \label{proof_eqn:b_upperbd1}
\end{align}
where $C_5=\max(C_1,C_2)$. Thus we have 
\begin{equation} \label{proof_eqn:rume_cpt_upperbd1}
    \Pb\left(|\RUME(X_{1:d_1}) - f_1| + |\RUME(X_{(\Delta+1):(\Delta+d_1)}) - f_{\Delta + 1}|>\zeta_{d_1,\Delta+d_1}\right)\leq \delta_{\Delta+d_1},
\end{equation}
where the final inequality follows from \eqref{proof_eqn:b_upperbd1}. Meanwhile, we can relate $\kappa$ to $\zeta_{d_1,\Delta+d_1}$. From the definition of $d_1$, we observe that
$$\sqrt{\frac{\log(1/\delta_{2\Delta})}{d_1}} \leq \frac{\kappa}{4C_5\phi},$$
Combining this with \Cref{assume-1}\ref{assum: rumecpt}, specifically $\kappa\geq 4C_5\phi \varepsilon^{1-1/v}$, we obtain 
\begin{equation}\label{proof_eqn:minkappa_fm}
    \kappa\geq4C_5\phi \max\left(\varepsilon^{1-1/v},\sqrt{\frac{\log(1/\delta_{2\Delta})}{d_1}}\right)\geq 2\zeta_{d_1,\Delta+d_1}.
\end{equation}
We can now prove the lower bound for $\Pb(\hat{t}>\Delta+d_1)$:
\begin{align*}
     \Pb(\hat{t}>\Delta+d_1)\
\leq& \Pb(Z_{d_1,\Delta+d_1}<\zeta_{d_1,\Delta+d_1}) \\
    =& \Pb(|\RUME(X_{1:d_1})-\RUME(X_{(\Delta+1):(\Delta+d_1)})|<\zeta_{d_1, \Delta+d_1}) \\
    \leq & \mathbb{P}(\kappa - |\RUME(X_{1:d_1})-f_1| - |\RUME(X_{(\Delta+1):(\Delta+d_1)}) - f_{\Delta + 1}| < \zeta_{d_1, \Delta+d_1}) \\
    = & \mathbb{P}(|\RUME(X_{1:d_1}) - f_1| + |\RUME(X_{(\Delta+1):(\Delta+d_1)}) - f_{\Delta + 1}| > \kappa - \zeta_{d_1, \Delta+d_1}) \\
    \leq & \mathbb{P}\left(|\RUME(X_{1:d_1}) - f_1| + |\RUME(X_{(\Delta+1):(\Delta+d_1)}) - f_{\Delta + 1}| > \zeta_{d_1,\Delta+d_1} \right)\leq \delta_{\Delta+d_1}.
\end{align*}
where the third inequality follows from \eqref{proof_eqn:minkappa_fm} and the fourth inequality follows from \eqref{proof_eqn:rume_cpt_upperbd1}. 

(d) For notational simplicity, define
$$d_2=\left\lceil\max\left\{\frac{4}{v\log\kappa},\frac{2}{\log(1/(4e\varepsilon))}\right\}\log\bigg(\frac{3\Delta^3}{\alpha}\bigg)\right\rceil\asymp \max\left\{\frac{1}{v\log\kappa},\frac{1}{\log(1/\varepsilon)}\right\}\log\bigg(\frac{\Delta}{\alpha}\bigg).$$
Our objective is to prove that the detection delay is bounded by $d_2$ with high probability:
\begin{equation*}
\mathbb{P}(\Delta < \hat{t} \leq \Delta + d_2)= \mathbb{P}(\hat{t} > \Delta) - \mathbb{P}(\hat{t} > \Delta + d_2) \geq 1 - \alpha.
\end{equation*}
From \eqref{proof_eqn:t1e_fm_early}, we have $\mathbb{P}(\hat{t} > \Delta) \geq 1 - \sum_{t=2}^{\Delta}\sum_{s=1}^{\lfloor t/2\rfloor} \delta_t$. Given that $\alpha$ is the infinite sum of $\delta_t$, the lower bound $1 - \alpha$ holds provided we show:
\begin{equation*}
\mathbb{P}(\hat{t} > \Delta + d_2) \leq \delta_{\Delta+d_2}.
\end{equation*}
By construction, $d_2$ satisfies the requirements of Lemma \ref{prop2} ($d_2= h_{2\Delta} \geq h_{\Delta+d_2}$). Therefore, by considering the event $\mathcal{E}_{d_2,\Delta+d_2}$, we have that with probability $1-\delta_{\Delta+d_2}$,
$$|\med(X_{1:d_2}) - f_1| + |\med(X_{(\Delta+1):(\Delta+d_2)}) - f_{\Delta + 1}| \leq c_{d_2,\Delta+d_2}+c_{d_2,\Delta+d_2}=\chi_{d_2,\Delta+d_2}.$$
Thus we have 
\begin{equation} \label{proof_eqn:med_cpt_upperbd}
    \Pb(|\med(X_{1:\Delta}) - f_1| + |\med(X_{(\Delta+1):(\Delta+d_2)}) - f_{\Delta + 1}|>2c_{d_2,\Delta+d_2})\leq \delta_{\Delta+d_2}.
\end{equation}
Meanwhile, we can relate $\kappa$ to $c_{d_2,\Delta+d_2}$. From the definition of $d_2$, the following inequality holds.
\begin{equation}\label{eqn:dkappa}
    d_2\geq \frac{4}{v\log \kappa}\log\frac{3\Delta^3}{\alpha}.
\end{equation}
Next, under Assumption \ref{assum: medcpt}, specifically that $\kappa\geq 64(2e)^{2/v}\phi$, we can show
\begin{equation}\label{eqn:kappa_cond}
    \log(\kappa)\leq2\log\frac{\kappa}{8(2e\phi^{v})^{1/v}(1-\varepsilon)^{1/v}}.
\end{equation}
Combining the last two inequalities \eqref{eqn:dkappa} and \eqref{eqn:kappa_cond} gives
\begin{align*}
    d_2\geq \frac{2\log\frac{3\Delta^3}{\alpha}}{v\log({\kappa}/{8(2e\phi^{v})^{1/v}(1-\varepsilon)^{1/v}})},
\end{align*}
which is equivalent to 
\begin{equation} \label{proof:originalkappa}
    \kappa\geq \frac{8(2e\phi^{v}(1-\varepsilon))^{1/v}}{(\alpha/(3\Delta^3))^{2/dv}}.
\end{equation}
Using the definition of $d_2$ again, we have both
\begin{align}
    (\alpha/(3\Delta^3))^{2/d_2}&\geq4e\varepsilon, \nonumber\\
    (\alpha/(3\Delta^3))^{2/d_2}-2e\varepsilon&\geq \frac{1}{2}(\alpha/(3\Delta^3))^{2/d_2}. \label{proof:getridepsilon}
\end{align}
Finally, combining the inequalities \eqref{proof:originalkappa} and \eqref{proof:getridepsilon} implies 
\begin{equation}\label{proof_eqn:kappa_c}
    \kappa\geq 4 \frac{(2e\phi^{v}(1-\varepsilon))^{1/v}}{((\alpha/(3\Delta^3))^{2/d_2}-2e\varepsilon)^{1/v}}=4c_{d_2,2\Delta}\geq4c_{d_2,\Delta+d_2}.
\end{equation}
Finally, we can now prove the lower bound for $\Pb(\hat{t}>\Delta+d_2)$.
    \begin{align*}
     \Pb(\hat{t}>\Delta+d_2)\leq& \Pb(Z_{d_2,\Delta+d_2}<\chi_{d_2,\Delta+d_2}) \\
    =& \Pb(|\med(X_{1:d_2})-\med(X_{(\Delta+1):(\Delta+d_2)})|<\chi_{d_2,\Delta+d_2}) \\
    \leq & \mathbb{P}(\kappa - |\med(X_{1:d_2})-f_1| - |\med(X_{(\Delta+1):(\Delta+d_2)}) - f_{\Delta + 1}| < \chi_{d_2,\Delta+d_2}) \\
    = & \mathbb{P}(|\med(X_{1:d_2}) - f_1| + |\med(X_{(\Delta+1):(\Delta+d_2)}) - f_{\Delta + 1}| > \kappa - 2c_{d_2,\Delta+d_2}) \\
    \leq & \mathbb{P}(|\med(X_{1:d_2}) - f_1| + |\med(X_{(\Delta+1):(\Delta+d_2)}) - f_{\Delta + 1}| > 2c_{d_2,\Delta+d_2})
    \leq \delta_{\Delta+d_2}.
\end{align*}
where the third inequality is due to \eqref{proof_eqn:kappa_c} and the fourth inequality follows from \eqref{proof_eqn:med_cpt_upperbd}.
\end{proof}

\section{Technical details and proofs of results in Section \ref{sect:hdtest}}\label{sect:hdtestproof}

\subsection{Notation and preliminaries}
In this subsection, we introduce the notation and probabilistic preliminaries needed for the design and analysis of our robust mean testing procedure, RobustMeanTest (\Cref{alg:efficient-mean-tester}). 

\subsubsection*{Good and bad sets} Let $\mathcal{Y}=\{Y_1,\ldots, Y_n\}$ be the sample of $n$ random variables generated as follows.
Let $d_i\sim \mathrm{Bern}(\varepsilon)$ denote the contamination label, so that $(Y_i \mid d_i=0)\sim F$ and $(Y_i \mid d_i=1)\sim H_i$. Fix a truncation radius $R>0$, to be specified later.
We define the \emph{good} and \emph{bad} index sets as
\[
\mathfrak{G} = \{ i\in[n] : \norm{Y_i-\mu}_2 \le R,\ d_i=0 \},
\qquad
\mathcal{B} = [n]\setminus \mathfrak{G} .
\]
Throughout this section, $\mathcal{Y}$, $\mathfrak{G}$, and $\mathcal{B}$ will always refer to these sets. Define
\[
q = 1-\Pb(\norm{Y_1-\mu}_2 \le R,\ d_1=0),
\qquad
\gamma = \Pb_{Y\sim F}(\norm{Y-\mu}_2 > R).
\]
Then
\begin{align*}
q
&= 1-\Pb_{Y_1\sim F}(\norm{Y_1-\mu}_2 \le R)\,\Pb(d_1=0) \\
&= 1-(1-\gamma)(1-\varepsilon)
\;\le\; \varepsilon+\gamma .
\end{align*}
Since $|\mathcal{B}|\sim\mathrm{Binom}(n,q)$, Bernstein's inequality (Theorem 2.9.5, \citealp{Vershynin_2025}) implies that,
with probability at least $1-\delta$, 
\begin{equation}\label{eq:badsetsize}
|\mathcal{B}|\leq un, \quad \text{ where } u=q+\sqrt{\frac{2q\log(1/\delta)}{n}}+\frac{2\log(1/\delta)}{3n}.
\end{equation}
Denote event $A=\{|\mathcal{B}|\leq un\}$. For our proofs below, we will assume our variables $\varepsilon, p, n,\delta$ satisfy $u\leq 0.08$.

\subsubsection*{Choice of parameters}
We next describe the choices of $R$ and $\gamma$, together with an upper bound on $u$, for two classes of inlier distributions $\mathcal{D}$. These parameter values follow from the concentration result in \Cref{lem:cov_matrix_conc}.

\begin{itemize}
\item
If $\mathcal{D}=\mathcal{G}_{\theta,M}^p$, set
\begin{align*}
R^2
&=
p + M^2 \cdot O\!\left(\log^{2/\theta}(n) + \sqrt{p\log n}\right),\\
\gamma
&\le
2\exp\!\left(
- C_\theta
\min\left\{
\frac{(R^2-p)^2}{pM^4},
\left(\frac{R^2-p}{M^2}\right)^{\theta/2}
\right\}
\right)
\le \frac{1}{n},
\end{align*}
where the first inequality for $\gamma$ follows from \Cref{prop:quadraticweibulltail}, and $C_\theta>0$ is a constant depending only on $\theta$. The second inequality follows from the above choice of $R$. Thus, up to logarithmic factors, we have
$$u\lesssim \varepsilon + 1/n,$$
since $\sqrt{\varepsilon/n}$ is dominated by $\varepsilon + 1/n$.
\item
If $\mathcal{D}=\mathcal{P}_{v,\phi}^p$ for $v\geq 4$, set
\begin{align}
R^2
&=
p +O\left(\sqrt{\max(n,p)}\right), \label{eqn:fm_R}\\
\gamma
&\le
O\left(\frac{p^{v/4}}{(R^2-p)^{v/2}}\right)
= \min\left(\frac{1}{20}, \frac{1}{20}\left(\frac{p}{n}\right)^{v/4}\right). \label{eqn:fm_gamma}
\end{align}
where the first inequality for $\gamma$ follows from Markov's inequality (using the $v$-th moment information) and the second inequality follows from the above choice of $R$. Thus, up to logarithmic factors, we have
\begin{equation}\label{eqn:u_bound}
    u\lesssim \varepsilon + \min(1,(p/n)^{v/4}).
\end{equation}
For the proofs below, we sometimes use the fact that for any given $v\geq 4$ and $\phi>0$, we have $\mathcal{P}_{v,\phi}\subseteq \mathcal{P}_{4,\psi}$ for some large enough $\psi$. In the proofs, it should be clear from context when we are specifying $u\lesssim \varepsilon + 1$ or $u\lesssim \varepsilon+(p/n)^{v/4}$.
\end{itemize}

\subsubsection*{Weights and linear-algebraic quantities}
Our algorithm will assign weights to each point, that we will monotonically decrease over time.
For any $n$, let $\Gamma_n$ denote the set of valid weights:
\[
\Gamma_n = \{w \in \R^n: w_i \in [0, 1] \mbox{ for all $i = 1, \ldots, n$}\}.
\]
Recall that for any set $\mathcal{M} \subseteq \mathcal{Y}$, $\mathbf{1}_\mathcal{M} \in \Gamma_n$ denotes the indicator vector for $\mathcal{M}$, and $\mathbf{1} = \mathbf{1}_S$.

Let $N$ be some value to be specified later.
Given a set of points $\mathcal{Y} = \{ Y_1, \ldots, Y_n \}$, we associate it weight vectors $w^{(t)} \in \Gamma_n$, for $i = 1, \ldots, n$ and $t = 1,  \ldots, N$ where initially we set $w^{(1)} = \mathbf{1}$.
For any such weight vector $w$, we let
\[
\Sum(w, \mathcal{Y}) = \sum_{i \in \mathcal{Y}} \sqrt{w_i} Y_i \; , \mbox{and} \; M (w, \mathcal{Y}) = \sum_{i \in \mathcal{Y}} w_i Y_i Y_i^\top.\]
When the context is clear, we will drop the $\mathcal{Y}$ from the notation for simplicity.
For any set $\mathcal{M}$, and for any set of weights $w$ on $\mathcal{Y}$, we let $w_\mathcal{M}$ denote the set of weights restricted to the indices in $\mathcal{M} \cap \mathcal{Y}$.
We also let $\Gram (w, \mathcal{Y})$  be the $n \times n$ matrix given by
\[
\Gram (w,\mathcal{Y})_{ij} = \sqrt{w_i w_j} \iprod{Y_i, Y_j}.
\]
Note that by design, the nontrivial eigenvalues of $\Gram (w,\mathcal{Y})$ and $M(w,\mathcal{Y})$ are identical.

We define the following important set:
\begin{equation}\label{eqn:important_set}
    \Lambda_n = \{ w \in \Gamma_n : \norm{\mathbf{1}_{G}- w_{G}}_1 \leq 5 \norm{\mathbf{1}_B - w_\mathcal{B}}_1 \},
\end{equation}
that is, $\Lambda_n$ is the set of weights where we have removed at most five times as much weight from the good samples as we have removed from the the bad samples.

\subsubsection*{Downweighting scheme} We will be using the following downweighting scheme:
\begin{proposition}[\cite{Canonne2023}, Fact 8.13]
    \label{fact:1d-filter}
    Let $w \in \Lambda_n$, and let $\tau_1, \ldots, \tau_n$ be a set of non-negative scores satisfying $\sum_{i \in \mathfrak{G}} w_i \tau_i \leq 5 \sum_{i \in \mathcal{B}} w_i \tau_i$.
    Let $w' \in \Gamma_n$ be defined by 
    \[
    w'_i = \left(1 - \frac{\tau_i}{\max_{i \in \mathcal{Y}} \tau_i} \right) w_i.
    \]
    Then $\supp{w'} \subset \supp{w}$, and moreover $w' \in \Lambda_n$.
\end{proposition}
The downweighting procedure defined in Proposition \ref{fact:1d-filter} ensures that the influence of contaminated points is systematically diminished. The key feature of this procedure is that the weights before and after downweighting remain in the set $\Lambda_n$, which simplifies the analysis of the algorithm.

\subsubsection*{Concentration of second moment matrix} Our filtering algorithm relies on inspecting the operator norm of the second moment matrix. Thus, we need the following concentration inequalities.
\begin{lemma}\label{lem:cov_matrix_conc0}
Let $Y_1,\ldots, Y_n$ be independent random variables in $\R^p$ drawn from \eqref{eqn:huberhd}. 
\begin{enumerate}[label=(\textbf{\alph*})]
    \item  If $n\leq p$, then for both $F\in\mathcal{P}_{4,\phi}^p$ and $F\in\mathcal{G}_{\theta,M}^p$, with probability at least $1-\delta$, we have
    $$\norm{\Gram(\mathfrak{G})-pI}_{\op}\leq O(p\log(2p/\delta))+n\kappa^2.$$
    \item  If $n\geq p$, then for both $F\in\mathcal{P}_{4,\phi}^p$ and $F\in\mathcal{G}_{\theta,M}^p$, with probability at least $1-\delta$, we have
    $$\norm{M(\mathfrak{G})-nI}_{\op}\leq O(\sqrt{np}\log(2p/\delta))+\varepsilon n+n\kappa^2.$$
\end{enumerate}
\end{lemma}

If an $\varepsilon$-net argument is used, the slower tail decay of sub-Weibull distributions incurs a $p^{1/\theta}$ dependence in the rate. To bypass this challenge, we employ matrix Bernstein inequality.

\subsection{Key features of the testing algorithm}\label{subsect:detail_routine}
\subsubsection*{Choice of test statistic} The key insight is that once these weights are refined appropriately, the norm $\|\Sum(w, \mathcal{Y})\|_2^2$ becomes a reliable proxy for the presence of a mean shift. Lemma \ref{lem:regular-weights-imply-tester} formalises this by showing that for $\mathcal{D}=\mathcal{G}^p_{\theta,M}$, the weighted sum's magnitude effectively distinguishes between $\mu = \mathbf{0}$ and the alternative hypothesis, provided that $\kappa_0$ is appropriately chosen.

\begin{lemma}
\label{lem:regular-weights-imply-tester}
Consider $\mathcal{D}=\mathcal{G}^p_{\theta,M}$. Let $w \in \Gamma_n$ be a weight vector satisfying $n\geq \norm{w_\mathfrak{G}}_1\geq (1-7u)n$. 
\begin{enumerate}[label=(\textbf{\alph*})]
    \item With probability $1-\delta$, we have
    \begin{align}
        \norm{\Sum(w,\mathfrak{G})}^2&= \norm{w_\mathfrak{G}}_1p+\left(\sum_{i\in \mathfrak{G}}\sqrt{w_i}\right)^2\kappa^2\pm  O(\kappa n\sqrt{n}\log^{1/\theta}(1/\delta)+n\sqrt{p}\log^{2/\theta}(1/\delta)).  \label{proofeqn:sumwg}
    \end{align}
    \item Suppose event $A$ and that \eqref{proofeqn:sumwg} holds. In addition, suppose we can find $\beta_1$ and $\beta_2$ satisfying
$\iprod{\Sum (w, \mathcal{Y}), \Sum (w, \mathcal{B})}= p  \|w_\mathcal{B}\|_1\pm un^2\kappa^2\pm O(\beta_1+\beta_2)$ and $\norm{\Sum (w, \mathcal{B})}^2 = O(\beta_2)$. Define $T_u=(1-7u)^2-2u$.
\begin{itemize}
    \item If $\mu = \mathbf{0}$, $O(\beta_1+\beta_2) \le 0.5T_un^2\kappa_0^2$ and
    \begin{align*}
    \kappa_0^2\gtrsim \frac{\sqrt{p}\log^{2/\theta}(1/\delta)}{n},
    \end{align*}
    then $$|\norm{\Sum (w, \mathcal{Y})}^2 - p \norm{w}_1| \leq 0.5T_un ^2\kappa_0^2. $$
    \item If $\mu \neq \mathbf{0}$, $O(\beta_1+\beta_2) \le T_un^2(\kappa^2-0.5\kappa_0^2)$ and
    \begin{align*}
    \kappa^2-0.5\kappa_0^2\gtrsim \frac{\sqrt{p}\log^{2/\theta}(1/\delta)}{n}+\frac{\kappa\log^{1/\theta}(1/\delta)}{\sqrt{n}},
    \end{align*}
    then $$|\norm{\Sum (w, \mathcal{Y})}^2 - p \norm{w}_1| > 0.5T_u n^2 \kappa_0^2. $$
\end{itemize}
In other words, the norm of the sum of the set of points distinguishes between the null and alternative hypotheses.
\end{enumerate}
\end{lemma}

Similarly, Lemma \ref{lem:regular-weights-imply-tester4mom} shows that for $\mathcal{D}=\mathcal{P}^p_{v,\phi}$, the magnitude of $\norm{\Sum(w,\mathcal{Y})}^2$ allows us to distinguish between $\mu = \mathbf{0}$ and the alternative hypothesis, provided that $\kappa_0$ is appropriately chosen.
\begin{lemma}
\label{lem:regular-weights-imply-tester4mom}
Consider $\mathcal{D}=\mathcal{P}^p_{v,\phi}$. Let $w \in \Gamma_n$ be a set of weights on $w$ that satisfies $n\geq \norm{w_\mathfrak{G}}_1\geq (1-7u)n$. 
\begin{enumerate}[label=(\textbf{\alph*})]
    \item With probability $1-\delta$, we have
    \begin{align}
        \norm{\Sum(w,\mathfrak{G})}^2&=p\norm{w_\mathfrak{G}}_1+\left(\sum_{i\in \mathfrak{G}}\sqrt{w_i}\right)^2 \kappa^2 \nonumber\\
        &\pm O(\kappa n^{\frac{9-v}{4}}p^{\frac{v-1}{4}}\wedge \kappa n^2+\kappa n\sqrt{n/\delta}+p^{2-\frac{2}{v}}n^{\frac{2}{v}}\wedge n^2+n\sqrt{p/\delta}).  \label{proofeqn:sumwg2}
    \end{align}
    \item Suppose event $A$ and that \eqref{proofeqn:sumwg2} holds. In addition, suppose we can find $\beta_1$ and $\beta_2$ satisfying
$\iprod{\Sum (w, \mathcal{Y}), \Sum (w, \mathcal{B})}= p  \|w_\mathcal{B}\|_1\pm un^2\kappa^2\pm O(\beta_1+\beta_2)$ and $\norm{\Sum (w, \mathcal{B})}^2 = O(\beta_2)$. Define $T_u=(1-7u)^2-2u$.
\begin{itemize}
    \item If $\mu = \mathbf{0}$, $O(\beta_1+\beta_2)\leq 0.5T_un^2\kappa_0^2$ and
    \begin{equation*}
        \kappa_0^2\gtrsim (p/n)^{2-2/v}\wedge 1+\frac{\sqrt{p}}{n\sqrt{\delta}},
    \end{equation*}
    then $$|\norm{\Sum (w, \mathcal{Y})}^2 - p \norm{w}_1| \leq 0.5T_u n ^2 \kappa^2_0.$$
    \item If $\mu \neq \mathbf{0}$, $O(\beta_1+\beta_2) \leq T_un^2(\kappa^2-0.5\kappa_0^2)$ and
    \begin{equation*}
        \kappa^2-0.5\kappa_0^2\gtrsim \kappa \left(\frac{p}{n}\right)^{(v-1)/4}\wedge \kappa+\frac{\kappa}{\sqrt{n\delta}}+\left(\frac{p}{n}\right)^{2-2/v}\wedge 1+\frac{\sqrt{p}}{n\sqrt{\delta}},
    \end{equation*}
    then $$|\norm{\Sum (w, \mathcal{Y})}^2 - p \norm{w}_1| > 0.5T_u n^2 \kappa^2 _0.$$
\end{itemize}
In other words, the norm of the sum of the set of points distinguishes between the null and alternative hypotheses.
\end{enumerate}
\end{lemma}

\subsubsection*{The filtering algorithm for \texorpdfstring{$n \leq p$}{n <= p}}
We first describe Algorithm~\ref{alg:spectralfilter1}, which filters the data points when the sample size is fewer than the number of dimensions. Let $\delta > 0$, and
\begin{equation}
\label{eq:gamma2}
    R_{f} = C_{\gamma} \left(unp\log(1/u)+(\sqrt{np}+p)\log(2p/\delta)+ n \kappa^2_0\right),
\end{equation}
for some constant $C_{\gamma}$ sufficiently large. 

\begin{algorithm}[H]
\caption{Spectral filtering for $n \leq p$ (GramFilter$(\{Y_i\}_{i=1}^n, R_{f})$)}
\label{alg:spectralfilter1}
\begin{algorithmic}
\INPUT{Dataset $Y_1,\ldots,Y_n \in \R^p$, Filtering radius $R_{f} > 0$}
\State Let $w^{(1)} \leftarrow \mathbf{1}$, and let $t \leftarrow 1$
\While{$\norm{\Gram (w^{(t)}, \mathcal{Y}) - p\cdot \mathrm{diag}(w^{(t)})}_\op \geq 5 R_{f}$}
    \State Let $v$ be the top singular vector of $\Gram (w^{(t)}, \mathcal{Y}) - D(w^{(t)})$
    \State For all $i$, $\tau_i \gets \frac{v_i^2}{w_i^{(t)}} \mathbb{I}[w^{(t)}_i > 0]$\Comment{If $w_i^{(t)} = 0$, we set $\tau_i = 0$.}
    \State $w^{(t + 1)}_i \leftarrow \left( 1  - \frac{\tau_i}{\max_j \tau_j}\right) w^{(t)}_i$
    \State $t \gets t + 1$
\EndWhile
\State \textbf{Return} $w^{(t)}$
\end{algorithmic}
\end{algorithm}

Algorithm~\ref{alg:spectralfilter1} begins by initialising the weights as $w^{(1)}=\mathbf{1}$.
Then, for $t = 1$ until termination, we proceed as follows.
For any $w \in \Gamma_n$, let $D(w) = p \cdot \mathrm{diag} (w)$.
Let $\lambda$ denote the top singular value of $\Gram (w, \mathcal{Y}) - D(w)$, and let $v$ be its associated singular unit vector (if there are multiple, choose any).
If $\lambda < 5 R_{f}$, then terminate.
Otherwise, for all $k \in \mathcal{Y}$, let $\tau_k = \frac{v_k^2}{w_k^{(t)}} \mathbb{I}[w^{(t)}_k > 0]$ (where $\tau_k$ defaults to $0$ when $w_k^{(t)} = 0$), and proceed to sort the samples in decreasing order of $\tau_k$.
Then, define $w^{(t + 1)}$ by 
\[
w^{(t + 1)}_k = \left( 1  - \frac{\tau_k}{\max_j \tau_j}\right) w^{(t)}_k.
\]
\Cref{lem:spectralfilter1} below shows that at termination, Algorithm~\ref{alg:spectralfilter1} gives an upper bound on $\norm{\Sum(w^{(N)}, \mathcal{B})}^2$.
\begin{lemma}[\cite{Canonne2023}, Lemma 8.14]
\label{lem:spectralfilter1}
    Assume that $ C_{\gamma}\kappa_0^2/2\geq  \kappa^2$. 
    Then, with probability at least $1-\delta$, Algorithm~\ref{alg:spectralfilter1} terminates in $N \leq 6un$ iterations, each taking $O(pn^2)$ time, and outputs $w^{(N)} \in \Lambda_n$ such that for all $\mathcal{M} \subset \mathcal{Y}$ with $|\mathcal{M}| \leq un$,
\[
\norm{\Sum(w^{(N)}, \mathcal{M})}^2 
= \norm{w^{(N)}_\mathcal{M}}_1\, p \pm O(un\,R_{f}).
\]
\end{lemma}

\begin{remark}\label{rem:weightsum}
    Given event $A$, defined in \eqref{eq:badsetsize}, the number of iterations $N$ being at most $6un$ ensures $\| \mathbf{1}_\mathcal{B} - w^{(N)}_\mathcal{B} \|_1 \leq un$, which implies $\| \mathbf{1}_\mathfrak{G} - w_\mathfrak{G}^{(N)} \|_1 \leq 5un$ using \eqref{eqn:important_set}. Since $|\mathfrak{G}| \geq (1 - u)n$ and $\sqrt{w_i} \geq w_i$ for $w_i \in [0, 1]$, it follows that
\begin{align*}
\sum_{i \in \mathfrak{G}} \sqrt{w_i} \geq |\mathfrak{G}| - \norm{\mathbf{1}_\mathfrak{G} - w_\mathfrak{G}}_1 \geq (1 - 6u)n.
\end{align*}
\end{remark}

\subsubsection*{The filtering algorithm for \texorpdfstring{$n > p$}{n > p}}
We now describe Algorithm~\ref{alg:spectralfilter2}, which filters the data points when the sample size is larger than the number of dimensions.
Let $R_{f}$ be as in~\eqref{eq:gamma2}. 

\begin{algorithm}[H]

\caption{Spectral filtering for $n> p$ (MomentFilter$(\{Y_i\}_{i=1}^n,R_{f},u)$)}
\label{alg:spectralfilter2}

    \begin{algorithmic}
\INPUT Dataset $Y_1,\ldots,Y_n \in \R^p$, Filtering radius $R_{f} > 0$
\State Let $w^{(1)} \leftarrow \mathbf{1}$, and let $t \leftarrow 1$
\While{$\norm{M (w^{(t)}, \mathcal{Y}) - n I}_\op \geq 5 R_{f}$}
    \State Let $v$ be the top singular vector of $M(w, \mathcal{Y}) - n I$
    \State For all $i$, $\tau_i \gets \iprod{v, Y_i}^2 \mathbb{I}[w^{(t)}_i > 0]$. Sort the indices in decreasing order by $\tau_i$. \Comment{By relabelling indices, for simplicity of notation assume that the $i$'s are initially sorted}
    \State Let $L$ be the smallest index such that $\sum_{i \leq L} w^{(i)}_t\geq 2un$.
    \State $w^{(t + 1)}\gets\begin{cases}
        \left( 1 - \frac{\tau_i}{\tau_1} \right) w^{(t)}_i & \text{if $i \leq L$};\\
        w^{(t)}_i & \text{if $i > L$}.
    \end{cases}$
    \State $t \gets t + 1$
\EndWhile
\State \textbf{Return} $w^{(t)}$
\end{algorithmic}
\end{algorithm}

Algorithm~\ref{alg:spectralfilter2} begins by initialising the weights as $w^{(1)} = \mathbf{1}$. Then, for $t =1$ until termination, we proceed as follows. 
Let $\lambda$ be the top singular value of $M(w^{(t)}) - n I$, and let $v$ be its associated singular value (if there are multiple, again choose one arbitrarily).
If $\lambda < 5 R_{f}$, then terminate.
Otherwise, for all $k \in \mathfrak{G}$, let $\tau_k = \iprod{v, Y_k}^2 \mathbb{I}[w_k > 0]$.
Proceed to sort the samples in decreasing order of $\tau_k$.
As before, by relabelling indices, assume that $\tau_1 \geq \tau_2 \geq \dots \geq \tau_n$.
Let $L$ be the smallest index so that $\sum_{k \leq L} w_k^{(t)} \geq 2 u n$, and define $w^{(t + 1)}$ by 
\begin{equation*}
w^{(t + 1)}_k = \left\{ \begin{array}{ll}
         \left( 1 - \frac{\tau_k}{\tau_1} \right) w^{(t)}_k & \mbox{if $k \leq L$},\\
        w^{(t)}_k & \mbox{if $k > L$}.\end{array} \right. 
\end{equation*}
\Cref{lem:spectralfilter2} gives an upper bound on $\norm{\Sum(w^{(N)}, \mathcal{B})}^2$ for the output $w^{(N)}$ returned by \Cref{alg:spectralfilter2}.
\begin{lemma}[\cite{Canonne2023}, Lemma 8.15]
\label{lem:spectralfilter2}
    Assume $C_\gamma\kappa_0^2/2\ge \kappa^2$. With probability $1-2\delta$, Algorithm~\ref{alg:spectralfilter2} terminates in $N$ iterations for some $N \leq 6un$, each taking $O(pn^2)$ time, and outputs $w^{(N)} \in \Lambda_n$ such that for all $\mathcal{M} \subset \mathcal{Y}$ with $|\mathcal{M}| \leq un$,
    \[
        \norm{\Sum (w^{(N)}, \mathcal{M})}_2^2 \leq 10 R_{f} u n.
    \]
\end{lemma}

\begin{remark}\label{rmk:boundTB}
    The above proof also shows that for all $\mathcal{M}\subset \mathcal{Y}$ such that $|\mathcal{M}|\leq 2un$, we have 
    $$\norm{\Sum(w^{(N)},\mathcal{M})}^2_2\leq 20unR_{f}.$$
\end{remark}

\begin{remark}\label{rmk:symmetrize}
    The condition $C_\gamma \kappa_0^2/2 \geq \kappa^2$ can be relaxed as follows. By pairing data points and taking pairwise differences, e.g. $Z_j = Y_{2j} - Y_{2j-1}$ for $j\in\{1,\ldots,\lfloor n/2\rfloor\}$, we create a dataset of size $\lfloor n/2\rfloor$ with mean 0. When we apply the same filtering approach to this dataset, the $n\kappa_0^2$ term in \eqref{eq:gamma2}, that originally arose from the rank-1 contribution $O(n)\mu\mu^\top$ in the upper bound of $\norm{M(\mathfrak{G})-nI}_{\op}$, will not appear in the empirical second moment matrix of the pairwise-difference dataset. Hence, we can remove the condition $C_{\gamma}\kappa_0^2/10\geq \kappa^2$ using this approach, though the minimum signal strength required for type I and II error control will remain unchanged.
\end{remark}

\subsubsection*{Bounding row sums}
\Cref{lem:spectralfilter1} and \Cref{lem:spectralfilter2} show that the filtering subroutines return an output $w^{(N)}$ under which the contribution from $\norm{\Sum(w^{(N)},\mathcal{B})}_2^2$ can be controlled by a suitable upper bound. Now, we further refine the weights through Algorithm~\ref{alg:rowsumfilter}.

\begin{algorithm}[H]
\caption{Bounding row sums (RowSumFilter$(\{Y_i\}_{i=1}^n, w,u)$)}
\label{alg:rowsumfilter}

    \begin{algorithmic}
        \INPUT{Dataset $Y_1,\ldots,Y_n \in \R^p$, Weights $w\in [0,1]^n$, Maximum contamination proportion $u$}
        \State For all $i$, $\tau_i \gets \left|\iprod{\sqrt{w_i} Y_i, \Sum (w, \mathcal{Y})} - w_i p\right| \cdot \mathbb{I} [w_i > 0]$.
        \State Sort the indices in decreasing order by $\tau_i$.
        \Comment{By relabelling indices, for simplicity of notation assume that the $i$'s are initially sorted}
\State Set $w_i \leftarrow 0$ for all $i \leq u n$.
        \State \textbf{return} $w$
    \end{algorithmic}
\end{algorithm}
Algorithm~\ref{alg:rowsumfilter} simply removes the set of $O(un)$ points whose row sums have largest deviation from what we expect.
More formally, given a set of weights $w \in \Lambda_n$, we will let 
\begin{equation}
\label{eq:rowsumscore}
\tau_k = \left|\iprod{\sqrt{w_k} Y_k, \sum_{j \in \mathcal{Y}} \sqrt{w_j} Y_j} - w_k p\right| \cdot \mathbb{I} [w_k > 0].
\end{equation}
We then sort the indices in decreasing order by $\tau_k$.
Again for simplicity of notation, assume that after some suitable reindexing we have that $\tau_1 \geq \tau_2 \geq \ldots \geq \tau_n$. Then, we replace $w_k$ with $0$ for all $k \le u n$. \Cref{lem:rowsumfilter} gives an upper bound on the row sums $\iprod{\Sum(w',\mathcal{Y}),\Sum(w',\mathcal{B})}$, where $w'$ is the output of \Cref{alg:rowsumfilter}.

\begin{lemma}
\label{lem:rowsumfilter}
    Let $w'$ be the output of Algorithm~\ref{alg:rowsumfilter}. If the event with probability $1-\delta$ in Lemma~\ref{lem:rowsumfilter-1} or Lemma~\ref{lem:rowsumfilter-1b} holds, then, for all $\mathcal{M} \subset \mathcal{Y}$ with $|\mathcal{M}| \leq u n$, we have that 
\[
\sum_{i \in \mathcal{M}, j \in [n]} \sqrt{w_i' w_j'} \iprod{Y_i, Y_j} = p \|w'_\mathcal{M}\|_1\pm un^2\kappa^2 \pm O(\beta_1+\beta_2),
\]
where $\beta_1,\beta_2$ are as defined in Lemma~\ref{lem:rowsumfilter-1} (or Lemma~\ref{lem:rowsumfilter-1b}).
\end{lemma}

Although the output $w^{(N)}$ of the filtering subroutines is modified in the final row-sum step, the upper bounds on
$\norm{\Sum(w',\mathcal{M})}_2^2$ for all $\mathcal{M}\subset \mathcal{Y}$ with $|\mathcal{M}|\le un$
still follow from Lemmas~\ref{lem:spectralfilter1} and~\ref{lem:spectralfilter2}. Indeed, if $\mathcal{X}$ denotes the set of indices whose weights are set to zero in this final step, then $\Sum(w',\mathcal{M})=\Sum(w^{(N)},\mathcal{M}\setminus \mathcal{X})$. Since $|\mathcal{M}\setminus\mathcal{X}|\le |\mathcal{M}|\le un$, the bounds from Lemmas~\ref{lem:spectralfilter1} and~\ref{lem:spectralfilter2} apply directly to $\mathcal{M}\setminus\mathcal{X}$.

\subsection{Proof of Proposition \ref{thm:poly-time-main}}
\begin{proof}
    The runtime of the algorithm is clearly dominated by the runtime of the spectral filters, which both run in time $O(un^2 p \min (n, p))$, and the runtime of computing $\Sum (w, \mathcal{Y})$, which is $O(n p)$. 
    
    Conditioned on event $A$, which occurs with probability $1 - \delta/4$, the following results hold:
    \begin{itemize}
        \item Weight Stability: By \Cref{rem:weightsum}, the weights $w$ returned by RowSumFilter in \Cref{alg:efficient-mean-tester} satisfy $$\sum_{i \in \mathfrak{G}} \sqrt{w_i} \geq |\mathfrak{G}| - \|\mathbf{1}_G - w_\mathfrak{G}\|_1 \geq (1 - 6u)n.$$ Accounting for the operations in Step 7, $\sum_{i \in \mathfrak{G}} \sqrt{w_i}$ is at least $(1 - 7u)n$ in the worst case.
        \item Assumptions in \eqref{eq:gram-assume} or \eqref{eq:cov-assume} are satisfied with probability $1-\delta/4$.
        \item Inner Product Bound: \Cref{lem:rowsumfilter-1} implies that with probability at least $1 - \delta/4$, the bound on $\langle \Sum(w, \mathcal{Y}), \Sum(w, \mathcal{B}) \rangle$ holds.
        \item Good set sum bound: the upper bound on $\|\Sum(w, \mathfrak{G})\|^2$ in \eqref{proofeqn:sumwg} is satisfied with  probability at least $1 - \delta/4$.
    \end{itemize}
    By a union bound over these three failure events, each occurring with probability at most $\delta/4$, all aforementioned conditions hold simultaneously with probability at least $1 - \delta$. 

    To control for type I error, \Cref{lem:regular-weights-imply-tester} gives us the conditions $0.5T_un^2\kappa_0^2\geq O(\beta_1(0)+\beta_2)$ and
    $$\kappa_0^2\gtrsim \frac{\sqrt{p}\log^{2/\theta}(1/\delta)}{n}.$$
    Similarly, to control for type II error, we require $T_un^2(\kappa^2-0.5\kappa_0^2)\geq O(\beta_1(0)+\beta_2)$ and
    $$\kappa^2-0.5\kappa_0^2\gtrsim \frac{\sqrt{p}\log^{2/\theta}(1/\delta)}{n}+\frac{\kappa\log^{1/\theta}(1/\delta)}{\sqrt{n}},$$
    by conditions in \Cref{lem:regular-weights-imply-tester}. Meanwhile, the results of \Cref{lem:spectralfilter1} or \Cref{lem:spectralfilter2} (depending on whether $n>p$) and \Cref{lem:rowsumfilter} show that the terms $\beta_1$ and $\beta_2$ are 
    \begin{align*}
        \beta_1(\kappa) &= un^{3/2} \sqrt{p} \log^{2/\theta}(n/\delta) + \kappa n\sqrt{un}\log^{1/\theta}(4/\delta). \\
        \beta_2 &= unR_{f} = u^2n^2p\log(1/u) + un(\sqrt{np}+p)\log(2p/\delta) + un^2\kappa_0^2.
    \end{align*}
    Substituting $\beta_1, \beta_2$ into the earlier conditions in gives the final result. 
    For type I error control, we require $n\gtrsim 1$ and
    \begin{align*}
        \kappa_0^2&\gtrsim_{\log} \frac{\sqrt{p}}{n}+ \frac{u\sqrt{p}}{\sqrt{n}}+\varepsilon^2p\log(1/\varepsilon)+\frac{p}{n^2}\asymp_{\log} \frac{\sqrt{p}}{n}+\varepsilon^2p\log(1/\varepsilon)+\frac{p}{n^2},
    \end{align*}
    since $\frac{u\sqrt{p}}{\sqrt{n}}\lesssim u^2p+1/n$.
    For type II error control, we require
    \begin{align}
        \kappa^2-0.5\kappa_0^2&\gtrsim_{\log} \frac{\sqrt{p}}{n}+\frac{\kappa}{\sqrt{n}}+\frac{u\sqrt{p}}{\sqrt{n}}+\frac{\kappa\sqrt{u}}{\sqrt{n}}+\varepsilon^2p\log(1/\varepsilon)+\frac{p}{n^2}+(\varepsilon+n^{-1})\kappa_0^2   \nonumber\\
        &\gtrsim_{\log} \frac{\sqrt{p}}{n}+\frac{\kappa}{\sqrt{n}}+\varepsilon^2p\log(1/\varepsilon)+\frac{p}{n^2}+(\varepsilon+n^{-1})\kappa_0^2\label{proof_eqn:t2e_condition1},
    \end{align}
    since $\frac{u\sqrt{p}}{\sqrt{n}}$ and $\frac{\kappa\sqrt{u}}{\sqrt{n}}$ are dominated by other terms. For any $\kappa_0$ satisfying $2/C_\gamma\leq \kappa_0^2/\kappa^2\leq \bar{c}$, where $1/\bar{c}> 0.5+O(\varepsilon+n^{-1})$, \eqref{proof_eqn:t2e_condition1} is satisfied as long as
    $$\kappa^2\gtrsim_{\log} \varepsilon^2p\log(1/\varepsilon)+\frac{\sqrt{p}}{n}+\frac{p}{n^2}.$$
\end{proof}

\subsection{Theoretical guarantees of Algorithm \ref{alg:efficient-mean-tester} under finite-moment assumption}
Similar to \Cref{thm:poly-time-main}, \Cref{thm:poly-time-main_fm} provides theoretical guarantees of \Cref{alg:efficient-mean-tester} when inlier distributions are drawn from the finite-moment class, yielding control of both type I and type II errors.
\begin{proposition}[Testing for $\mathcal{P}^{p}_{v,\phi}$ class]
\label{thm:poly-time-main_fm}
   Let $n,p \in \mathbb{N}$, $\delta > 0$ and $\varepsilon \in[0,0.08)$. Let $\{Y_i\}_{i\in[n]}$ be independently generated according to \eqref{eqn:huberhd}, with $\mathcal{D}=\mathcal{P}^{p}_{v,\phi}$ for some $v\geq 4$ and $\phi>0$. Assume 
   \begin{equation}\label{eqn:u_val2}
       u=\varepsilon+\frac{1}{20}\min\left(1,\left(\frac{p}{n}\right)^{v/4}\right)+\sqrt{\frac{2(\varepsilon+0.05\min(1, (p/n)^{v/4})\log(4/\delta)}{n}}+\frac{2\log(4/\delta)}{3n}\leq 0.08.
   \end{equation}
Then \Cref{alg:efficient-mean-tester} with detection sensitivity factor input $T_u=(1-7u)^2-2u$ has the following guarantees.
  \begin{enumerate}[label=(\textbf{\alph*})]
      \item The algorithm has runtime $O([\varepsilon+\min\{1,(p/n)^{v/4}\}]n^2 p \min(n, p) + np)$.
      \item If $\mu=\mathbf{0}$ and the input $\kappa_0$ satisfies
              \begin{equation*}
            \kappa_0^2\gtrsim\begin{cases}
                \left[\frac{\sqrt{p}}{n\sqrt{\delta}}+\varepsilon^2p\log(1/\varepsilon)+\left(\frac{p}{n}\right)^{2-2/v}+\frac{p^{1+v/2}}{n^{v/2}}\right]\cdot\mathrm{polylog}(v, n, p),&\text{if $n>p$,}\\
                p \cdot\mathrm{polylog}(v, n, p),  &\text{if $n\leq p$,}
            \end{cases} 
        \end{equation*}
      then the algorithm outputs 0 with probability at least $1-\delta$.
      \item If $\kappa = \|\mu\|_2$ satisfies $0<2/C_\gamma\leq \kappa_0^2/\kappa^2 \leq \bar{c}$ for some absolute constant $\bar{c}>0$, and 
        \begin{equation*}
            \kappa^2\gtrsim\begin{cases}
                \left[\frac{\sqrt{p}}{n\sqrt{\delta}}+\varepsilon^2p\log(1/\varepsilon)+\left(\frac{p}{n}\right)^{2-2/v}+\frac{p^{1+v/2}}{n^{v/2}}\right]\cdot\mathrm{polylog}(v, n, p),&\text{if $n>p$,}\\
                p \cdot\mathrm{polylog}(v, n, p),  &\text{if $n\leq p$,}
            \end{cases} 
        \end{equation*}
      then the algorithm outputs 1 with probability at least $1-\delta$.
  \end{enumerate}
\end{proposition}
The condition \eqref{eqn:u_val2} plays a similar role as \eqref{eqn:u_val}, but is adapted to the setting where the inlier distribution possesses only finite $v$-th moments. In particular, $u$ is defined as an upper bound on the fraction of outliers (i.e. contaminated samples and inlier samples far from the mean) within the dataset, and the bound holds with probability at least $1-\delta/4$.

Comparing with \Cref{thm:poly-time-main}, we observe two key differences. Firstly, the requirements on $\kappa_0$ and $\kappa$ have a polynomial dependence on $\delta$ in \Cref{thm:poly-time-main_fm}, while that in \Cref{thm:poly-time-main} have a logarithmic dependence on $\delta$. Having polynomial dependence is not an issue if we consider $\delta$ as fixed and large, but it is not ideal when it is used as a subroutine to the change point problem. We address this issue in \Cref{subsect:testalgo2}. Secondly, we now have a new regime for $n\leq p$, where we observe that the minimum signal strength is of order $p$ and does not improve as $n$ increases. This could be an artefact of the proof and detailed reasoning is provided in \Cref{rem:contam_rate_fm}.

The proof of \Cref{thm:poly-time-main_fm} will follow closely to the proof of \Cref{thm:poly-time-main}. We will consider the same four conditions that hold with probability $1-\delta$. Only the choice of constants $\beta_1$ and $\beta_2$ differ.
\begin{proof}[Proof of \Cref{thm:poly-time-main_fm}]
    By the condition that $u\leq 0.08$, we need $1/n\leq 0.08$, implying $n\gtrsim 1$.

    Conditioned on event $A$, which occurs with probability $1 - \delta/4$, the following results hold:
    \begin{itemize}
        \item Weight Stability: By \Cref{rem:weightsum}, the weights $w$ returned by RowSumFilter in \Cref{alg:efficient-mean-tester} satisfy $$\sum_{i \in \mathfrak{G}} \sqrt{w_i} \geq |\mathfrak{G}| - \|\mathbf{1}_G - w_\mathfrak{G}\|_1 \geq (1 - 6u)n.$$ Accounting for the operations in Step 7, $\sum_{i \in \mathfrak{G}} \sqrt{w_i}$ is at least $(1 - 7u)n$ in the worst case.
        \item Assumptions in \eqref{eq:gram-assume} or \eqref{eq:cov-assume} are satisfied with probability $1-\delta/4$.
        \item Inner Product Bound: \Cref{lem:rowsumfilter-1b} implies that with probability at least $1 - \delta/4$, the bound on $\langle \Sum(w, \mathcal{Y}), \Sum(w, \mathcal{B}) \rangle$ holds.
        \item Good set sum bound: the upper bound on $\|\Sum(w, \mathfrak{G})\|^2$ in \eqref{proofeqn:sumwg} is satisfied with  probability at least $1 - \delta/4$.
    \end{itemize}
    By a union bound over these three failure events, each occurring with probability at most $\delta/4$, all aforementioned conditions hold simultaneously with probability at least $1 - \delta$. 

    By \Cref{lem:regular-weights-imply-tester4mom}, the conditions on $\kappa_0^2$ are:
    \begin{itemize}
    \item (type I error control condition) if $\mu = \mathbf{0}$, $O(\beta_1+\beta_2)\leq 0.5T_un^2\kappa_0^2$ and
    $$\kappa_0^2\gtrsim (p/n)^{2-2/v}\wedge 1+\frac{\sqrt{p}}{n\sqrt{\delta}},$$
    then $|\norm{\Sum (w, \mathcal{Y})}^2 - p \norm{w}_1| \leq 0.5T_u n ^2 \kappa^2_0$, and
    \item (type II error control condition) if $\mu \neq \mathbf{0}$, $O(\beta_1+\beta_2) < n^2T_u(\kappa^2-0.5\kappa_0^2)$ and
    $$\kappa^2-0.5\kappa_0^2\gtrsim \kappa \left(\frac{p}{n}\right)^{(v-1)/4}\wedge \kappa+\frac{\kappa}{\sqrt{n\delta}}+\left(\frac{p}{n}\right)^{2-2/v}\wedge 1+\frac{\sqrt{p}}{n\sqrt{\delta}},$$
    then $|\norm{\Sum (w, \mathcal{Y})}^2 - p \norm{w}_1| > 0.5T_u n^2 \kappa^2 _0$.
    \end{itemize}
    Meanwhile, the results of \Cref{lem:spectralfilter1} or \Cref{lem:spectralfilter2} (depending on whether $n>p$) and \Cref{lem:rowsumfilter} show that the terms $\beta_1$ and $\beta_2$ used above
    \begin{align*}
        \beta_1(\kappa) &= 
up^{2-2/v}n^{2/v}\wedge un^2+n\sqrt{up/\delta}+\kappa un^{1+1/v}p^{1-1/v}\wedge \kappa un^2+\kappa n\sqrt{u/\delta}. \\
        \beta_2 &= unR_{f} = u^2n^2p\log(1/u) + un(\sqrt{np}+p)\log(2p/\delta) + un^2\kappa_0^2.
    \end{align*}
By a union bound, all required events in \Cref{thm:poly-time-main} hold simultaneously with probability at least $1-\delta$. Therefore, we can obtain the result by just substituting in $\beta_1$ and $\beta_2$ and solving for $\kappa_0^2$.

From $\beta_1$, we have that under $\mu=\mathbf{0}$, for type I error control, we require
$$\kappa_0^2\gtrsim u\left(\frac{p}{n}\right)^{2-2/v}\wedge u+\frac{\sqrt{up}}{n\sqrt{\delta}}.$$
Under $\mu\neq\mathbf{0}$, for type II error control, we require
$$\kappa^2-0.5\kappa_0^2\gtrsim_{\log} u\left(\frac{p}{n}\right)^{2-2/v}\wedge u+\frac{\sqrt{up}}{n\sqrt{\delta}}+\kappa u \left(\frac{p}{n}\right)^{1-1/v}\wedge \kappa u + \frac{\kappa\sqrt{u}}{n\sqrt{\delta}}.$$
Note that most terms are dominated by the previous terms in \Cref{lem:regular-weights-imply-tester4mom}. 

From $\beta_2$, we have that under $\mu=\mathbf{0}$, for type I error control, we require
$$\kappa_0^2\gtrsim_{\log} u^2p\log(1/u)+u\sqrt{\frac{p}{n}}+\frac{up}{n}+u\kappa_0^2.$$
Under $\mu\neq\mathbf{0}$, for type II error control, we require
$$\kappa^2-0.5\kappa_0^2\gtrsim_{\log} u^2p\log(1/u)+u\sqrt{\frac{p}{n}}+\frac{up}{n}+u\kappa_0^2.$$

Combining all the results above, to control type I error, we require
\begin{equation}\label{proof_eqn:minkappa0}
    \kappa_0^2\gtrsim_{\log} \left(\frac{p}{n}\right)^{2-2/v}\wedge 1+\frac{\sqrt{p}}{n\sqrt{\delta}}+u^2p\log(1/u)+u\sqrt{\frac{p}{n}}+\frac{up}{n}+u\kappa_0^2.
\end{equation}
This simplifies to 
\begin{equation*}
    \kappa_0^2\gtrsim_{\log}\begin{cases}
        \left(\frac{p}{n}\right)^{2-2/v}+\frac{\sqrt{p}}{n\sqrt{\delta}}+\varepsilon^2p+\frac{p^{1+v/2}}{n^{v/2}},&\text{if $n>p$,}\\
        p,  &\text{if $n\leq p$.}
    \end{cases} 
\end{equation*}

To control type II error, we require
\begin{align}
    \kappa^2-0.5\kappa_0^2\gtrsim_{\log}  \kappa \left(\frac{p}{n}\right)^{(v-1)/4}\wedge \kappa+\frac{\kappa}{\sqrt{n\delta}}+\left(\frac{p}{n}\right)^{2-2/v}\wedge 1+\frac{\sqrt{p}}{n\sqrt{\delta}}+ \kappa u \left(\frac{p}{n}\right)^{1-1/v}\wedge \kappa u \nonumber\\
+u^2p\log(1/u)+u\sqrt{\frac{p}{n}}+\frac{up}{n}+u\kappa_0^2. \label{proof_eqn:minkappa1}
\end{align}
Defining $u'=\varepsilon+\left(\frac{p}{n}\right)^{v/4}$, this simplifies to
\begin{equation}\label{proof_eqn:t2e_condition_fm}
    \kappa^2-0.5\kappa_0^2\gtrsim_{\log}\begin{cases}
        \kappa \left(\frac{p}{n}\right)^{\frac{v-1}{4}}+\frac{\kappa}{\sqrt{n\delta}}+\kappa u' \left(\frac{p}{n}\right)^{1-1/v}+\left(\frac{p}{n}\right)^{2-\frac{2}{v}}+\frac{\sqrt{p}}{n\sqrt{\delta}}+\varepsilon^2p+\frac{p^{1+v/2}}{n^{v/2}}+u'\kappa_0^2,&\text{if $n>p$,}\\
        \kappa+p,  &\text{if $n\leq p$.}
    \end{cases} 
\end{equation}
Now, if $\kappa_0$ satisfies $2/C_\gamma\leq \kappa_0^2/\kappa^2\leq \bar{c}$, where $1/\bar{c}> 0.5+O(\varepsilon+n^{-1})$, then \eqref{proof_eqn:t2e_condition_fm} will be satisfied as long as 
\begin{equation*}
    \kappa^2\gtrsim_{\log}\begin{cases}
        \left(\frac{p}{n}\right)^{2-\frac{2}{v}}+\frac{\sqrt{p}}{n\sqrt{\delta}}+\varepsilon^2p+\frac{p^{1+v/2}}{n^{v/2}},&\text{if $n>p$,}\\
        p,  &\text{if $n\leq p$.}
    \end{cases} 
\end{equation*}
\end{proof} 

\begin{remark}\label{rem:contam_rate_fm}
    \Cref{thm:poly-time-main_fm} shows that, for the class $\mathcal{P}_{v,\phi}^p$, the minimum values of $\kappa$ and $\kappa_0$ required for type I and II error control do not improve
    with $n$ in the regime $n\le p$. This feature can be traced to its proof. The analysis of \Cref{alg:efficient-mean-tester} requires controlling the number of inliers which lie outside a $\ell_2$-ball centred at the mean with radius~$R$, which happens with probability $\gamma$ with an upper bound given in \eqref{eqn:fm_gamma}. Hence the issue is
    tied to the use of Markov's inequality and the choice of $R$ in \eqref{eqn:fm_R}. When $n>p$, one can choose $R^2-p\asymp \sqrt{n}$ so that the bound is of order $(p/n)^{v/4}$. However, when $n\leq p$, this choice of $R$ is not useful since $(p/n)^{v/4}\geq O(1)$. To keep the proportion of truncated inliers under control, we instead take $R^2-p\asymp \sqrt{p}$, which yields a constant-order bound on $\gamma$. Consequently, in the regime $n\le p$, the outlier proportion $u$ given in \eqref{eqn:u_bound} is of constant order, so the term $u^2p\log(1/u)$ is of order $p$. This term is then dominant in \eqref{proof_eqn:minkappa0} and
    \eqref{proof_eqn:minkappa1}. The same reasoning also applies to \Cref{thm:poly-time-main-mom}, since the proof relies on \Cref{alg:efficient-mean-tester} as a subroutine within \Cref{alg:efficient-mean-tester-mom}.
\end{remark}

\subsection{Proof of Proposition \ref{thm:poly-time-main-mom}}
\begin{proof}
\noindent\textbf{Step 1: Reduction to block-wise guarantees.} Suppose there exist thresholds $$t_4 < 0.5 T_u \kappa_0^2 < t_5$$ such that for every block $i \in \{1,\dots,k\}$:
\begin{itemize}
    \item (\textbf{Null}) if $\mu = 0$, then $\Pb(U_i > t_4) \leq 1/4$,
    \item (\textbf{Alternative}) if $\|\mu\|_2 \geq \kappa_0$, then $\Pb(U_i < t_5) \leq 1/4$.
\end{itemize}
Then, by the property of medians in \Cref{prop:med},
\[
\Pb\!\left(\med(U_{1:k}) < 0.5 T_u \kappa_0^2 \right)
\leq
\begin{cases}
\omega & \text{if } \mu = 0,\\
1-\omega & \text{if } \|\mu\|_2 \geq \kappa_0,
\end{cases}
\]
which establishes the theorem. It remains to verify the existence of such $t_4$ and $t_5$.

\medskip

\noindent\textbf{Step 2: Good event and contamination control.}
For a fixed block $\mathcal{Y}_i$, decompose it as $\mathcal{Y}_i = \mathfrak{G}_i \cup \mathcal{B}_i$ into inliers and outliers. We have $|\mathcal{Y}_i|=\lfloor n/K \rfloor=n_0$ for all $i\in\{1,\ldots,K\}$. By Bernstein’s inequality and the definition of $u$, the event
\[
A_i = \{|\mathcal{B}_i| \leq u n_0\}
\]
holds with probability at least $1 - 1/(16K)$.
Under event $A_{i}$, the following results hold
    \begin{itemize}
        \item Weight Stability: By \Cref{rem:weightsum}, the weights $w$ returned by RowSumFilter in \Cref{alg:efficient-mean-tester-mom} satisfy $$\sum_{j \in \mathfrak{G}_i} \sqrt{w_j} \geq |\mathfrak{G}_i| - \|\mathbf{1}_{\mathfrak{G}_i} - w_{\mathfrak{G}_i}\|_1 \geq (1 - 6u)n_0.$$ Accounting for the operations in Step 9, $\sum_{j \in \mathfrak{G}_i} \sqrt{w_j}$ is at least $(1 - 7u)n_0$ in the worst case.
        \item Assumptions in \eqref{eq:gram-assume} or \eqref{eq:cov-assume} are satisfied with probability $1-\delta/(4K)$.
        \item Inner Product Bound: \Cref{lem:rowsumfilter-1b} implies that with probability at least $1 - \delta/(4K)$, the bound on $\langle \Sum(w, \mathcal{Y}_i), \Sum(w, \mathcal{B}_i) \rangle$ holds.
        \item Good set sum bound: the upper bound on $\|\Sum(w, \mathfrak{G}_i)\|^2$ in \eqref{proofeqn:sumwg2} is satisfied with  probability at least $1 - 1/(16K)$.
    \end{itemize}
    By a union bound over four failure events and over all $i\in\{1,\ldots, K\}$, each occurring with probability at most $1/(16K)$, all aforementioned conditions hold simultaneously with probability at least $3/4$.

\medskip

\noindent\textbf{Step 3: Expansion of the statistic.} On this event,
\begin{align*}
\norm{\Sum(w,\mathcal{Y}_i)}_2^2
= \norm{\Sum(w,\mathfrak{G}_i)}_2^2
+ 2 \iprod{\Sum(w,\mathcal{Y}_i), \Sum(w,\mathcal{B}_i)}
- \norm{\Sum(w,\mathcal{B}_i)}_2^2.
\end{align*}
By \Cref{lem:regular-weights-imply-tester4mom}, we substitute $\delta=1/(4K)$ to get
\begin{align*}
    \norm{\Sum(w,\mathfrak{G}_i)}_2^2= p\norm{w_{\mathfrak{G}_i}}_1&+\left(\sum_{i\in \mathfrak{G}_i}\sqrt{w_i}\right)^2 \kappa^2\\
&\pm O(\kappa n_0^{(9-v)/4}p^{(v-1)/4}\wedge \kappa n_0^2+\kappa n_0\sqrt{n_0K}+p^{2-2/v}n_0^{2/v}\wedge n^2_0+n_0\sqrt{pK}).
\end{align*}
Secondly, by \Cref{lem:spectralfilter2}, we substitute $\delta=1/(4K)$ to get
$$\norm{\Sum(w,\mathcal{B}_i)}_2^2\leq O(un_0R_{f}),$$
where $R_{f}$ is defined as
\begin{equation*}
    R_{f} = C_{\gamma} \left(un_0p\log(1/u)+(\sqrt{n_0p}+p)\log(8pK)+ n_0 \kappa^2_0\right).
\end{equation*}
Finally, by \Cref{lem:rowsumfilter}, we substitute $\delta=1/(4K)$ to get
\begin{align*}
    \iprod{\Sum(w,\mathcal{Y}_i), \Sum(w,\mathcal{B}_i)}=p\norm{w_\mathcal{B}}_1\pm un_0^2\kappa^2\pm 
    O(up^{2-2/v}n_0^{2/v}\wedge un_0^2+n_0\sqrt{upK})\\
    \pm O(\kappa un_0^{1+1/v}p^{1-1/v}\wedge \kappa un_0^2+\kappa n_0\sqrt{uK}+un_0R_{f}).
\end{align*}

\medskip 

\noindent\textbf{Step 4: Null case ($\mu = 0$).}
By concentration bounds above, there exists
\[
t_4 = O\!\Big(p^{2-2/v}n_0^{2/v} \wedge n_0^2 + n_0\sqrt{pK} + n_0 u R_{f} \Big)
\]
such that $\mathbb{P}(U_i \ge t_4) \le 1/4$ for all $i$. For successful type I error control we require $$t_4<0.5T_un_0^2\kappa_0^2.$$ 
Solving for $\kappa_0$ in the exact same way as the proof of \Cref{thm:poly-time-main_fm} gives the condition
\begin{equation*}
    \kappa_0^2\gtrsim_{\log}\begin{cases}
        \left(\frac{p}{n_0}\right)^{2-2/v}+\frac{\sqrt{p}}{n_0}+\varepsilon^2p\log(1/\varepsilon)+\frac{p^{1+v/2}}{n_0^{v/2}},&\text{if $n_0>p$,}\\
        p,  &\text{if $n_0\leq p$.}
    \end{cases} 
\end{equation*}
\vspace{2mm}
\medskip

\noindent\textbf{Step 5: Alternative case ($\norm{\mu}_2 = \kappa$).}
Note that $\sum_{j\in \mathfrak{G}_i} \sqrt{w_j}\geq (1-7u)^2$.
By concentration bounds above, there exists
\[
t_5 = [(1-7u)^2-2u]n_0^2\kappa^2- O\!\Big(\kappa n_0^{(9-v)/4}p^{(v-1)/4}\wedge \kappa n_0^2+\kappa n_0^{3/2} \sqrt{K}+p^{2-2/v}n_0^{2/v}\wedge n^2_0 + n_0\sqrt{pK} + n_0 uR_{f}\Big)
\]
such that $\mathbb{P}(U_i \ge t_5) \le 1/4$ for all $i$. For successful type II error control we require 
\begin{equation}\label{proof_eqn:t2e_mom}
    0.5T_un_0^2\kappa_0^2<t_5.
\end{equation}
Assuming $\kappa_0$ satisfies $2/C_\gamma\leq \kappa_0^2/\kappa^2\leq \bar{c}$, where $1/\bar{c}> 0.5+O(\varepsilon+n^{-1})$, then
\begin{equation*}
    \kappa^2\gtrsim_{\log}\begin{cases}
        \left(\frac{p}{n_0}\right)^{2-2/v}+\frac{\sqrt{p}}{n_0}+\varepsilon^2p\log(1/\varepsilon)+\frac{p^{1+v/2}}{n_0^{v/2}},&\text{if $n_0>p$,}\\
        p,  &\text{if $n_0\leq p$.}
    \end{cases} 
\end{equation*}
is a sufficient condition for \eqref{proof_eqn:t2e_mom} to hold, by the proof of \Cref{thm:poly-time-main_fm}.
\end{proof}

\subsection{Proof of Proposition \ref{fact:1d-filter}}
\begin{proof}
$\supp{w'} \subset \supp{w}$ is immediate as $w_i'$ sets one of the non-zero weights in $w$ to 0 and the zero weights stay 0. For the second part,
\begin{align*}
    \norm{\mathbf{1}_{\mathfrak{G}}- w'_{\mathfrak{G}}}_1&=\sum_{i\in \mathfrak{G}} |1- (1-\frac{\tau_i}{\max_{i\in \mathcal{Y}} \tau_i})w_i|= \norm{\mathbf{1}_{\mathfrak{G}}- w_{\mathfrak{G}}}_1+\sum_{i\in \mathfrak{G}} \frac{w_i\tau_i}{\max_{j\in \mathcal{Y}} \tau_j}\\
    &\leq5\norm{\mathbf{1}_{\mathcal{B}}- w_{\mathcal{B}}}_1+5\sum_{i\in \mathcal{B}} \frac{w_i\tau_i}{\max_{j\in \mathcal{Y}} \tau_j}=5 \norm{\mathbf{1}_B - w_\mathcal{B}'}_1.
\end{align*}
\end{proof}

\subsection{Proof of Lemma \ref{lem:cov_matrix_conc0}}
\begin{proof} First we consider $n\geq p$. By expanding $M(\mathfrak{G})$ and using triangle inequality, we have
    \begin{align*}
        &\norm{M(\mathfrak{G})-nI}_{\op}\\
        &=\norm{\sum_{i=1}^n\zeta_i(Y_i-\mu)(Y_i-\mu)^\top+\sum_{i=1}^n\zeta_i\mu (Y_i-\mu)^\top+\sum_{i=1}^n\zeta_i (Y_i-\mu)\mu^\top+\sum_{i=1}^n\zeta_i\mu\mu^\top-nI}_{\op}\\
        &\leq n \norm{\mu}_2^2+\left(\norm{\sum_{i=1}^n\zeta_i(Y_i-\mu)(Y_i-\mu)^\top-nI}_{\op}+2\norm{\sum_{i=1}^n\zeta_i (Y_i-\mu)}_2\norm{\mu}_2\right).
    \end{align*} 
   For $F\in\mathcal{P}_{v,\phi}^p$, by \Cref{lem:cov_matrix_conc}, with probability at least $1-\delta$, we have
    \begin{equation}\label{proofeqn:covnorm}
        \norm{\sum_{i\in \mathfrak{G}}(Y_i-\mu)(Y_i-\mu)^\top-nI}_{\op}\leq O((\sqrt{np}+p)\log(2p/\delta))+\varepsilon n.
    \end{equation}
    \Cref{lem:normbound} gives us a bound for the cross term. With probability $1-\delta$, we have
    \begin{align}\label{proofeqn:normX22}
        \norm{\sum_{i=1}^n \zeta_i (Y_i-\mu)}_2
        &\leq O(\sqrt{np}\log(1/\delta)).
    \end{align}
    Adding up all the terms from \eqref{proofeqn:covnorm} and \eqref{proofeqn:normX22} gives
    \begin{align*}
        \norm{M(\mathfrak{G})-nI}_{\op}
        &\leq n\kappa^2+ O(\sqrt{np}\log(4p/\delta))+\varepsilon n+O(\kappa\sqrt{np}\log(2/\delta)+\kappa\sqrt{np})\\
        &=n\kappa^2+O(\sqrt{np}\log(p/\delta))+\varepsilon n,
    \end{align*}
    with probability $1-\delta$. The same argument holds for $F\in\mathcal{\mathfrak{G}}_{\theta,M}^p$ as $\mathcal{G}_{\theta,M}
    \subset \mathcal{P}_{4,\phi}$ for some $\phi>0$.

    Secondly, for $n \leq p$, suppose $\mathfrak{G}=\{i_1,\ldots,i_m\}$. Define the data matrix
\[
Y_G = [\,Y_{i_1}\; Y_{i_2}\;\cdots\; Y_{i_m}\,] \in \mathbb{R}^{p\times m}.
\]
The Gram matrix is
\[
\Gram(\mathfrak{G}) = (Y_{i_a}^\top Y_{i_b})_{a,b=1}^m = Y_G^\top Y_G .
\]
Consider
\[
Y_\mathfrak{G} Y_\mathfrak{G}^\top = \sum_{i\in \mathfrak{G}} Y_i Y_i^\top=M(\mathfrak{G}).
\]
The matrices $M(\mathfrak{G})$ and $\Gram(\mathfrak{G})$ have the same nonzero eigenvalues. By results above, we have that the maximum eigenvalue of $M(\mathfrak{G})$ satisfy $$\lambda_{\max}(\Gram(\mathfrak{G}))=\lambda_{\max}(M(\mathfrak{G}))\leq n+O(p\log(2p/\delta))+n\kappa^2,$$
and zero eigenvalues of algebraic multiplicity $n-p$. This means that for all $i\in[n]$, 
$$-p\leq \lambda_i(\Gram(\mathfrak{G})-pI)\leq n-p+O(p\log(2p/\delta))+n\kappa^2\leq O(p\log(2p/\delta))+n\kappa^2,$$
i.e. $\norm{\Gram(\mathfrak{G})-pI}_\op\leq O(p\log(2p/\delta))+n\kappa^2.$
\end{proof}

\subsection{Proof of Lemma \ref{lem:regular-weights-imply-tester}}\label{subsect:prf_regular-weights-imply-tester}
\begin{proof}
We can expand $\norm{\Sum (w, \mathcal{Y})}^2$ as
    \begin{align}
        \norm{\Sum (w, \mathcal{Y})}^2 &= \norm{\Sum (w, \mathfrak{G})}^2 + 2 \iprod{\Sum (w, \mathfrak{G}), \Sum (w, \mathcal{B})} + \norm{\Sum (w, \mathcal{B})}^2 \nonumber \\
        &= \norm{\Sum (w, \mathfrak{G})}^2 + 2 \iprod{\Sum (w, \mathcal{Y}), \Sum (w, \mathcal{B})} - \norm{\Sum(w, \mathcal{B})}^2. \label{proof_eqn:expansion}
    \end{align}
The first term can be found by expanding $\norm{\Sum(w,\mathfrak{G})}^2$ as
    \begin{align}
        \norm{\Sum(w,\mathfrak{G})}^2&=\norm{\mu\sum_{i\in \mathfrak{G}}\sqrt{w_i}+\sum_{i\in \mathfrak{G}}\sqrt{w_i}(Y_i-\mu)}^2 \nonumber\\
        &=\left(\sum_{i\in \mathfrak{G}}\sqrt{w_i}\right)^2 \norm{\mu}^2+2\left(\sum_{i\in \mathfrak{G}}\sqrt{w_i}\right)\mu^\top\sum_{i\in \mathfrak{G}}\sqrt{w_i}(Y_i-\mu)+\norm{\sum_{i\in \mathfrak{G}} \sqrt{w_i}(Y_i-\mu)}^2. \label{proof_eqn:sumwGexp}
    \end{align}
    We bound the second term in \eqref{proof_eqn:sumwGexp}.
Condition on the index set $\mathfrak{G}=\{i_1,\ldots,i_m\}$. By \Cref{lem:conditioning_reduction}, we have for any $t>0$ that
\begin{align*}
&\Pb_{Y_{1:n}\sim D}\left(
\left|\mu^\top\sum_{i\in \mathfrak{G}}\sqrt{w_i}(Y_i-\mu)\right|
\ge t
\;\middle|\;
\mathfrak{G}=\{i_1,\ldots,i_m\}
\right)\\
&\le4\Pb_{Y_{i_{1:m}}\sim F}
\left(\left|
\sum_{j=1}^m \sum_{k=1}^p
\mu_k \sqrt{w_{i_j}}(Y_{i_j,k}-\mu_k)
\right|\ge t\right).
\end{align*}
By Lemma~\ref{lemma:weib_concentration}, 
\[
t=O\left(\norm{\mu}_2\sqrt{\norm{w_{i_{1:m}}}_1 \log(4/\delta)}+\norm{\mu}_\infty\log^{1/\theta}(4/\delta)\right)
\]
satisfies
\[
\Pb_{Y_{i_{1:m}}\sim F}\!\left(\left|\sum_{j=1}^m \sum_{k=1}^p\mu_k \sqrt{w_{i_j}}(Y_{i_j,k}-\mu_k)\right|\ge t\right)\le \delta/4.
\]
Since $\norm{w_{i_{1:m}}}_1\le n$, we may choose the larger threshold
\[
t=O(\,\norm{\mu}_2\,\sqrt{n}\log^{1/\theta}(4/\delta)),
\]
which does not depend on $m$.
With this choice, the bound holds uniformly over all realizations of $\mathfrak{G}$,
and thus the conditioning on $\mathfrak{G}$ can be removed. Consequently, with probability $1-\delta$,
$$\left|\mu^\top\sum_{i\in \mathfrak{G}}\sqrt{w_i}(Y_i-\mu)\right|\le
O(\,\norm{\mu}_2\,\sqrt{n}\log^{1/\theta}(4/\delta)).$$
We bound the third term in \eqref{proof_eqn:sumwGexp} similarly. Condition on the index set $\mathfrak{G}=\{i_1,\ldots,i_m\}$. By \Cref{lem:conditioning_reduction}, we have for any $t>0$ that
    \begin{align*}
        &\Pb_{Y_{1:n}\sim D}\left(\left|\norm{\sum_{i\in \mathfrak{G}} \sqrt{w_i}(Y_i-\mu)}_2^2-p\norm{w_\mathfrak{G}}\right|\geq t\bigg|\mathfrak{G}=\{i_1,\ldots, i_m\}\right)\\
        &\leq 4\Pb_{Y_{i_{1:m}}\sim F}\left(\left|\norm{\sum_{j=1}^m \sqrt{w_{i_j}}(Y_{i_j}-\mu)}_2^2-p\norm{w_{i_{1:m}}}\right|\geq t\right).
    \end{align*} 
    By \Cref{lem:normbound_subW}, 
\[
t=O\left(\norm{w_{i_{1:m}}}_1\sqrt{ p\log(4/\delta)}+\norm{w_{i_{1:m}}}_1\log(4/\delta)+\log^{2/\theta}(4/\delta)\right)
\]
satisfies
\[
\Pb_{Y_{i_{1:m}}\sim F}\left(\left|\norm{\sum_{j=1}^m \sqrt{w_{i_j}}(Y_{i_j}-\mu)}_2^2-p\norm{w_{i_{1:m}}}\right|\geq t\right)\le \delta/4.
\]
Since $\norm{w_{i_{1:m}}}_1\le n$, we may choose the larger threshold
\[
t=O\left(n\sqrt{p}\log^{2/\theta}(4/\delta)\right),
\]
which does not depend on $m$.
With this choice, the bound holds uniformly over all realizations of $\mathfrak{G}$,
and thus the conditioning on $\mathfrak{G}$ can be removed. Consequently, with probability $1-\delta$,
\[
\norm{\sum_{i\in \mathfrak{G}}\sqrt{w_i}(Y_i-\mu)}_2^2
=\norm{w_\mathfrak{G}}_1 p\pm O\!\left(n\sqrt{p}\log^{2/\theta}(4/\delta)
\right).
\]
Thus, we have with probability $1-2\delta$ that
    \begin{align*}
        \norm{\Sum(w,\mathfrak{G})}^2&= \norm{w_\mathfrak{G}}_1p+\left(\sum_{i\in \mathfrak{G}}\sqrt{w_i}\right)^2\norm{\mu}^2\\
        &\quad\pm  O(\norm{\mu}_2\sqrt{n}\log^{1/\theta}(1/\delta))\left(\sum_{i\in \mathfrak{G}}\sqrt{w_i}\right)\pm O(n\sqrt{p}\log^{2/\theta}(1/\delta)).
    \end{align*}
If $\mu=\mathbf{0}$, we have
\begin{align*}
    |\norm{\Sum (w, \mathcal{Y})}^2_2-p\norm{w}_1| \leq O\left(np^{1/2}\log^{2/\theta}(1/\delta)+\beta_1+\beta_2\right).
\end{align*}
Our procedure declares that $\mu=\mathbf{0}$ if the RHS is less than $0.5T_u\kappa_0^2n^2$. Thus we require
$$0.5T_u\kappa_0^2n^2\geq O\left(np^{1/2}\log^{2/\theta}(1/\delta)+\beta_1+\beta_2\right),$$
implying that we need $T_u \kappa_0^2n^2\geq O(\beta_1+\beta_2)$, and
$$\kappa_0^2\gtrsim \frac{\sqrt{p}\log^{2/\theta}(1/\delta)}{0.5T_un}\asymp \frac{\sqrt{p}\log^{2/\theta}(1/\delta)}{n}.$$
If $\norm{\mu}_2\neq 0$, since $n\geq\sum_{i\in \mathfrak{G}} \sqrt{w_i}\geq \sum_{i\in \mathfrak{G}} w_i\geq (1-7u)n$,
\begin{align*}
    |\norm{\Sum (w, \mathcal{Y})}^2_2-p\norm{w}_1| &\geq T_u n^2\kappa^2-O(\kappa n^{3/2}\log^{1/\theta}(1/\delta)+n\sqrt{p}\log^{2/\theta}(1/\delta)+\beta_1+\beta_2).
\end{align*}
Our procedure declares that $\norm{\mu}_2\geq \kappa_0$ if the RHS is more than $0.5T_u\kappa_0^2n^2$, thus we require
\begin{align*}
    T_un^2\kappa^2-O(\kappa n^{3/2}\log^{1/\theta}(1/\delta)+n\sqrt{p}\log^{2/\theta}(1/\delta)+\beta_1+\beta_2)\geq 0.5T_un^2\kappa_0^2,
\end{align*}
which simplifies to the conditions $\kappa^2> 0.5\kappa_0^2$, $T_un^2(\kappa^2-0.5\kappa_0^2)\geq O(\beta_1+\beta_2)$ and
\begin{align*}
    \kappa^2-0.5\kappa_0^2\gtrsim \kappa n^{-1/2}\log^{1/\theta}(1/\delta)+n^{-1}\sqrt{p}\log^{2/\theta}(1/\delta).
\end{align*}
\end{proof}

\subsection{Proof of Lemma \ref{lem:regular-weights-imply-tester4mom}}\label{subsect:prf_regular-weights-imply-tester4mom}
\begin{proof}
    We can expand $\norm{\Sum (w, \mathcal{Y})}^2$ as
    \begin{align*}
        \norm{\Sum (w, \mathcal{Y})}^2 &= \norm{\Sum (w, \mathfrak{G})}^2 + 2 \iprod{\Sum (w, \mathfrak{G}), \Sum (w, \mathcal{B})} + \norm{\Sum (w, \mathcal{B})}^2 \\
        &= \norm{\Sum (w, \mathfrak{G})}^2 + 2 \iprod{\Sum (w, \mathcal{Y}), \Sum (w, \mathcal{B})} - \norm{\Sum(w, \mathcal{B})}^2.
    \end{align*}
    
The first term can be found by expanding $\norm{\Sum(w,\mathfrak{G})}^2$ as
    \begin{align*}
        \norm{\Sum(w,\mathfrak{G})}^2&=\norm{\mu\sum_{i\in \mathfrak{G}}\sqrt{w_i}+\sum_{i\in \mathfrak{G}}\sqrt{w_i}(Y_i-\mu)}^2\\
        &=\left(\sum_{i\in \mathfrak{G}}\sqrt{w_i}\right)^2 \norm{\mu}^2+2\left(\sum_{i\in \mathfrak{G}}\sqrt{w_i}\right)\sum_{i\in \mathfrak{G}}\sqrt{w_i}\mu^\top(Y_i-\mu)+\norm{\sum_{i\in \mathfrak{G}} \sqrt{w_i}(Y_i-\mu)}^2.
    \end{align*}
    
    To bound the second term in the expansion, we consider its mean and variance. Firstly, for the mean, we have
    \begin{align*}
        \left|\E_{Y_{i_{1:m}}\sim F_R}\left[\sum_{k=1}^{m} \sqrt{w_{i_k}}\mu^\top (Y_{i_k}-\mu)\right]\right|&=\left(\sum_{k=1}^{m} \sqrt{w_{i_k}}\right)|\mu^\top\E_{Y\sim F_R}[Y_{i_k}-\mu]|\\
        &\lesssim  m\norm{\mu}_2\gamma^{1-1/v}\lesssim n^{(5-v)/4}p^{(v-1)/4}\norm{\mu}_2 \wedge n\norm{\mu}_2,
    \end{align*}
    where the last inequality follows from \Cref{lem:approx_mean}, since $F\in \mathcal{P}^p_{v,\phi}$. Meanwhile, for the variance, we have
    \begin{align*}
        \Var_{Y\sim F_R}\left[\mu^\top (Y-\mu)\right]&=\mu^\top \Var_{Y\sim F}(Y-\mu|\norm{Y-\mu}_2\leq R) \mu \\
        &\leq\norm{\mu}_2^2 \norm{I+\Var_{Y\sim F}(Y-\mu|\norm{Y-\mu}_2\leq R)-I}_\op\\
        &\leq(1+2\psi^2\sqrt{\gamma}) \norm{\mu}_2^2=O(\norm{\mu}_2^2),
    \end{align*}
     where the last inequality follows from \Cref{lem:approx_cov}. Thus, we have 
     \begin{align*}
        \Var_{Y_{i_{1:m}}\sim F_R}\left[\sum_{k=1}^{m} \sqrt{w_{i_k}}\mu^\top (Y_{i_k}-\mu)\right]&=\sum_{k=1}^{m} w_{i_k}\Var_{Y\sim F_R}\left[\mu^\top (Y-\mu)\right]\leq O(n\norm{\mu}_2^2),
    \end{align*}
     Thus by Chebyshev's inequality, we have
    \begin{equation}\label{proof_eqn:k0_fm}
        \left|\sum_{i\in \mathfrak{G}}\sqrt{w_i}\mu^\top(Y_i-\mu)\right|\leq \norm{\mu}_2\cdot O(n^{(5-v)/4}p^{(v-1)/4}\wedge n+\sqrt{n/\delta}).
    \end{equation}
    
To bound the third term,
\begin{align*}
&\Pb_{Y_{1:n}\sim D}
\!\left(
\left|\norm{\sum_{i\in \mathfrak{G}} \sqrt{w_i}(Y_i-\mu)}^2_2-p\norm{w_\mathfrak{G}}_1\right|\ge t
\;\middle|\; \mathfrak{G}=\{i_1,\ldots,i_m\} \right)\\
&=\,\Pb_{Y_{i_{1:m}}\sim F_R}\!\left(
\left|\norm{\sum_{k=1}^m \sqrt{w_{i_k}}(Y_{i_k}-\mu)}^2_2
-p\sum_{k=1}^{m} w_{i_k}\right|\ge t\right),
\end{align*}
where $F_R$ is the distribution of $Y \sim F$ conditioned on the event $\|Y - \mu\|_2 \leq R$. Therefore, it remains to bound 
\begin{align*}
    &\norm{\sum_{k=1}^m \sqrt{w_{i_k}}(Y_{i_k}-\mu)}^2_2-p\sum_{k=1}^{m} w_{i_k}\\
    &=\underbrace{\sum_{j=1}^m w_{i_j}(\norm{Y_{i_j}-\mu}_2^2-p)}_{K_1}+ \underbrace{\sum_{j=1}^m \sum_{k=1, k\neq j}^m \sqrt{w_{i_j}w_{i_k}}(Y_{i_j}-\mu)^\top (Y_{i_k}-\mu)}_{K_2}.\\
\end{align*}
    
    We first analyse $K_1$. Since $F\in\mathcal{P}^p_{4,\psi}$, then by \Cref{lem:approx_mean},  we have
    \begin{align*}
        \left|\E_{Y_{i_{1:m}}\sim F_R}\left[\sum_{k=1}^{m} w_{i_k}(\norm{Y_{i_k}-\mu}_2^2-p)\right]\right|&= \left|\norm{w_{i_{1:m}}}_1(\E_{Y_1\sim F}[\norm{Y_1-\mu}^2_2|\norm{Y_1-\mu}\leq R]-p)\right|\\
        &\leq 2\norm{w_{i_{1:m}}}_1 \psi^2\gamma^{1/2}\leq O(m(p/n)^{1/2} \wedge m)\leq O(\sqrt{pn}\wedge n).
    \end{align*}
    Meanwhile, since $\gamma\leq 1/20$,
    \begin{align*}
        \Var_{Y\sim F_R}\left[\norm{Y-\mu}_2^2-p\right]&\le \E_{Y\sim F} [(\norm{Y-\mu}_2^2-p)^2 |\norm{Y-\mu}_2<R]\\
        &\le (1-\gamma)^{-1}\E_{Y\sim F} [\norm{Y-\mu}_2^4-p^2]\\
        &\le \frac{20}{19}p(\psi^4-1).
    \end{align*}
    Therefore, we have
    \begin{align*}
        \Var_{Y_{i_{1:m}}\sim F_R}\left[\sum_{k=1}^{m} w_{i_k}(\norm{Y_{i_k}-\mu}_2^2-p)\right]&= \norm{w_{i_{1:m}}}_2^2 \Var_{Y\sim F_R}\left[\norm{Y-\mu}_2^2-p\right]\leq O(n^2p).
    \end{align*}
    Finally, by Chebyshev's inequality,
    \begin{align*}
        \Pb_{Y_{i_{1:m}}\sim F_R}\left(|\sum_{k=1}^{m} w_{i_k}(\norm{Y_{i_k}-\mu}^2-p)-\E\left[\sum_{k=1}^{m} w_{i_k}(\norm{Y_{i_k}-\mu}^2-p)\right]|>t\right)\leq O(n^2p)/t^2,
    \end{align*}
    i.e.~with probability at least $1-\delta$, we have
    \begin{align*}
        \left|\sum_{k=1}^{m} w_{i_k}(\norm{Y_{i_k}-\mu}^2-p)-p\norm{w_{i_{1:m}}}_1\right|\leq O(n\sqrt{p/\delta}).
    \end{align*}

    Next, we bound $K_2$. Since $F\in\mathcal{P}^p_{4,\psi}$, then by \Cref{lem:approx_mean},  we have
    \begin{align*}
        &\left|\E_{Y_{i_{1:m}}\sim F_R}\left[\sum_{j=1}^m \sum_{k=1, k\neq j}^m \sqrt{w_{i_j}w_{i_k}}(Y_{i_j}-\mu)^\top (Y_{i_k}-\mu)\right]\right|\\
        &= \left|(\sum_{k=1}^m \sqrt{w_{i_k}})^2\E_{Y_1\sim F}[(Y_{1}-\mu)|\norm{Y_1-\mu}_2\leq R]^\top\E_{Y_2\sim F}[(Y_{2}-\mu)|\norm{Y_2-\mu}_2\leq R]\right|\\
        &\lesssim m^2\gamma^{2(1-1/v)}\lesssim n^2(p/n)^{2(1-1/v)}\wedge n^2\asymp p^{2-2/v}n^{2/v}\wedge n^2.
    \end{align*}
    Meanwhile, since $\gamma\leq 1/20$,
    \begin{align*}
        &\Var_{Y_1,Y_2\sim F_R}\left[(Y_{1}-\mu)^\top (Y_{2}-\mu)\right]\\
        &\le \E_{Y_1,Y_2\sim F} [(Y_{1}-\mu)^\top (Y_{2}-\mu)(Y_{2}-\mu)^\top (Y_{1}-\mu) |\max_{i\in\{1,2\}}\norm{Y_i-\mu}_2<R]\\
        &\le (1-\gamma)^{-2}\mathrm{tr}\left\{\E_{Y_1,Y_2\sim F} \left[(Y_{1}-\mu)(Y_{1}-\mu)^\top (Y_{2}-\mu)(Y_{2}-\mu)^\top\right]\right\}\\
        &\le (1-\gamma)^{-2}\mathrm{tr}(I_p)\\
        &\le 2p.
    \end{align*}
    Therefore, by Chebyshev's inequality,
    \begin{align*}
        \Var_{Y_{i_{1:m}}\sim F_R}\left[\sum_{j=1}^m \sum_{k=1, k\neq j}^m \sqrt{w_{i_j}w_{i_k}}(Y_{i_j}-\mu)^\top (Y_{i_k}-\mu)\right]&= \norm{w_{i_{1:m}}}_1^2 \Var_{Y,Y'\sim F_R}\left[(Y-\mu)^\top (Y'-\mu)\right]\\
        &\leq 2n^2p,
    \end{align*}
    Concluding the above, with probability at least $1-\delta$, we have
\begin{equation}\label{proof_eqn:k2_fm}
    \norm{\sum_{i\in \mathfrak{G}} \sqrt{w_{i}}(Y_{i}-\mu)}^2_2=p\norm{w_\mathfrak{G}}_1\pm O\left(p^{2-2/v}n^{2/v}\wedge n^2+n\sqrt{p/\delta}\right).
\end{equation}

Combining \eqref{proof_eqn:k0_fm} and \eqref{proof_eqn:k2_fm}, with probability at least $1-\delta$, we have
\begin{align*}
        \norm{\Sum(w,\mathfrak{G})}^2&=p\norm{w_\mathfrak{G}}_1+\left(\sum_{i\in \mathfrak{G}}\sqrt{w_i}\right)^2 \norm{\mu}^2\pm \left(\sum_{i\in \mathfrak{G}}\sqrt{w_i}\right)\norm{\mu}_2O(n^{(5-v)/4}p^{(v-1)/4}\wedge n+\sqrt{n/\delta})\\
        &\quad\pm O(p^{2-2/v}n^{2/v}\wedge n^2+n\sqrt{p/\delta})\\
        &=p\norm{w_\mathfrak{G}}_1+\left(\sum_{i\in \mathfrak{G}}\sqrt{w_i}\right)^2 \kappa^2\pm O(\kappa n^{\frac{9-v}{4}}p^{\frac{v-1}{4}}\wedge \kappa n^2+\kappa n\sqrt{n/\delta})\\
        &\quad\pm O(p^{2-\frac{2}{v}}n^{\frac{2}{v}}\wedge n^2+n\sqrt{p/\delta}).
    \end{align*}
If $\mu= \mathbf{0}$, then we have
\begin{align*}
    |\norm{\Sum (w, \mathcal{Y})}^2_2-p\norm{w}_1| \le O(p^{2-\frac{2}{v}}n^{\frac{2}{v}}\wedge n^2+n\sqrt{p}\delta^{-1/2}+\beta_1+\beta_2).
\end{align*}
Our procedure declares that $\mu =\mathbf{0}$ if the RHS is less than $0.5T_u\kappa_0^2n^2$, thus we require
$$O(p^{2-\frac{2}{v}}n^{\frac{2}{v}}\wedge n^2+n\sqrt{p}\delta^{-1/2}+\beta_1+\beta_2)\leq 0.5T_un^2\kappa_0^2,$$
which simplifies to the conditions $O(\beta_1+\beta_2)\leq 0.5T_un^2\kappa_0^2$ and
$$\kappa_0^2\gtrsim \begin{cases}
    (p/n)^{2-2/v}+n^{-1}\sqrt{p/\delta}  &\text{if $n\gtrsim p$,}\\
   1+n^{-1}\sqrt{p/\delta} &\text{if $n\lesssim p$.}
\end{cases} $$
If $\norm{\mu}_2\neq 0$, since $n\geq\sum_{i\in \mathfrak{G}} \sqrt{w_i}\geq \sum_{i\in \mathfrak{G}} w_i\geq (1-7u)n$, we have
\begin{align*}
    |\norm{\Sum (w, \mathcal{Y})}^2_2-p\norm{w}_1| &\geq [(1-7u)^2-2u]n^2\kappa^2-O(\kappa n^{\frac{9-v}{4}}p^{\frac{v-1}{4}}\wedge \kappa n^2+\kappa n^{3/2}\delta^{-1/2})\\
    &\quad-O(p^{2-\frac{2}{v}}n^{\frac{2}{v}}\wedge n^2+n\sqrt{p}\delta^{-1/2}+\beta_1+\beta_2).
\end{align*}
Our procedure declares that $\norm{\mu}_2\neq 0$ if the RHS is more than $0.5T_u\kappa_0^2n^2$, thus we require
\begin{align*}
    T_un^2\kappa^2-O(\kappa n^{\frac{9-v}{4}}p^{\frac{v-1}{4}}\wedge \kappa n^2+\kappa n^{3/2}\delta^{-1/2}+p^{2-\frac{2}{v}}n^{\frac{2}{v}}\wedge n^2+n\sqrt{p}\delta^{-1/2}+\beta_1+\beta_2)\geq  0.5T_un^2\kappa_0^2,
\end{align*}
which simplifies to the conditions $\kappa^2> 0.5\kappa_0^2$, $T_u(\kappa^2-0.5\kappa_0^2)n^2\geq O(\beta_1+\beta_2)$ and 
\begin{align*}
        \kappa^2-0.5\kappa_0^2 \gtrsim O(\kappa (p/n)^{\frac{v-1}{4}}\wedge \kappa +\kappa/\sqrt{n\delta} +(p/n)^{2-\frac{2}{v}}\wedge 1+n^{-1}\sqrt{p/\delta}).
    \end{align*}
\end{proof}

\subsection{Proof of Lemma \ref{lem:spectralfilter1}}
\begin{proof}
    The runtime per iteration is clearly dominated by the time it takes to find the top singular vector of the centred gram matrix, which can be done in time $O(p n^2)$.
    
    We will show that for all $t = 1, \ldots, N$, we have that $w^{(t)} \in \Lambda_n$.
    First, we demonstrate how this proves the overall lemma.
    First, note that after each iteration, some new $w_i$ (with the maximum $\tau_i$) becomes $0$, so after $6 u n$ iterations, we have removed at least $6 u n$ mass from $w$. By definition of $\Lambda_n$ in \eqref{eqn:important_set}, this means we have removed at least $un$ mass from the bad coordinates $w_\mathcal{B}$, at which point no further updates can maintain the invariant that $w^{(i)} \in \Lambda_n$.
    
    Next, we observe that if $w^{(N)} \in \Lambda_n$, then since we terminated, we must have that
    \[\norm{\Gram (w^{(N)}) - D(w^{(N)})} \leq 5 R_{f}.\]
    But then, for all $\mathcal{M}$ with $|\mathcal{M}| \leq u n$, let $\mathbf{1}_\mathcal{M} \in \R^n$ be the indicator vector on the set $\mathcal{M}$.
    Then, we have that 
    \begin{equation*}
        \norm{\Sum(w^{(N)}, \mathcal{M})}^2 = \mathbf{1}_\mathcal{M}^\top \Gram(w^{(N)}) \mathbf{1}_\mathcal{M}= \norm{w^{(N)}_\mathcal{M}}_1 p \pm O(un \cdot R_{f}),
    \end{equation*}
    as claimed.

    Thus, it suffices to prove the invariant that $w^{(t)} \in \Lambda_n$ for all $t = 1, \ldots, N$.
    We proceed by induction.
    Clearly $w^{(1)} \in \Lambda_n$.
    Now, suppose $w^{(t)} \in \Lambda_n$ for some $t < N$.
    Since we have not yet terminated, this implies that 
    \[
    \lambda = \norm{\Gram(w^{(t)}) - D(w^{(t)})} \ge 5 R_{f}.
    \]

    By \Cref{lem:cov_matrix_conc0}, with probability $1-\delta$,
    \begin{equation}
    \label{eq:gram-assume}
    \norm{\Gram(\mathfrak{G}) - p I}_\op \leq O(p\log(2p/\delta))+n\kappa^2.
    \end{equation}
    By assumption that $C_\gamma\kappa_0^2/2\geq\kappa^2$, \eqref{eq:gram-assume} can be upper bounded by
    \begin{align*}
        \norm{\Gram (w^{(t)}, \mathfrak{G}) - D (w^{(t)}_{\mathfrak{G}})} \leq \frac{R_{f}}{2} \le \frac{\lambda}{10},
    \end{align*}
    since we are multiplying $\Gram(\mathfrak{G})-pI$ by a diagonal matrix $\mathrm{diag}(\sqrt{w_1},\ldots, \sqrt{w_n})$ with eigenvalues between 0 and 1. We claim that this implies that $5 \sum_{i \in \mathcal{B}} v_i^2 > \sum_{i \in \mathfrak{G}} v_i^2$.
    Indeed, suppose not, and let $v_G$ denote the restriction of $v$ onto the coordinates in $\mathfrak{G}$, and let $v_B$ denote the restriction of $v$ onto the coordinates in $\mathcal{B}$.
    This means that
    \begin{align*}
        \left| v^\top \left(\Gram (w^{(t)}) - D(w) \right) v \right| &= \left| v_G^\top \left( \Gram (w^{(t)}, \mathfrak{G}) - D(w^{(t)}_{\mathfrak{G}}) \right) v_G \right. \\
        &\left. +2 v_G^\top \left(\Gram (w^{(t)}) - D(w^{(t)})\right) v_B + v_B^\top \left( \Gram(w^{(t)}, \mathcal{B}) - D(w^{(t)}_B) \right) v_B \right| \\
        &\leq \frac{\lambda}{10}\norm{v_G}^2 + 2 \lambda \norm{v_G} \norm{v_B}  + \norm{v_B}^2 \lambda \leq 0.996 \lambda,
    \end{align*}
    where the last inequality holds because $\|v_B\|^2+\|v_G\|^2 = \|v\|^2 = 1$ and $\|v_B\|^2 \le \frac{1}{6}.$
    But this is a contradiction since $v$ is the top singular vector of the centered Gram matrix.
    Therefore, by \Cref{fact:1d-filter} and the definition of $\tau_i$, we obtain that $w^{(t + 1)} \in \Lambda_n$, as claimed.
    (Note that if $w_i^{(t)} = 0$, the $i$th row and column of both $\Gram(w^{(t)}, \mathcal{Y})$ and $D(w^{(t)})$ are $0$ so $v_i = 0$, which means $w_i^{(t)} \cdot \tau_i = v_i^2$ even if $w_i^{(t)} = 0$.)
    This completes the proof.
\end{proof}

\subsection{Proof of Lemma \ref{lem:spectralfilter2}}
\begin{proof}
    As before, the per-iteration runtime is dominated by the runtime of PCA, which is $O(n p^2)$. 

    Throughout the rest of this proof, we will be working under event $A$ and the event 
    \begin{equation}
        \label{eq:cov-assume}
        W=\left\{\norm{M (\mathfrak{G}) - n I} \leq \frac{R_{f}}{2}\right\}.
    \end{equation} 
    The events $A$ and $W$ each hold with probability at least $1-\delta$. Thus, the event $A\cap W$ holds with probability at least $1-2\delta$. To see why $\Pb(W)\geq 1-\delta$, we notice that with probability $1-\delta$,
    $$\norm{M (\mathfrak{G}) - n I}_\op \leq O(\sqrt{np}\log(2p/\delta))+\varepsilon n+n\kappa^2\leq \frac{R_{f}}{2},$$
    where the first inequality follows from \Cref{lem:cov_matrix_conc0}, and the second inequality follows from the assumption $C_\gamma \kappa_0^2/2\geq \kappa^2$, for a large enough $C_\gamma$.
    
    We will again inductively show that for all iterations $t$, we have that $w^{(t)} \in \Lambda_n$. We first show how to prove the lemma, assuming this claim. In this case, we can bound the number of iterations $N$ in the same way as \Cref{lem:spectralfilter1}. Moreover, by construction, at termination we have that $\norm{M(w^{(N)},\mathcal{Y}) - n I}_\op \le 5 R_{f}$.
    Now suppose that there was some subset $\mathcal{M}$ with $|\mathcal{M}| \leq un$ that had
    \[
    \norm{\Sum (w^{(N)}, \mathcal{M})}^2_2 > 10 R_{f} un.
    \]
    Then, there is a unit vector $v \in \R^p$ so that
\begin{equation} \label{eq:unit-v-big}
    \sum_{i \in \mathcal{M}} (w^{(N)}_i)^{1/2} \iprod{v, Y_i} > \sqrt{10 R_{f} u n} \; .
\end{equation}
    Thus, we have that
\begin{equation} \label{eq:v-bound}
    \sum_{i \in \mathcal{M} \cup \mathcal{B}} w_i^{(N)} \langle v, Y_i \rangle^2 \ge \sum_{i \in \mathcal{M}} w^{(N)}_i \iprod{v, Y_i}^2 \geq 10 R_{f} \; .
\end{equation}
    Above, the first inequality holds because every $w_i^{(N)}$ is nonnegative, and the second inequality follows from \eqref{eq:unit-v-big} and the Cauchy-Schwarz inequality.
    
    Thus, we have
    \begin{align*}
        \norm{ M \left(w^{(N)}, \mathcal{Y}\right) - n I}_\op &= \norm{ M \left(w^{(N)}, \mathfrak{G}\setminus (\mathfrak{G}\cap \mathcal{M})\right)+M \left(w^{(N)}, \mathcal{M}\cup \mathcal{B} \right) - n I}_\op\\
        &\geq \norm{M \left(w^{(N)}, \mathcal{M}\cup \mathcal{B} \right)}_\op-\norm{ M \left(w^{(N)}, \mathfrak{G}\setminus (\mathfrak{G}\cap \mathcal{M})\right) - n I}_\op\\
        &\geq \norm{M \left(w^{(N)}, \mathcal{M}\cup \mathcal{B} \right)}_\op-\norm{M(\mathfrak{G})-M \left(w^{(N)}, \mathfrak{G}\setminus (\mathfrak{G}\cap \mathcal{M})\right)}_\op\\&\quad-\norm{M(\mathfrak{G})-nI}_\op\\
        &\geq 10R_{f}-R_{f}-R_{f}/2>5R_{f},
    \end{align*}
    where the first term comes from \eqref{eq:v-bound} and the third term comes from \eqref{eq:cov-assume}. To justify the bound on the second term, we first define
    $$(a_{\mathcal{M}})_i=\begin{cases}
            1 &\text{if $i\in \mathfrak{G}\cap \mathcal{M}$,}\\
            1-w_i^{(N)} &\text{if $i\in \mathfrak{G}\setminus (\mathfrak{G}\cap \mathcal{M})$.}
        \end{cases}$$
    Note that the matrix $M(\mathfrak{G})-M \left(w^{(N)}, \mathfrak{G}\setminus (\mathfrak{G}\cap \mathcal{M})\right)$ has weight vector $a_T$, with $$\norm{a_T}_1=\sum_{i\in \mathfrak{G}\cap \mathcal{M}} 1+ \sum_{i\in \mathfrak{G}\setminus \mathfrak{G}\cap \mathcal{M}} (1-w_i^{(N)})\leq \sum_{i\in \mathfrak{G}\cap \mathcal{M}} 1+ \sum_{i\in \mathfrak{G}} (1-w_i^{(N)})\leq un + 6un=7un.$$ Thus we can bound the second term by \Cref{fact:small-subset-deviations}, defining the set $J=\{a_T: \mathcal{M}\subset \mathcal{Y}, |\mathcal{M}|\le un \}$. This gives us a contradiction.

    Thus, as before, it suffices to prove that $w^{(t)} \in \Lambda_n$ for all iterations until termination.
    We will do so inductively.
    As before, the base case $t = 1$ is trivial.
    Now suppose that $w^{(t)} \in \Lambda_n$ for some $t < N$.
    Since we have not yet terminated, this means that $\norm{M (w^{(t)}, \mathcal{Y}) - n I} > 5 R_{f}$.
    Then,~\eqref{eq:cov-assume} and \Cref{fact:small-subset-deviations} together immediately imply that
\begin{align} \label{eq:B_bound}
    \sum_{i \in \mathcal{B}} w_i^{(t)}\iprod{v, Y_i}^2&=\norm{M(w^{(t)}, \mathcal{Y})-M(w^{(t)}, \mathfrak{G})}_\op \nonumber\\
    &> \norm{M(w^{(t)}, \mathcal{Y})-nI}_\op-\norm{M(w^{(t)}, \mathfrak{G})-M(\mathfrak{G})}_\op-\norm{M(\mathfrak{G})-nI}_\op \nonumber \\
    &\geq 5 R_{f} -R_{f} -R_{f}/2\geq 3R_{f},
\end{align}
    for $v$ the top eigenvector of $M(w^{(t)}, \mathcal{Y})-n I$. On the other hand, since $\sum_{i \le L} w_i^{(t)} \le 2 un + 1$ and since $w^{(t)} \in \Lambda_n$ means we have removed at most $6 un$ mass from all samples, this means $L \le 8 un + 1 \le 10 un$. So, \Cref{fact:small-subset-deviations} implies that 
\begin{equation} \label{eq:G_bound}
    \sum_{i \leq L, i \in \mathfrak{G}} w_i^{(t)} \iprod{v, Y_i}^2 \leq R_{f}.
\end{equation}
    By definition of $L$, every $\langle v, Y_i \rangle^2$ for $i \le L$ is larger than every $\langle v, Y_i \rangle^2$ for $i \in \mathcal{B} \backslash [L]$. Therefore, since $\sum_{i \in \mathcal{B} \backslash [L]} w_i^{(t)} \le |\mathcal{B} \backslash [L]| \le un$ but $\sum_{i \le L} w_i^{(t)} \ge 2 un$, we have
\begin{equation} \label{eq:BI_bound}
    2  \left(\sum_{i \in \mathcal{B}} w_i^{(t)} \langle v, Y_i \rangle^2 - \sum_{i \le L} w_i^{(t)} \langle v, Y_i \rangle^2\right) \le 2  \sum_{i \in \mathcal{B} \backslash [L]} w_i^{(t)} \langle v, Y_i \rangle^2 \le \sum_{i \le L} w_i^{(t)} \langle v, Y_i \rangle^2.
\end{equation}
    Along with \eqref{eq:B_bound}, \eqref{eq:BI_bound} implies that
\begin{equation} \label{eq:I_bound}
    \sum_{i \le L} w_i^{(t)} \langle v, Y_i \rangle^2 \ge 2 R_{f}.
\end{equation}
    Hence, by combining \eqref{eq:I_bound} with \eqref{eq:G_bound}, we have
    \[
    \sum_{i \leq L, i \in \mathcal{B}} w_i^{(t)} \iprod{v, Y_i}^2 \geq R_{f} \geq \sum_{i \leq L, i \in \mathfrak{G}} w_i^{(t)} \iprod{v, Y_i}^2,
    \]
    and so the result for $w^{(t+1)}$ immediately follows from \Cref{fact:1d-filter}.
\end{proof}

\subsection{Proof of Remark \ref{rmk:boundTB}}
\begin{proof}
    We have already shown that $w^{(t)} \in \Lambda_n$ for all iterations until termination.
    Now suppose that there was some subset $\mathcal{M}\subset \mathcal{Y}$ with $|\mathcal{M}| \leq 2un$ that had
    \[
    \norm{\Sum (w^{(N)}, \mathcal{M})}^2_2 > 20 R_{f} un.
    \]
    Then, there is a unit vector $v \in \R^p$ so that
\begin{equation} \label{eq:unit-v-big2}
    \sum_{i \in \mathcal{M}} (w^{(N)}_i)^{1/2} \iprod{v, Y_i} > \sqrt{20 R_{f} un}.
\end{equation}
        Thus, we have that
\begin{equation*} 
    \sum_{i \in \mathcal{M} \cup \mathcal{B}} w_i^{(N)} \langle v, Y_i \rangle^2 \ge \sum_{i \in \mathcal{M}} w^{(N)}_i \iprod{v, Y_i}^2 \geq 10 R_{f}.
\end{equation*}
    Above, the first inequality holds because every $w_i^{(N)}$ is nonnegative, and the second inequality follows from \eqref{eq:unit-v-big2} and the Cauchy-Schwarz inequality since $|\mathcal{M}|\leq an$.
    Thus, we have
    \begin{align*}
        \norm{ M \left(w^{(N)}, \mathcal{Y}\right) - n I} &= \norm{ M \left(w^{(N)}, \mathfrak{G}\setminus (\mathfrak{G}\cap \mathcal{M})\right)+M \left(w^{(N)}, \mathcal{M}\cup \mathcal{B} \right) - n I}_\op\\
        &\geq \norm{M \left(w^{(N)}, \mathcal{M}\cup \mathcal{B} \right)}_\op-\norm{M(\mathfrak{G})-M \left(w^{(N)}, \mathfrak{G}\setminus (\mathfrak{G}\cap \mathcal{M})\right)}_\op\\
        &\quad-\norm{M(\mathfrak{G})-nI}_\op\\
        &\geq 10R_{f}-R_{f}-R_{f}/2>5R_{f},
    \end{align*}
    using the same argument as \Cref{lem:spectralfilter2}. This gives us a contradiction.
\end{proof}

\subsection{Proof of Lemma \ref{lem:rowsumfilter}}
The high probability event required for Lemma \ref{lem:rowsumfilter} to hold is the result of the following two lemmas corresponding to concentration properties of different heavy-tailed classes. The lemmas below are modified from Lemma 8.16 of \cite{Canonne2023}.
\begin{lemma}
\label{lem:rowsumfilter-1}
Let $w$ be the output of Algorithm~\ref{alg:spectralfilter2}.
If $F\in \mathcal{\mathfrak{G}}_{\theta,M}^p$, then for any $\mathcal{M} \subset \mathfrak{G}$ with $|\mathcal{M}| \leq u n$, with probability at least $1-\delta$, we have 
\[
\sum_{i \in \mathcal{M}, j \in \mathcal{Y}} \sqrt{w_i w_j} \iprod{Y_i, Y_j} = p \|w_\mathcal{M}\|_1+\sum_{i \in \mathcal{M}, j \in \mathcal{Y}} \sqrt{w_i w_j} \kappa^2 \pm O(\beta_1+\beta_2),
\] 
where $\beta_1(\kappa)=O\left(
un\sqrt{np}\log^{2/\theta}({8n}/{\delta})+\kappa n\sqrt{un}\log^{1/\theta}(4/\delta)\right), \beta_2= unR_{f}$. 
\end{lemma}

\begin{lemma}
\label{lem:rowsumfilter-1b}
Let $w$ be the output of Algorithm~\ref{alg:spectralfilter2}.
If $F\in \mathcal{P}_{4,\phi}^p$, then for any $\mathcal{M} \subset \mathfrak{G}$ with $|\mathcal{M}| \leq u n$, with probability at least $1-\delta$, we have 
\[
\sum_{i \in \mathcal{M}, j \in \mathcal{Y}} \sqrt{w_i w_j} \iprod{Y_i, Y_j} = p \|w_\mathcal{M}\|_1 +\sum_{i \in \mathcal{M}, j \in \mathcal{Y}} \sqrt{w_i w_j} \kappa^2+ O(\beta_1+\beta_2),
\] 
where $\beta_1(\kappa)=O\left(
up^{2-2/v}n^{2/v}\wedge un^2+n\sqrt{up/\delta}+\kappa un^{1+1/v}p^{1-1/v}\wedge \kappa un^2+\kappa n\sqrt{u/\delta}\right), \beta_2= unR_{f}$. 
\end{lemma}

Assuming the above results, we prove \Cref{lem:rowsumfilter} below.
\begin{proof}[Proof of \Cref{lem:rowsumfilter}]
    Recall the definition of $\tau_i$ from \eqref{eq:rowsumscore}.
    First, we note that for any $\mathcal{M} \subset \mathfrak{G}$ with $|\mathcal{M}| \le u n$, $\sum_{i \in \mathcal{M}} \tau_i \le un^2 \kappa^2 +O(\beta_1+\beta_2)$. To see why, we can split $\mathcal{M}$ into $\mathcal{M}^+$ and $\mathcal{M}^-$, where $i \in \mathcal{M}^+$ if $\langle \sqrt{w_i} Y_i, \sum_{j \in \mathcal{Y}} \sqrt{w_j} Y_j \rangle \ge w_i p$ and $i \in \mathcal{M}^-$ otherwise. Then, since $|\mathcal{M}^+|, |\mathcal{M}^-| \le u n$, Lemma~\ref{lem:rowsumfilter-1} implies that both $\sum_{i \in \mathcal{M}^+} \tau_i$ and $\sum_{i \in \mathcal{M}^-} \tau_i$ are at most $un^2 \kappa^2 +O(\beta_1+\beta_2)$.

    Since $\sum_{i \in \mathcal{M}} \tau_i \le un^2 \kappa^2 +O(\beta_1+\beta_2)$ for any subset $\mathcal{M}$ of $\mathfrak{G}$ of size at most $u n$, and since we sorted the $\tau_i$'s in decreasing order, this implies $\sum_{i \in \mathcal{M}} \tau_i \le un^2 \kappa^2 +O(\beta_1+\beta_2)$ for any subset $\mathcal{M}$ of $\mathcal{Y} \backslash [un]$ of size at most $u n$. If $w_i$ represents the values of $w$ before setting the top $u n$ indices to $0$, and $w_i'$ represents the values of $w$ afterwards (i.e., $w_i' = 0$ for $i \le \varepsilon n$ and $w_i' = w_i$ for $i > \varepsilon n$), then 
\begin{align}
    \left|\sum_{i \in \mathcal{M}, j \in \mathcal{Y}} \sqrt{w_i' w_j} \langle Y_i, Y_j \rangle - p \cdot \|w'_T\|_1\right| 
    &\le \sum_{i \in \mathcal{M}} \left|\left\langle \sqrt{w_i'} Y_i, \sum_{j \in \mathcal{Y}} \sqrt{w_j} Y_j \right\rangle - w_i' p\right| \nonumber \\
    &= \sum_{i \in \mathcal{M} \backslash [u n]} \tau_i \le un^2\kappa^2+O(\beta_1+\beta_2). \label{eq:w'-bound-1}
\end{align}
    Next, we have
\begin{align}
    \sum_{i \in \mathcal{M}, j \in [un]} \sqrt{w_i' w_j} \langle Y_i, Y_j \rangle &= \sum_{i \in \mathcal{M} \backslash [un], j \in [un]} \sqrt{w_i w_j} \langle Y_i, Y_j \rangle = \pm O(unR_{f}), \label{eq:w'-bound-2}
\end{align}
    by the same argument as in \eqref{eq:a3-bound2}, since $\mathcal{M} \backslash [un]$ and $[un]$ are disjoint sets in $\mathcal{Y}$ and have size at most $un$. By subtracting \eqref{eq:w'-bound-2} from \eqref{eq:w'-bound-1}, we obtain the desired bound
\[\left|\sum_{i \in \mathcal{M}, j \in \mathcal{Y}} \sqrt{w_i' w_j'} \langle Y_i, Y_j \rangle - p \|w'_T\|_1\right| = \left|\sum_{i \in \mathcal{M}, j \in \mathcal{Y} \backslash [un]} \sqrt{w_i' w_j} \langle Y_i, Y_j \rangle - p \|w'_T\|_1\right| \le un\kappa^2+ O(\beta_1+\beta_2). \]
\end{proof}

We now provide the proof of Lemmas \ref{lem:rowsumfilter-1} and Lemma \ref{lem:rowsumfilter-1b}.
\begin{proof}[Proof of Lemma \ref{lem:rowsumfilter-1}]
We condition on the index set $\mathfrak{G}=\{i_1,\ldots,i_m\}$ and on the event $A$, i.e. $m \geq (1-u)n$. 
Fix an arbitrary subset $\mathcal{M}=\{i_1,\ldots,i_\tau\}\subset \mathfrak{G}$ with $\tau\le un$. 
Recalling that $\mathcal{Y}=\mathfrak{G}\cup \mathcal{B}$, we decompose
\begin{align*}
\sum_{i\in \mathcal{M},\,j\in \mathcal{Y}}\sqrt{w_iw_j}\iprod{Y_i,Y_j}- p \|w_\mathcal{M}\|_1
&=
\underbrace{\sum_{i\in \mathcal{M},\,j\in \mathfrak{G}}\sqrt{w_iw_j}\iprod{Y_i,Y_j}- p \|w_\mathcal{M}\|_1}_{B_1}
+
\underbrace{\sum_{i\in \mathcal{M},\,j\in \mathcal{B}}\sqrt{w_iw_j}\iprod{Y_i,Y_j}}_{B_2}.
\end{align*}

We first control the term $B_1$. Conditioning on $\mathfrak{G}=\{i_1,\ldots,i_m\}$, we may write
\begin{align*}
B_1=\underbrace{\sum_{k=1}^{\tau}\sum_{j=1}^{m} \sqrt{w_{i_k}w_{i_j}}
\iprod{Y_{i_k}-\mu,Y_{i_j}-\mu}-p\sum_{k=1}^{\tau} w_{i_k}}_{B_3}+ \underbrace{\sum_{k=1}^{\tau}\sum_{j=1}^{m}\sqrt{w_{i_k}w_{i_j}}
\iprod{Y_{i_k}-\mu,\mu}}_{B_4}\\
+\underbrace{\sum_{k=1}^{\tau}\sum_{j=1}^{m}\sqrt{w_{i_k}w_{i_j}}
\iprod{\mu,Y_{i_j}-\mu}}_{B_5}+\underbrace{\sum_{k=1}^{\tau}\sum_{j=1}^{m}\sqrt{w_{i_k}w_{i_j}} \norm{\mu}_2^2}_{B_6}.
\end{align*}
By \Cref{lem:conditioning_reduction}, for any $t>0$ we have that 
\begin{align*}
&\Pb_{Y_{1:n}\sim D}
\!\left(
\big|\iprod{\sum_{k\in \mathcal{M}}\sum_{j\in \mathfrak{G}}\sqrt{w_{k}w_{j}}
\iprod{Y_{k}-\mu,Y_{j}-\mu}}-p\norm{w_\mathcal{M}}_1\big|\ge t
\;\middle|\;
\mathfrak{G}=\{i_1,\ldots,i_m\}
\right)\\
&\le
4\,\Pb_{Y_{i_{1:m}}\sim F_R}\!\left(
\left|\sum_{k=1}^{\tau}\sum_{j=1}^{m}\sqrt{w_{i_k}w_{i_j}}\iprod{Y_{i_k}-\mu,Y_{i_j}-\mu}
-p\sum_{k=1}^{\tau} w_{i_k}\right|\ge t\right).
\end{align*}
By Lemma~\ref{lem:crossterm}, 
\[
t=O\left(\tau\sqrt{mp}\log^{2/\theta}\!\left({8m}/{\delta}\right)\right)
\]
satisfies
\[
\Pb_{Y_{i_{1:m}}\sim F}\!\left(
\left|\sum_{k=1}^{\tau}\sum_{j=1}^{m}\sqrt{w_{i_k}w_{i_j}}\iprod{Y_{i_k}-\mu,Y_{i_j}-\mu}
-p\sum_{k=1}^{\tau} w_{i_k}\right|\ge t\right)
\le \delta.
\]

Finally, since $\tau\le un$ and $m\le n$, we may choose the larger threshold
\[
t=O\left(un\sqrt{np}\log^{2/\theta}\!\left({8n}/{\delta}\right)\right),
\]
which depends only on $n$ and not on the particular choice of $\mathfrak{G}$ or $\mathcal{M}$. Hence, for any fixed subset $\mathcal{M}\subset[n]$ with
$|\mathcal{M}|\le un$, we may remove the conditioning on $\mathfrak{G}$. Consequently, we have
\begin{equation}\label{eq:a1-bound_subW} 
    |B_3|=p\norm{w_\mathcal{M}}_1+O\!\left(un\sqrt{np}\log^{2/\theta}\!\left({8n}/{\delta}\right)\right)
\end{equation}
with probability $1-\delta$.

To bound $B_4$, by \Cref{lem:conditioning_reduction}, we have for any $t>0$ that
\begin{align*}
&\Pb_{Y_{1:n}\sim D}\left(
\left|\mu^\top\sum_{k=1}^\tau \sqrt{w_{i_k}}(Y_{i_k}-\mu)\right|\ge t \;\middle|\; \mathfrak{G}=\{i_1,\ldots,i_m\} \right)\\
&\quad\le4\Pb_{Y_{i_{1:m}}\sim F}
\!\left(
\left|
\sum_{r=1}^p \sum_{k=1}^\tau
\mu_k \sqrt{w_{i_k}}(Y_{i_k,r}-\mu_r)
\right|
\ge t
\right).
\end{align*}
By Lemma~\ref{lemma:weib_concentration}, 
\[
t=O\left(\norm{\mu}_2\sqrt{\norm{w_{i_{1:\tau}}}_1 \log(4/\delta)}+\norm{\mu}_\infty\log^{1/\theta}(4/\delta)\right)
\]
satisfies
\[
\Pb_{Y_{i_{1:m}}\sim F}\!\left(\left|\sum_{r=1}^p \sum_{k=1}^\tau
\mu_k \sqrt{w_{i_k}}(Y_{i_k,r}-\mu_r)\right|\ge t\right)\le \delta/4.
\]
Since $\norm{w_{i_{1:\tau}}}_1\le un$, we may choose the larger threshold
\[
t=O(\,\norm{\mu}_2\,\sqrt{un}\log^{1/\theta}(4/\delta)),
\]
which does not depend on $m$.
With this choice, the bound holds uniformly over all realizations of $\mathfrak{G}$,
and thus the conditioning on $\mathfrak{G}$ can be removed. Consequently, with probability $1-\delta$,
$$\left|B_4\right|\le
O(\,\norm{\mu}_2\,n\sqrt{un}\log^{1/\theta}(4/\delta)).$$

Similarly, by swapping $m,\tau$ in the above calculation, we have
$$\left|B_5\right|\le
O(\,\norm{\mu}_2\,un\sqrt{n}\log^{1/\theta}(4/\delta))$$
with probability $1-\delta$.

Finally, we bound $B_2$. Since $\mathcal{M}, \mathcal{B}$ are disjoint sets in $\mathcal{Y}$ of size at most $u n$, we can use \Cref{lem:spectralfilter1} or \Cref{lem:spectralfilter2} and \Cref{rmk:boundTB} to obtain
\begin{align}
    A_2 &= \frac{1}{2} \left(\left\|\Sum(w, \mathcal{M} \cup \mathcal{B})\right\|^2 - \left\|\Sum(w, \mathcal{M})\right\|^2 - \left\|\Sum(w, \mathcal{B})\right\|^2\right)= \pm O(unR_{f}). \label{eq:a2-bound_subW}
\end{align}
Combining results from \eqref{eq:a1-bound_subW} and \eqref{eq:a2-bound_subW} gives the final result.
\end{proof}

\begin{proof}[Proof of Lemma \ref{lem:rowsumfilter-1b}]
We condition on the index set $\mathfrak{G}=\{i_1,\ldots,i_m\}$ and on the event $A$, i.e. $m \geq (1-u)n$. 
Fix an arbitrary subset $\mathcal{M}=\{i_1,\ldots,i_\tau\}\subset \mathfrak{G}$ with $\tau\le un$. 
Recalling that $\mathcal{Y}=\mathfrak{G}\cup \mathcal{B}$, we decompose
\begin{align*}
\sum_{i\in \mathcal{M},\,j\in \mathcal{Y}}\sqrt{w_iw_j}\iprod{Y_i,Y_j}- p \|w_\mathcal{M}\|_1
&=
\underbrace{\sum_{i\in \mathcal{M},\,j\in \mathfrak{G}}\sqrt{w_iw_j}\iprod{Y_i,Y_j}- p \|w_\mathcal{M}\|_1}_{B_1}
+
\underbrace{\sum_{i\in \mathcal{M},\,j\in \mathcal{B}}\sqrt{w_iw_j}\iprod{Y_i,Y_j}}_{B_2}.
\end{align*}

We first control the term $B_1$. Conditioning on $\mathfrak{G}=\{i_1,\ldots,i_m\}$, we may write
\begin{align*}
B_1=\underbrace{\sum_{k=1}^{\tau}\sum_{j=1}^{m} \sqrt{w_{i_k}w_{i_j}}
\iprod{Y_{i_k}-\mu,Y_{i_j}-\mu}-p\sum_{k=1}^{\tau} w_{i_k}}_{B_3}+ \underbrace{\sum_{k=1}^{\tau}\sum_{j=1}^{m}\sqrt{w_{i_k}w_{i_j}}
\iprod{Y_{i_k}-\mu,\mu}}_{B_4}\\
+\underbrace{\sum_{k=1}^{\tau}\sum_{j=1}^{m}\sqrt{w_{i_k}w_{i_j}}
\iprod{\mu,Y_{i_j}-\mu}}_{B_5}+\underbrace{\sum_{k=1}^{\tau}\sum_{j=1}^{m}\sqrt{w_{i_k}w_{i_j}} \norm{\mu}_2^2}_{B_6}.
\end{align*}

We first analyse the diagonal terms of $B_3$. Since $F\in\mathcal{P}^p_{4,\psi}$, then by \Cref{lem:approx_mean},  we have
    \begin{align*}
        \left|\E_{Y_{i_{1:\tau}}\sim F_R}\left[\sum_{k=1}^{\tau} w_{i_k}(\norm{Y_{i_k}-\mu}_2^2-p)\right]\right|&= \left|\norm{w_{i_{1:\tau}}}_1(\E_{Y_1\sim F}[\norm{Y_1-\mu}^2_2|\norm{Y_1-\mu}\leq R]-p)\right|\\
        &\leq 2\tau \psi^2\gamma^{1/2}\leq O(un(p/n)^{1/2}\wedge un)\leq O(u\sqrt{pn} \wedge un).
    \end{align*}
    Meanwhile, since $\gamma\leq 1/20$,
    \begin{align*}
        \Var_{Y\sim F_R}\left[\norm{Y-\mu}_2^2-p\right]&\le \E_{Y\sim F} [(\norm{Y-\mu}_2^2-p)^2 |\norm{Y-\mu}_2<R]\\
        &\le (1-\gamma)^{-1}\E_{Y\sim F} [\norm{Y-\mu}_2^4-p^2]\\
        &\le 2p(\psi^4-1).
    \end{align*}
    Therefore, we have
    \begin{align*}
        \Var_{Y_{i_{1:\tau}}\sim F_R}\left[\sum_{k=1}^{\tau} w_{i_k}(\norm{Y_{i_k}-\mu}_2^2-p)\right]&= \norm{w_{i_{1:\tau}}}_2^2 \Var_{Y\sim F_R}\left[\norm{Y-\mu}_2^2-p\right]\leq 2\psi^4u^2n^2p.
    \end{align*}

    Finally, by Chebyshev's inequality, with probability at least $1-\delta$, we have
    \begin{align*}
        |\sum_{k=1}^{\tau} w_{i_k}(\norm{Y_{i_k}-\mu}^2-p)-p\norm{w_{i_{1:\tau}}}_1|\leq O(un\sqrt{p/\delta}).
    \end{align*}

The off-diagonal terms in $B_3$ can be bounded as follows. By \Cref{lem:approx_mean},  we have
    \begin{align*}
        &\left|\E_{Y_{i_{1:m}}\sim F_R}\left[\sum_{j=1}^\tau \sum_{k=1, k\neq j}^m \sqrt{w_{i_j}w_{i_k}}(Y_{i_j}-\mu)^\top (Y_{i_k}-\mu)\right]\right|\\
        &=\left|\sum_{j=1}^\tau \sum_{k=1, k\neq j}^m \sqrt{w_{i_j}w_{i_k}} \E_{Y_1\sim F}[(Y_{1}-\mu)|\norm{Y_1-\mu}_2\leq R]^\top\E_{Y_2\sim F}[(Y_{2}-\mu)|\norm{Y_2-\mu}_2\leq R]\right|\\
        &\leq\sum_{j=1}^\tau \sum_{k=1, k\neq j}^m \sqrt{w_{i_j}w_{i_k}} \norm{\E_{Y_1\sim F}[(Y_{1}-\mu)|\norm{Y_1-\mu}_2\leq R]}_2^2\\
        &\lesssim un^2\gamma^{2-2/v}\lesssim up^{2-2/v}n^{2/v} \wedge un^2.
    \end{align*}
    Meanwhile, since $\gamma\leq 1/20$,
    \begin{align*}
        &\Var_{Y_1,Y_2\sim F_R}\left[(Y_{1}-\mu)^\top (Y_{2}-\mu)\right]\\&\le \E_{Y_1,Y_2\sim F} [(Y_{1}-\mu)^\top (Y_{2}-\mu)(Y_{2}-\mu)^\top (Y_{1}-\mu) |\max_{i\in\{1,2\}}\norm{Y_i-\mu}_2<R]\\
        &\le (1-\gamma)^{-2}\mathrm{tr}(\E_{Y_1,Y_2\sim F} [(Y_{1}-\mu)(Y_{1}-\mu)^\top (Y_{2}-\mu)(Y_{2}-\mu)^\top])\\
        &\le (1-\gamma)^{-2}\mathrm{tr}(I_p)\le 2p.
    \end{align*}
    Therefore,
    \begin{align*}
        \Var_{Y_{i_{1:m}}\sim F_R}\left[\sum_{j=1}^\tau \sum_{k=1, k\neq j}^m \sqrt{w_{i_j}w_{i_k}}(Y_{i_j}-\mu)^\top (Y_{i_k}-\mu)\right]&= un^2\Var_{Y,Y'\sim F_R}\left[(Y-\mu)^\top (Y'-\mu)\right]\\
        &\leq 2un^2p,
    \end{align*}
    Concluding the above, with probability at least $1-\delta$, we have
\begin{equation}\label{proof_eqn:b3_fm}
    B_3=p\norm{w_\mathcal{M}}_1\pm O\left(up^{2-2/v}n^{2/v}\wedge un^2+n\sqrt{up/\delta}\right).
\end{equation}

To bound $B_4$, by \Cref{lem:approx_mean},
\begin{align*}
    \left|\E\left[\sum_{k=1}^{\tau}\sum_{j=1}^{m}\sqrt{w_{i_k}w_{i_j}}
\iprod{Y_{i_k}-\mu,\mu}\right]\right|\leq \kappa m\tau\norm{\E_{Y\sim F_R}[Y]-\mu}_2&\leq O( \kappa un^2 \gamma^{1-1/v})\\
&=O(\kappa un^{1+1/v}p^{1-1/v}\wedge \kappa un^2).
\end{align*}
Meanwhile, by \Cref{lem:approx_cov},
\begin{align*}
    \Var_{Y\sim F_R}[\mu^\top (Y-\mu)]&=\mu^\top(I+\E_{Y\sim F}[(Y-\mu)(Y-\mu)^\top-I\mid \norm{Y-\mu}_2<R])\mu\\
    &\leq\kappa^2(1+2\psi^2\sqrt{\gamma})=O(\kappa^2).
\end{align*}
Thus we have
\begin{align*}
    \Var_{Y\sim F_R}\left[\sum_{k=1}^{\tau}\sum_{j=1}^{m}\sqrt{w_{i_k}w_{i_j}}\iprod{Y_{i_k}-\mu,\mu}\right]\leq O(m\tau\kappa^2)\leq O(un^2\kappa^2).
\end{align*}

Combining the above, with probability $1-\delta$, we have 
\[
\left|B_4\right|=\left|\sum_{k=1}^{\tau}\sum_{j=1}^{m}\sqrt{w_{i_k}w_{i_j}}\iprod{Y_{i_k}-\mu,\mu}\right|
\leq O(\kappa un^{1+1/v}p^{1-1/v}\wedge \kappa un^2+\kappa n\sqrt{u/\delta}).
\]
With this choice, the bound holds uniformly over all realizations of $\mathfrak{G}$,
and thus the conditioning on $\mathfrak{G}$ can be removed. 

Similarly, by swapping $m,\tau$ in the above calculation, we have
$$\left|B_5\right|\leq O(\kappa un^{1+1/v}p^{1-1/v}\wedge \kappa un^2+\kappa n\sqrt{u/\delta})$$
with probability $1-\delta$.

    Finally, we bound $A_2$. Since $\mathcal{M}, \mathcal{B}$ are disjoint sets in $\mathcal{Y}$ of size at most $u n$, we can use \Cref{lem:spectralfilter1} or \Cref{lem:spectralfilter2} (depending on whether $n>p$) and \Cref{rmk:boundTB} to obtain
\begin{align}
    A_2 &= \frac{1}{2} \left(\left\|\Sum(w, \mathcal{M} \cup \mathcal{B})\right\|^2 - \left\|\Sum(w, \mathcal{M})\right\|^2 - \left\|\Sum(w, \mathcal{B})\right\|^2\right)= \pm O(unR_{f}). \label{eq:a3-bound2}
\end{align}
Adding \eqref{proof_eqn:b3_fm} and \eqref{eq:a3-bound2} gives the final result.
\end{proof}

\section{Technical details and proofs of results in Section \ref{sect:high_dim_cpt}}\label{sect:hdcptproof}
\subsection{Properties of the pairwise difference dataset}
We first briefly recall the data generation process in online change point detection. For each $i\in[n]$, $X_i$ is independently generated from $Q_i$, where $Q_i$ is defined in \eqref{huber} by
\[
Q_i = (1-\varepsilon_i)F_i + \varepsilon_i H_i, \qquad i \in \{1,\ldots,n\},
\]
where $\varepsilon_i \le \varepsilon$ for all $i$. Suppose we observe $\{X_1,\ldots, X_t\}$, where $t\in\mathbb{N}$. For $1 \le s \le \lfloor t/2 \rfloor$, define the pairwise difference
\[
Y_{s,t} = \frac{X_{t-s+1} - X_s}{\sqrt{2}\sigma}.
\]
\Cref{prop:properties_Ds} shows that $\{Y_{s,t}\}_{s=1}^{\lfloor t/2 \rfloor}$ is equal in distribution to a Huber contamination model. In particular, in the pre-change regime ($t \le \Delta$), the inlier distribution will have zero mean, while in the post-change regime ($t>\Delta\geq s$), the inlier distribution will share a common mean with norm $\kappa/\sqrt{2}\neq 0$, allowing us to apply our robust mean testing procedure.

\begin{proposition}[Huber structure of pairwise differences]
\label{prop:properties_Ds}
Fix $t \in \mathbb{N}$. For each $s \in \{1,\dots,\lfloor t/2 \rfloor\}$, the random variable $Y_{s,t}$ follows a dynamic Huber $2\varepsilon$-contamination model
\[
Y_{s,t} \sim Q_{s,t} = (1-2\varepsilon) D_{s,t} + 2\varepsilon H_{s,t},
\]
where $D_{s,t}$ is an inlier distribution with the following properties:
\begin{enumerate}[label=(\textbf{\alph*})]
        \item $\E_{Y\sim D_{s,t}}[Y]=(f_{t-s+1}-f_{s})/(\sqrt{2}\sigma)$. In particular, 
        $$\norm{\E_{Y\sim D_{s,t}}[Y]}_2=\begin{cases}
            0 &\text{ if $t\leq \Delta$, }\\
            \kappa/\sqrt{2} &\text{ if $t> \Delta\ge s$. }
        \end{cases}$$
        \item $\Var_{Y\sim D_{s,t}}[Y]=I_p$, thus $D_{s,t}\in \mathcal{P}^{p}_{2,1}$ for all $s\in \{1,\dots,\lfloor t/2 \rfloor\}$.
        \item Assuming further that $F_i\in \mathcal{G}^{p}_{\theta,M}$ for all $i\in\mathbb{N}$, then $D_{s,t}\in\mathcal{G}^{p}_{\theta,\sqrt{2}M/\sigma}$ for all $s\in \{1,\dots,\lfloor t/2 \rfloor\}$.
    \end{enumerate}
\end{proposition}
\begin{proof}
    Let $B_i$ denote the event that $X_i$ is not contaminated, i.e. $X_i$ is drawn from the inlier distribution $F_i$. 
Then $\mathbb{P}(B_i) = 1-\varepsilon_i$, and for any fixed $s,t$ we have
\[
\mathbb{P}(B_{t-s+1} \cap B_s)
= (1-\varepsilon_{t-s+1})(1-\varepsilon_s)\ge (1-\varepsilon)^2
\ge 1 - 2\varepsilon,
\]
since $\varepsilon_i \le \varepsilon$ for all $i$. Conditional on the event $B_{t-s+1} \cap B_s$, the random variable $Y_{s,t}$ has the same distribution as
\[
(Z_{t-s+1} - Z_s)/(\sqrt{2}\sigma),
\]
where $Z_{t-s+1} \sim F_{t-s+1}$ and $Z_s \sim F_s$ are independent. Direct evaluation of the mean and variance gives the result.
\end{proof} 

\subsection{Proof of Theorem \ref{thm:cpt_hd_subW}}
We first state a corollary of \Cref{thm:poly-time-main}, which is needed in the proof of \Cref{thm:cpt_hd_subW}.
\begin{corollary}\label{rmk:poly-time-main}
    In the setting of \Cref{thm:poly-time-main}, provided that $\kappa \gtrsim \varepsilon \sqrt{p \log(1/\varepsilon)}$, and that the signal size input $\kappa_0$ satisfies $2/C_\gamma\leq \kappa_0^2/\kappa^2\leq \bar{c}$, we can use the same proof to show that the minimum sample size required to keep both type I error and type II error below $\delta$ is given by 
    \begin{equation*}
    G(p,\kappa,\varepsilon,\delta,\mathcal{G}^{p}_{\theta,M}) \asymp_{\log} \begin{cases}
        \frac{\sqrt{p}}{\kappa^2}, &\text{if $\varepsilon\sqrt{p\log(1/\varepsilon)}\lesssim \kappa \lesssim 1$,}\\
        \frac{\sqrt{p}}{\kappa}, &\text{if $1\lesssim\kappa \lesssim \sqrt{p}$,}\\
        1, &\text{if $\kappa \gtrsim \sqrt{p}$,}
    \end{cases}
    \end{equation*}
    up to poly-logarithmic factors in $p$, $1/\kappa$, and 
$1/\delta$, and constant factors depending only on $\theta$ and $M$. 
\end{corollary}

\begin{proof}[Proof of Theorem \ref{thm:cpt_hd_subW}] 

    Throughout the proof, we choose $C>0$ such that $$h_t= G(p,\kappa,\varepsilon,\delta_t,\mathcal{G}^p_{\theta,M}),$$
    where $\delta_t=(4\alpha)/[t(r+1)(r+2)]$. This is possible by \Cref{rmk:poly-time-main}. 

    (a) We can write the event $\{\hat{t}=\infty\}$ as
    \begin{equation*}
        \{\hat{t}=\infty\}=\bigcap_{t\geq 2} \bigcap_{\substack{h_{t} \leq s \leq \lfloor t/2\rfloor, \\ u(2\varepsilon, s, \delta_{t}) \leq 0.08}} \left\{ \mathrm{RobustMeanTest}\left(\{Y_{i,t}\}_{i=1}^s; \frac{\kappa_0}{\sqrt{2}}, \delta_{t}, 2\varepsilon, C_\gamma, T_u\right) = 0 \right\}.
    \end{equation*}
    By a union bound argument, it holds that 
    \begin{align}
        \mathbb{P}_\infty(\hat{t}<\infty)&\leq \sum_{t=2}^{\infty} \sum_{\substack{h_{t} \leq s \leq \lfloor t/2\rfloor, \\ u(2\varepsilon, s, \delta_{t}) \leq 0.08}} \mathbb{P}_\infty\left[\mathrm{RobustMeanTest}\left(\{Y_{i,t}\}_{i=1}^s; \frac{\kappa_0}{\sqrt{2}}, \delta_{t}, 2\varepsilon, C_\gamma, T_u\right)\right] \nonumber\\
        &\leq \sum_{t=2}^{\infty}\sum_{s=1}^{\lfloor t/2\rfloor} \frac{4\alpha}{t^2(t+1)}\leq \sum_{t=2}^{\infty} \frac{2\alpha}{t(t+1)}=\alpha, \label{proof_eqn:hdcpt_t1e}
    \end{align}
    where the second inequality follows from \Cref{thm:poly-time-main}.

(b) Note that $(X_1,\ldots, X_\Delta)$ have the same law under $\Pb_\infty$ and $\Pb_\Delta$. Thus,
    \begin{equation}\label{proof_eqn:hdcpt_t1e2}
    \mathbb{P}_\Delta(\hat{t}\leq\Delta)=\mathbb{P}_\infty(\hat{t}<\Delta)\leq \sum_{t=2}^{\Delta}\sum_{s=1}^{\lfloor t/2\rfloor} \frac{4\alpha}{t^2(t+1)}<\alpha,
    \end{equation}
    where the first inequality follows from the fact that $\{\hat{t}\leq \Delta\}\subseteq\{\hat{t}<\infty\}$ and the last inequality follows from \eqref{proof_eqn:hdcpt_t1e}.

(c) Denote  $$d=h_{2\Delta}=G\left(p,\kappa,\varepsilon, \delta_{2\Delta}, \mathcal{G}^p_{\theta,M}\right).$$
We note that \Cref{assum: hd_cpt} implies that
$\Delta\geq d$. Therefore, we have that $\E[Y_{s,t}]=\kappa$ for $t>\Delta$ and $s\leq d$, i.e. all samples in $\{Y_{s,t}\}_{s=1}^d$ has the same mean. In addition, we have that $d\geq G(p,\kappa,\varepsilon, \delta_{\Delta+d})$ since we are working with the same testing problem but with different type I error probability, specifically $\delta_{2\Delta}\leq \delta_{\Delta+d}$. Therefore, we can upper bound failure probability by
    \begin{align*}
     \Pb(\hat{t}>\Delta+d)
&\leq\Pb(\mathrm{RobustMeanTest}(\{Y_{i,\Delta+d}\}_{i=1}^d;\kappa_0, \delta_{\Delta+d}, 2\varepsilon, T_u)=0) \\
&\leq \frac{4\alpha}{(\Delta+d)^{2}(\Delta+d+1)}.
\end{align*}
where the second inequality follows from \Cref{thm:poly-time-main}, since the required assumptions are satisfied for our choice of $d$. Combining with \eqref{proof_eqn:hdcpt_t1e2}, we have 
\begin{align*}
     \Pb(\Delta<\hat{t}\leq\Delta+d)
&=\Pb(\hat{t}>\Delta)-\Pb(\hat{t}>\Delta+d)\\
&\geq 1-\sum_{t=2}^{\Delta}\sum_{s=1}^{\lfloor t/2\rfloor} \frac{4\alpha}{t^2(t+1)}- \frac{4\alpha}{(\Delta+d)^{2}(\Delta+d+1)}\\
&\geq 1-\alpha.
\end{align*}
\end{proof}

\subsection{Proof of Theorem \ref{thm:cpt_hd_fm}}
We first state a corollary of \Cref{thm:poly-time-main-mom}, which is needed in the proof of \Cref{thm:cpt_hd_subW}.

\begin{corollary}\label{rmk:poly-time-main-mom}
    In the setting of \Cref{thm:poly-time-main-mom}, provided that $\kappa \gtrsim \varepsilon \sqrt{p \log(1/\varepsilon)}$, and that the signal size input satisfies $2/C_\gamma\leq \kappa_0^2/\kappa^2 \leq \bar{c}$, we can use the same proof to show that the minimum sample size required to keep both type I error and type II error below $\delta$ is given by 
\begin{equation*}
   \tilde{G}(p,\kappa,\varepsilon,\delta,\mathcal{P}^p_{v,\phi})\asymp_{\log} \begin{cases}
            {\sqrt{p}}/{\kappa^2}, &\text{if $\varepsilon\sqrt{p}\log(1/\varepsilon)\lesssim \kappa \lesssim p^{-[(2v-2)\vee (v+4)]/(4v-8)}$.}\\
            {p^{1+2/v}}/{\kappa^{4/v}}+{p}/{\kappa^{v/(v-1)}}, &\text{if $p^{-[(2v-2)\vee (v+4)]/(4v-8)}\lesssim \kappa \lesssim \sqrt{p}$.}\\
            1, &\text{if $ \kappa \gtrsim \sqrt{p}$.}
        \end{cases}
\end{equation*}
up to poly-logarithmic factors in $p$, $1/\kappa$, and $1/\omega$ and constant factors depending only on $\phi$. 
\end{corollary}

The proof of \Cref{thm:cpt_hd_fm} is exactly the same as that of \Cref{thm:cpt_hd_subW}, up to replacing $G(p,\kappa,\varepsilon,\delta_t,\mathcal{G}_{\theta,M}^p)$ by $\tilde{G}(p,\kappa,\varepsilon,\delta_t,\mathcal{P}_{v,\phi}^p)$ and $\mathrm{RobustMeanTest}$ by $\mathrm{RobustMeanTest}_\mathrm{MoM}$.  

\section{Supplementary information for Section \ref{sect:sim_study}}\label{sect:hdcptalgo}
\subsection{Modified multivariate change point detection algorithm}\label{append:sim_study}
For clarity, we provide the full pseudocode for \Cref{alg:cpd_highd_sim} below. This algorithm is a minor modification of \Cref{alg:cpd_highd} that allocates the type I error budget $\alpha$ more efficiently.
\begin{algorithm}[H]
\caption{Online change point detection via mean testing (modified)}\label{alg:cpd_highd_sim}
\begin{algorithmic}
\INPUT Dataset $\{X_u\}_{u\in\mathbb{N}}$, Class of inlier distributions $\mathcal{D}=\{\mathcal{G}_{\theta,M}^p,\mathcal{P}_{v,\phi}^p\}$, Signal size input $\kappa_0>0$, Standard deviation $\sigma>0$, False alarm probability $\alpha\in(0,1)$, Minimum sample size $h_t$, Outlier control threshold $\Omega$, Contamination level $\varepsilon\in [0,\Omega)$, Filtering strength $C_\gamma$, Detection sensitivity factor $T_u$, Group number constant $K_c$ (for $\mathcal{D}=\mathcal{P}_{v,\phi}^p$ only) 
\State $t \gets 2$
\State $r \gets 1$
\State $\mathrm{FLAG} \gets 0$
\While{$\mathrm{FLAG} = 0$}
    \State $t \gets t + 1$
    \State $\text{tested} \gets \text{false}$
    \State $\delta_t\gets \frac{4\alpha}{t(r+1)(r+2)}$
    \State $Y_{s,t} \gets (X_{t-s+1} - X_s)/(\sqrt{2}\sigma), \forall 1\leq s\leq \lfloor t/2 \rfloor$
    \For{$s= h_t$ \textbf{to}  $\lfloor t/2 \rfloor$}
        \If{$\mathcal{D}=\mathcal{G}_{\theta,M}^p$}
    \State $u \gets 2\varepsilon + \frac{1}{n}
+ \sqrt{\frac{2(2\varepsilon + 1/n)\log(4/\delta)}{n}}
+ \frac{2\log(4/\delta)}{3n}$
        \If{$u\leq \Omega$}
        \State $\mathrm{FLAG}=\mathrm{RobustMeanTest}(\{Y_{i,t}\}_{i=1}^s;\kappa_0/\sqrt{2}, \delta_t, 2\varepsilon, C_\gamma, T_u)$
        \EndIf
        \Else
        \State $K\leftarrow \lceil K_c \log(1/\delta_t) \rceil$
        \State $n\gets \lfloor n/K \rfloor$
        \State $u \gets 2\varepsilon + \frac{1}{20}\min\left(1,\left(\frac{p}{n_0}\right)^{v/4}\right)
+ \sqrt{\frac{2(2\varepsilon + 0.05\min((p/n_0)^{v/4},1))\log(16K)}{n_0}}
+ \frac{2\log(16K)}{3n_0}$
        \If{$u\leq \Omega$}
        \State $\mathrm{FLAG}=\mathrm{RobustMeanTest}_{\text{MoM}}(\{Y_{i,t}\}_{i=1}^s;\kappa_0/\sqrt{2}, K, 2\varepsilon, C_\gamma, T_u)$
        \EndIf
        \EndIf
        \If{$\mathrm{FLAG} = 1$} \textbf{break} \EndIf
    \EndFor
    \If{$\text{tested} = \text{true}$} \State $r \gets r + 1$ \EndIf
\EndWhile
\State \Return $t$
\end{algorithmic}
\end{algorithm}

\subsection{Faceted plots for sensitivity analysis}\label{append:facet_plot}
\Cref{fig:individual-cgamma} provides a faceted version of \Cref{fig:hd_misspec}, with the empirical probabilities shown separately for each value of $C_\gamma$.
\begin{figure}[htbp]
     \centering
     \begin{subfigure}[b]{0.48\textwidth}
         \centering
         \includegraphics[width=\textwidth]{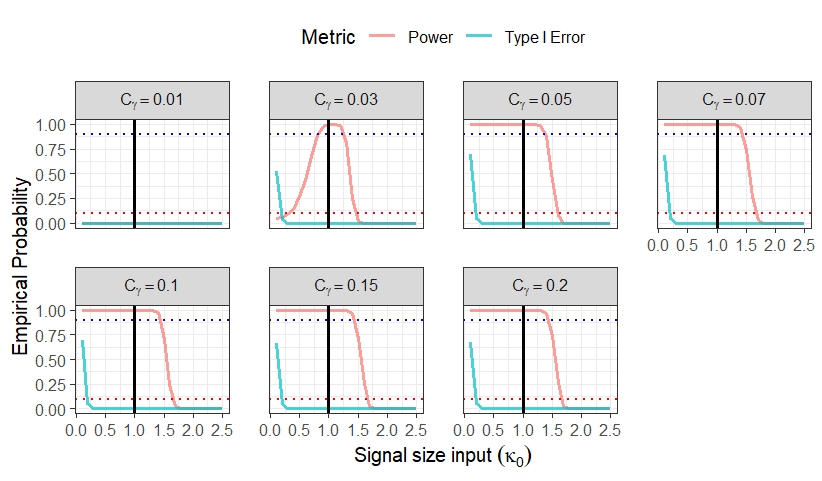}
         \caption{Laplace distribution, $p=10, \varepsilon=0.01$.}
     \end{subfigure}
     \hfill
     \begin{subfigure}[b]{0.48\textwidth}
         \centering
         \includegraphics[width=\textwidth]{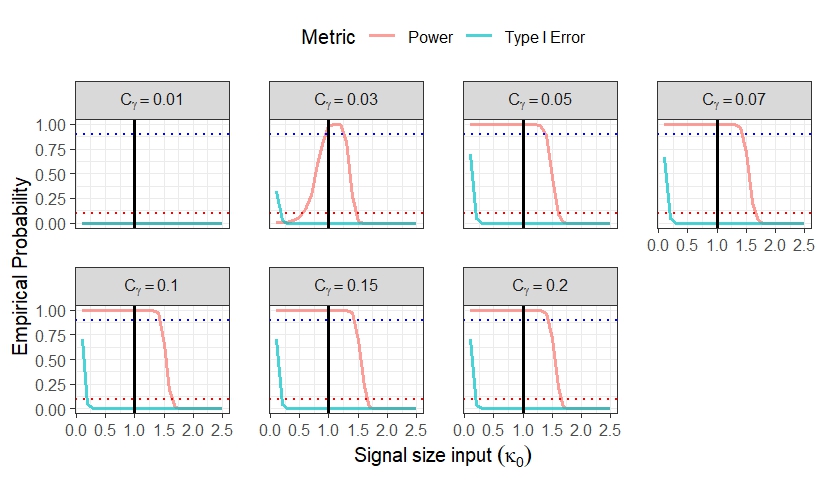}
         \caption{$t$-distribution, $p=10, \varepsilon=0.01$.}
     \end{subfigure}
     \begin{subfigure}[b]{0.48\textwidth}
         \centering
         \includegraphics[width=\textwidth]{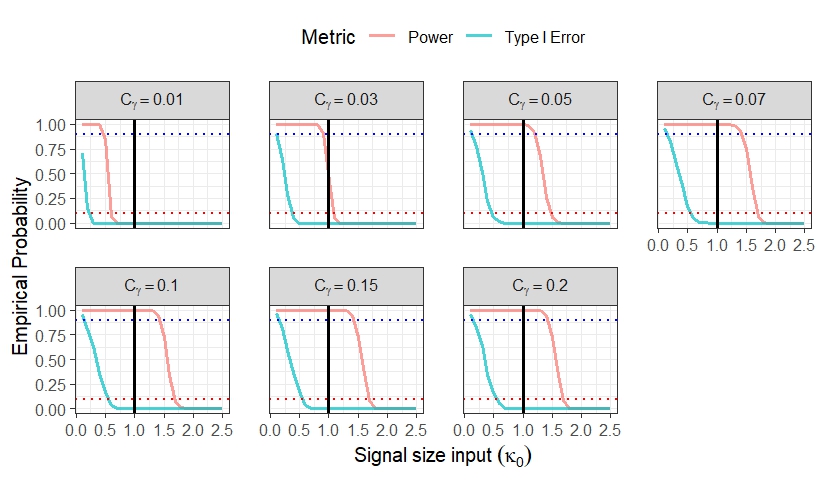}
         \caption{Laplace distribution, $p=600, \varepsilon=0.01$.}
     \end{subfigure}
     \hfill
     \begin{subfigure}[b]{0.48\textwidth}
         \centering
         \includegraphics[width=\textwidth]{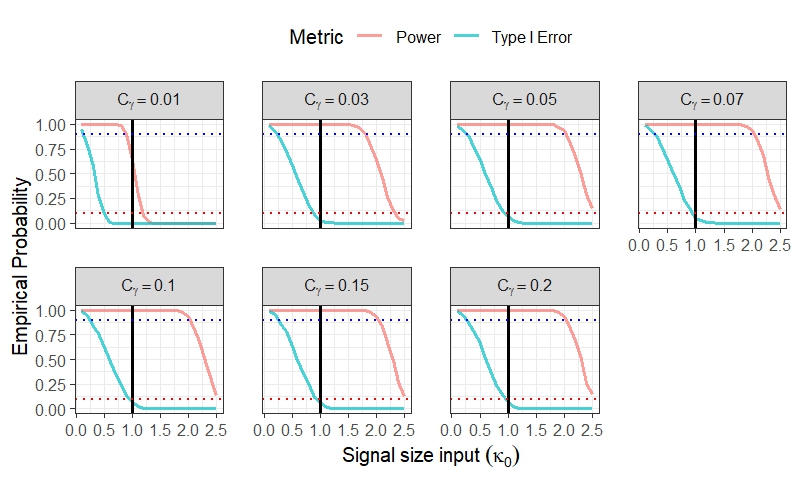}
         \caption{$t$-distribution, $p=600, \varepsilon=0.01$.}
     \end{subfigure}
     \caption{Empirical type I error and power in testing between hypotheses in \eqref{eqn:hypotheses} using \Cref{alg:efficient-mean-tester} under different specifications of $\kappa_0$, shown separately for each value of the filtering parameter $C_\gamma$. True signal size is $\kappa=1$, indicated by a vertical black line. For each value of $\kappa_0$ and $C_\gamma$, we report the proportion of simulation runs which resulted in false rejection (type I error) and correct rejections (Power). The horizontal red line indicates the 10\% threshold for type I error control and the horizontal blue line indicates the 90\% threshold for power.}
     \label{fig:individual-cgamma}
\end{figure}

\newpage
\section{Auxiliary lemmas}\label{sect:auxillarylem}
\subsection{Basic properties of sub-Weibull random variables}\label{subsect:subw_prop}
\begin{definition}[Orlicz norms] \label{defn:orlicz}
    Let $f: [0, \infty) \rightarrow [0, \infty)$ be a non-decreasing function with $f(0) = 0$. The f-Orlicz norm of a real-valued random variable $X$ is
    $$\norm{X}_f = \inf\{t > 0 : \E f(|X|/t) \leq 1\}.$$ 
\end{definition}

\begin{definition} [Sub-Weibull random variables] \label{defn:subWeib}
    A random variable $X$ is sub-Weibull with parameter $\theta > 0$,
denoted sub-Weibull$(\theta)$, if
$$\norm{X}_{\psi_\theta} < \infty,$$
with the function $\psi_\theta$ defined by $\psi_\theta(x) = \exp(x^\theta)- 1$ for $x \geq 0$. 
\end{definition}

\begin{lemma} [\citealt{Vladimirova_Girard_Nguyen_Arbel_2020}, Theorem 2.1] \label{thm:subWeib}
    Let X be a sub-Weibull$(\theta)$ random variable with $ \theta>0$ and $\norm{X}_{\psi_\theta} = M < \infty$. Then, we have the following properties.
    \begin{enumerate}[label=(\textbf{\alph*})]
        \item For any $x \geq 0$, it holds that
    \[
        \mathbb{P}(|X| \geq x) \leq 2 \exp\{-(x/M)^\theta\}.
    \]
        \item There exists an absolute constant $K_{\theta} > 0$ such that
        \[
            (\E|X|^k)^{1/k} \leq K_\theta Mk^{1/\theta}, \quad\forall k\geq 1.
        \]
    \end{enumerate}
\end{lemma}

We now provide two tail bound results from literature for sums and quadratic forms of independent sub-Weibull random variables respectively. \Cref{prop:quadraticweibulltail} below can be viewed as an extension of the Hanson--Wright inequality \citep{Hanson1971bound}.
\begin{lemma}[\cite{Hao_2019}, Theorem 3.1]\label{lemma:weib_concentration}
    Suppose $\{X_i\}_{i=1}^n$ are independent sub-Weibull random variables with $\norm{X_i}_{\psi_\theta}\leq M$. Then there exists an absolute constant $C_\theta$ only depending on $\theta$ such that for any $\mathbf{a}\in\mathbb{R}^n$ and $0<\delta<e^{-2}$,
    \begin{equation*}
        \left|\sum_{i=1}^na_iX_i-\E\left[\sum_{i=1}^na_iX_i\right]\right|\leq C_\theta M \left(\norm{\mathbf{a}}_2\sqrt{\log(1/\delta)}+\norm{\mathbf{a}}_\infty\log^{1/\theta}(1/\delta)\right)
    \end{equation*}
    with probability $1-\delta$.
\end{lemma}

\begin{proposition}[\citealp{Gotze2019}, Proposition 1.5] \label{prop:quadraticweibulltail}
Let $\theta \in (0,1] \cup \{2\}$, $A = (a_{ij}) \in \mathbb{R}^{n \times n}$ be a symmetric matrix and $X_1, \dotsc, X_n$ be independent mean zero sub-Weibull random variables of order $\theta$, with $\mathbb{E}X_i^2 = \sigma_i^2$ and $\|X_i\|_{\psi_\theta} \leq M$ for all $i \in \mathbb{Z}^+$ and for some $M > 0$. Then, there exists a constant $C_\theta > 0$, depending only on $\theta$, such that for any $x \geq 0$, we have
\begin{equation*}
    \mathbb{P}\biggl(\Bigl|\sum_{1 \leq i, j \leq n} a_{ij}X_i X_j - \sum_{i=1}^n a_{ii} \sigma_i^2\Bigr| \geq x\biggr) \leq 2\exp\biggl( - \frac{1}{C_{\theta}} \eta_\theta(x/M^2; A) \biggr),
\end{equation*}
where
\[
\eta_\theta(x; A) = \min\Biggl\{ \biggl(\frac{x}{\|A\|_{\mathrm{F}}}\biggr)^2, \frac{x}{\|A\|_2}, \biggl(\frac{x}{\|A\|_{2 \rightarrow \infty}}\biggr)^{\frac{2\theta}{2+\theta}}, \biggl(\frac{x}{\|A\|_{\max}}\biggr)^{\frac{\theta}{2}}  \Biggr\}.
\]
\end{proposition}
\begin{remark}
    \Cref{prop:quadraticweibulltail} can be restated as follows. Given the assumptions of \Cref{prop:quadraticweibulltail}, we have
\begin{align*}
    \left|\sum_{1 \leq i, j \leq n} a_{ij}X_i X_j - \sum_{i=1}^n a_{ii} \sigma_i^2\right| \leq C_\theta' M^2\bigg(&\norm{A}_F\sqrt{\log(1/\delta)}+\norm{A}_2\log(1/\delta)\\
    &+\norm{A}_{2\rightarrow \infty}\log^{\frac{2\theta}{2+\theta}}(1/\delta)+\norm{A}_{\max}\log^{2/\theta}(1/\delta)\bigg)
\end{align*}
for some $C_\theta'>0$ with probability at least $1-\delta$.
\end{remark}

The following lemmas in \cite{Zhang_Wei_2022} relate the $\psi_\theta$-Orlicz norm of random variable $X$ and $\psi_{\theta/2}$-Orlicz norm of $X^2$.
\begin{lemma}[\citealp{Zhang_Wei_2022}, Proposition 2]\label{cor:multiplied_norm}
    Let $X$, $Y$ be random variables such that\\ $\max(\norm{X}_{\psi_\theta}, \norm{Y}_{\psi_\theta})<\infty$ for some $\theta\in(0,\infty)$. Then,
    $$\norm{XY}_{\psi_{\theta/2}}=\norm{X}_{\psi_\theta}\norm{Y}_{\psi_\theta}.$$
\end{lemma}

\begin{lemma}[\citealp{Zhang_Wei_2022}, Corollary 4]\label{cor:squared_norm}
    Let $X$ be a random variable such that $\norm{X}_\theta<\infty$ for some $\theta\in(0,\infty)$. Then 
    $$\norm{X^2}_{\psi_{\theta/2}}=\norm{X}_{\psi_\theta}^2.$$
\end{lemma}

The following proposition shows that the Orlicz norm of non-centered and centered random variables are of the same order. 
\begin{proposition}[\citealp{Gotze2019}, Corollary A.5] \label{prop:centering_norm}
    For any $\theta>0$ and real random variable $X$, we have
    \begin{equation*}
        \norm{X-\E X}_{\psi_\theta}\leq K_\theta\left(1+\left(\frac{\log2}{2}(\theta e)^{1/\theta}\right)^{-1/\theta}\right)\norm{X}_{\psi_\theta},
    \end{equation*}
    where $K_\theta=\begin{cases}
        2^{1/\theta} &\text{ if $\theta\in(0,1)$}\\
        1 &\text{ if $\theta\geq 1$}
    \end{cases}$.
\end{proposition}

\subsection{Properties of univariate robust mean estimators}
A key part in our proposed univariate change point procedure is the use of robust estimators, namely the Optimal Robust Univariate Mean Estimator \citep{pmlr-v108-prasad20a} and the median. We will outline their statistical properties in the following section, which are crucial in proving results on our change point procedure in Section \ref{sect:univariate}.

\subsubsection*{Medians}
We study the concentration properties of medians under contamination and different heavy-tailed distributions. We first prove a concentration inequality for the median where the inlier distributions in the class $\mathcal{G}_{\theta,M}$.

\begin{lemma}[Median concentration with inlier class $\mathcal{G}_{\theta,M}$]  \label{thm:median_subWeib} 
    Let $\delta\in(0,1)$ and $0\leq\varepsilon<\delta^{2/n}(2e)^{-1}$. Suppose $\{X_i\}_{i=1}^n$ is a sequence of independent random variables, drawn from a dynamic $\varepsilon$-Huber contamination model with inlier class $\mathcal{D} = \mathcal{G}_{\theta,M}$. Assume further that all inlier distributions have zero mean. Then,
    \begin{equation*}
        \Pb\Biggl(\bigl|\med(X_{1:n})\bigr|\leq M\biggl\{\log(4e)+\log\frac{1-\varepsilon}{\delta^{2/n}-2e\varepsilon}\biggr\}^{1/\theta}\Biggr)\geq 1-\delta.
    \end{equation*}
\end{lemma}

\begin{proof}
Let $s = M\Bigl\{\log(4e)+\log\frac{1-\varepsilon}{\delta^{2/n}-2e\varepsilon}\Bigr\}^{1/\theta}$. Denote $\mathcal{I}=\{i \in [n]:X_i\sim F\}$ and define the random variable $$N_n=\sum_{i=1}^n\mathbbm{1}(\{i\in\mathcal{I},|X_i|\geq s\}\cup \{i\notin \mathcal{I}\}).$$ 
Then, by a union bound and Lemma \ref{thm:subWeib}, we have
\begin{equation}\label{eqn:proof_p_s}
        \Pb(\{i\in\mathcal{I},|X_i|\geq s\}\cup \{i\notin \mathcal{I}\})\leq\varepsilon+2(1-\varepsilon)\exp\left(-\left(\frac{s}{M}\right)^\theta\right)=:p_s,
    \end{equation}
Consequently,
    \begin{align*}
        \Pb(|\med(X_{1:n})|> s)&\leq\Pb\left(\sum_{i=1}^n\mathbbm{1}(|X_i|\geq s)>\frac{n}{2}\right)\leq\Pb\left(N_n>\frac{n}{2}\right)\\
        &=\Pb\left(N_n>np_s\left(1+\left(\frac{1}{2p_s}-1\right)\right)\right)\\
        &\leq\exp\left(np_s\left(\frac{1}{2p_s}-1-\frac{1}{2p_s}\log\left(\frac{1}{2p_s}\right)\right)\right)\\
        &\leq\exp\left(\frac{n}{2}\left(1-\log\left(\frac{1}{2p_s}\right)\right)\right)\\
        &=\left({2ep_s}\right)^{n/2}=\delta,
    \end{align*}
where the first inequality follows from the fact that $|\med(X_{1:n})|> s$ implies at least $n/2$ data points $\{X_j:1\leq j\leq n\}$ satisfy $|X_j|>s$, the second inequality is by the fact that
$N_n\geq\sum_{i=1}^n1(|X_i|\geq s)$ almost surely,
the third inequality is due to the multiplicative Chernoff bound  (\Cref{lem:multiCher}), the fourth inequality is obtained by discarding -1 inside the bracket, and the final equality is achieved by substituting $s=M\left[\log(4e)+\log\frac{1-\varepsilon}{\delta^{2/n}-2e\varepsilon}\right]^{1/\theta}$.
\end{proof}

Using a similar proof technique, we prove a concentration inequality for the median where the inlier distributions in the class $\mathcal{P}_{v,\phi}$.
\begin{lemma}[Median concentration with inlier distribution in $\mathcal{P}_{v,\phi}$]  \label{thm:median_moment}
    Let $\delta\in(0,1)$ and $0\leq\varepsilon<\delta^{2/n}(2e)^{-1}$. Suppose $\{X_i\}_{i=1}^n$ are independent random variables, and are distributed according to a dynamic $\varepsilon$-Huber contamination model \eqref{huber} with inlier class $\mathcal{D}=\mathcal{P}_{v,\phi}$. Assume further that all inlier distributions have zero mean. Then the median $\med(X_{1:n})$ satisfies
    \begin{equation*}
        \Pb\left(|\med(X_{1:n})|\leq \phi\left(2e\frac{1-\varepsilon}{\delta^{2/n}-2e\varepsilon}\right)^{1/v}\right)\geq 1-\delta.
    \end{equation*}
\end{lemma}

\begin{proof}
    Let $s=\phi\left(2e\frac{1-\varepsilon}{\delta^{2/n}-2e\varepsilon}\right)^{1/v}$. Following the proof of lemma \ref{thm:median_subWeib}, under this setting, $p_s$ in equation \eqref{eqn:proof_p_s} will become 
    \begin{equation*}
        p_s=\varepsilon+(1-\varepsilon)\frac{\phi^v}{s^v},
    \end{equation*}
    by Markov's inequality and the fact that $\E[|X_1|^{v}]=\phi^v<\infty$.
    Thus, we have
    \begin{align*}
        \Pb(|\med(X_{1:n})|> s)\leq
        \left(2ep_s\right)^{n/2}=\delta,
    \end{align*}
where we substituted $s=\phi\left(2e\frac{1-\varepsilon}{\delta^{2/n}-2e\varepsilon}\right)^{1/v}$ in the final equality.
\end{proof}

Apart from the concentration properties, medians can also converts constant-probability guarantees into exponentially small failure probabilities.
\begin{proposition}\label{prop:med}
    Suppose $\{X_i\}_{i=1}^n$ are independent and identically distributed. If there exists $c\in\mathbb{R}$ such that $\Pb(X_i>c)\le 1/4$ for all $i$, then 
\[
\mathbb{P}\!\left( \med(X_{1:n}) > c \right)
\le \exp(-n/8).
\]
In particular, if $n \ge 8 \log(1/\omega)$, where $\omega\in (0,1)$, then
\[
\mathbb{P}\!\left( \med(X_{1:n}) > c \right)
\le \omega.
\]
\end{proposition}

\begin{proof} We have
    \begin{align*}
        \Pb\left(\med(X_{1:n})> c\right)
        &\leq\Pb\left(\sum_{i=1}^n\mathbbm{1}\left(T_i> c\right)\geq\frac{n}{2}\right)\\
        &=\Pb\left(\frac{1}{n}\sum_{i=1}^n[\mathbbm{1}(T_i> c)-\Pb(T_i> c)]\geq\frac{1}{2}-\Pb(T_i> c)\right)\\
        &=\Pb\left(W\geq\frac{1}{2}-\Pb(T_i> c)\right)\\
        &\leq\Pb\left(W\geq\frac{1}{4}\right)\leq e^{-n/8},
    \end{align*}
    where $W$ is a sub-Gaussian random variable with variance parameter $1/(4n)$. 
\end{proof}

\begin{remark}
    By a very similar argument, if there exists $c\in\mathbb{R}$ such that $\Pb(X_i<c)\le 1/4$ for all $i$, then 
\[
\mathbb{P}\!\left( \med(X_{1:n}) < c \right)
\le \exp(-n/8).
\]
\end{remark}

\subsubsection*{Optimal Robust Univariate Mean Estimation}
The Optimal Robust Univariate Mean Estimator \citep{pmlr-v108-prasad20a} enables robust estimation of the inlier mean under the dynamic Huber contamination model \eqref{huber}, even when $F$ is heavy-tailed. A slightly-modified version of their approach is described in \Cref{algo:RUME}. \Cref{algo:RUME} is similar to that of trimmed means \citep{Lugosi_Mendelson_2021}. Both approaches split the dataset into two subsamples of equal size. One subsample is used to learn a truncation interval, and the other is used to compute the resulting truncated empirical mean. The difference lies in how they pick the truncation intervals. While trimmed means estimates the $q$-th sample quantile and $(1-q)$-th sample quantile for some $q\in(0,1)$, \Cref{algo:RUME} picks a shortest interval $\hat{I}$ containing $1-2q$ fraction of observations.\footnote{While it is possible that there may be more than one possible shortest interval, we can pick any shortest interval at random in this case as the proof still works. Thus, for the following proof we will assume there is only one such shortest interval $\hat{I}$.} 

\begin{algorithm}[H]
\centering
\caption{Robust Univariate Mean Estimation (RUME)}
\begin{algorithmic}
\INPUT{Dataset $\{z_i\}_{i=1}^{2n}$, Corruption Level $\varepsilon$, Confidence Level $\delta$}
\State Split the data into two subsets: $\mathcal{Z}_1 = \{z_i\}_{i=1}^{n}$ and $\mathcal{Z}_2 = \{z_i\}_{i=n+1}^{2n}$. 
\State Let $\varepsilon' = \max \left(\varepsilon,\frac{\log (1/\delta)}{n}\right)$.
\State Let $\hat{I} = [a,b]$ be the shortest interval containing $\left\lfloor n \left(1-2\varepsilon'-2\sqrt{\varepsilon'\frac{\log(1/\delta)}{n}}-\frac{\log(1/\delta)}{n}\right)\right\rfloor$ points in $\mathcal{Z}_1$.
\OUTPUT $\frac{1}{\sum_{i=n+1}^{2n}\mathbb{I}\{z_i \in \hat{I}\}} \sum_{i=n+1}^{2n} z_i \mathbb{I}\{z_i \in \hat{I}\}$
\end{algorithmic}
\label{algo:RUME}
\end{algorithm}

Lemma 3 of \cite{pmlr-v108-prasad20a} shows that RUME is information theoretically optimal in the case where the inlier distribution $F$ has finite variance $\sigma^2$. The estimation error of RUME is of the order $O(\sigma\sqrt{\varepsilon'})$. \cite{Li_2021} has generalised the proof to study the dynamic Huber contamination model setting stated in Definition \ref{modelasmmp} under the finite variance assumption of $F$. Here, we further generalise the proof for $F\in\mathcal{G}_{\theta,M}$ and $F\in\mathcal{P}_{v,\phi}$.

\begin{lemma}[Estimation error of RUME]\label{lem:rume_moment}
Suppose $\{Z_i\}_{i=1}^{2h}$ are independent and distributed according to a dynamic Huber $\varepsilon$-contamination model \eqref{huber} with means $f_i=\mu$ and $F_i=F_0$ for all $i\in\{1,\ldots,2h\}$. Suppose 
\[
\varepsilon'=\max\left\{\varepsilon,\frac{\log(1/\delta)}{h}
\right\},\]
satisfies 
\begin{equation}\label{A1_moment}
    2\varepsilon'+2\sqrt{\varepsilon'\frac{\log(1/\delta)}{h}}+\frac{\log(1/\delta)}{h} < \frac{1}{2}, \;\; \mathrm{and} \;\; \delta \leq \min(1/h,1/4),
\end{equation}
then 
\begin{enumerate}[label=(\textbf{\alph*})]
    \item assuming further that $F_0\in\mathcal{P}_{v,\phi}$, it holds that with probability at least $1-\delta/2$, 
\begin{equation*}
    |\mathrm{RUME}(\{Z_i\}_{i=1}^{2h})-\mu| \leq \phi\left( C_1(\varepsilon')^{1-1/v}+C_2\sqrt{\frac{\log(1/\delta)}{h}}\right);
\end{equation*}
    \item assuming further that $F_0\in\mathcal{G}_{\theta,M}$, it holds that with probability at least $1-\delta/2$, 
    \begin{equation*}
        |\mathrm{RUME}(\{Z_i\}_{i=1}^{2h})-\mu|
    \leq C_3 \varepsilon'\log^{1+1/\theta} \left(\frac{1}{\varepsilon'}\right)+C_4\sqrt{\frac{\log(1/\delta)}{h}},
    \end{equation*}
\end{enumerate}
where $C_1,C_2,C_3,C_4 > 0$ are absolute constants. 
\end{lemma}
\Cref{lem:rume_moment} shows that there are two regimes for the RUME estimation error, depending on whether the contamination term ($[\varepsilon']^{1-1/v}$ in finite moment setting or $\varepsilon'\log^{1+1/\theta} \left(\frac{1}{\varepsilon'}\right)$ in sub-Weibull setting) or the sampling term $\sqrt{\log(1/\delta)/h}$ dominates. The existence of these two regimes imply the optimality for mean estimation under heavy-tailed data and Huber contamination. In the absence of contamination ($\varepsilon=0$), we recover the sub-Gaussian rate of estimation $O(\sqrt{\log(1/\delta)/h})$ for both classes of inlier distributions. When contamination level is large, the RUME estimation error matches the information theoretic lower bound for mean estimation under Huber contamination when $F_0\in\mathcal{P}_{v,\phi}$ \citep{Steinhardt2018}. 

\begin{proof}
(a) Without loss of generality, we can take $\mu = 0$. Let $I^*$ be the interval $(-(\phi^{v}/\varepsilon')^{1/v},(\phi^{v}/\varepsilon')^{1/v})$ and $F_0(I^*)$ denotes the probability that one sample drawn from \eqref{huber} is distributed according to $F_0$ and lies in $I^*$. 
Then we have 
$$F_0(I^*) = \mathbb{P}(Z_i \in I^*\, \text{and}\, Z_i \sim F_0) = \mathbb{P}(Z_i \in I^* |Z_i \sim F_0) \mathbb{P}(Z_i \sim F_0) \geq (1-\varepsilon')(1-\varepsilon)\geq 1-2\varepsilon',$$ where the first inequality follows from Chebyshev's inequality \citep[e.g.][Corollary 1.6.3]{Vershynin_2025}.

Now let $X_i = \mathbbm{1}\{Z_i \sim F_0 \,\,\text{and}\,\,Z_i \in  I^* \}$ and $F_0^h(I^*) = \sum_{i=1}^h X_i/h$. Note that $X_i$ is a Bernoulli random variable with success probability $F_0(I^*)$. Therefore, using the Bernstein inequality for bounded random variables \citep[e.g.~Theorem 2.9.5 in][]{Vershynin_2025}, we have with probability at least $1-\delta$  
\begin{align}
    F_0^h(I^*) - F_0(I^*) &\geq -\sqrt{F_0(I^*)(1-F_0(I^*))}\sqrt{\frac{2\log(1/\delta)}{h}}-\frac{2\log(1/\delta)}{3h}, \nonumber \\
    F_0^h(I^*) &\geq 1-2\varepsilon'-\sqrt{2\varepsilon'(1-2\varepsilon')\frac{2\log(1/\delta)}{h}}-\frac{2\log(1/\delta)}{3h}\label{I*numa},
\end{align}
since $F_0(I^*)(1-F_0(I^*))$ is a decreasing function of $F_0(I^*)$ when $F_0(I^*)>1/2$. 

Let $g_h(\varepsilon,\delta) = \varepsilon+\sqrt{\varepsilon\frac{2\log(1/\delta)}{h}}+\frac{2\log(1/\delta)}{3h}$ and $\hat{I} = [a,b]$ be a shortest interval containing $h(1-g_h(2\varepsilon',\delta))$ points in $\mathcal{Z}_1$. Since $I^*$ also contains at least $h(1-g_h(2\varepsilon',\delta))$ points due to (\ref{I*numa}) with probability at least $1-\delta$, and $\hat{I}$ is the shortest interval containing at least $h(1-g_h(2\varepsilon',\delta))$ points, we must have
\[
\mathrm{length}(\hat{I})\leq \mathrm{length}(I^*) = 2\left(\frac{\phi^{v}}{\varepsilon'}\right)^{1/v}.
\]
with probability at least $1-\delta$. Further, by given assumption that $g_h(2\varepsilon',\delta) < 1/2$, then both $\hat{I}$ and $I^*$ contain more than half of the data in $\mathcal{Z}_1$. As a result, these two intervals must intersect and we have  
\begin{equation}\label{eq1a}
    |z - \mu |\leq 4\left(\frac{\phi^{v}}{\varepsilon'}\right)^{1/v} \quad \forall z \in \hat{I}.
\end{equation}

Next, we control the error of the final estimator. Let $|\hat{I}| = \sum_{Z_i\in \mathcal{Z}_2}\mathbbm{1}\{Z_i \in \hat{I}\}$ be the number of points from the second sample and lie in $\hat{I}$. Similarly, let $|\hat{I}_{H}|$ and $|\hat{I}_{F_0}|$ denote the number of points that lie in $\hat{I}$ and are not distributed according to $F_0$ (i.e.\ adversarial point) and according to $F_0$ respectively. Note that 
\begin{equation*}
    \left|\frac{1}{|\hat{I}|} \sum_{Z_i \in \hat{I}\cap \mathcal{Z}_2}Z_i  \right| \leq T_1+T_2,
\end{equation*}
where 
\begin{equation*}
    T_1 = \left|\frac{1}{|\hat{I}|} \sum_{\substack{Z_i \in \hat{I}\cap\mathcal{Z}_2 \\ Z_i \not\sim F_0}}Z_i \right| \quad \text{and} \quad T_2 = \left|\frac{1}{|\hat{I}|} \sum_{\substack{Z_i \in \hat{I}\cap\mathcal{Z}_2 \\ Z_i \sim F_0}}Z_i  \right|.
\end{equation*}

\textbf{Control of $T_1$}: Using \eqref{eq1a}, we have
\[
T_1 \leq \frac{|\hat{I}_{H}|}{|\hat{I}|}\max_{\substack{Z_i \in \hat{I}\cap\mathcal{Z}_2\\Z_i\not\sim F_0 }}|Z_i| \leq  4\frac{|\hat{I}_H|}{h} \frac{h}{|\hat{I}_{F_0}|} \left(\frac{\phi^{v}}{\varepsilon'}\right)^{1/v}.
\]
Again, by Bernstein's inequality, we have, with probability at least $1-\delta$, that
\begin{equation}\label{H0bound}
\frac{|\hat{I}_H|}{h} \leq \frac{\sum_{i: Z_i \in \mathcal{Z}_2} \mathbbm{1}\{Z_i \not \sim F_0\}}{h}
    \leq  \varepsilon+\sqrt{\varepsilon(1-\varepsilon)}\sqrt{\frac{2\log(1/\delta)}{h}}+\frac{2\log(1/\delta)}{3h} \leq g_h(\varepsilon,\delta) \leq 3.1 \varepsilon'.
\end{equation} 
Thus, we bound
\begin{equation}\label{t1a}
    T_1 \leq 12.4\phi\varepsilon'^{1-1/v}\frac{h}{|\hat{I}_{F_0}|},
\end{equation}
with probability at least $1-\delta$. 

\textbf{Control of $T_2$}: To control $T_2$, we write it as 

\[
T_2 = \left|\frac{|\hat{I}_{F_0}|}{|\hat{I}|}\left[\frac{1}{|\hat{I}_{F_0}|}\sum_{\substack{Z_i\in \hat{I}\cap\mathcal{Z}_2\\Z_i \sim F_0}}Z_i \right]\right| \leq T_{2a} +T_{2b},
\]
where 
\[
T_{2a}=\frac{|\hat{I}_{F_0}|}{|\hat{I}|}\left[\frac{1}{|\hat{I}_{F_0}|}\sum_{\substack{Z_i\in \hat{I}\cap\mathcal{Z}_2\\Z_i \sim F_0}}Z_i-\mathbb{E}[Z|Z \in \hat{I}, Z \sim F_0]\right] \quad \text{and} \quad T_{2b} = \frac{|\hat{I}_{F_0}|}{|\hat{I}|}\left|\mathbb{E}[Z|Z \in \hat{I}, Z \sim F_0]\right|.
\]
In $T_{2a}$, since conditional on $Z_i \in \hat{I}$, each $Z_i$ is a bounded random variable with $|Z_i-\mathbb{E}(Z_i)|\leq \text{length}(\hat{I}) = 2(\phi^{v}/\varepsilon')^{1/v}$ and they are independent of each other, we can again use Bernstein inequality. By Lemma \ref{lem:variance_shift}, for any event $E$ that occurs with probability at least $P(E)$, we can upper bound the conditional variance of $Z_i$ by
\[
\mathbb{E}_{Z\sim F_0}\left[\left(Z-\mathbb{E}[Z|Z\in E]\right)^2|Z\in E\right]\leq \left(\frac{\phi^{v}}{P(E)}\right)^{2/v}.
\]
Denote $F_0(\hat{I})$ to be the probability that $Z_i$ is distributed according to $F_0$ and lies in $\hat{I}$. Then, we have with probability at least $1-\delta$

\begin{equation}\label{t2aa}
    T_{2a} \leq \sqrt{\frac{2\log(1/\delta)}{|\hat{I}_{F_0}|}\left(\frac{\phi^{v}}{F_0(\hat{I})}\right)^{2/v}} + 4\left(\frac{\phi^{v}}{\varepsilon'}\right)^{1/v}\frac{\log(1/\delta)}{3|\hat{I}_{F_0}|},
\end{equation}
by the Bernstein inequality. 

For $T_{2b}$, we first notice 
\begin{align*}
   \mathbb{E}_{Z\sim F_0}[Z|Z \not\in \hat{I}] &= \frac{\mathbb{E}_{Z\sim F_0}[Z \mathbbm{1}_{Z \not\in \hat{I}}]}{F_0(\hat{I}^c)} \\
    & \leq \left(\frac{\mathbb{E}_{Z\sim F_0}[Z^v]}{F_0(\hat{I}^c)}\right)^{1/v} \\
    &= \left(\frac{\phi^{v}}{F_0(\hat{I}^c)}\right)^{1/v},
\end{align*}
where $F_0(\hat{I}^c)$ is the probability that $Z$ is distributed according to $F_0$ but does not lie in $\hat{I}$ and we use Hölder's inequality in the second line. Combining with fact that (as $\mu=0$)
\[
\left|\mathbb{E}\left[Z|Z\in \hat{I}\right]\right|F_0(\hat{I}) = F_0(\hat{I}^c)\left|\mathbb{E}\left[Z|Z\not\in \hat{I}\right]\right|,
\]
we have 
\begin{equation}\label{t2ba}
    T_{2b} \leq \phi \frac{F_0(\hat{I}^c)^{1-1/v}}{F_0(\hat{I})}.
\end{equation}

Combining \eqref{t1a}, \eqref{t2aa}, and \eqref{t2ba}, with probability at least $1-3\delta$, we have

\begin{equation}\label{t1t2a}
    |\mathrm{RUME}| \leq 12.4 \phi\varepsilon'^{1-1/v}\frac{h}{|\hat{I}_{F_0}|}+\left(\frac{\phi^{v}}{F_0(\hat{I})}\right)^{1/v}\sqrt{\frac{2\log(1/\delta)}{|\hat{I}_{F_0}|}} +\left(\frac{\phi^{v}}{\varepsilon'}\right)^{1/v}\frac{4\log(1/\delta)}{3|\hat{I}_{F_0}|}+\phi \frac{F_0(\hat{I}^c)^{1-1/v}}{F_0(\hat{I})}.
\end{equation}

Let $|\hat{h}_{H}|$ denote the number of points in $\mathcal{Z}_1$ which are not drawn from $F_0$ and lie in $\hat{I}$, and $|\hat{h}_{F_0}|$ denote the number of points in $\mathcal{Z}_1$ which are drawn from $F_0$ and lie in $\hat{I}$. To get the claimed bound, we need to study $F_0(\hat{I})$ and $F_0^h(\hat{I}) = \frac{|\hat{h}_{F_0}|}{\sum_{Z_i \in \mathcal{Z}_1}\mathbbm{1}_{Z_i \sim F_0}}$. Note that $F_0^h(\hat{I})$ is a sample version of $F_0(\hat{I})$.  Similar to \eqref{H0bound}, with probability at least $1-\delta$, we can upper bound $\frac{|\hat{h}_{H}|}{h}$ again by 
\[
\frac{|\hat{h}_{H}|}{h} \leq \frac{\sum_{i: Z_i \in \mathcal{Z}_1} \mathbbm{1}\{Z_i \not \sim F_0\}}{h}
    \leq  g_h(\varepsilon,\delta)
\]
using the Bernstein inequality.
Since $|\hat{h}_{H}|+|\hat{h}_{F_0}| = h(1-g_h(2\varepsilon',\delta))$, we have with probability at least $1-\delta$,
\begin{equation}\label{pn}
    |\hat{h}_{F_0}| \geq h\left(1-g_h(2\varepsilon',\delta)-g_h(\varepsilon,\delta)\right).
\end{equation}
Note that by \eqref{pn}, we have
$$F_0^h(\hat{I}) = \frac{|\hat{h}_{F_0}|}{\sum_{Z_i \in \mathcal{Z}_1}\mathbbm{1}_{Z_i \sim F_0}}\geq\frac{|\hat{h}_{F_0}|}{h}\geq1-g_h(2\varepsilon',\delta)-g_h(\varepsilon,\delta).$$

Consequently, we have 
\begin{align*}
    F_0^h(\hat{I^c})&\leq g_h(2\varepsilon',\delta)+g_h(\varepsilon,\delta) \nonumber \\
    &\leq 3\varepsilon' + (\sqrt{\varepsilon'(1-\varepsilon')}+\sqrt{2\varepsilon'(1-2\varepsilon')})\sqrt{\frac{2\log(1/\delta)}{h}} + \frac{4\log(1/\delta)}{3h} \nonumber\\
    &\leq 8\varepsilon'.
\end{align*}
Using the relative deviation lemma from empirical process theory \citep[e.g. Theorem 7 in][]{Bousquet2004}, we can finally bound $F_0(\hat{I}^c)$ as  
\begin{equation}\label{vc}
    F_0(\hat{I}^c) \leq F_0^h(\hat{I}^c)+2\sqrt{F_0^h(\hat{I}^c)\frac{\log(S_{\mathcal{F}}(2h))+\log(4/\delta)}{h}}+4 \frac{\log(S_{\mathcal{F}}(2h))+\log(4/\delta)}{h},
\end{equation}
with probability at least $1-\delta$. Since the VC dimension for intervals in $\mathbb{R}$ is 2, we have $S_{\mathcal{F}}(2h) \leq (2h+1)^2$ by the Sauer-Shelah Lemma (e.g.~Theorem 8.3.9 in \cite{Vershynin_2025}). Substituting the upper bound on $S_{\mathcal{F}}(2h)$ and $F_0^h(\hat{I^c})$ into equation (\ref{vc}), we get with probability at least $1-2\delta$
\begin{align}
     F_0(\hat{I}^c) &\leq 8\varepsilon'+2\sqrt{8\varepsilon'}\left(\sqrt{\frac{2\log(2h+1)}{h}}+\sqrt{\frac{\log(4/\delta)}{h}}\right)+4 \left(\frac{2\log(2h+1)}{h}+\frac{\log(4/\delta)}{h}\right) \nonumber\\
     &\leq 8\varepsilon'+2\sqrt{8\varepsilon'}\left(\sqrt{\frac{4\log(h)}{h}}+\sqrt{\frac{2\log(1/\delta)}{h}}\right)+4 \left(\frac{4\log(h)}{h}+\frac{2\log(1/delta)}{h}\right) \nonumber \\
     &\leq 52\varepsilon', \label{empia}
\end{align}
where the first inequality uses the fact that $\sqrt{a+b}\leq \sqrt{a}+\sqrt{b}$ for any $a,b\geq 0$; the second inequality relies on $\log(4/\delta)\leq 2\log(1/\delta) \;\forall\delta\leq1/4$ and $\log(2h+1)\leq2\log(h) \; \forall h\geq3$; and the third inequality uses $\delta\leq1/h$. Under the event $\{F_0(\hat{I}^c)\leq52\varepsilon'\}$, we also have $F_0(\hat{I})\geq 1/2$ provided that $\varepsilon'\leq 1/104$ and $|\hat{I}_{F_0}|\geq h/8$ happens with probability at least $1-\delta$ in this event, since $e^{-9h/64}\leq\delta$.

Combining $|\hat{I}_{F_0}|\geq h/8$, \eqref{t1t2a} and \eqref{empia} and using the fact that $\log(2/\delta)\leq 2\log(1/\delta)\; \forall \delta\leq 1/4$, we have that with probability at least $1-6\delta$,
\begin{align*}
    |\mathrm{RUME}-\mu|&\leq 100 \phi\varepsilon'^{1-1/v}+(2\phi^{v})^{1/v}\sqrt{\frac{2\log(1/\delta)}{h/8}} +\left(\frac{\phi^{v}}{\varepsilon'}\right)^{1/v}\frac{4\log(1/\delta)}{3(h/8)}+10\phi (52\varepsilon')^{1-1/v}\\
    &\leq C_1 \varepsilon'^{1-1/v}+C_2\sqrt{\frac{\log(1/\delta)}{h}}
\end{align*}
for some absolute constants $C_1,C_2>0$.

(b) We follow the same proof as (i) with minor modifications. We let $I^*$ be the interval 
$$I^*=(-M(\log(2/\varepsilon'))^{1/\theta},\,M(\log(2/\varepsilon'))^{1/\theta}).$$ 
This satisfies $F_0(I^*)\geq 1-2\varepsilon'$. Thus we have with probability at least $1-\delta$,
\[
\mathrm{length}(\hat{I})\leq \mathrm{length}(I^*) = 2M(\log(2/\varepsilon'))^{1/\theta}.
\]
and since $\hat{I}$ and $I^*$ contains more than half of the data in $\mathcal{Z}_1$, the two intervals must intersect and we have
$$|z - \mu |\leq 4M(\log(2/\varepsilon'))^{1/\theta}, \quad \forall z \in \hat{I}.$$
Thus, we have with probability at least $1-2\delta$,
\begin{equation}\label{t1b}
    T_1 \leq \frac{4g_h(\varepsilon,\delta)}{1-g_h(2\varepsilon',\delta)}M(\log(2/\varepsilon'))^{1/\theta} \leq 25M\varepsilon'(\log(2/\varepsilon'))^{1/\theta},
\end{equation}
where the last inequality follows from the fact that $g_h(\varepsilon,\delta)\leq 3.1\varepsilon'$ and $g_h(2\varepsilon',\delta)<1/2$.

In $T_{2a}$, since conditional on $Z_i \in \hat{I}$, each $Z_i$ is a bounded random variable with $|Z_i-\mathbb{E}(Z_i)|\leq \text{length}(\hat{I}) = 2M(\log(2/\varepsilon'))^{1/\theta}$ and they are independent of each other, we can again use the Bernstein inequality. Let $\sigma^2$ be the variance of $F_0$. By Lemma \ref{lem:variance_shift}, for any event $E$ that occurs with probability at least $P(E)$, we have
\[
\mathbb{E}_{Z\sim F_0}\left[\left(Z-\mathbb{E}[Z|Z\in E]\right)^2|Z\in E\right]\leq \frac{\sigma^2}{P(E)},
\]
thus we can obtain an upper bound for the conditional variance of $Z_i$. Denote $F_0(\hat{I})$ to be the probability that $Z_i$ is distributed according to $F_0$ and lies in $\hat{I}$. Then, we have with probability at least $1-\delta$

\begin{equation}\label{t2ab}
    T_{2a} \leq \sqrt{\frac{2\log(1/\delta)}{|\hat{I}_{F_0}|}\left(\frac{\sigma^2}{F_0(\hat{I})}\right)} + 4M\left(\log \frac{2}{\varepsilon'}\right)^{1/\theta}\frac{\log(1/\delta)}{3|\hat{I}_{F_0}|},
\end{equation}
by the Bernstein inequality. 

For $T_{2b}$, since $F_0$ is sub-Weibull, \Cref{thm:subWeib} implies that there exists an absolute constant $K$ such that $[\mathbb{E}_{Z\sim F_0}(Z^v)]^{1/v}\leq K v^{1/\theta}$ for any $v\geq 1$. Thus, 
\begin{align*}
   \mathbb{E}_{Z\sim F_0}[Z|Z \not\in \hat{I}] &= \frac{\mathbb{E}_{Z\sim F_0}[Z \mathbbm{1}_{Z \not\in \hat{I}}]}{F_0(\hat{I}^c)} \leq \left(\frac{\mathbb{E}_{Z\sim F_0}[Z^v]}{F_0(\hat{I}^c)}\right)^{1/v}= Kv^{1/\theta} F_0(\hat{I}^c)^{-1/v}.
\end{align*}
where $F_0(\hat{I}^c)$ is the probability that $Z$ is distributed according to $F_0$ but does not lie in $\hat{I}$ and we use Hölder's inequality in the second line. Combining with fact that (as $\mu=0$)
\[
\left|\mathbb{E}\left[Z|Z\in \hat{I}\right]\right|F_0(\hat{I}) = F_0(\hat{I}^c)\left|\mathbb{E}\left[Z|Z\not\in \hat{I}\right]\right|,
\]
and assuming $F_0(\hat{I}) \geq 1/2$, we have 
\begin{equation*}
    T_{2b} \leq 2Kv^{1/\theta} F_0(\hat{I}^c)^{1-1/v}.
\end{equation*}
We choose $v=\log(1/\varepsilon')/\log\log(1/\varepsilon')$. Since $\varepsilon'\leq 1/4$ by assumption \eqref{A1_moment}, we have $\frac{\log(1/\varepsilon')}{\log \log 4}\geq v\geq 1$. Combining this with \eqref{empia}, the upper bound on $T_{2b}$ then becomes
\begin{equation}\label{t2bb}
    T_{2b} \leq 2Kv^{1/\theta} (52\varepsilon')^{1-1/v}\leq \frac{104K}{\log \log 4}\varepsilon' \left(\log \frac{1}{\varepsilon'}\right)^{1+1/\theta},
\end{equation}
and holds with probability at least $1-\delta$.

Combining \eqref{t1b}, \eqref{t2ab} and \eqref{t2bb}, we get with probability at least $1-3\delta$
\begin{align}
    & |\mathrm{RUME-\mu}| \leq 25M\varepsilon'\log^{1/\theta}\left(\frac{2}{\varepsilon'}\right)+\sqrt{\frac{2\log(1/\delta)}{|\hat{I}_{F_0}|}\frac{\sigma^2}{F_0(\hat{I})}} \nonumber \\
    & \hspace{2cm} + 4M\log^{1/\theta}\left(\frac{2}{\varepsilon'}\right)\frac{\log(1/\delta)}{3|\hat{I}_{F_0}|}+\frac{104K}{\log \log 4}\varepsilon' \log^{1+1/\theta} \left(\frac{1}{\varepsilon'}\right). \label{t1t2b}
\end{align}

Equations $|\hat{I}_{F_0}|\geq h/8$ and \eqref{empia} still hold as they do not depend on $F_0$. Combining $|\hat{I}_{F_0}|\geq h/8$,  (\ref{empia}) and (\ref{t1t2b}), and using the fact that $\log(2/\delta)\leq 2\log(1/\delta)\; \forall \delta\leq 1/4$, we have that with probability at least $1-5\delta$,
\begin{align*}
    |\mathrm{RUME-\mu}|
    \leq C_3 \varepsilon'\left(\log \frac{1}{\varepsilon'}\right)^{1+1/\theta}+C_4\sqrt{\frac{\log(1/\delta)}{h}}.
\end{align*}
for some absolute constants $C_3,C_4>0$. Re-parametrise $5\delta$ as $\delta/2$ and update constants $C_3,C_4$ to get desired bound.
\end{proof}

The minimum sample size required for RUME estimation is presented in \Cref{prop2}.
\begin{lemma}\label{prop2} 
Consider the setting of Theorem \ref{thm:rume_cpt}. Let $\alpha>0$, $\delta_t=8\alpha/(3t^3-3t)$ and $\varepsilon'_{s,t}=\max\left(\varepsilon,\frac{1}{s}\log(1/\delta_t)\right)$ for all $t\in\mathbb{N}$ and $t\geq 2$. If we set the following detection thresholds for the RUME estimator
\begin{align*}
    h_t=\begin{cases}
        \left\lceil20\log(1/\delta_t)\right\rceil  &\text{if $\varepsilon < 0.1$,}\\
        \left\lceil\frac{2}{0.5-\sqrt{2\varepsilon(1-2\varepsilon)}}\log(1/\delta_t)\right\rceil  &\text{if $0.1\leq\varepsilon<0.25$,} 
    \end{cases}
\end{align*}
then $\forall t\geq2 \;\forall s\in[h_t, \lfloor t/2\rfloor]$, we have $\delta_t \leq \min(1/s,1/4)$ and
\begin{equation} \label{A1a}
        2\varepsilon'_{s,t}+2\sqrt{\varepsilon'_{s,t}\frac{2\log(1/\delta_t)}{s}}+\frac{2\log(1/\delta_t)}{s} < \frac{1}{2}. 
    \end{equation}
\end{lemma}

\begin{proof}
(1) Firstly, to show $\delta_t\leq \min(1/s,1/2)$, we note that
\begin{equation*}
    \delta_t=\frac{8\alpha}{3t(t^2-1)}\leq\frac{8}{9t}\leq \min\left(\frac{2}{t},\frac{1}{2}\right)\leq\min\left(\frac{1}{s},\frac{1}{2}\right),
\end{equation*}
where we used the fact that $t\geq 2$ and $\alpha<1$ in the first and second inequality and the fact that $s\leq \lfloor t/2 \rfloor$ in the third inequality.

Define 
$$h(\varepsilon)=\frac{2}{0.5-\sqrt{2\varepsilon(1-2\varepsilon)}}.$$
We first note that \eqref{A1a} can be simplified by completing the square to give
    \begin{gather*}
        \quad s> 2\left(\frac{\sqrt{\log(1/\delta_t)}}{\sqrt{1/2-\varepsilon'_{s,t}}-\sqrt{\varepsilon'_{s,t}}}\right)^2 = h(\varepsilon'_{s,t})\log(1/\delta_t).
    \end{gather*}

(2) \textbf{Case 1: $\varepsilon<0.1$.} Since $s\geq h_t$, we have
\begin{equation*}
    \frac{2}{s}\log(1/\delta_t)\leq\frac{2}{h_t}\log(1/\delta_t)\leq \min(0.1,(\kappa/C_\lambda)^{2})\leq 0.1.
\end{equation*}
Thus, we have by definition of $\varepsilon'_{s,t}$ that $ \varepsilon'_{s,t}\leq\max(\varepsilon, 0.1)=0.1$.
Since $h(\varepsilon'_{s,t})\leq20,\; \forall\varepsilon'_{s,t}\leq0.1$ by the fact that $h$ is an increasing function in this range, the choice of $h_t$ satisfies \eqref{A1a}.

\textbf{Case 2: $\varepsilon>0.1$.} Since $s\geq h_t$, we have
\begin{equation*}
    \frac{2}{s}\log(1/\delta_t)\leq\frac{2}{h_t}\log(1/\delta_t)\leq 0.5-\sqrt{2\varepsilon(1-\varepsilon)}< 0.1.
\end{equation*}
Thus by definition of $\varepsilon'_{s,t}$, we have that $ \varepsilon'_{s,t}=\max(\varepsilon, \frac{1}{s}\log(1/\delta_t))=\varepsilon$. Meanwhile, by definition of $h_t$, 
\begin{equation*}
    s\geq h_t>h(\varepsilon)\log(1/\delta_t)=h(\varepsilon'_{s,t})\log(1/\delta_t).
\end{equation*}
Thus the choice of $h_t$ satisfies \eqref{A1a}. 
\end{proof}

\subsection{Concentration inequalities for the univariate setting}
\begin{lemma}[Multiplicative Chernoff bound, e.g. Corollary 4.9 in \cite{Mitzenmacher_Upfal_2005}]\label{lem:multiCher}
Let $p\in(0,1)$ and $n\in\mathbb{N}$. Given $\{W_i\}_{i=1}^n\overset{iid}{\sim} \mathrm{Bern}(p)$, we have 
\begin{equation*}
    \Pb\left[\sum_{i=1}^n W_i\geq (1+\delta)np\right]\leq \exp(np(\delta-(1+\delta)\log(1+\delta))).
\end{equation*}
\end{lemma}

\begin{proof}
For any $t>0$, the moment generating function of $W_i$ satisfies
\[
\E[e^{tW_i}]
= (1-p) + p e^t.
\]
Using the inequality $1+x \le e^x$, we obtain
\[
(1-p) + p e^t
= 1 + p(e^t-1)
\le \exp\!\big(p(e^t-1)\big).
\]
By independence,
\[
\E\!\left[\exp\!\left(t S_n\right)\right]
= \prod_{i=1}^n \E[e^{tW_i}]
\le
\exp\!\big(np(e^t-1)\big).
\]

Thus for any $t>0$,
\begin{align*}
    \Pb\!\left(S_n \ge (1+\delta)np\right)
&\le e^{-t(1+\delta)np}
\E\!\left[e^{t S_n}\right]\\
&\leq \exp\!\Big(np(e^t-1 - t(1+\delta))\Big)\\
&\leq\exp\!\Big(-np\big[(1+\delta)\log(1+\delta)-\delta\big]\Big),
\end{align*}
where the first inequality follows from Chernoff bound, the second inequality follows from the moment generating function above, and the last inequality follows from an optimization over $t>0$, giving the choice $t=\log(1+\delta)$.
\end{proof}

\begin{lemma}[Theorem 3 in \cite{Rosenthal_1970}]\label{lemma:rosenthal}
    Let $v\in(2,\infty)$ and $X_1,\ldots,X_n$ be mean zero independent random variables in $L^v$. Then,
    \begin{equation*}
        \left(\E\left|\sum_{i=1}^n X_i\right|^v\right)^{1/v}\leq2^{v/4+1/2}\sqrt{v}\max\left[\left(\sum_{i=1}^n\E|X_i|^v\right)^{1/v},\left(\sum_{i=1}^n\E|X_i|^2\right)^{1/2}\right].
    \end{equation*}
\end{lemma}

\begin{lemma}[Conditional pth moment bound, \citealp{prasad_2019_unified}]\label{lem:variance_shift} Suppose that $Y$ is sampled from a distribution $P^*$ with mean $\mu$ and bounded absolute $p^{th}$ central moment $\phi^{p}$. Then, for any event $A$ which occurs with probability at least $1 - \delta$, the variance of the conditional distribution satisfies
\[ \E[(y - \E[y|A])^2 | A] \leq \left(\frac{\phi^p}{1-\delta}\right)^{2/p} \].
\end{lemma}
\begin{proof}
Let $\mu_A = E[y|A]$, $d = \mu_A - \mu$. Observe the following,
\begin{align*} 
\E[(y - \mu_A)^2 | A ] = \E[(y - \mu - d)^2 | A] &= \E[ ({(y - \mu)^2 - 2d(y - \mu) + d^2) }| A ] \\
& = \E[(y - \mu)^2 | A] - d^2 \\
& \leq \E[(y - \mu)^2 | A]  \\
& = \frac{| \E((y - \mu)^2 \mathbbm{1}(A)) |}{P(A)}\\ 
&\leq \frac{ \E[|y- \mu |^{p}]^{\frac{2}{p}} (\E[\mathbbm{1}(A)^{q}]^{1/q})}{{P(A)}}\\
&\leq \left(\frac{\phi^p}{1-\delta}\right)^{2/p},
\end{align*}
where $p,q > 1$ are such that $2/p + 1/q = 1$.
\end{proof}

\subsection{Concentration inequalities for the multivariate setting}
Throughout this subsection, we work under the following setup, which is also adopted in the proofs in \Cref{sect:hdtestproof}. Let $X_1,\ldots,X_n$ be independent random variables in $\mathbb{R}^p$ drawn from the Huber contamination model
\[
    X_i \sim D_i := (1-\varepsilon)F + \varepsilon H_i,
    \qquad i=1,\ldots,n,
\]
where $F\in\mathcal{D}$ is the inlier distribution and $H_1,\ldots,H_n$ are arbitrary contamination distributions. Equivalently, we may represent the Huber contamination model by introducing independent contamination indicators
$$d_i \overset{\mathrm{i.i.d.}}{\sim} \operatorname{Bern}(\varepsilon),
\qquad i=1,\ldots,n.$$
Then $d_i=0$ indicates that the observation is drawn from the inlier distribution $F$, while $d_i=1$ indicates that it is drawn from the contamination distribution $H_i$. Thus,
$$(X_i\mid d_i=0)\sim F,
\qquad
(X_i\mid d_i=1)\sim H_i.$$ 
Recall that we define the set of inliers that are close to the mean as 
$$\mathfrak{G}=\{i\in[n] : d_i=0, \norm{X_i-\mu}_2\le R \}.$$ 
We therefore introduce the following random variables to aid the exposition. For each $i\in[n]$, let
$$\zeta_i=\mathbbm{1}\left\{ d_i=0,\ \norm{X_i-\mu}_2 \leq R \right\}.$$
We also define
$$\gamma=\mathbb{P}_{X\sim F}\left( \norm{X-\mu}_2 > R \right),\qquad \Sigma=\mathbb{E}_{X\sim F}\left[(X-\mu)(X-\mu)^\top\mid\norm{X-\mu}_2 \leq R\right].$$

\begin{lemma}[Reduction to the inlier distribution]
\label{lem:conditioning_reduction}
Assume that $\mathcal{D}=\mathcal{G}^p_{\theta,M}$, where we have $\gamma\leq 1/n$. For any measurable function
$\Phi:\mathbb{R}^{m\times p}\to\mathbb{R}$ and any $t>0$,
\[
\Pb_{Y_{1:n}\sim D}
\!\left(
|\Phi(Y_i\,:\,i\in \mathfrak{G})|\ge t
\;\middle|\;
\mathfrak{G}
\right)
\le
4\,
\Pb_{Y_{i_{1:m}}\sim F}
\!\left(
|\Phi(Y_{i_1},\ldots,Y_{i_m})|\ge t
\right),
\]
where $\mathfrak{G}=\{i_1,\ldots,i_m\}$.
\end{lemma}

\begin{proof}
Conditioning on the realization $\mathfrak{G}=\{i_1,\ldots,i_m\}$, by the definition
of $\mathfrak{G}$ we have
\[
d_{i_j}=0
\quad\text{and}\quad
\norm{Y_{i_j}-\mu}_2\le R
\qquad\forall j\in[m].
\]
Therefore,
\begin{align*}
\Pb_{Y_{1:n}\sim D}\left(|\Phi(Y_i\,:\,i\in \mathfrak{G})|\ge t\mid\mathfrak{G}\right)&=
\Pb_{Y_{i_{1:m}}\sim F}\left(
|\Phi(Y_{i_1},\ldots,Y_{i_m})|\ge t\mid \norm{Y_{i_j}-\mu}_2\le R,\ \forall j\in[m] \right)\\
&\le\frac{\Pb_{F}(|\Phi(Y_{i_1},\ldots,Y_{i_m})|\ge t)}{(1-\gamma)^m}\leq 4\Pb_{F}(|\Phi(Y_{i_1},\ldots,Y_{i_m})|\ge t),
\end{align*}
where the first inequality follows from the definition of conditional probability, and the second inequality follows from 
$$(1-\gamma)^m\ge(1-1/n)^n \ge \frac{1}{4},$$
since $\gamma\le 1/n$.
\end{proof}

\begin{lemma}\label{lem:normbound}
With probability $1-\delta$,
    \begin{align*}
        \norm{\sum_{i=1}^n \zeta_i(X_i-\mu)}_2\leq O(\sqrt{np}\log(1/\delta)).
    \end{align*}
\end{lemma}
\begin{proof}
We first define $\mu^*=\E_{Y\sim F}[Y|\norm{Y-\mu}_2<R]$. Then, we have
    \begin{equation}\label{proofeqn:normX2} 
    \norm{\sum_{i=1}^n \zeta_i (Y_i-\mu)}_2\leq\underbrace{\norm{\sum_{i=1}^n [\zeta_i (Y_i-\mu)-(1-\varepsilon)(1-\gamma)(\mu^*-\mu)]}_2}_{B_1}+\underbrace{n\norm{\mu^*-\mu}_2}_{B_2}.
    \end{equation}
    We will bound $B_1$ in \eqref{proofeqn:normX2} by \Cref{thm:vector_berstein}. Define the random variable $$V_i=\zeta_i (Y_i-\mu)-(1-\varepsilon)(1-\gamma)(\mu^*-\mu).$$ Then $\E[V_i]=0$,  $\norm{V_i}_2\leq 2R$ and
    \begin{align*}
        \E\norm{V}_2^2&=\E[\zeta_i\norm{Y_i-\mu}_2^2]-(1-\varepsilon)^2(1-\gamma)^2\norm{\mu^*-\mu}^2_2\leq p,
    \end{align*}
    and 
    \begin{align*}
  \sigma^2&=\norm{\E[VV^\top]}_\op=\norm{\E[\zeta(Y-\mu)(Y-\mu)^\top]-(1-\varepsilon)^2(1-\gamma)^2 \mu^*\mu^{*\top}}_\op\\
        &=\norm{(1-\varepsilon)\E_{Y\sim F}[(Y-\mu)(Y-\mu)^\top \mathbbm{1}(\norm{Y-\mu}_2<R)]-(1-\varepsilon)^2(1-\gamma)^2 (\mu^*-\mu)(\mu^{*}-\mu)^\top}_\op\\
        &\leq \norm{\E_{Y\sim F}[(Y-\mu)(Y-\mu)^\top]}_\op=1.
    \end{align*}
    Therefore, by \Cref{thm:vector_berstein}, we have 
    \begin{align*}
    \mathbb{P}(B_1 \ge \sqrt{np} + t)\leq \mathbb{P}(B_1 \ge \sqrt{n\E\norm{V}_2^2} + t) \le \exp \left( - \frac{t^2/2}{n + 4R\sqrt{pn} + 2tR/3} \right).
    \end{align*}
    In other words, with probability at least $1-\delta$, we have
    \begin{align*}
        \norm{\sum_{i=1}^n [\zeta_i (Y_i-\mu)-(1-\varepsilon)(1-\gamma)\mu^*]}_2&\leq \sqrt{np}+\sqrt{(2n+8R\sqrt{np})\log(1/\delta)}+\frac{4R}{3}\log(1/\delta)\\
        &\leq O(\sqrt{np}\log(1/\delta)).
    \end{align*}
    
    Meanwhile, an upper bound for $B_2$ follows directly from \Cref{lem:approx_mean}. 
    $$ n\norm{\mu^*-\mu}_2\leq 2n\phi\gamma^{3/4}.$$
    Under the assumption $F\in \mathcal{P}^p_{4,\phi}$, we can bound this further by
    $$n\norm{\mu^*-\mu}_2\leq 2\phi n\left(\frac{p\psi^4}{(R^2-p)^2}\right)^{3/4}\leq 2n^{1/4}p^{3/4}.$$
    
    Adding up all the terms in \eqref{proofeqn:normX2} gives
    \begin{equation} 
    \norm{\sum_{i=1}^n \zeta_i (Y_i-\mu)}_2\leq  O(\sqrt{np}\log(1/\delta)+n^{1/4}p^{3/4})\leq O(\sqrt{np}\log(1/\delta))
    \end{equation}
    with probability $1-\delta$. The same argument holds for $F\in\mathcal{G}_{\theta,M}^p$ as $\mathcal{G}_{\theta,M}
    \subset \mathcal{P}_{4,\phi}$ for some $\phi>0$.
\end{proof}

\begin{lemma}\label{lem:normbound_subW}
Let $a\in\R^n$ be a fixed vector with $a_i\in[0,1]$ for all $i\in[n]$. Suppose $\{Z_i\}_{i=1}^n$ are i.i.d. random vectors in $\R^p$ where each coordinate of $Z_1$ is independent with  $\norm{Z_{11}}_{\psi_\theta}\leq M$, zero mean and unit variance. Then with probability $1-\delta$,
    $$\norm{\sum_{i \leq n} a_iZ_i}^2_2\leq \norm{a}_2^2p+2C_\theta M^2(\norm{a}_2^2\sqrt{p\log(1/\delta)}+\norm{a}_2^2\log(1/\delta)+\norm{a}_\infty\log^{2/\theta}(1/\delta)),$$
    where $C_\theta>0$ is a constant depending only on $\theta$. Furthermore, treating $\theta$ as a constant, with probability $1-\delta$, we have
    \begin{align*}
        \norm{\sum_{i \leq n} a_iZ_i}_2\leq O(\norm{a}_2\sqrt{p}+\sqrt{\norm{a}_2\log(1/\delta)}+\norm{a}_\infty^{1/2}\log^{1/\theta}(1/\delta)).
    \end{align*}
\end{lemma}
\begin{proof}
    First we note that
  $$\norm{\sum_{i \leq n} a_iZ_i}_2^2=\sum_{j\leq p} \bigg(\sum_{i\leq n}  a_{i}Z_{ij}\bigg)^2=\sum_{j\leq p} \sum_{i\leq n} \sum_{k\leq n}  a_{i}a_{k}Z_{ij}Z_{kj}\overset{d}{=}\sum_{s=1}^{np}\sum_{t=1}^{np} B_{st}Z_{s}Z_t,$$
  where $\{Z_s\}_{i=1}^{np}$ are independent and identically distributed as $Z_{11}$ and $B\in\R^{np\times np}$ is the following block diagonal matrix 
   $$B=\begin{pmatrix}
        B^{\text{block}} & 0 & \cdots & 0\\
        0 & B^{\text{block}} & \cdots & 0\\
        \vdots& \vdots & \ddots & \vdots\\
        0 & 0 & \cdots & B^{\text{block}}
    \end{pmatrix}$$
    and $B_{\text{block}}=aa^\top\in\R^{n\times n}$ . By evaluating the matrix norms of $B$, namely
    \begin{align*}
        \norm{B}_2&=\norm{a}_2^2, \quad\norm{B}_{\mathrm{max}}=\norm{a}_\infty,\\
        \norm{B}_F^2&=\sum_{i=1}^p \sum_{j,k=1}^n a_j^2a_k^2=p\norm{a}_2^4,\\
         \norm{B}_{2\rightarrow\infty}&=\max_{j\in[n]}\sqrt{\sum_{k=1}^n(B^{\text{block}}_{jk})^2}=\max_{j\in[n]}\sqrt{\sum_{k=1}^n(a_ja_k)^2}=\norm{a}_2\norm{a}_\infty,
    \end{align*}
    and consequently by \Cref{prop:quadraticweibulltail}, with probability at least $1-\delta$, we have
    \begin{align*}
        & \Biggl|\norm{\sum_{i \leq n} a_iZ_i}_2^2-p\norm{a}_2^2\Biggr| \\
        \leq & C_\theta M^2(\norm{a}_2^2\sqrt{p\log(1/\delta)}+\norm{a}_2^2\log(1/\delta)+\norm{a}_2\norm{a}_\infty\log^{\frac{2\theta}{2+\theta}}(1/\delta)+\norm{a}_\infty\log^{2/\theta}(1/\delta))\\
        \leq & 2C_\theta M^2(\norm{a}_2^2\sqrt{p\log(1/\delta)}+\norm{a}_2^2\log(1/\delta)+\norm{a}_\infty\log^{2/\theta}(1/\delta)).\\
    \end{align*}
    Thus we have 
    \begin{align*}
        \norm{\sum_{i \leq n} a_iZ_i}_2&\leq \norm{a}_2\sqrt{p}+\sqrt{2C_\theta} M((\norm{a}_2^4p\log(1/\delta))^{1/4}+\norm{a}_2\sqrt{\log(1/\delta)}+\norm{a}_\infty^{1/2}\log^{1/\theta}(1/\delta))\\
        &\leq O(\norm{a}_2\sqrt{p}+\norm{a}_2\sqrt{\log(1/\delta)}+\norm{a}_\infty^{1/2}\log^{1/\theta}(1/\delta)),
    \end{align*}
    where the second inequality follows from using AM-GM inequality on the term $(\norm{a}_2^4p\log(1/\delta))^{1/4}$.
\end{proof}

\begin{lemma}\label{lem:crossterm}
Suppose $\{Z_i\}_{i=1}^m$ are i.i.d. random vectors in $\R^p$ where each coordinate of $Z_1$ is independent with  $\norm{Z_{11}}_{\psi_\theta}\leq M$, zero mean and unit variance. For any subset $\mathcal{M}\subset [m]$ of size at most $\tau\leq m$, 
    $$\left|\iprod{\sum_{i\in \mathcal{M}}\sqrt{w_i}X_i, \sum_{i=1}^m\sqrt{w_i}X_i}\right|= p\norm{w_\mathcal{M}}_1 \pm O(\tau\sqrt{mp}\log^{2/\theta}(2m/\delta)),$$
with probability $1-\delta$.
\end{lemma}

\begin{proof}
    For any $i = 1, \ldots, m$, note that
    \begin{align*}
    \iprod{\sqrt{w_i}X_i, \Sum (w,\mathfrak{G})} &= w_i\norm{X_i}^2 + \iprod{\sqrt{w_i}X_i, \Sum (w,\mathfrak{G}\setminus \{i\})} \; .        
    \end{align*}
    
    To bound the first term on the LHS, we first note that $\norm{X_{ij}}_{\psi_\theta}=M$ implies that $\norm{X_{ij}^2}_{\psi_{\theta/2}}=M^2$ by \cref{cor:multiplied_norm}. Thus, for a given $i\in [m]$, by \cref{lemma:weib_concentration}, with probability at least $1-\delta$, 
\begin{equation}\label{proof_eqn:chi2X}
    \left|\|\sqrt{w_i}X_i\|^2 - w_ip\right| \leq O\left(\sqrt{w_ip\log(1/\delta)} + \log^{2/\theta}(1/\delta)\right)\leq O(\sqrt{w_ip}\log^{2/\theta}(1/\delta)),
\end{equation}
Therefore, with probability at least $1-\delta$, we have for all $i \in [m]$ that
\begin{equation*}
    \left|\|\sqrt{w_i}X_i\|^2 - w_ip\right|\leq O(\sqrt{w_ip}\log^{2/\theta}(n/\delta)).
\end{equation*}
Assuming this holds for all $i\in [m]$, then for any subset $\mathcal{M} \subset [m]$ such that $|\mathcal{M}| \le \tau$, we have
$$\sum_{i \in \mathcal{M}} w_i\|X_i\|^2 = p\norm{w_\mathcal{M}}_1 \pm O(\sqrt{p}|\mathcal{M}|  \log^{2/\theta}(m/\delta))= p \norm{w_\mathcal{M}}_1 \pm O(\tau \sqrt{p} \log^{2/\theta} (m/\delta)).$$

To bound the second term on the RHS, we note that for any $i,k\in[m]$ with $i\neq k$ and $j\in[p]$, we have $\norm{X_{ij}X_{kj}}_{\psi_{\theta/2}}\leq \norm{X_{ij}}^2_{\psi_{\theta}}=M^2$ by \Cref{cor:multiplied_norm}. Using \Cref{lemma:weib_concentration}, for fixed $i$, with probability $1-\delta$,
    $$\left|\sum_{k=1, k\neq i}^m\sum_{j=1}^p\sqrt{w_iw_k}X_{ij}X_{kj}\right|\leq O(\sqrt{pw_i\norm{w}_1\log(1/\delta)}+\log^{2/\theta}(1/\delta))\leq O(\sqrt{pw_i\norm{w}_1}\log^{2/\theta}(1/\delta)).$$
    Thus, with probability $1-\delta$, we have for all $i\in[m]$ that
    \begin{align*}
        \left|\iprod{\sqrt{w_i}X_i, \Sum (w,[m]\setminus\{i\})}\right|&=\left|\sum_{k=1, k\neq i}^m \sum_{j=1}^p\sqrt{w_iw_k}X_{ij}X_{kj}\right|\\
        &\leq O(\sqrt{pw_i\norm{w}_1}\log^{2/\theta}(m/\delta))\\
        &\leq O(\sqrt{mp}\log^{2/\theta}(m/\delta)).
    \end{align*}
    Condition on the two events above holding. Then, for any fixed $\mathcal{M}$ satisfying $|\mathcal{M}| \leq \varepsilon n$, we have with probability $1-2\delta$ that 
    \begin{align}
        \left\langle \Sum (w,\mathcal{M}), \Sum(w,\mathfrak{G}) \right\rangle &= \sum_{i \in \mathcal{M}} \iprod{\sqrt{w_i}X_i, \Sum (w,\mathfrak{G})} \nonumber\\
        &= \sum_{i \in \mathcal{M}} w_i\norm{X_i}_2^2 + \sum_{i \in \mathcal{M}} \iprod{\sqrt{w_i}X_i, \Sum (w,\mathfrak{G} \setminus \{i\})} \nonumber\\
        &= p\norm{w_\mathcal{M}}_1 \pm O(\tau\sqrt{mp}\log^{2/\theta}(m/\delta)), \label{eq:a1-bound}
    \end{align}
    as claimed.
\end{proof}

\begin{lemma}[Concentration of covariance matrix]\label{lem:cov_matrix_conc}
Let $\delta\in(0,1)$. Given the setup, 
\begin{enumerate}[label=(\textbf{\alph*})]
    \item if we assume further that $F\in\mathcal{P}_{v,\phi}^p$ for $v\geq 4$, then by picking
        $$R^2=p+O\left(\sqrt{\max(n,p)}\right),$$
    with probability at least $1-\delta$, we have
    $$\norm{\sum_{i\in \mathfrak{G}}(X_i-\mu)(X_i-\mu)^\top-nI}_\op\leq O((\sqrt{np}+p)\log(2p/\delta))+\varepsilon n.$$
    \item if we assume further that $F\in\mathcal{G}_{\theta,M}^p$, then by picking
$$R^2=p+M^2 \cdot O\left(\log^{2/\theta}(n)+\sqrt{p\log(n)}\right),$$ 
    with probability at least $1-\delta$, we have
    \begin{align*}
        \norm{\sum_{i\in \mathfrak{G}}(X_i-\mu)(X_i-\mu)^\top-nI}_\op\leq O((\sqrt{np}+p)\log(2p/\delta))+\varepsilon n.
    \end{align*}    
\end{enumerate}
\end{lemma}
\begin{proof}
To simplify notation, write $\bar{X}_i=X_i-\mu$. We can write
    \begin{align}
        \norm{\sum_{i\in \mathfrak{G}}\bar{X}_i\bar{X}_i^\top-nI}_\op\leq&\norm{\sum_{i=1}^n \zeta_i \bar{X}_i\bar{X}_i^\top-n(1-\varepsilon)(1-\gamma)\Sigma}_\op+n(1-\varepsilon)(1-\gamma)\norm{\Sigma-I}_\op+(\varepsilon+\gamma) n. \label{eqn:emp_decompose}
    \end{align}
    Since for all $i$, we have $\norm{\zeta_i \bar{X}_i\bar{X}_i^\top-\E[\zeta_i \bar{X}_i\bar{X}_i^\top]}_\op \leq 2R^2$ and
    \begin{align*}
        \E[\zeta_i \bar{X}_i\bar{X}_i^\top]&=\E[\mathbbm{1}(\zeta_i=1) \bar{X}_i\bar{X}_i^\top]\\
        &=\Pb(d_i=0)\E_{X_i\sim F}[\mathbbm{1}(\norm{\bar{X}_i}_2\le R) \bar{X}_i\bar{X}_i^\top]\\
        &=\Pb(d_i=0)\Pb_{X_i\sim F}(\norm{\bar{X}_i}_2\le R)\E_{X_i\sim F}[ \bar{X}_i\bar{X}_i^\top|\norm{\bar{X}_i}_2\le R]=(1-\varepsilon)(1-\gamma)\Sigma,
    \end{align*} 
    we can apply matrix Bernstein inequality (\Cref{thm:matrix_bernstein}) to bound the first term on the RHS of \eqref{eqn:emp_decompose}. With probability $1-\delta$, we have
    \begin{align*}
        \norm{\sum_{i=1}^n \zeta_i \bar{X}_i\bar{X}_i^\top-n(1-\varepsilon)(1-\gamma)\Sigma}_\op\leq\sqrt{2\sigma^2\log(2p/\delta)}+\frac{4R^2}{3}\log(2p/\delta),
    \end{align*}
    and the variance parameter $\sigma^2$ satisfies
    \begin{align*}
        \sigma^2&=\norm{\sum_{i=1}^n \{\E[\zeta_i\bar{X}_i\bar{X}_i^\top\norm{\bar{X}_i}_2^2]-(1-\varepsilon)^2(1-\gamma)^2\Sigma^2\}}_\op\\
        &\leq\norm{\sum_{i=1}^n \E_{X_i\sim F}[\bar{X}_i\bar{X}_i^\top\norm{\bar{X}_i}_2^2]}_\op\\
        &\leq n(\E[X_{1i}^4]+p-1),
    \end{align*}
    where the first inequality follows from $(1-\varepsilon)\E_{X_i\sim F}[\bar{X}_i\bar{X}_i^\top\norm{\bar{X}_i}_2^2]\succeq\E[\zeta_i\bar{X}_i\bar{X}_i^\top\norm{\bar{X}_i}_2^2]-(1-\varepsilon)^2\gamma^2\Sigma^2\succeq0$ and the last inequality follows from evaluating the expectation. For $(i,j)\in\{1,\ldots,n\}$ such that $i\neq j$, we have
    \begin{align*}
        (\E_{\bar{X}_1\sim F}[\bar{X}_1\bar{X}_1^\top\norm{\bar{X}_1}_2^2)_{ij}&=\sum_{k=1}^p\E_{X_1\sim F}[\bar{X}_{1i}\bar{X}_{1j}\bar{X}_{1k}^2]=0,
    \end{align*}
    while for $i=j$, we have
    \begin{align*}
        (\E_{\bar{X}_1\sim F}[\bar{X}_1\bar{X}_1^\top\norm{\bar{X}_1}_2^2)_{ii}&=\sum_{k=1}^p\E_{\bar{X}_1\sim F}[\bar{X}_{1i}^2\bar{X}_{1k}^2]\\
        &=\E_{\bar{X}_1\sim F}[\bar{X}_{1i}^4]+\sum_{k\neq i}\E_{X_1\sim F}[\bar{X}_{1i}^2\bar{X}_{1k}^2]\\
        &\leq \E_{X_1\sim F}[\bar{X}_{1i}^4]+p-1=O(p).
    \end{align*}

    (a) Consider the finite $v$-th moment assumption for some $v\geq 4$, i.e. $\E[\bar{X}_{11}^v]\leq \phi^v$. The second term on the RHS of \eqref{eqn:emp_decompose} can be bounded by \Cref{lem:approx_cov}:
    \begin{align*}
        \norm{I-\Sigma}_\op&\leq O(\gamma^{1-2/v}).
    \end{align*}
    By Markov's inequality, we can bound $\gamma$ by
    $$\gamma=\Pb(\norm{\bar{X}_1}_2>R)\leq \frac{\E\left[\left(\norm{\bar{X}_1}^2_2-p\right)^{v/2}\right]}{(R^2-p)^{v/2}}\lesssim \frac{p^{v/4}}{(R^2-p)^{v/2}}.$$
    For the final inequality, since $\E_{X_1\sim F}[X_{1i}^2-1]=0$, we can use Rosenthal inequality (\Cref{lemma:rosenthal}) to bound the numerator.
    \begin{align*}
        & \E\left[\left(\norm{\bar{X}_1}^2_2-p\right)^{v/2}\right] \\
        =& \E\left[\left(\sum_{i=1}^p(X_{1i}^2-1)\right)^{v/2}\right]\leq \left(C_v\left(p^{2/v}\E[|X_{1i}^2-1|^{v/2}]^{2/v}+\sqrt{p(\E[X_{1i}^4]-1)}\right)\right)^{v/2}\\
        \lesssim & p+p^{v/4}\asymp p^{v/4},
    \end{align*}
    since we have that $\E[X_{1i}^4]=O(1)$ for $v\geq 4$ and
    \begin{align*}
        \E[|X_{1i}^2-1|^{v/2}]=2^{v/2-1}(1+\E[|X_{1i}|^v])\leq 2^{v/2-1}(1+\phi^v)=O(1).
    \end{align*}

    Therefore, using \eqref{eqn:emp_decompose}, we have 
    \begin{align*}
        \norm{M(\mathfrak{G})-nI}_\op\lesssim \sqrt{np\log(2p/\delta)}+R^2\log(2p/\delta)+\frac{np^{v/4}}{(R^2-p)^{v/2}}+\varepsilon n.
    \end{align*}
    with probability $1-\delta$. By choosing $R^2=p+O(\sqrt{\max(n,p)})$,
    with probability at least $1-\delta$, we have
    \begin{align*}
        \norm{M(\mathfrak{G})-nI}_\op&\leq O((\sqrt{np}+p)\log(2p/\delta))+\varepsilon n.
    \end{align*}
    
    (b) Consider the sub-Weibull assumption. The second term on the RHS of \eqref{eqn:emp_decompose} can be bounded by \Cref{lem:approx_cov} since sub-Weibull distributions have finite 4th moment.
    \begin{align*}
        \norm{I-\Sigma}_\op&\leq 2 \sqrt{\gamma(\psi^4-1)}. 
    \end{align*}
    By \Cref{prop:quadraticweibulltail}, we have that 
    \begin{align*}
       \gamma= \Pb_{X\sim F}(\norm{\bar{X}}_2>R)
        &\leq\Pb_{X\sim F}(|\norm{\bar{X}}_2^2-p|>R^2-p)\\
        &\leq 2 \exp\left(-C_{\theta}\min
        \left\{\frac{(R^2-p)^2}{pM^4}, \left(\frac{R^2-p}{M^2}\right)^{\theta/2}\right\}\right)\\
        &\leq 2 \exp \left[-C_{\theta} \left(\frac{R^2-p}{M^2}\right)^{\theta/2}\right] 
            +2 \exp\left[-C_{\theta}\frac{(R^2-p)^2}{pM^4}\right].
    \end{align*}
    Therefore, using \eqref{eqn:emp_decompose}, we have 
    \begin{align*}
        \norm{M(\mathfrak{G})-nI}_\op\leq& \sqrt{2np\log(2p/\delta)}+\frac{4R^2}{3}\log(2p/\delta)+{\sqrt{2}n\phi^2}\exp \left[-\frac{C_{\theta}}{2}\left(\frac{R^2-p}{M^2}\right)^{\theta/2}\right]\\
        &+\sqrt{2}n\exp\left[-C_{\theta}\frac{(R^2-p)^2}{pM^4}\right]+\varepsilon n,
    \end{align*}
    with probability $1-\delta$. By choosing 
    $$R^2=p+M^2 \left(C_\theta^{-2/\theta}\log^{2/\theta}(n)+C_\theta^{-1/2}\sqrt{p\log(n)}\right),$$ 
    then with probability at least $1-\delta$, we have
    \begin{align*}
        \norm{M(\mathfrak{G})-nI}\leq O((\sqrt{np}+p)\log(2p/\delta))+\varepsilon n.
    \end{align*}    
\end{proof}

\begin{lemma}\label{fact:small-subset-deviations}
    Let $\delta>0$ and $\lceil 10 un \rceil\leq |\mathfrak{G}|$. Consider a finite set $J\subset \{w:w\in \Gamma_n, \norm{w}\leq 10 un\}$. With probability at least $1-\delta$,  
    \begin{equation*}
        \sup_{w\in J}\norm{\sum_{i\in \mathfrak{G}}w_iX_iX_i^\top}_\op\lesssim \sqrt{unp\log(2p/\delta)+unp\log(|J|)}+R^2(\log(2p/\delta)+\log(|J|))+un.
    \end{equation*}
    In particular, if $|J|=\exp[O(un\log(1/u))]$, then with probability at least $1-\delta$, we have
    \begin{align*}
        \sup_{w\in J}\norm{\sum_{i\in \mathfrak{G}}w_iX_iX_i^\top}_\op&\lesssim \sqrt{unp\log(2p/\delta)+u^2n^2p\log(1/u)}+R^2(\log(2p/\delta)+un\log(1/u))+un.
    \end{align*}
\end{lemma}

\begin{proof}
Suppose $w\in \Gamma_n$ is a fixed vector satisfying $\norm{w}_1\leq 10 un$. Then
    \begin{align}\label{eqn:emp_decompose2}
        \norm{\sum_{i\in \mathfrak{G}}w_iX_iX_i^\top}_\op\leq\norm{\sum_{i=1}^n \zeta_i w_i X_iX_i^\top-\norm{w}_1(1-\varepsilon)(1-\gamma)\Sigma}_\op+\norm{w}_1(1-\varepsilon)\norm{(1-\gamma)\Sigma}_\op.
    \end{align}
    Since for all $i$, we have $\norm{\zeta_iw_i X_iX_i^\top-w_i\E[\zeta_i X_iX_i^\top]}_\op \leq 2w_iR^2\leq 2R^2$ and
    \begin{align*}
        \E[\zeta_i w_iX_iX_i^\top]&=w_i\E[\mathbbm{1}(\zeta_i=1) X_iX_i^\top]\\
        &=w_i\Pb(d_i=0)\E_{X_i\sim F}[\mathbbm{1}(\norm{X_i}_2\le R) X_iX_i^\top]\\
        &=w_i\Pb(d_i=0)\Pb_{X_i\sim F}(\norm{X_i}_2\le R)\E_{X_i\sim F}[ X_iX_i^\top|\norm{X_i}_2\le R]=w_i(1-\varepsilon)(1-\gamma)\Sigma,
    \end{align*} 
    we can apply matrix Bernstein inequality (\Cref{thm:matrix_bernstein}) to bound the first term on the RHS of \eqref{eqn:emp_decompose}. With probability $1-\delta$, we have
    \begin{align*}
        \norm{\sum_{i=1}^n \zeta_i X_iX_i^\top-\norm{w}_1(1-\varepsilon)(1-\gamma)\Sigma}_\op\leq\sqrt{2\sigma^2\log(2p/\delta)}+\frac{4R^2}{3}\log(2p/\delta),
    \end{align*}
    and the variance parameter $\sigma^2$ satisfies
    \begin{align*}
        \sigma^2&=\norm{\sum_{i=1}^n \left\{ \E[\zeta_iw_i^2X_iX_i^\top\norm{X_i}_2^2]-w_i^2(1-\varepsilon)^2(1-\gamma)^2\Sigma^2\right\}}_\op\\
        &\leq\norm{\sum_{i=1}^n \E[\zeta_iw_i^2X_iX_i^\top\norm{X_i}_2^2]}_\op\\
        &\leq \norm{w}_2^2 (\psi^4+p-1),
    \end{align*}
    where the first inequality follows from $$(1-\varepsilon)\E[X_iX_i^\top\norm{X_i}_2^2]\succeq\E[\zeta_iX_iX_i^\top\norm{X_i}_2^2]-(1-\varepsilon)^2(1-\gamma)^2\Sigma^2\succeq0,$$ and the last inequality follows from evaluating the expectation, in the similar way as the proof of \Cref{lem:cov_matrix_conc}.
    
    Combining the above results, the upper bound with probability $1-\delta$ in \eqref{eqn:emp_decompose2} for a fixed $w$ becomes
    \begin{align*}
        \norm{\sum_{i\in \mathfrak{G}}w_iX_iX_i^\top}_\op\lesssim \sqrt{unp\log(2p/\delta)}+R^2\log(2p/\delta)+un.
    \end{align*}
    Thus, by replacing $\delta$ by $\delta/|J|$, we have
    \begin{align*}
        \sup_{w\in J}\norm{\sum_{i\in \mathfrak{G}}w_iX_iX_i^\top}_\op\lesssim \sqrt{unp\log(2p/\delta)+unp\log(|J|)}+R^2(\log(2p/\delta)+\log(|J|))+un.
    \end{align*}
    with probability $1-\delta$.
\end{proof}

\begin{lemma}[Vector Bernstein, {\citealp[Corollary~8.45]{foucart2013mathematical}}]\label{thm:vector_berstein}
    Let $a$ be a fixed vector with entries $a_i\in[0,1], \forall i\in\{1,\ldots,n\}$. Let $\mathbf{Y}_1, \dots, \mathbf{Y}_n$ be independent copies of a random vector $\mathbf{Y}$ on $\mathbb{R}^p$ satisfying $\mathbb{E}\mathbf{Y} = 0$. Assume $\|\mathbf{Y}\|_2 \le R$ for some $K > 0$. Let
\begin{equation*}
Z = \left\| \sum_{\ell=1}^n a_\ell \mathbf{Y}_\ell \right\|_2, \quad \mathbb{E}Z^2 = n\mathbb{E}\|\mathbf{Y}\|_2^2, 
\end{equation*}
and
\[
\sigma^2 = \sup_{\|\mathbf{x}\|_2 \le 1} \mathbb{E} |\langle \mathbf{x}, \mathbf{Y} \rangle|^2=\norm{\E[\mathbf{Y} \mathbf{Y}^\top]}_\op.
\]
Then, for $t > 0$,
\begin{equation*}
\mathbb{P}(Z \ge \sqrt{\mathbb{E}Z^2} + t) \le \exp \left( - \frac{t^2/2}{n\sigma^2 + 2R\sqrt{\mathbb{E}Z^2} + tR/3} \right). 
\end{equation*}
\end{lemma}

\begin{lemma}[Matrix Bernstein inequality, {\citealp[Theorem~1.4]{tropp2012user}}]\label{thm:matrix_bernstein}
Consider a finite sequence $\{M_k\}_{k=1}^n$ of independent, random, self-adjoint matrices in $\mathbb{R}^{p\times p}$. Assume that each random matrix satisfies
\[
\mathbb{E}M_k = 0
\quad\text{and}\quad
\norm{M_k}_{\op} \le R
\quad \text{almost surely}.
\]
Then, for all $t \ge 0$,
\[
\mathbb{P}\!\left\{
\norm{\sum_{k=1}^n M_k}_\op \ge t
\right\}
\le
2p \exp\!\left(
\frac{-t^2/2}{\sigma^2 + Rt/3}
\right),
\]
where
\[
\sigma^2 = \left\| \sum_{k=1}^n \mathbb{E}(M_k^2) \right\|_\op.
\]
\end{lemma}

\begin{lemma}[General mean shift, \citealp{prasad_2019_unified}] \label{lem:approx_mean}
Suppose that $Z\in\mathbb{R}^p$ is sampled from a distribution $F\in\mathcal{P}_{2,\sigma}^p$ with mean $\mu$. Then, for any event $E$ which occurs with probability at least $1 - \delta \geq \frac{1}{2}$,
\[ \norm{\mu - \E[Z|E]}_2 \leq  2\sigma \delta^{1/2}. \]
Furthermore, if $F\in\mathcal{P}_{4,\psi}^p$ and $\sigma=1$, then with probability at least $1 - \delta \geq \frac{1}{2}$,
\[ \norm{\mu - \E[Z|E]}_2 \leq  2\psi \delta^{3/4}. \]
In general, if $F\in\mathcal{P}_{v,\phi}^p$ for some $v\geq 2$ and $\sigma=1$, then with probability at least $1 - \delta \geq \frac{1}{2}$,
\[ \norm{\mu - \E[Z|E]}_2 \leq  O( \delta^{1-1/v}). \]
\end{lemma}

\begin{proof}
For any event $E$, let $\mathbbm{1}(E)$ denote the indicator variable for $E$. 
\begin{align*}
    \norm{\E [Z|E ] - \mu}_2 &= \frac{\norm{\E((Z - \mu ) \mathbbm{1}(E))}_2}{P(E)} \\
    &=P(E)^{-1}\norm{\E[(Z-\mu)\mathbbm{1}(E^c)]} \\
    &=P(E)^{-1}\sup_{u\in\mathbb{R}^p:\norm{u}_2=1} |\E[u^\top(Z-\mu)\mathbbm{1}(E^c)]| 
    \\
    &\leq P(E)^{-1}\sup_{u\in\mathbb{R}^p:\norm{u}_2=1} \E[(u^\top(Z-\mu))^{v}]^{1/v} \E[\mathbbm{1}(E^c)^{v/(v-1)}]^{(v-1)/v} \\
    &\leq 2\sup_{u\in\mathbb{R}^p:\norm{u}_2=1} \E[(u^\top(Z-\mu))^{v}]^{1/v} \delta^{1-1/v}.
\end{align*}
where the second equality follows from $$\E[(Z-\mu)\mathbbm{1}(E)]+ \E[(Z-\mu)\mathbbm{1}(E^c)]= \E[Z-\mu]=0,$$ and the first inequality follows from Hölder's inequality. 

For $F\in\mathcal{P}_{2,\sigma}^p$, we have 
\begin{align*}
    \norm{\E [Z|E ] - \mu}_2 \leq 2\sup_{u\in\mathbb{R}^p:\norm{u}_2=1} \E[(u^\top(Z-\mu))^{2}]^{1/2} \delta^{1/2}\leq 2 \norm{\sigma^2I_p}_\op^{1/2}\delta^{1/2}.
\end{align*}

For $F\in\mathcal{P}_{4,\psi}^p$, denote $\bar{Z}=Z-\mu$. Then we have
\begin{align*}
    \E[(u^\top(Z-\mu))^{4}]&=\E(\sum_{k=1}^p u_k \bar{Z}_k)^4\\
    &=\sum_{k=1}^p u_k^4 \E[\bar{Z}_k^4]+3\sum_{m,k\in[p],m\neq k} u_k^2 u_m^2\E[\bar{Z}^2_k]^2\\
    &=\psi^4\sum_{k=1}^p u_k^4+3\left(\left(\sum_{m=1}^p u_k^2\right)^2- \sum_{k=1}^p u_k^4\right)\\
    &=3+(\psi^4-3)\norm{u}_4^4
\end{align*}
Since $\norm{u}_4\leq \norm{u}_2=1$,
\begin{align*}
    \max_{u\in\R^p:\norm{u}_2=1}\E[(u^\top(Z-\mu))^{4}]&=\psi^4,
\end{align*}
Concluding the above, we have
\begin{align*}
    \norm{\E [Z|E ] - \mu}_2 \leq 2\sup_{u\in\mathbb{R}^p:\norm{u}_2=1} \E[(u^\top(Z-\mu))^{4}]^{1/4} \delta^{3/4}\leq 2 \psi\delta^{3/4}.
\end{align*}

For general $v\geq 2$, we note that for any fixed $u\in\R^p$ with $\norm{u}_2=1$, we have $\E[u_i (Z_i-\mu_i)]=0$ for any $i$. Thus, we can apply Rosenthal inequality (\Cref{lemma:rosenthal}) to get
$$\E|\iprod{u, Z-\mu}|^v=\E\left[\left|\sum_{i=1}^p u_i(Z_i-\mu_i)\right|^v\right]\leq [C_v (\phi\norm{u}_v+\norm{u}_2)]^v\leq C_v'(\phi+1)^v,$$
where $C_v, C_v'$ are constants depending only on $v$ and the last inequality follows from $\norm{u}_v\leq \norm{u}_2=1$. Thus, we have
$$\sup_{u\in\mathbb{R}^p:\norm{u}_2=1} \E[|\iprod{u, Z-\mu}|^{v}]^{1/v}\leq [C_v'(\phi+1)^v]^{1/v}=O(1),$$
leading to the final result
$$\norm{\E [Z|E ] - \mu}_2\leq O(\delta^{1-1/v}).$$
\end{proof}

\begin{lemma}[General sample covariance matrix shift] \label{lem:approx_cov}
Let $v\geq 4$. Suppose that $Z\in\mathbb{R}^p$ is sampled from a distribution $F\in \mathcal{P}_{v,\phi}^p$ with mean $\mathbf{0}$ and variance $I$. Then, for any event $E$ which occurs with probability at least $1 - \delta \geq \frac{1}{2}$,
\[ \norm{I - \E[ZZ^\top|E]}_\op \leq  O(\delta^{1-2/v}). \]
In particular, if $v=4$, 
\[ \norm{I - \E[ZZ^\top|E]}_\op \leq  2\sqrt{(\psi^4-1)\delta}. \]
\end{lemma}

\begin{proof}
For any event $E$, let $\mathbbm{1}(E)$ denote the indicator variable for $E$. 
\begin{align}
    \norm{\E [ZZ^\top|E ] - I}_\op &= \frac{\norm{\E [(ZZ^\top-I)  \mathbbm{1}(E)]}_\op}{P(E)} \nonumber \\
    &=P(E)^{-1}\norm{\E [(ZZ^\top-I)  \mathbbm{1}(E^c)]}_\op \nonumber\\
    &=P(E)^{-1}\sup_{u\in\mathbb{R}^p:\norm{u}_2=1} |\E[ (\iprod{u, Z}^2-1)\mathbbm{1}(E^c)]| \nonumber
    \\
    &\leq P(E)^{-1}\sup_{u\in\mathbb{R}^p:\norm{u}_2=1} \E[|\iprod{u, Z}^2-1|^{v/2}]^{2/v} \E[\mathbbm{1}(E^c)^{v/(v-2)}]^{1-2/v} \nonumber\\
    &\leq 2\sup_{u\in\mathbb{R}^p:\norm{u}_2=1} \E[|\iprod{u, Z}^2-1|^{v/2}]^{2/v} {\delta}^{1-2/v} \label{eqn:cov_shift_v}.
\end{align}
where the second equality follows from $$\E [(ZZ^\top-I)  \mathbbm{1}(E)]+ \E [(ZZ^\top-I)  \mathbbm{1}(E^c)]= \E[ZZ^\top-I]=0,$$ and the first inequality follows from Hölder's inequality. For the case of $v=4$, we can evaluate \eqref{eqn:cov_shift_v} exactly using a similar argument as \Cref{lem:approx_mean}:
\begin{align*}
    \sup_{u\in\mathbb{R}^p:\norm{u}_2=1}\E[(\iprod{u, Z}^2-1)^{2}]&=\sup_{u\in\mathbb{R}^p:\norm{u}_2=1}\E[\iprod{u, Z}^4-2\iprod{u, Z}^2+1]\\
    &=\sup_{u\in\mathbb{R}^p:\norm{u}_2=1} 3+(\psi^4-3)\sum_{k=1}^p u_k^4-2+1\\
    &=\psi^4-1.
\end{align*}
allowing us to conclude that 
$$\norm{\E [ZZ^\top|E ] - I}_\op\leq 2\sqrt{(\psi^4-1)\delta}.$$

For general $v\geq 4$, we note that for any fixed $u\in\R^p$ with $\norm{u}_2=1$, we have 
$$|\iprod{u, Z}^2-1|^{v/2}\leq 2^{v/2-1}(1+|\iprod{u, Z}|^v).$$
Since $\E[u_i Z_i]=0$, we can apply Rosenthal inequality (\Cref{lemma:rosenthal}) to get
$$\E|\iprod{u, Z}|^v=\E\left[\left|\sum_{i=1}^p u_iZ_i\right|^v\right]\leq [C_v (\phi\norm{u}_v+\norm{u}_2)]^v\leq C_v'(\phi+1)^v,$$
where $C_v, C_v'$ are constants depending only on $v$ and the last inequality follows from $\norm{u}_v\leq \norm{u}_2=1$. Thus, we have
$$\sup_{u\in\mathbb{R}^p:\norm{u}_2=1} \E[|\iprod{u, Z}^2-1|^{v/2}]^{2/v}\leq [2^{v/2-1}(1+C_v'(\phi+1)^v)]^{2/v}=O(1),$$
leading to the final result
$$\norm{\E [ZZ^\top|E ] - I}_\op\leq O(\delta^{1-2/v}).$$
\end{proof}

\end{document}